\documentclass[12pt,a4paper,reqno]{amsart}

\usepackage[margin=2.8cm,footskip=1cm]{geometry}
\usepackage[utf8]{inputenc}
\usepackage[T1]{fontenc}
\usepackage[english]{babel}
\usepackage{amsmath}
\usepackage{amsfonts}
\usepackage{amssymb}
\usepackage{amsthm}
\usepackage{ucs}
\usepackage{graphicx}
\usepackage{subfig}
\usepackage{xcolor}
\usepackage{dsfont}
\usepackage{wasysym}
\usepackage{physics}
\usepackage{mathtools}
\usepackage{xifthen}
\usepackage{enumitem}
\usepackage{stmaryrd}
\usepackage{tikz,pgfplots}
\usepackage{float}
\usepackage{soul}
\usepackage{titlesec}
\usepackage{hyperref}
\definecolor{colorlinks}{RGB}{0, 24, 168}
\definecolor{colorcites}{RGB}{124, 10, 2}
\hypersetup{
	colorlinks=true,
	linkcolor=colorlinks,
	citecolor=colorcites,
	urlcolor=colorlinks,
	pdfborder={0 0 0}
}
\usepackage[font={small}]{caption}
\usepackage{algorithm2e}
\usepackage{constants}

\newconstantfamily{ImpCst}{
	symbol=\mathrm{c}
}

\newcommand{\lrangle}[1]{\langle #1 \rangle}

\newcommand{\given}{\, |\,}
\newcommand{\bgiven}{\, \big|\,}

\newcommand{\ind}{\mathds{1}}

\titleformat{\section}{\normalfont\large\scshape\bfseries}{\thesection.}{1em}{\centering}
\titleformat{\subsection}{\normalfont\large\bfseries}{\thesubsection.}{1em}{}
\titleformat{\subsubsection}{\normalfont\normalsize\bfseries}{\thesubsubsection.}{1em}{}
\titlespacing{\subsubsection}{0cm}{0.2cm}{0.1cm}

\makeatletter
\def\l@section{\@tocline{1}{0pt}{1pc}{3pc}{}}
\def\l@subsection{\@tocline{2}{0pt}{1pc}{3pc}{}}
\makeatother

\newcommand{\N}{\mathbb{N}}

\newcommand{\R}{\mathbb{R}}
\newcommand{\Z}{\mathbb{Z}}

\newcommand{\bbE}{\mathbb{E}}
\newcommand{\bbG}{\mathbb{G}}
\newcommand{\bbH}{\mathbb{H}}
\newcommand{\bbL}{\mathbb{L}}
\newcommand{\bbS}{\mathbb{S}}
\newcommand{\bbV}{\mathbb{V}}

\newcommand{\calA}{\mathcal{A}}
\newcommand{\calC}{\mathcal{C}}

\newcommand{\calE}{\mathcal{E}}
\newcommand{\calF}{\mathcal{F}}

\newcommand{\calK}{\mathcal{K}}

\newcommand{\calP}{\mathcal{P}}

\newcommand{\calR}{\mathcal{R}}

\newcommand{\calT}{\mathcal{T}}
\newcommand{\calU}{\mathcal{U}}

\newcommand{\calW}{\mathcal{W}}

\newcommand{\rmb}{\mathrm{b}}
\newcommand{\rmd}{\mathrm{d}}
\newcommand{\rme}{\mathrm{e}}
\newcommand{\rmf}{\mathrm{f}}
\newcommand{\rmi}{\mathrm{i}}
\newcommand{\rmw}{\mathrm{w}}

\newcommand{\rmC}{\mathrm{C}}

\newcommand{\rmE}{\mathrm{E}}

\newcommand{\rmT}{\mathrm{T}}

\newcommand{\fcone}{\mathcal{Y}^\blacktriangleleft}
\newcommand{\bcone}{\mathcal{Y}^\blacktriangleright}
\newcommand{\bend}{\mathbf{b}}
\newcommand{\fend}{\mathbf{f}}
\newcommand{\diam}{\mathrm{Diamond}}

\newcommand{\CPts}{\textnormal{CPts}}
\newcommand{\rCPts}{\mathrm{rCPts}}

\newcommand{\Skel}{\mathrm{Sk}}

\newcommand{\concatenate}{\circ}
\newcommand{\displace}{\mathrm{X}}

\newcommand{\irr}{\mathrm{\scriptscriptstyle ir}}

\newcommand{\SetRootMarkBackCont}{\mathfrak{B}_L}
\newcommand{\SetRootMarkForwCont}{\mathfrak{B}_R}
\newcommand{\SetRootDiaCont}{\mathfrak{A}}

\newcommand{\SetRootMarkBackContFV}{\mathfrak{B}_{L,n}^{\mathrm{good}}}
\newcommand{\SetRootMarkForwContFV}{\mathfrak{B}_{R,n}^{\mathrm{good}}}

\newcommand{\MixMeas}{\mathrm{MixMe}}

\newcommand{\LinInt}{\mathrm{LI}}
\newcommand{\bbP}{\mathbb{P}}
\newcommand{\CovMat}{\mathrm{D}}

\newcommand{\slab}{\mathcal{S}}
\newcommand{\slabCP}{\slab\CPts}
\newcommand{\goodCl}{\mathrm{GCl}}

\newcommand{\scale}{\mathrm{sc}}

\newcommand{\crossCl}{\mathcal{C}_{\mathrm{cross}}}
\newcommand{\rectangle}{\mathrm{rect}}
\newcommand{\FVBox}{\mathcal{R}}

\newcommand{\feCst}{\alpha}
\newcommand{\stateSwapCst}{c_{\mathrm{swap}}}

\newcommand{\doubleCross}{\mathrm{DoubleCross}}

\newcommand{\OZDecompmeas}{\mathrm{OZDec}}
\newcommand{\OZmeas}{\mathrm{OZ}}
\newcommand{\OZWalkmeas}{\mathrm{OZwalk}}

\newcommand{\restrict}{\mathrm{restr}}
\newcommand{\Env}{\mathrm{Env}}
\newcommand{\OutBnd}{\mathrm{OutBnd}}
\newcommand{\Fill}{\mathrm{Fill}}

\newcommand{\Forget}{\mathrm{Forget}}

\newcommand{\hf}{\mathtt{HF}}
\newcommand{\spin}{\mathtt{Spin}}
\newcommand{\at}{\mathtt{AT}}
\newcommand{\atrc}{\mathtt{ATRC}}
\newcommand{\matrc}{\mathtt{mATRC}}
\newcommand{\fk}{\mathtt{FK}}
\newcommand{\potts}{\mathtt{Potts}}

\newcommand{\edwardSokal}{\mathtt{ES}}

\newcommand{\fkfree}{\mathrm{f}}
\newcommand{\fkwired}{\mathrm{w}}
\newcommand{\atrcfree}{0}
\newcommand{\atrcwired}{1}

\newcommand{\clusters}{\kappa}
\newcommand{\clusterSet}{\mathrm{cl}}
\newcommand{\loops}{\mathrm{loop}}

\newcommand{\tvd}{\mathrm{d}_{\mathrm{TV}}}
\renewcommand{\diameter}{\mathrm{diam}}

\newcommand{\svc}{\mathbf{c}}
\newcommand{\svcb}{\mathbf{c}_{\mathrm{b}}}

\newcommand{\partialex}{\partial^{\mathrm{ex}}}
\newcommand{\partialin}{\partial^{\mathrm{in}}}
\newcommand{\partialedge}{\partial^{\mathrm{edge}}}

\def\Hloop#1#2{
\draw[blue] ({#1 + cos(-45)/(2*sqrt(2))},{#2+0.5 + sin(-45)/(2*sqrt(2))}) arc (-45:-135:{1/(2*sqrt(2))}) ;
\draw[blue] ({#1 + cos(45)/(2*sqrt(2))},{#2-0.5 + sin(45)/(2*sqrt(2))}) arc (45:135:{1/(2*sqrt(2))}) ;
}
\def\Vloop#1#2{
\draw[blue] ({#1 + 0.5 + cos(135)/(2*sqrt(2)) },{#2 + sin(135)/(2*sqrt(2))}) arc (135:225:{1/(2*sqrt(2))}) ;
\draw[blue] ({#1 - 0.5 + cos(-45)/(2*sqrt(2))},{#2 + sin(-45)/(2*sqrt(2))}) arc (-45:45:{1/(2*sqrt(2))}) ;
}

\def\DrawTile#1#2#3{
\draw[#3] ({#1-0.5},{#2}) -- ({#1},{#2+0.5}) -- ({#1+0.5},{#2}) -- ({#1},{#2-0.5}) -- ({#1-0.5},{#2}) ;
}

\def\Hedge#1#2#3{
\draw[#3] ({#1-0.5},#2) -- ({#1+0.5},#2) ;
}

\def\Vedge#1#2#3{
\draw[#3] (#1,{#2-0.5}) -- (#1,{#2+0.5}) ;
}

\def\UParrow#1#2#3{
\draw[#3] ({#1-0.1},{#2-0.1}) -- ({#1},{#2}) -- ({#1+0.1},{#2-0.1}) ;
}
\def\DOWNarrow#1#2#3{
\draw[#3] ({#1-0.1},{#2+0.1}) -- ({#1},{#2}) -- ({#1+0.1},{#2+0.1}) ;
}
\def\LEFTarrow#1#2#3{
\draw[#3] ({#1+0.1},{#2-0.1}) -- ({#1},{#2}) -- ({#1+0.1},{#2+0.1}) ;
}
\def\RIGHTarrow#1#2#3{
\draw[#3] ({#1-0.1},{#2-0.1}) -- ({#1},{#2}) -- ({#1-0.1},{#2+0.1}) ;
}

\theoremstyle{plain}
\newtheorem{lemma}{Lemma}[section]
\newtheorem{theorem}[lemma]{Theorem}
\newtheorem{proposition}[lemma]{Proposition}
\newtheorem{corollary}[lemma]{Corollary}

\newtheorem{remark}[lemma]{Remark}

\newtheorem{claim}[lemma]{Claim}

\theoremstyle{definition}
\newtheorem{definition}[lemma]{Definition}

\newif\ifpic
\pictrue

\title[Order-Disorder Interface in 2D Potts]{Discontinuous transition in 2D Potts:\\ I. order-disorder interface convergence}

\author{Moritz Dober}
\address{Fakultät für Mathematik, Universität Wien, Vienna, Austria}
\email{moritz.dober@univie.ac.at}

\author{Alexander Glazman}
\address{Karlsruhe Institute of Technology, Karlsruhe, Germany and TU Wien, Vienna, Austria}
\email{alexander.glazman@kit.edu}

\author{Sébastien Ott}
\address{Institute of Mathematics, EPFL, 1015 Lausanne, Switzerland}
\email{ott.sebast@gmail.com}

\date{\today}

\begin{document}

\begin{abstract}
	We study a $q$-state Potts model on the square grid when~$q>4$ at the point~$T_c(q)$ of its (discontinous) transition.
	This model exhibits exactly~$q+1$ extremal Gibbs measures: $q$ ordered (monochromatic) and one disordered (free).
	The current work deals with the Dobrushin order--disorder boundary conditions on a finite~$N\times N$ box.
	Our main result is that this interface is a well-defined object, has \(\sqrt{N}\) fluctuations, and converges to a Brownian bridge under diffusive scaling.
	The same holds also for the corresponding FK-percolation model for all~$q>4$.

	Our proofs rely on a coupling between FK-percolation, the six-vertex model, and the random-cluster representation of an Ashkin--Teller model (ATRC), and on a detailed study of the latter.
	The coupling relates the interface in FK-percolation to a long subcritical cluster in the ATRC model.
	For this cluster we develop a ``renewal picture'' \emph{à la} Ornstein-Zernike.
	This is based on fine mixing properties of the ATRC model that we establish using the link to the six-vertex model and its height function.
	Along the way, we derive various properties of the Ashkin-Teller model, such as Ornstein-Zernike asymptotics for its two-point function.
	
	In a companion work, we provide a detailed study of the Potts model under order-order Dobrushin conditions. 
	We show emergence of a free layer of width $\sqrt{N}$ between the two ordered phases ({\em wetting}) and establish convergence of its boundaries to two Brownian bridges conditioned not to intersect.
\end{abstract}

\maketitle

\setcounter{tocdepth}{1}
\tableofcontents

\section{Introduction and results}
\label{sec:intro}

The Potts model is a classical model of statistical mechanics introduced in 1952~\cite{Pot52}.
Each vertex of a graph is assigned one of $q$ states (colours), with states of adjacent vertices interacting with a strength depending on the temperature $T>0$ of the system. 
At~$q=2$, this corresponds to the seminal Ising model.
The Potts model becomes increasingly ordered as the temperature decreases, and a phase transition occurs on lattices $\Z^{d}$ with $d \geq 2$ at some transition temperature \(T_c(q,d)>0\). 
Depending on $q$ and $d$, the transition is either discontinuous (first-order) or continuous (higher-order).
Such a rich behaviour has brought a lot of attention to the Potts model.

\subsubsection*{Interface in the planar Potts model.}

Our work is restricted to dimension~$d=2$.
In this case, planar duality and a correlation inequality (available when~$q\geq 1$) have allowed to a watershed of progress in the phase diagram of the Potts model in the last two decades:
\begin{itemize}
	\item the transition occurs at the self-dual point~$T_c(q):=[\ln (1+\sqrt{q})]^{-1}$~\cite{BefDum12a};
	\item the transition is continuous when~$q =2,3,4$, in a sense that, at~$T_c(q)$, there is a unique Gibbs measure~\cite{DumSidTas17} (see also~\cite{GlaLam23} for another argument);
	\item the transition is discontinuous when~$q > 4$~\cite{DumGagHar21} (see also~\cite{RaySpi20} for a short proof), and any Gibbs measure can be written as a linear combination of~$q+1$ extremal Gibbs measures ($q$ monochromatic and one free)~\cite{GlaMan23}.
\end{itemize}
We focus on discontinuous transitions ($q>4$) and study interfaces at~$T_c(q)$ separating different states.
The structure of extremal Gibbs measures (described above) leads to two natural definitions of Dobrushin boundary conditions:
\begin{itemize}
	\item {\em order-disorder}: one half of the boundary is of a fixed colour and the other one is free (no colour assigned);
	\item {\em order-order}: both halves of the boundary are assigned different fixed colours.
\end{itemize}
The phenomenology is quite different in the two cases, see Fig.~\ref{Fig:Potts_interface_simul}.

\begin{figure}
	\includegraphics[scale=0.13]{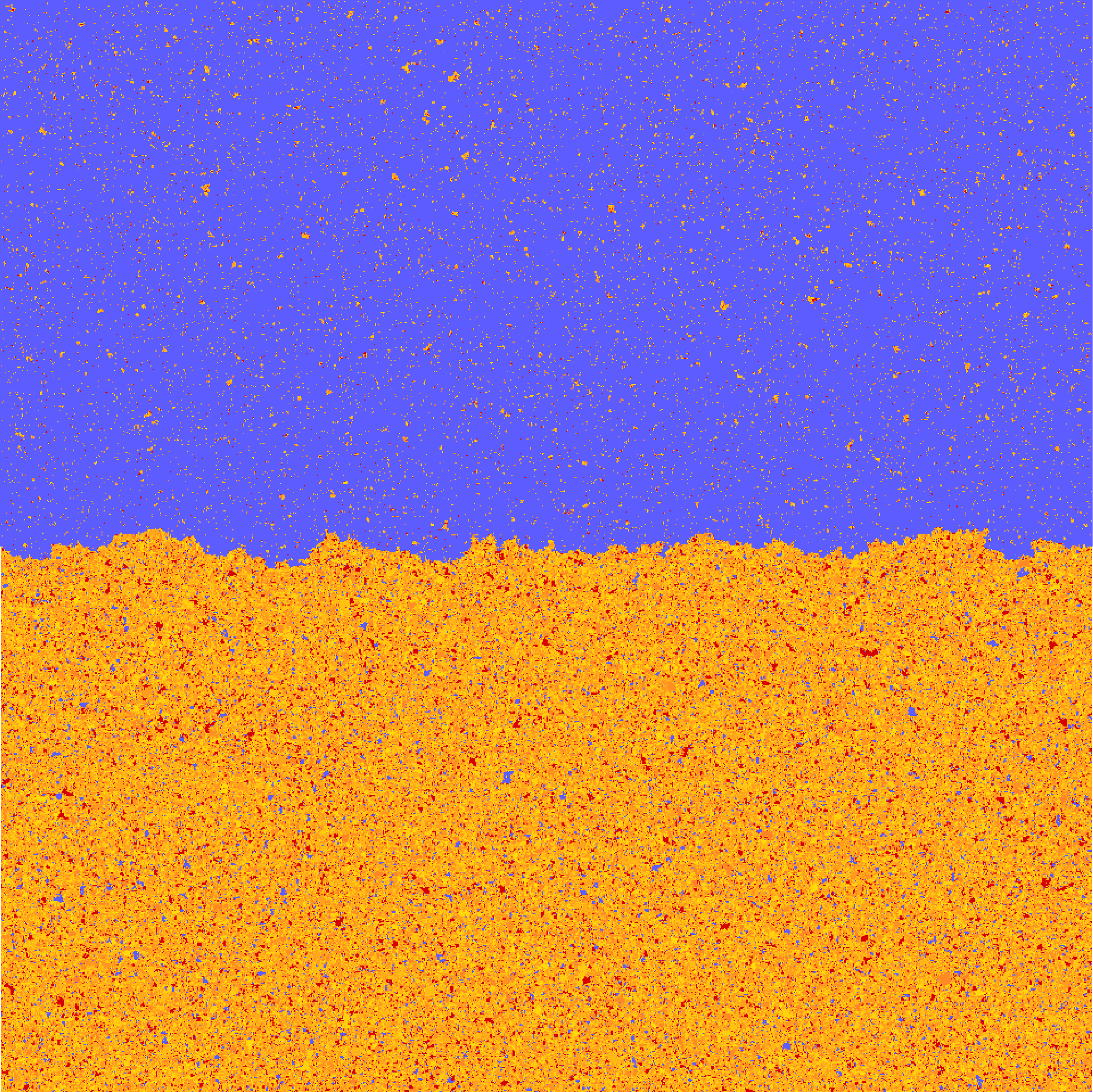}
	\hspace*{0.4cm}
	\includegraphics[scale=0.13]{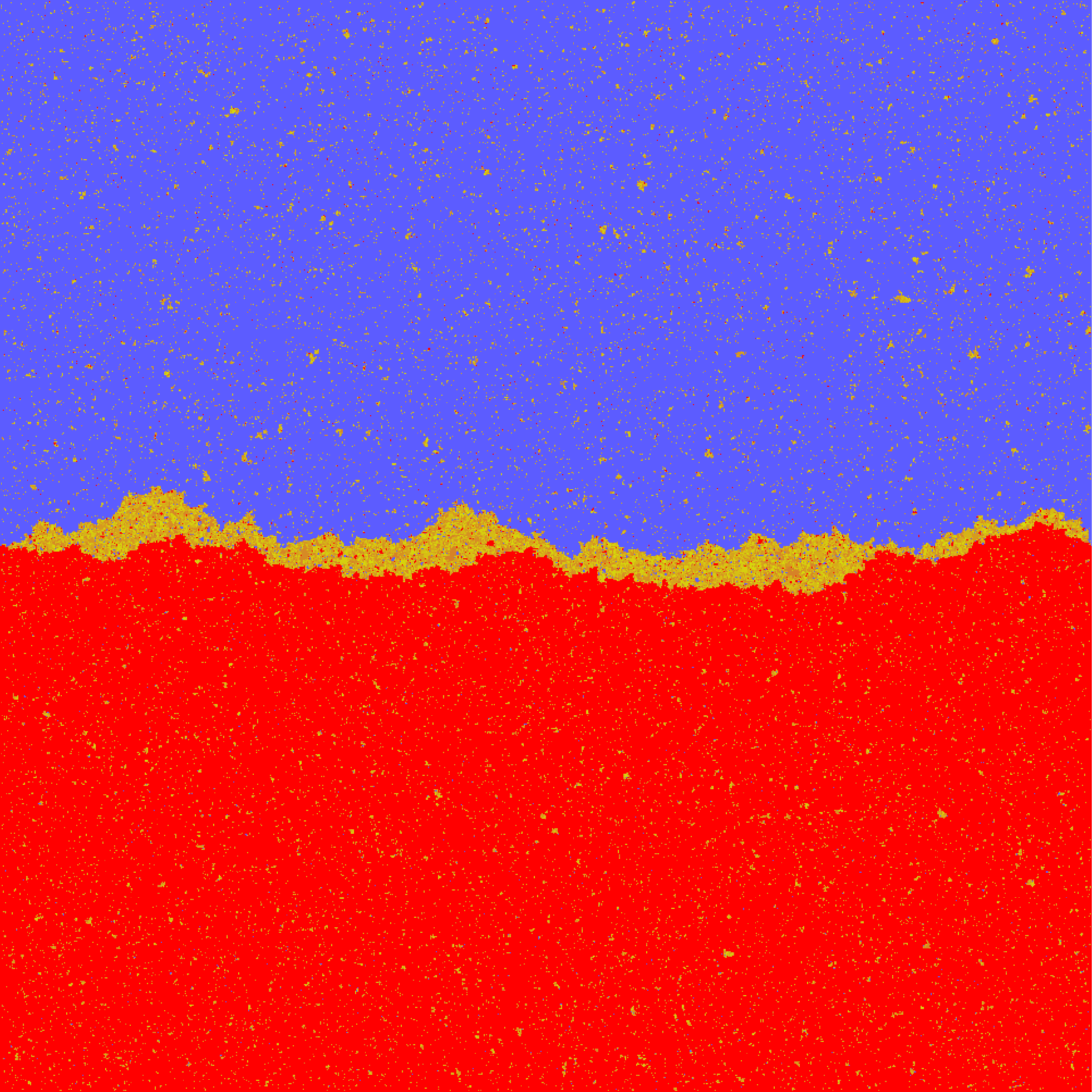}
	\caption{Sample of a 1000x1000 Potts model with 25 colours at \(T_c(25)\) with Dobrushin boundary conditions. Colours are: blue for the first, red for the second and interpolate between yellow and orange for colours 3 to 25. Left: order-disorder interface; upper part has blue b.c., bottom has white b.c. (no colour favoured). Right: order-order interface; upper part has blue b.c., bottom has red b.c..}
	\label{Fig:Potts_interface_simul}
\end{figure}

The main result of the current paper is convergence of the interface under order-disorder conditions to the Brownian bridge in the diffusive limit.
Results of such precision were previously proven for~$T<T_c(q)$ and the order-order interface. Indeed, planar duality transfers this problem to the study of the typical geometry of long clusters at \(T>T_c\), for which a quasi-renewal structure was developed in~\cite{CamIofVel08}.
We do use similar ideas, but the situation is more involved: duality does not help directly as we work at the transition point with order-disorder boundary conditions, where duality only maps the problem to itself.
Instead, we build on~\cite{BaxKelWu76} and~\cite{GlaPel23} to construct a sequence of combinatorial mappings relating the interface in the FK-percolation to a suitable long subcritical cluster in the Ashkin-Teller model.
We extend the study of~\cite{AouDobGla24} to establish mixing properties in the latter model.
This opens a way to obtain a renewal picture and invariance principle for long subcritical clusters in the ATRC model using the methods developed in~\cite{CamIofVel03,CamIofVel08,AouOttVel24}.
Finally, we transport back the convergence result to the FK and Potts models.

Let us also mention the work~\cite{MesMirSal+1991} that showed~$\sqrt{N}$ fluctuations when~$q$ is taken to be large enough.
Our results imply this for all~$q>4$.
In this generality, the only progress was a recent complete characterisation of Gibbs measures in the planar Potts model~\cite{GlaMan23}.
This work proves that all Gibbs measures are invariant to translations, and this implies that the interface (if defined at all) exhibits diverging fluctuations.

The case of the order-order conditions is addressed in our companion paper~\cite{DobGlaOtt26}: we show emergence of a free layer of width $\sqrt{N}$ between the two ordered phases ({\em wetting}) and establish convergence of its boundaries to a Brownian watermelon (two Brownian bridges conditioned not to intersect).

In particular, we prove that, when~$q>4$, in both cases of Dobrushin conditions, as the mesh of the lattice tends to zero, the scaling limit of the interface in the Potts model is a straight line.
To compare with what happens in the case of continuous phase transition, let us describe the expected behaviour  for $q=2,3,4$.
The mesmerizing physics conjecture from 1980s postulates conformal invariance.
Schramm's~\cite{Sch00} geometric interpretation of this conjecture asserts that as one takes the scaling limit of the Potts model, the interface converges to a random fractal curve called the Schramm--Loewner Evolution.
This has been proven rigorously only at~$q=2$ (Ising model) by Smirnov et al~\cite{Smi10,CheSmi11,CheDumHonKemSmi14}.

\subsubsection*{The ``Ornstein-Zernike'' (OZ) theory.}

It is a (non-rigorous) picture introduced in~\cite{OrnZer14,Zer16} of how correlation functions in various models behave. Their main idea was to postulate a suitable \emph{renewal structure} satisfied by correlation functions, which leads to very precise expressions for the asymptotics of the said correlations.
The first rigorous implementation of this renewal picture was done by Abraham and Kunz~\cite{AbrKun77} using perturbative expansions. The modern approach, which started with the work of Chayes, and Chayes~\cite{ChaCha86} on the self-avoiding walk (SAW) and of Campanino, Chayes, and Chayes~\cite{CamChaCha91} on Bernoulli percolation, is based on creating a renewal structure \emph{à la} OZ for elongated subcritical objects (SAW or percolation clusters containing a distant point). Both these works rely on heavy combinatorial study, and the renewal steps are ``irreducible crossings of slabs''.

This idea was further developed by Ioffe~\cite{Iof98} for the ballistic phase of the self-avoiding walk, introducing key measure-tilting ideas coming from large deviation theory. His strategy was then extended to Bernoulli percolation in~\cite{CamIof02}, to the high-temperature Ising model~\cite{CamIofVel03} and to the FK-percolation~\cite{CamIofVel08}. 
Compared to the cases of SAW and Bernoulli percolation where the measure factorizes nicely, the structure obtained in the case of Ising and FK, while still geometrically being a concatenation of ``irreducible blocs'', is not a real renewal structure. Indeed, it is only a sequence of ``fast mixing kernels'', which study is heavier and performed in~\cite{CamIofVel03}.
The last step to finally obtain a true renewal structure from the fast mixing kernel picture was done in~\cite{OttVel18}.

These works on subcritical clusters of the FK-percolation all take place in any dimensions.
On \(\Z^2\), planar duality allows to rewrite the interface of the Potts model at \(T<T_c(q)\) as a subcritical FK-percolation cluster conditioned to contain \((0,0)\) and \((N,0)\).

Our contribution to this line of works is to derive a ``renewal picture'' for the long clusters of the planar random-cluster representation of the Ashkin--Teller model, for values of the parameters lying on the self-dual line.
This model is then related to the Potts interface using a suitable adaptation of a coupling introduced in~\cite{GlaPel23}.

\vspace*{4pt}

We now formally state our main results for the Potts model, the FK percolation and the Ashkin--Teller model.

\subsection{Potts model}

We view~$\Z^2$ both as a set of points on the plane having integer coordinates and as a graph (square grid) with edges linking points at distance one.
Denote by~$\bbE$ the set of edges in~$\Z^2$ and write~$i\sim j$ if~$\{i,j\}\in\bbE$.
Let~$G=(V,E)$ be a subgraph of~$\Z^2$.
Take parameters $T>0,\,q\in \{2,3,4,\dots\}$, and boundary conditions $\eta\in\{0,1,\dots,q\}^{\bbV}$. 
The Potts model on \(G\) with boundary conditions \(\eta\) is the probability measure on \(\{1,\dots,q\}^{V}\) given by
\begin{equation*}
	\potts_{G;T,q}^{\eta}(\sigma):=
	\tfrac{1}{Z_{{\potts}}} \cdot 
	\exp\Bigg[\tfrac1T \cdot \Bigg(
	\sum_{\{i,j\}\in E} \delta(\sigma_i,\sigma_j) + \sum_{i\in V, j\in V^c:\, i\sim j} \delta(\sigma_i,\eta_j)\Bigg)\Bigg],
\end{equation*}
where~$Z_{{\potts}}=Z_{{\potts}}(G,T,q,\eta)$ is the unique normalising constant (called {\em partition function}) that renders the above a probability measure, and~$\delta(x,y) = 1$ if~$x=y$ and~$\delta(x,y) = 0$ otherwise. Value~\(0\) for~\(\eta\) corresponds to free boundary conditions favoring none of the \(q\) possible states of the spins.

\begin{figure}
	\centering
	\ifpic
	\begin{tikzpicture}[scale=0.6]
		\foreach \i in {-2,...,2}{
			\foreach \j in {-2,...,2}{
				\filldraw[black] (\i,\j) circle (2pt);
			}
		}
		\foreach \i in {-3,...,2}{
			\foreach \j in {0,1,2}{
				\draw[very thick] (\i,\j)--(\i+1,\j) ;
			}
		}
		\foreach \i in {-2,...,2}{
			\foreach \j in {0,1,2}{
				\draw[very thick] (\i,\j)--(\i,\j+1) ;
			}
		}
		\draw[thick, dotted] (-3.5,-0.5)--(3,-0.5);
		\draw (3,-0.5) node[right]{$y=-1/2$} ;
		\draw (0.35,0.35) node{$0$};
	\end{tikzpicture}
	\hspace{2cm}
	\begin{tikzpicture}[scale=0.6]
		\foreach \i in {-2,...,2}{
			\foreach \j in {-2,...,2}{
				\filldraw[black] (\i,\j) circle (2pt);
			}
		}
		\foreach \i in {-2,...,2}{
			\filldraw[black] (\i,3) circle (2pt);
		}
		\foreach \j in {0,1,2}{
			\filldraw[black] (-3,\j) circle (2pt);
			\filldraw[black] (3,\j) circle (2pt);
		}
		\foreach \i in {-3,...,2}{
			\foreach \j in {0,1,2}{
				\draw (\i,\j)--(\i+1,\j) ;
			}
		}
		\foreach \i in {-2,...,1}{
			\foreach \j in {-1,-2}{
				\draw (\i,\j)--(\i+1,\j) ;
			}
		}
		\foreach \i in {-2,...,2}{
			\foreach \j in {-2,...,2}{
				\draw (\i,\j)--(\i,\j+1) ;
			}
		}
		\draw (0.35,0.35) node{$0$};
	\end{tikzpicture}
	\fi
	\caption{Left: wired-free Dobrushin boundary condition. Right: the graph~\(G_{2}\).}
	\label{fig:WiredFree_FK}
\end{figure}
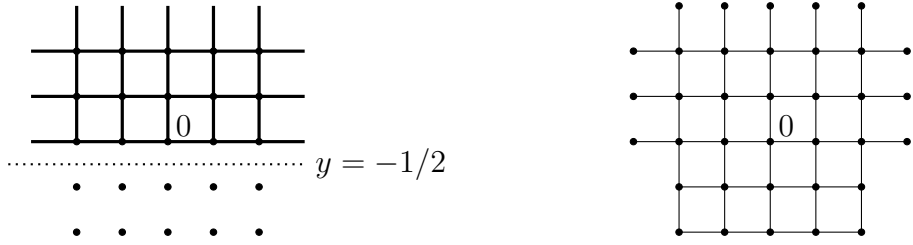

We say that~$\eta$ defines the order-disorder (1-free) Dobrushin boundary conditions (and denote it by~$1/\rmf$) if~$\eta((x,y))=\ind_{\geq 0}(y)$.
Identify~$\Lambda_n:=\{-n,\dots, n\}^2$ with the induced subgraph of~$\Z^2$ on this set of vertices.
The Dobrushin boundary conditions on~$\Lambda_n$ impose existence of an interface between color one and the rest.
This can be made explicit, but for brevity we choose to define directly the upper and lower discrete envelopes of this interface: $\Gamma_{\potts}^+$ and~$\Gamma_{\potts}^-$ respectively.
Given~$\sigma\in\{1,\dots,q\}^{\Lambda_n}$, define~$\bar\sigma \in\{0,1,\dots,q\}^{\Z^2}$ to be its extension outside of~$\Lambda_n$ by the Dobrushin boundary conditions: one on~$\Z\times \Z_{\geq 0}$ and zero on~$\Z\times \Z_{< 0}$. 
For~\(k =-n,\dots, n\), define
\begin{align*}
	\Gamma_{\potts}^{+,n}(k) & := \max\{y\in \Z:\ (k,y-1)\xleftrightarrow{\bar\sigma \neq 1}  (\Z\times \Z_{< 0})\setminus\Lambda_n\},\\
	\Gamma_{\potts}^{-,n}(k) & := \min\{y\in \Z:\ (k,y+1)\xleftrightarrow{\bar\sigma = 1, \text{ diag}}  (\Z\times \Z_{\geq 0})\setminus\Lambda_n\},
\end{align*}
where~$(k,y-1)\xleftrightarrow{\bar\sigma \neq1}  (\Z\times \Z_{<0})\setminus\Lambda_n$ states the existence of a path in~$\Z^2$ in going from~$(k,y-1)$ to~$(\Z\times \Z_{<0})\setminus\Lambda_n$ and 
consisting of vertices where~$\bar\sigma\neq 1$ and~$(k,y)\xleftrightarrow{\bar\sigma = 1, \text{ diag}}  (\Z\times \Z_{\geq 0})\setminus\Lambda_n$ states the existence of a path in~$\Z^2$ {\em with diagonal connectivity} going from~$(k,y)$ to~$(\Z\times \Z_{\geq 0})\setminus\Lambda_n$ and consisting of vertices where~$\bar\sigma= 1$.
By~$\Z^2$ with diagonal connectivity we mean a graph with the vertex-set~$\Z^2$ with edges linking vertices at distance at most~$\sqrt{2}$.
Define the rescaled linear interpolation of~$\Gamma_{\potts}^{+,n}$ and~$\Gamma_{\potts}^{-,n}$ by
\[
	\tilde{\Gamma}_{\potts}^{\pm,n}(t) := \tfrac{1}{\sqrt{n}} \big((1-\{2tn -n\})\Gamma_{\potts}^{\pm}(\lfloor 2tn -n\rfloor) + \{2tn -n\}\Gamma_{\potts}^{\pm}(\lceil 2tn -n\rceil)\big),
\]
where \(\lfloor\,\rfloor\), \(\lceil\, \rceil\), \(\{\,\}\) denote respectively lower and upper roundings and fractional part.

\begin{theorem}\label{thm:potts}
	Let~$q>4$ be integer and take~$T=T_c(q)$.
	For~$n\in\N$, sample~$\Gamma_{\potts}^{\pm,n}$ and~$\tilde{\Gamma}_{\potts}^{\pm,n}$ from~$\potts_{\Lambda_n;T_c(q),q}^{1/\rmf}$ as described above.
	Then, as~$n$ tends to infinity,
	\begin{enumerate}
		\item both \(\big(\tilde{\Gamma}_{\potts}^{+,n}(t)\big)_{t\in [0,1]}\) 
		and~\(\big(\tilde{\Gamma}_{\potts}^{-,n}(t)\big)_{t\in [0,1]}\) converge in law to \((c_q\rmb_t)_{t\in [0,1]}\), where \(\rmb_t\) is a standard Brownian bridge and \(c_q >0\) is some constant;
		\item the probability that~$\max_{k} \left|\Gamma_{\potts}^{+,n}(k) -\Gamma_{\potts}^{-,n}(k)\right|\geq C\ln(n)^2$ tends to zero.
	\end{enumerate}
\end{theorem}
The constant \(c_q\) has an explicit characterization, see Remark~\ref{rem:diffusivity_constant}.

\begin{remark}\label{rem:thm-extensions}
	We focus on a horizontal interface for simplicity, but the analysis can easily be adapted to other directions at the cost of slightly heavier notations. 
	It is however important that the sides of the box are perpendicular to the direction of the interface.
	Indeed, to study other geometry, one would have to exclude boundary wetting in a suitable percolation model (the ATRC model introduced in Section~\ref{sec:atrc}).
\end{remark}

\subsection{FK percolation}

The main tool in analyzing the Potts model is the Fortuin--Kasteleyn (FK) percolation (or random-cluster) model~\cite{ForKas72} that allows to express the spin-spin correlation function via connection probabilities.
We first define it and then state the relation between the two models.
Take a finite subgraph $G=(V,E)$ of $\Z^2$, parameters $p\in (0,1),\,q>0$ and boundary conditions $\xi\in\{0,1\}^{\bbE}$. 
We identify any~$\omega\in\{0,1\}^{\bbE}$ with the set of edges~$e\in \bbE$ for which~$\omega_e=1$ ({\em open edges}) and with the spanning subgraph of~$\Z^2$ defined by the open edges.
The FK-percolation model on \(G\) with boundary conditions \(\xi\) is the probability measure on $\{0,1\}^{\bbE}$ given by 
\begin{equation*}
	\fk_{G;p,q}^{\xi}(\eta):=\tfrac{1}{Z_{\fk}}\cdot p^{|\eta\cap E|}(1-p)^{|E\setminus\eta|}\,q^{\clusters_V(\eta)}\, \ind_{\eta=\xi \text{ on } \bbE\setminus E},
\end{equation*}
where $Z_{\fk}=Z_{\fk}(G,p,q,\xi)$ is the partition function and~$\clusters_V(\eta)$ is the number of connected components ({\em clusters}) of $(\Z^2,\eta)$ that intersect $V$.

The \emph{free} and \emph{wired} measures correspond to the choices $\xi\equiv 0$ and $\xi\equiv 1$, respectively, and we simply write $\fkfree$ and $\fkwired$ instead of $\xi$, respectively. 
We will be interested in the Dobrushin wired/free boundary conditions: $\xi_e =1$ if and only if~$e \subset\Z\times\Z_{\geq 0}$ (see Fig.~\ref{fig:WiredFree_FK}).
We denote these boundary conditions by~$1/0$.

For \(n\geq 1\), define the graph \(G_{n}=(V_{n},E_{n})\) by
\begin{equation*}
E_{n}:=\{e\in\bbE: e\subset\Lambda_{n}\}\cup\{e\in\bbE:e\subset\Z\times\Z_{\geq 0}\text{ and }e\cap\Lambda_{n}\neq\varnothing\},\quad V_{n}:=\bigcup_{e\in E_{n}}e.
\end{equation*}

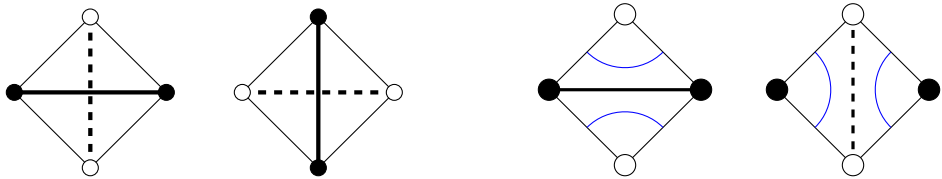
\begin{figure}
	\centering
	\begin{tikzpicture}
		\draw (0,0)--(1,1)--(0,2)--(-1,1)--(0,0);
		\draw[ultra thick] (1,1)--(-1,1);
		\draw[ultra thick, dashed] (0,0)--(0,2);
		\draw (3,0)--(4,1)--(3,2)--(2,1)--(3,0);
		\draw[ultra thick] (3,0)--(3,2);
		\draw[ultra thick, dashed] (4,1)--(2,1);
		\filldraw[fill=white] (0,2) circle(3pt);
		\filldraw[fill=white] (0,0) circle(3pt);
		\filldraw[fill=black] (-1,1) circle(3pt);
		\filldraw[fill=black] (1,1) circle(3pt);
		\filldraw[fill=white] (2,1) circle(3pt);
		\filldraw[fill=white] (4,1) circle(3pt);
		\filldraw[fill=black] (3,2) circle(3pt);
		\filldraw[fill=black] (3,0) circle(3pt);
	\end{tikzpicture}
	\hspace*{1.5cm}
	\begin{tikzpicture}[scale=2]
		\draw (0,-0.5)--(0.5,0)--(0,0.5)--(-0.5,0)--(0,-0.5);
		\draw (1.5,-0.5)--(2,0)--(1.5,0.5)--(1,0)--(1.5,-0.5);
		
		\Vloop{1.5}{0}
		\Hloop{0}{0}
		
		\draw[very thick] (-0.5,0)--(0.5,0); 
		\draw[very thick, dashed] (1.5,-0.5)--(1.5,0.5);
		
		\foreach \i in {0,1.5}{
			\filldraw[fill=black] ({\i-0.5},0) circle(2pt) ;
			\filldraw[fill=black] ({\i+0.5},0) circle(2pt) ;			
			\filldraw[fill=white] (\i,-0.5) circle(2pt) ;
			\filldraw[fill=white] (\i,0.5) circle(2pt) ;
		}
	\end{tikzpicture}
	\caption{Left: Tile associated to a mid-edge. Right: Tile centred at the middle of a horizontal primal edge (solid black) or its associated vertical dual edge (dashed black), with its two possible local loop configurations.}
	\label{fig:midEdgeTiles}
\end{figure}

The seminal Edwards--Sokal coupling~\cite{EdwSok88} states that, when~$q\geq 2$ is integer and~$p=1- \exp[-\tfrac{1}{T}]$, coloring clusters of~$\omega\sim \fk_{G_n;p,q}^{1/0}$ independently in colors~$1,2,\dots,q$ gives a spin configuration~$\sigma\sim \potts_{\Lambda_n;T,q}^{1/\rmf}$.
We will denote this coupling between the Potts model and the FK percolation 
by~\(\edwardSokal_{\Lambda_n,T,q}^{1/0}\).

We define an interface in the FK percolation forced by the Dobrushin boundary conditions using planar duality.
Note that the lattice dual to~$\Z^2$ is again a square lattice and we denote it by~$(\Z^2)^*$.
For each edge~$e$ of~$\Z^2$, denote by~$e^*$ the edge of~$(\Z^2)^*$ that is dual to~$e$, i.e. the unique edge of~$(\Z^2)^*$ that intersects~$e$; see Fig.~\ref{fig:midEdgeTiles}.
Given~$\omega\in \{0,1\}^{\bbE}$, define its dual by
\[
	\omega^*_{e^*}:= 1-\omega_e.
\]
The FK percolation at $p_c(q)$, given by
\[
	p_c(q) := 1- \exp[-\tfrac{1}{T_c(q)}] = \tfrac{\sqrt{q}}{\sqrt{q}+1},
\]
is known to enjoy the self-duality: if~$\omega\sim\fk_{G_n;T_c(q),q}^{1/0}$, then its dual~$\omega^*$ is also distributed as an FK-percolation with parameters~$q$ and~$p_c(q)$, but on a dual graph and under dual Dobrushin boundary conditions.
This self-dual nature is also revealed when looking at the loop representation of the FK percolation model.
Specifically, we draw two arcs next to every primal or dual edge of~$\omega$, as shown on Fig.~\ref{fig:midEdgeTiles}.
These arcs link together into loops separating primal and dual clusters and one interface tracing the boundary of the union of primal clusters attached to the upper boundary of~$\Lambda_n$.
Denote this interface by~$\Gamma_{\fk}$.
Remarkably, a standard application of the Euler's formula (see eg.~\cite[Lemma~3.9]{DumGagHar21}), allows to rewrite the distribution of~$\omega$ via loops which are symmetric with respect primal and dual configurations:
\begin{equation}\label{eq:FK_loop_expression}
	\fk_{G_n;p_c(q),q}^{1/0}(\omega) = \tfrac{1}{Z_{\sf loop}} \cdot \sqrt{q}^{\# \, {\rm loops} },
\end{equation}
where~$\# \, {\rm loops}$ stands for the number of loops in the loop representation of~$\omega$.
This is a part of the classical Baxter--Kelland--Wu~\cite{BaxKelWu76} coupling between the FK percolation and the six-vertex model; see Sections~\ref{sec:coupling:interfaces}~and~\ref{sec:bkw}.

As for the Potts model, we define the upper and the lower discrete envelops of~$\Gamma_{{\fk}}$:
\begin{align*}
	\Gamma_{\fk}^{+,n}(k) & := \max\{y\in \Z:\ (k\pm\tfrac12,y-\tfrac12)\xleftrightarrow{\omega^*} (\{\Z+\tfrac12\}\times \{\Z_{< 0} + \tfrac12\})\setminus\Lambda_n^*\},\\
	\Gamma_{\fk}^{-,n}(k) & := \min\{y\in \Z:\ (k,y+1)\xleftrightarrow{\omega} (\Z\times \Z_{\geq 0})\setminus\Lambda_n\},
\end{align*}
where \(\Lambda_n^*:=[-n,n]^2\cap(\Z^2)^*\).
The rescaled linear interpolations~$\tilde{\Gamma}_{\fk}^{\pm,n}$ of~$\Gamma_{\fk}^{\pm,n}$ are defined in the same way as in the Potts model.
Our main result for the FK percolation is an invariance principle for~$\tilde{\Gamma}_{\fk}^{\pm,n}$:

\begin{theorem}
	\label{thm:order-disorder:FK_loop}
	Let~$q>4$ be a real number and take~$p=p_c(q)$.
	For~$n\in\N$, sample~$\Gamma_{\fk}^{\pm,n}$ and~$\tilde{\Gamma}_{\fk}^{\pm,n}$ from~$\fk_{G_n;p_c(q),q}^{1/0}$ as described above.
	Then, as~$n$ tends to infinity, the convergence results from Theorem~\ref{thm:potts} hold for~$\Gamma_{\fk}^{\pm,n}$ and~$\tilde{\Gamma}_{\fk}^{\pm,n}$ in place of~$\Gamma_{\potts}^{\pm,n}$ and~$\tilde{\Gamma}_{\potts}^{\pm,n}$.
\end{theorem}

We draw attention to the precision of our control of the interface.
Previously, at~$p_c(q)$ when~$q>4$ (discontinuous transition), it remained open that~$\Gamma_{\fk}^{n}$ does not exhibit linear fluctuations.
Our results imply that this is indeed the case and, moreover, the probability that~$\Gamma_{\fk}^{n}$ exhibits linear fluctuations is in fact exponentially small.
This should be compared to the behavior of~$\Gamma_{\fk}^{n}$ at~$p_c(q)$ when~$q\in [1,4]$ (continuous transition): there it does exhibit linear fluctuations~\cite{DumSidTas17}.

Our proof goes via developing a random walk representation for long clusters in the Ashkin--Teller model.
We proceed by introducing the latter model and stating our results for it.

\subsection{Ashkin--Teller model}
\label{sec:ATdef}

Introduced in 1943~\cite{AshTel43} as a generalization of the Ising model to a four-component system, the Ashkin--Teller (AT) model can be viewed as a pair of interacting Ising models.
Take a finite subgraph $G=(V,E)$ of $\Z^2$, parameters $J,U>0$ and boundary conditions $\sigma,\sigma'\in\{0,\pm1\}^{\Z^2}$. 
The AT model on~$G$ with parameters~$J,U\in\R$ and boundary conditions~$(\sigma,\sigma')$ is a probability measures on pairs~$\tau,\tau'\in \{\pm 1\}^{V}$ given by
\begin{align*}
	{\at}_{G,J,U}^{\sigma, \sigma'}(\tau,\tau') 
	= \tfrac{1}{Z_\at} \cdot \exp 
	 \biggl[\sum_{\{i,j\}\in E} J(\tau_i\tau_j &+ \tau'_i\tau'_j) + U\tau_i\tau_j\tau'_i\tau'_j\\
	 + \sum_{i\in V,j\in V^c\colon i\sim j} J(\tau_i\sigma_j &+ \tau'_i\sigma'_j) + U\tau_i\sigma_j\tau'_i\sigma'_j \biggr]\,
\end{align*}
where~$Z_\at=Z_\at(G,J,U,\sigma,\sigma')$ is the partition function.
Taking~$\sigma= \sigma' \equiv 1$ we obtain plus-plus boundary conditions and taking~$\sigma= \sigma' \equiv 0$ we obtain free-free boundary conditions; we denote the corresponding measures by~${\at}_{G,J,U}^{+,+}$ and~${\at}_{G,J,U}^{\mathrm{f},\mathrm{f}}$ respectively.

We consider only~$J>0$, since flipping the sign of~$J$ corresponds to flipping the sign of~$\tau'$ and~$\tau$ at one of the two partite classes of~$\Z^2$.
Using the Ising-duality for~$\tau'$ and then for~$\tau$, the {\em self-dual curve} of the parameters was identified~\cite{MitSte71}:
\begin{equation}\label{eq:at-sd}
	\sinh 2J = e^{-2U}.
\end{equation}
A classical GKS inequality~\cite{KelShe68} implies that the correlations are monotone along the lines of a constant ratio~$J/U$.
The case $U=0$ gives two independent Ising models and the line~$J=U$ is in direct correspondence with the four-state Potts model and the FK percolation with the cluster-weight~$q=4$.
In addition, the three models are related on the self-dual line~\eqref{eq:at-sd}, with~$q>4$ corresponding to~$U>J$.
A precise coupling was constructed in~\cite{GlaPel23} via the six-vertex (square ice) model based on the works of Fan~\cite{Fan72} and Wegner~\cite{Weg72} and on the seminal Baxter--Kelland--Wu (BKW) correspondence~\cite{BaxKelWu76}; see also~\cite{HuaDenJacSal13}.
This coupling is crucial to the current work and is described in detail in Sections~\ref{sec:coupling}~and~\ref{sec:bkw}.

When~$U>J>0$, correlation inequalities~\cite{KelShe68} guarantee the existence of the infinite-volume limits $\at_{J,U}^{\mathrm{f},\mathrm{f}}$ and $\at_{J,U}^{+,+}$ of the free and monochromatic AT measures $\at_{G_n,J,U}^{\mathrm{f},\mathrm{f}}$ and $\at_{G_n,J,U}^{+,+}$, respectively, as~$G_n\nearrow \Z^2$.
Our first result states that their marginals on the single spin $\tau$ coincide.
  
\begin{proposition}
	Let $0<J<U$ satisfy $\sinh 2J=e^{-2U}$.
	Then, the measures~$\at_{J,U}^{\rmf,\rmf}$ and~$\at_{J,U}^{+,+}$ have the same marginal distribution on~$\tau$.
\end{proposition}

Denote the expectation operators with respect to $\at_{J,U}^{\mathrm{f},\mathrm{f}}$ and $\at_{J,U}^{+,+}$ by $\langle\,\cdot\,\rangle_{J,U}^{\mathrm{f},\mathrm{f}}$ and $\langle\,\cdot\,\rangle_{J,U}^{+,+}$, respectively, and define the inverse correlation length $\nu_{J,U}$ by setting, for $x\in\R^2$,
\begin{equation*}
\nu_{J,U}(x):=-\lim_{n\to\infty}\tfrac{1}{n}\ln\langle\tau_0\tau_{\lfloor nx \rfloor}\rangle_{J,U}^{\mathrm{f},\mathrm{f}}=-\lim_{n\to\infty}\tfrac{1}{n}\ln\langle\tau_0\tau_{\lfloor nx \rfloor}\rangle_{J,U}^{+,+},
\end{equation*}
where $\lfloor\,\cdot\,\rfloor$ is the componentwise integer part.
The existence of the limit is derived in a standard manner from correlation inequalities and a subadditive argument (see Section~\ref{sec:ATRC:RW_infinite_volume} for more details and references).
In~\cite{AouDobGla24}, it was shown that $\at_{\Lambda_n}^{+,+}$ admits exponential decay of correlations in $\tau$, which implies that $\nu>0$.
The following theorem establishes sharp Ornstein--Zernike-type asymptotics for the 2-point function.

\begin{theorem}
\label{thm:oz-at}
Let $0<J<U$ satisfy $\sinh 2J=e^{-2U}$.
The inverse correlation length $\nu_{J,U}$ is a norm on $\R^2$. Furthermore, uniformly in $|x|\to\infty$, 
\begin{equation*}
\langle\tau_0\tau_x\rangle_{J,U}^{\mathrm{f},\mathrm{f}}=\langle\tau_0\tau_x\rangle_{J,U}^{+,+}=\tfrac{g(x/|x|)}{\sqrt{|x|}}\,e^{-\nu(x)}\,(1+o(1)),
\end{equation*}
where $\nu = \nu_{J,U}$, and $g =g_{J,U}$ is a strictly positive analytic function on $\mathbb{S}^1$.
\end{theorem}

\subsection{Six-vertex model}
The six-vertex model~\cite{Pau35,Rys63} is a model of \emph{edge orientations} on a square lattice, which in the setting of this article is taken to be the medial graph of~\(\Z^2\).
While it serves merely as an intermediate step and a tool in the derivation of our main results, Theorems~\ref{thm:potts}-\ref{thm:oz-at}, we also obtain some results on the six-vertex model along the way, which we will \emph{informally} describe in this subsection. We refer to Sections~\ref{sec:combinatorial_mappings}~and~\ref{sec:height-func} for precise definitions.

The six-vertex model admits a representation in terms of \emph{height functions}. In the above setting, these are assignments of integers to the vertices of both~\(\Z^2\) and~\((\Z^2)^*\), satisfying the following property. A six-vertex height function~\(h:\Z^2\cup(\Z^2)^*\to\Z\) takes even values on~\(\Z^2\) and, for all~\(i\in\Z^2\) and~\(u\in(\Z^2)^*\) at a Euclidean distance of~\(1/\sqrt{2}\), it holds that~\(|h(i)-h(u)|=1\).
An edge~\(e=\{i,j\}\in\bbE\) (or its associated tile) with dual~\(e^*=\{u,v\}\) is said to be of type~5-6 if both~\(h(i)=h(j)\) and~\(h(u)=h(v)\), and otherwise of type~1-4.
Given a finite set~\(E\subset\bbE\) and a parameter~\(\svc>0\), each height function is assigned a weight given by 
\begin{equation*}
\mathrm{w}_{E;\svc}(h):=\svc^{|\{e\in E\text{ of type 5-6}\}|}.
\end{equation*}
Through the introduction of boundary conditions, measures on the set of height functions are defined by a proportionality relation to the above weight.
When~\(1\leq \svc\leq 2\), the height function under constant boundary condition is known to \emph{delocalise}~\cite{DumKarManOul20} (see also~\cite{Lis21,GlaLam23}) as the system exhausts the whole lattice, whereas it was shown to \emph{localise}~\cite{DumGagHar21,GlaPel23} when~\(\svc>2\).

In relation to the Potts and FK measures under Dobrushin boundary conditions, we consider six-vertex height function measures with parameter~\(\svc>2\) on square domains and with the following Dobrushin boundary condition (see Fig.~\ref{fig:six-v_hf_dobrushin}):~\(+1\) on~\((\Z^2)^*\) in the upper half-plane,~\(-1\) on~\((\Z^2)^*\) in the lower half-plane, and constant~\(0\) on~\(\Z^2\). 
In analogy to Theorems~\ref{thm:potts}-\ref{thm:order-disorder:FK_loop}, we derive the convergence under these height function measures of the Peierls contour between~\(+1\) and~\(-1\) in the odd heights under diffusive scaling. 
We prove this statement for measures with a modified parameter~\(\svcb\) on the boundary, as they can be coupled with the FK and Potts measures, which allows us to transfer the result to the latter and derive Theorems~\ref{thm:potts}-\ref{thm:order-disorder:FK_loop}. However, it should be noted that this boundary weight in fact complicates the analysis in the related AT model. The derivation of the statement for the measures without the modified boundary weight is simpler and can be undertaken using the same method.

\begin{figure}
	\includegraphics[scale=0.4]{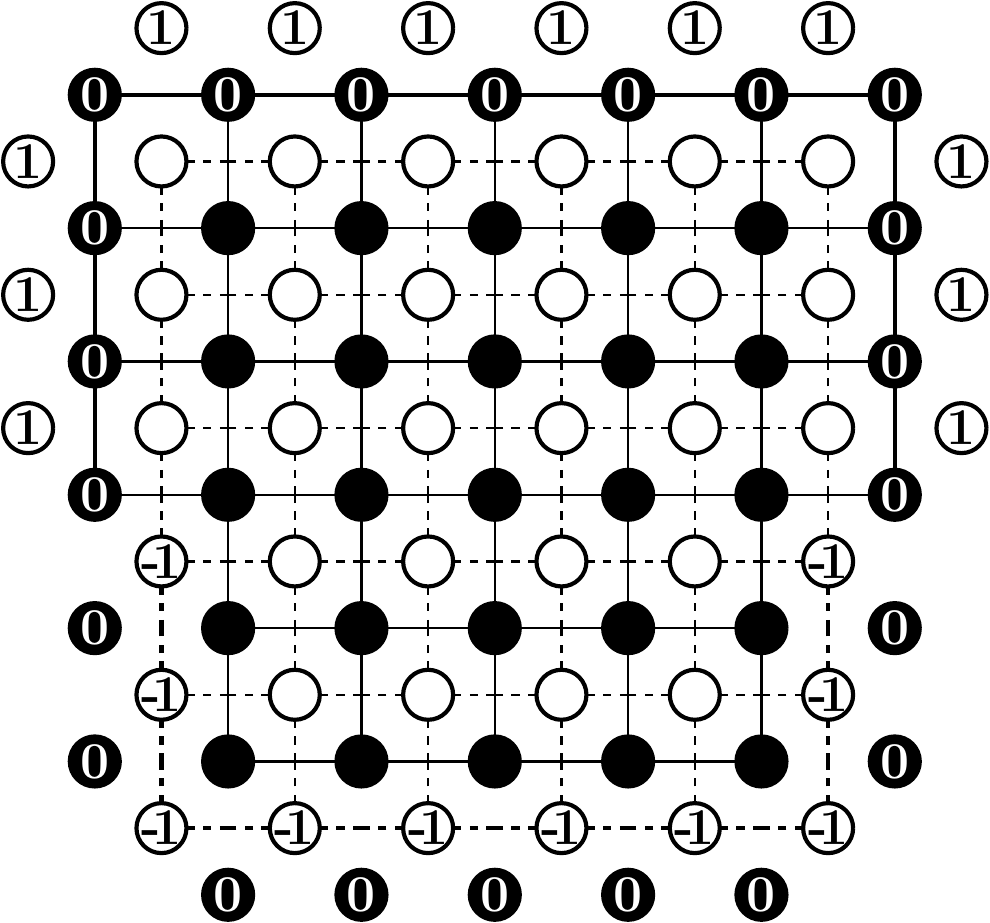}
	\caption{Dobrushin boundary condition for the six-vertex height function.}
	\label{fig:six-v_hf_dobrushin}
\end{figure}

It is also noteworthy that we establish exponential relaxation results for the height function measures with constant boundary conditions and potentially inhomogeneous weights~\(\svc\) and~\(\svcb\) at the boundary (see Section~\ref{sec:height-func-relax}).

\subsection{Summary of the paper: what is new?}

The main novelty of our work is to study the FK percolation and the Potts models via the graphical (or random-cluster) representation of the Ashkin--Teller model (ATRC model) introduced in~\cite{PfiVel97}.
At a first glance, the FK percolation model has been by now much better understood, see the classical book~\cite{Gri06} and more recent lecture notes~\cite{Dum17a}.
Moreover, the FK model has been used to establish some basic properties for the ATRC model~\cite{GlaPel23,AouDobGla24}.
However, compared to the FK percolation at its transition point when~$q>4$, the ATRC model at its self-dual line when~$U>J$ remarkably exhibits a unique Gibbs measure.
This brings more symmetries that play a key role in the study of interfaces.

This comes at some cost: the ATRC model is supported on {\em pairs} of edge configurations.
Thus, the domain Markov property is significantly weaker than in the FK percolation defined on a single edge configuration.
In particular, it is highly non-trivial to prove mixing properties of the ATRC model.
For this we use a coupling of the ATRC model to the six-vertex model, whose height function enjoys additional monotonicity properties.
Using the classical work of Alexander~\cite{Ale98} and a new general mixing result~\cite{Ott25}, we prove, for the regime of parameters considered in the rest of the work, that the ATRC model satisfies
\begin{itemize}
	\item the exponential ratio weak mixing property (Theorem~\ref{thm:ratio_weak_mixing_ATRC}),
	\item a restricted version of the exponential strong mixing property (Theorem~\ref{thm:strong_mixing_atrc}).
\end{itemize}

These mixing properties allow us to extend the arguments of~\cite{CamIofVel03,CamIofVel08,OttVel18, AouOttVel24} to the ATRC model and provide a renewal picture for the long subcritical ATRC clusters in infinite volume (Theorem~\ref{thm:OZ_for_ATRC_infinite_vol}).

Regarding the link with interfaces in FK/Potts, we point out a technical issue that the FK percolation model under standard Dobrushin boundary conditions is directly coupled only to the ATRC with rather involved ``boundary conditions'' (modified cluster weight on the boundary, we call it a modified ATRC model).
Fortunately, this model still satisfies the FKG lattice condition which allows for comparison with normal ATRC measures. Using this, we are able to transfer the results derived for the infinite volume ATRC to the modified, finite volume, measures. This is the content of Section~\ref{sec:invariance_principle}.

\subsubsection*{Organisation of the article.}

Section~\ref{sec:notations}: notation and conventions that will be used throughout the article.  

Section~\ref{sec:atrc}: definition of the ATRC model, its basic properties and statement of the results that we establish for this model, including the uniqueness of the ATRC Gibbs measure and mixing properties.

Section~\ref{sec:coupling:interfaces}: coupling between the FK percolation and modified ATRC models.
The coupling is very sensitive to boundary conditions, eg. note appearance of a different boundary-cluster weight in~\cite{GlaPel23}.
We extend this coupling to standard Dobrushin boundary conditions in the FK percolation at a price of rather inconvenient conditions in the ATRC model. In particular, we obtain a coupling between the FK interface and a crossing cluster under a modified version of a finite volume ATRC measure.

Section~\ref{sec:mixing}: proof of new weak mixing properties (Theorem~\ref{thm:ratio_weak_mixing_ATRC}) for the ATRC model and, in particular, uniqueness of the ATRC Gibbs measure. We adapt the classical works of Alexander~\cite{Ale92,Ale98,Ale04} and use the recent works on the ATRC model~\cite{GlaPel23,AouDobGla24}.

Section~\ref{sec:strong_mixing}: derivation of a restricted version of strong mixing (Theorem~\ref{thm:strong_mixing_atrc}) from the weak mixing of Section~\ref{sec:mixing}, using a general argument from~\cite{Ott25}.

Section~\ref{sec:ATRC:RW_infinite_volume}: development of the ``OZ theory'' (renewal picture) for long subcritical cluster in the ATRC model, in infinite volume, following~\cite{CamIofVel08, AouOttVel24}, and using the mixing established in Sections~\ref{sec:mixing} and~\ref{sec:strong_mixing}.
More precisely, we show that the typical geometry of clusters conditioned on containing two distant points is the same as the one of a directed random walk bridge between these points (Theorem~\ref{thm:OZ_for_ATRC_infinite_vol}).

Section~\ref{sec:invariance_principle}: using stochastic comparison and mixing properties, we couple the ``crossing cluster'' under the finite volume measures of Section~\ref{sec:coupling} with an infinite volume long cluster as studied in Section~\ref{sec:ATRC:RW_infinite_volume} (Theorem~\ref{thm:main_Inv_principle_coupling_with_RW}). From there, an invariance principle follows by importing the results of~\cite{Kov04}.

Section~\ref{sec:proofs_thms}: wrap up of the proofs of Theorems~\ref{thm:potts},~\ref{thm:order-disorder:FK_loop},~\ref{thm:oz-at}, and~\ref{thm:oz-atrc}.

\subsubsection*{Acknowledgements}
We want to say big thanks to Ioan Manolescu for frequent discussions that started already in 2019-20 and for pointing out that it is enough to work with tall rectangles. Parts of the work were accomplished during the visits of some of us to the Universities of Fribourg, Geneva, Innsbruck and Vienna and at the NCCR SwissMAP research station in Les Diablerets. We want to thank these institutions and all people involved for their hospitality and creating excellent working conditions.
This research was partially funded by the Austrian Science Fund (FWF) 10.55776/P34713 and 10.55776/AST1405524.

\subsubsection*{AI disclosure}
AI/LLMs have not been used in this work.

\section{Notations and conventions}
\label{sec:notations}

\vspace{5pt}

\noindent\textbf{General graphs.}
Let $\bbG=(\bbV,\bbE)$ be a graph. We simply write $xy=\{x,y\}$ for an edge $\{x,y\}\in\bbE$.
Given finite subsets $\Lambda\subset \bbV$ and $E\subset\bbE$, define
\begin{gather*}
	\bbE_\Lambda:=\{e\in\bbE :e\subset \Lambda\},\quad
	\bbV_E:=\bigcup_{e\in E}e,\\
	\partialin\Lambda:=\{x\in \Lambda:\exists y\in\bbV\setminus \Lambda,\,xy\in\bbE\},\quad
	\partialex\Lambda:=\{y\in \bbV\setminus \Lambda:\exists x\in \Lambda,\,xy\in\bbE\},\\
	\partialin E:=\{e\in E:\exists f\in\bbE\setminus E, e\cap f\neq\varnothing\},\quad
	\partialex E:=\{e\in\bbE\setminus E:\exists f\in E, e\cap f\neq\varnothing\},\\
	\partialedge\Lambda:=\{xy\in\bbE:x\in\Lambda,y\in\bbV\setminus\Lambda\}.
\end{gather*}
In case of ambiguity, we add \(\bbG\) as a subscript (and write, for example, \(\partialin_\bbG\Lambda\)) to emphasise that the boundary is taken in \(\bbG\).
The \emph{interior} of \(\Lambda\) (in \(\bbG\)) is given by \(\Lambda\setminus\partialin\Lambda\).
The subgraph \emph{induced} by $\Lambda\subset\bbV$ is given by $(\Lambda,\bbE_\Lambda)$, and the subgraph induced by \(E\subset\bbE\) is given by \((\bbV_E,E)\).
\medskip

\noindent\textbf{Lattices.}
We will mainly work on \(\Z^2\) with nearest-neighbour edges, and on its dual. We will denote the \emph{primal lattice} by \(\bbL_{\bullet} = \Z^2\) and its {\em dual} by \(\bbL_{\circ}= (1/2,1/2)+\Z^2\). Denote by \(\bbE^{\bullet}\) the nearest-neighbour edges between sites in \(\bbL_{\bullet}\) (primal edges), by \(\bbE^{\circ}\) the nearest-neighbour edges between sites in \(\bbL_{\circ}\) (dual edges).

\medskip

\noindent\textbf{Duality.}
Each edge~\(e\in\bbE^{\bullet}\) intersects a unique edge of~\(\bbE^{\circ}\), we denote it by~\(e^*\).
For a set of primal (or dual) edges \(E\), define \(*E = \{e^*:\ e\in E\}\).
We say that $E\subset\bbE^\bullet$ is \emph{simply lattice-connected} if both the subgraphs induced by $E$ and by $*(\bbE^\bullet\setminus E)$ are connected.
As a convention, sets of edges or of dual edges will be identified with the corresponding sets of mid-points whenever the meaning is clear from the context.

\medskip

\noindent\textbf{Tiles.}
To each primal-dual pair of edges \(e,e^*\), associate a {\em tile} \(t\) given by the convex hull of their endpoints and define \(e_t:=e\); see Fig.~\ref{fig:midEdgeTiles}.
Define~\(\bbL_{\diamond}\) as the set of all tiles, and let~\(\bbE^\diamond\) be the set of all pairs of adjacent tiles.
Note that~\((\bbL_{\diamond},\bbE^{\diamond})\) is the medial graph of~\((\bbL_{\bullet},\bbE^{\bullet})\), since tiles can be identified with midpoints of edges in~\(\bbE^{\bullet}\).

\medskip

\noindent\textbf{Standard rectangular domains.}
Define the upper and lower half planes by 
\begin{equation*}
	\bbH^+ := \R\times \R_{\geq 0},\quad \bbH^- := \R\times \R_{<0},
\end{equation*}
and, for \(n,m\geq 0\), set (see Fig.~\ref{fig:FK_tilesDomain})
\begin{align*}
	\Lambda_{n,m} &:= \{-n, \dots, n\}\times\{-m, \dots, m\},\\
	\Lambda'_{n,m} &:=\Big(\big([-n-1,n+1]\times [0,m+1]\big)\cup\big([-n,n]\times[-m,0]\big)\Big)\cap\bbL_\circ.
\end{align*}

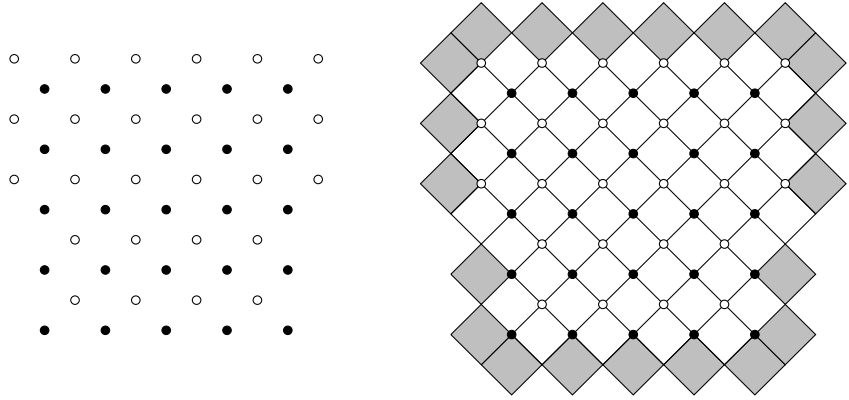
\begin{figure}
	\ifpic
	\begin{tikzpicture}[scale=0.8]
		\foreach \i in {-2,...,2}{
			\foreach \j in {-2,...,2}{
				\filldraw[fill=black] (\i,\j) circle(2pt);
			}
		}
		\foreach \i in {-3,...,2}{
			\foreach \j in {0,...,2}{
				\filldraw[fill=white] ({\i+0.5},{\j+0.5}) circle(2pt);
			}
		}
		\foreach \i in {-2,...,1}{
			\foreach \j in {-2,-1}{
				\filldraw[fill=white] ({\i+0.5},{\j+0.5}) circle(2pt);
			}
		}
		{\filldraw[white](0,-3) circle(2pt);}
	\end{tikzpicture}
	\hspace{1cm}
	\begin{tikzpicture}[scale=0.8]
		\foreach \i\j in {-2.5/3, -1.5/3, -0.5/3, 0.5/3, 1.5/3, 2.5/3, -2/-2.5, -1/-2.5, 0/-2.5, 1/-2.5, 2/-2.5, -3/0.5, -3/1.5, -3/2.5, 3/0.5, 3/1.5, 3/2.5, -2.5/-1, -2.5/-2, 2.5/-1, 2.5/-2}{
			\filldraw[fill=lightgray] ({\i-0.5},\j)--(\i,{\j+0.5})--({\i+0.5},\j)--(\i,{\j-0.5})--({\i-0.5},\j);
		}
		
		\foreach \i in {-2.5,...,2.5}{
			\foreach \j in {0,...,2}{
				\draw ({\i-0.5},\j)--(\i,{\j+0.5})--({\i+0.5},\j)--(\i,{\j-0.5})--({\i-0.5},\j);
			}
		}
		
		\foreach \i in {-1.5,...,1.5}{
			\foreach \j in {-2,-1}{
				\draw ({\i-0.5},\j)--(\i,{\j+0.5})--({\i+0.5},\j)--(\i,{\j-0.5})--({\i-0.5},\j);
			}
		}

		\foreach \i in {-2,...,2}{
			\foreach \j in {-2,...,2}{
				\filldraw[fill=black] (\i,\j) circle(2pt);
			}
		}
		
		\foreach \i in {-3,...,2}{
			\foreach \j in {0,...,2}{
				\filldraw[fill=white] ({\i+0.5},{\j+0.5}) circle(2pt);
			}
		}
		\foreach \i in {-2,...,1}{
			\foreach \j in {-2,-1}{
				\filldraw[fill=white] ({\i+0.5},{\j+0.5}) circle(2pt);
			}
		}
	\end{tikzpicture}
	\fi
	\caption{Left: the sets $\Lambda_{2,2}$ (solid) and $\Lambda'_{2,2}$ (hollow). Right: the inner tiles \(A^\rmi_{2,2}\) (white) and the boundary tiles \(\partial A_{2,2}\) (grey).}
	\label{fig:FK_tilesDomain}
\end{figure}

Define \(\mathcal{D}_{n,m}:=\Lambda_{n,m}\cup\Lambda'_{n,m}\) and the graph \(G_{n,m}=(V_{n,m},E_{n,m})\) by (see Fig.~\ref{fig:WiredFree_FK})
\begin{equation*}
V_{n,m}=\Lambda_{n,m}\cup\big(\partialex\Lambda_{n,m} \cap \bbH^{+}\big),\quad E_{n,m}=\bbE_{V_{n,m}}\setminus\bbE_{\bbL_\bullet\setminus\Lambda_{n,m}}.
\end{equation*}
Let~\(A_{n,m}\subset\bbL_\diamond\) be the set of tiles with at least one corner in $\mathcal{D}_{n,m}$ and~\(\partial A_{n,m}\subset A_{n,m}\) (boundary tiles) be the set of tiles with precisely one corner in \(\mathcal{D}\); see Fig.~\ref{fig:FK_tilesDomain}.
Define also~\(\partial^+ A=\{t\in\partial A:t\subseteq\bbH^+\},\ \partial^- A=\partial A\setminus\partial^+ A\) and~\(A_{n,m}^\rmi:= A_{n,m}\setminus\partial A_{n,m}\).

Define also (see Fig.~\ref{fig:K} and~\ref{fig:bc_FK_loops_spins})
\begin{equation*}
	v_L(n):= (-n-1,0),\quad v_R(n):=(n+1,0).
\end{equation*}

\medskip

\noindent\textbf{Augmented rectangular domains.}
We also introduce~\(E_{n,m}\) augmented by some \emph{upper} boundary edges. First, define the set of boundary edges and its upper and lower parts by
\begin{equation}\label{eq:def-e-b-pm}
	E_{\rmb,n,m}:=\{e_t:t\in\partial A_{n,m}\}
	\qquad\text{and}\qquad
	E_{\rmb,n,m}^\pm=\{e_t:t\in\partial^\pm A_{n,m}\}. 
\end{equation}
Then define the augmented sets of edges and vertices by
\begin{equation*}
\bar{E}_{n,m}=E_{n,m} \cup E_{\rmb,n,m}^+\qquad\text{and}\qquad\bar{V}_{n,m}=\bbV_{\bar{E}_{n,m}}.
\end{equation*}
We now define the augmented domain and its dual (see Fig.~\ref{fig:K}):
\begin{itemize}
	\item Set \(K_{n,m}:=(\bar{V}_{n,m},\bar{E}_{n,m})\),
	\item and let \(K^*_{n,m}\) be the graph obtained from \((\bbV_{*\bar{E}_{n,m}},*\bar{E}_{n,m})\) by identifying the vertices in \(\partialin_{\bbL_\circ}\bbV_{*\bar{E}_{n,m}}\).
\end{itemize}

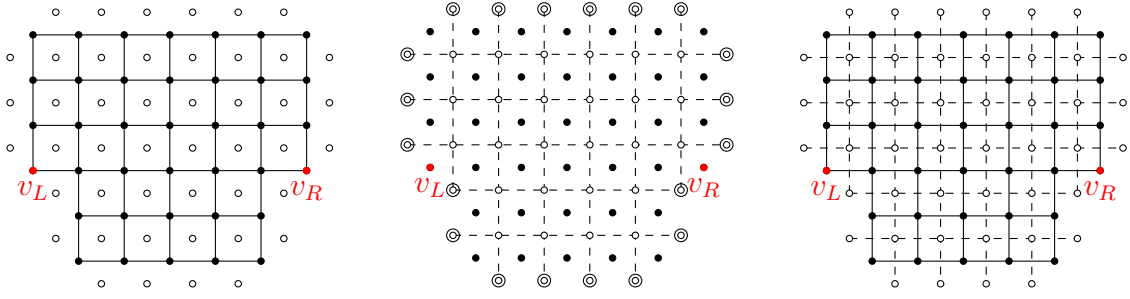
\begin{figure}
	\ifpic
	\begin{tikzpicture}[scale=0.6]
		\foreach \j in {0,...,3}{
			\draw[black] (-3,\j)--(3,\j);
		}
		\foreach \j in {-2,...,-1}{
			\draw[black] (-2,\j)--(2,\j);
		}
		\foreach \i in {-2,...,2}{
			\draw[black] (\i,3)--(\i,-2);
		}
		\draw[black] (-3,3)--(-3,0);
		\draw[black] (3,3)--(3,0);
		
		\foreach \i in {-3,...,3}{
			\foreach \j in {0,...,3}{
				\filldraw[fill=black] (\i,\j) circle(2pt);
			}
		}
		\foreach \i in {-2,...,2}{
			\foreach \j in {-2,...,-1}{
				\filldraw[fill=black] (\i,\j) circle(2pt);
			}
		}
		\foreach \i in {-4,...,3}{
			\foreach \j in {0,1,2}{
				\filldraw[fill=white] ({\i+0.5},{\j+0.5}) circle(2pt);
			}
		}
		\foreach \i in {-3,...,2}{
			\foreach \j in {-1,-2,3}{
				\filldraw[fill=white] ({\i+0.5},{\j+0.5}) circle(2pt);
			}
		}
		\foreach \i in {-2,...,1}{
			\filldraw[fill=white] ({\i+0.5},{-3+0.5}) circle(2pt);
		}
		\filldraw[red] (-3,0) circle(2pt) node[below]{$v_L$};
		\filldraw[red] (3,0) circle(2pt) node[below]{$v_R$};
	\end{tikzpicture}
	\hspace{0.6cm}
	\begin{tikzpicture}[scale=0.6]
		\foreach \i in {-2,...,1}{
			\draw[dashed] ({\i+0.5},3.5)--({\i+0.5},-2.5);
		}
		\draw[dashed] (-2.5,-0.5)--(-2.5,3.5);
		\draw[dashed] (2.5,-0.5)--(2.5,3.5);
		\foreach \j in {0.5,1.5,2.5}{
			\draw[dashed] (-3.5,\j)--(3.5,\j);
		}
		\foreach \j in {-0.5,-1.5}{
			\draw[dashed] (-2.5,\j)--(2.5,\j);
		}
		
		\foreach \i in {-3,...,3}{
			\foreach \j in {0,...,3}{
				\filldraw[fill=black] (\i,\j) circle(2pt);
			}
		}
		\foreach \i in {-2,...,2}{
			\foreach \j in {-2,...,-1}{
				\filldraw[fill=black] (\i,\j) circle(2pt);
			}
		}
		\foreach \i in {-4,...,3}{
			\foreach \j in {0,1,2}{
				\filldraw[fill=white] ({\i+0.5},{\j+0.5}) circle(2pt);
			}
		}
		\foreach \i in {-3,...,2}{
			\foreach \j in {-1,-2,3}{
				\filldraw[fill=white] ({\i+0.5},{\j+0.5}) circle(2pt);
			}
		}
		\foreach \i in {-2,...,1}{
			\filldraw[fill=white] ({\i+0.5},{-3+0.5}) circle(2pt);
		}
		\filldraw[red] (-3,0) circle(2pt) node[below]{$v_L$};
		\filldraw[red] (3,0) circle(2pt) node[below]{$v_R$};
		
		\foreach \i in {-2,...,1}{
			\draw (\i+0.5,3.5) circle(4pt);
			\draw (\i+0.5,-2.5) circle(4pt);
		}
		\foreach \j in {0,...,2}{
			\draw (-3.5,\j+0.5) circle(4pt);
			\draw (3.5,\j+0.5) circle(4pt);
		}
		\foreach \j in {-1.5,-0.5,3.5}{
			\draw (-2.5,\j) circle(4pt);
			\draw (2.5,\j) circle(4pt);
		}
	\end{tikzpicture}
	\hspace{0.6cm}
	\begin{tikzpicture}[scale=0.6]
		\foreach \j in {0,...,3}{
			\draw[black] (-3,\j)--(3,\j);
		}
		\foreach \j in {-2,...,-1}{
			\draw[black] (-2,\j)--(2,\j);
		}
		\foreach \i in {-2,...,2}{
			\draw[black] (\i,3)--(\i,-2);
		}
		\draw[black] (-3,3)--(-3,0);
		\draw[black] (3,3)--(3,0);
		
		\foreach \i in {-2,...,1}{
			\draw[dashed] ({\i+0.5},3.5)--({\i+0.5},-2.5);
		}
		\draw[dashed] (-2.5,-0.5)--(-2.5,3.5);
		\draw[dashed] (2.5,-0.5)--(2.5,3.5);
		\foreach \j in {0.5,1.5,2.5}{
			\draw[dashed] (-3.5,\j)--(3.5,\j);
		}
		\foreach \j in {-0.5,-1.5}{
			\draw[dashed] (-2.5,\j)--(2.5,\j);
		}
		
		\foreach \i in {-3,...,3}{
			\foreach \j in {0,...,3}{
				\filldraw[fill=black] (\i,\j) circle(2pt);
			}
		}
		\foreach \i in {-2,...,2}{
			\foreach \j in {-2,...,-1}{
				\filldraw[fill=black] (\i,\j) circle(2pt);
			}
		}
		\foreach \i in {-4,...,3}{
			\foreach \j in {0,1,2}{
				\filldraw[fill=white] ({\i+0.5},{\j+0.5}) circle(2pt);
			}
		}
		\foreach \i in {-3,...,2}{
			\foreach \j in {-1,-2,3}{
				\filldraw[fill=white] ({\i+0.5},{\j+0.5}) circle(2pt);
			}
		}
		\foreach \i in {-2,...,1}{
			\filldraw[fill=white] ({\i+0.5},{-3+0.5}) circle(2pt);
		}
		\filldraw[red] (-3,0) circle(2pt) node[below]{$v_L$};
		\filldraw[red] (3,0) circle(2pt) node[below]{$v_R$};
	\end{tikzpicture}
	\fi
	\caption{Left: the graph \(K_{2,2}\). Center: the graph \(K^*_{2,2}\) (vertices surrounded by a circle are identified). Right: planar duality relation between their edges.}
	\label{fig:K}
\end{figure}

\medskip

\noindent\textbf{Connectivity events.}
Given $\Lambda,\Delta\subset\bbL_\bullet$ and a percolation configuration~$\omega \in \{0,1\}^{\bbE^\bullet}$, we write $\Lambda\xleftrightarrow{\omega}\Delta$ for the event that~$\Lambda$ and $\Delta$ are connected by a path in the graph~$(\bbL_\bullet,\omega)$.
If~$\Lambda=\{i\}$ and~$\Delta=\{j\}$, we simply write~$i \xleftrightarrow{\omega} j$. 
We omit $\omega$ from the notation when it cannot lead to any confusion. 	

\medskip

\noindent\textbf{Agreements of spins along edges.}
For an edge~$e=ij\in \mathbb{E}^\bullet$ and a spin configuration~$\sigma_\bullet\in \{+1,-1\}^{\bbL_\bullet}$, we write~$\sigma_\bullet \sim e$ if~$\sigma_\bullet(i)=\sigma_\bullet(j)$;
for~$\xi\subset\mathbb{E}^\bullet$, we write~$\sigma_\bullet \sim \xi$ if~$\sigma_\bullet\sim e$ for every~$e\in\xi$ (in other words, $\sigma_\bullet$ is constant on clusters of~$\xi$).
We use similar notation for edges in~$\mathbb{E}^\circ$ and~$\sigma_\circ\in \{+1,-1\}^{\bbL_\circ}$.

\medskip

\noindent\textbf{Parameters.}
The couplings go through standard ``expansion--resummation'' of Boltzmann weights combined with extensive use of planarity. As for each step one will write the weight associated with a given model as a sum of weights for a ``more expanded'' model, several parameters will come into play.
Below we list the parameters, as well as the algebraic relations linking them:
\begin{gather}
	q>4,\quad \beta = \beta_c(q) = \ln(1+\sqrt{q}),\quad \lambda>0,\quad \svc>2,\quad U>J>0,\nonumber\\
	\sqrt{q} = e^{\lambda} + e^{-\lambda},\quad \svc = e^{\lambda/2} + e^{-\lambda/2} = \coth(2J),\label{eq:parameters_bulk}\\
	\sinh(2J) = e^{-2U}.\nonumber
\end{gather}
The parameter \(\svc\) also has a ``boundary version'':
\begin{equation}
	\svcb = e^{\lambda/2}>1.\label{eq:parameters_bnd}
\end{equation}
For the remainder of the article, we fix \( q > 4 \), along with the corresponding parameters above (which are uniquely determined by \( q \)).

\medskip

\noindent\textbf{Constants.}
Constants like \(c,c_1,C,C_1,C',\dots\) are constants which can change from line to line and which can depend on the parameters unless explicitly stated. They are independent of the system size, \(n\), which will be our main ``variable'' quantity.

\medskip

\noindent\textbf{Measures.}
For a probability measure \(P\) and an event \(A\) with \(P(A)>0\), we denote \(P(\,\cdot\, \given A)\) the probability measure \(P\) conditioned on \(A\): \(B\mapsto P(B\given A)\). Also, we denote \(P(\,\cdot\, ; A)\) the measure \(B\mapsto P(B\cap A)\). We also refer the reader to Appendix~\ref{app:tot_var_dist} for reminders about the total variation distance, that we denote \(\tvd\), and some of its basic properties.

\section{Ashkin--Teller random-cluster model}
\label{sec:atrc}

Like the Potts models, the AT model has a random-cluster (RC) representation, called the {\em ATRC model}, and introduced in~\cite{ChaMac97,PfiVel97}. We will first introduce the model and state some of its basic properties. This will be followed by the statement of our results.

\subsection{Definition and basic properties}

Let $\bbG=(\bbV,\bbE)$ be a graph, and let $G=(V,E)$ be a finite subgraph of $\bbG$.
Let~$\eta_\tau,\eta_{\tau\tau'}\in\{0,1\}^{\bbE}$ with $\eta_\tau\subseteq\eta_{\tau\tau'}$.
The ATRC model on~$G$ with parameters~$U>J>0$ under boundary conditions~$(\eta_\tau,\eta_{\tau\tau'})$ is the probability measure on $\{0,1\}^{\bbE}\times\{0,1\}^{\bbE}$ given by 
\begin{multline}
\label{eq:atrc_def}
	\atrc_{G;J,U}^{\eta_\tau,\eta_{\tau\tau'}}(\omega_\tau,\omega_{\tau\tau'})
	= \tfrac{1}{Z} \cdot \ind_{\omega_\tau\subseteq\omega_{\tau\tau'}} \cdot  \rmw_\tau^{\abs{\omega_\tau\cap E}}\, \rmw_{\tau\tau'}^{\abs{(\omega_{\tau\tau'}\setminus\omega_\tau)\cap E}}\,2^{\clusters_{V}(\omega_\tau)+\clusters_{V}(\omega_{\tau\tau'})}\\
	\cdot  \prod_{e\in \bbE\setminus E}\,\ind_{\omega_\tau(e)=\eta_\tau(e)}\ind_{\omega_{\tau\tau'}(e)=\eta_{\tau\tau'}(e)},
\end{multline}
where~$Z=Z_{\atrc}^{\eta_\tau,\eta_{\tau\tau'}}(G,J,U)$ is the partition function, $\clusters_{V}(\cdot)$ is the number of clusters that intersect $V$, and the weights are given by
\begin{equation}\label{eq:atrc_weights}
	\rmw_\tau=e^{2U}(e^{2J}-e^{-2J})\quad\text{and}\quad \rmw_{\tau\tau'}=e^{2(U-J)}-1.
\end{equation}
When $\eta_\tau\equiv 0$ (resp. $\eta_\tau\equiv 1$), we write $\atrcfree$ (resp. $\atrcwired$) instead of $\eta_\tau$ in the superscript, and analogously for $\eta_{\tau\tau'}$.
Given finite subsets $\Lambda\subset\bbV$ and $E\subset\bbE$, we write $\atrc_{\Lambda;J,U}^{\eta_\tau,\eta_{\tau\tau'}}$ and $\atrc_{E;J,U}^{\eta_\tau,\eta_{\tau\tau'}}$ for the measures on the graphs $(\Lambda,\bbE_\Lambda)$ and $(\bbV_E,E)$, respectively.
\medskip

\noindent\textbf{Finite energy.} There exists a constant $c=c(J,U)>0$ such that the following holds. For any finite subgraph $G=(V,E)$, any $e\in E$ and any $a,b\in\{0,1\}^{E}\) with \(a\subseteq b\)
\begin{equation}
	\label{eq:fe_atrc}\tag{FE}
	\atrc_{G;J,U}^{\eta_\tau,\eta_{\tau\tau'}}\big((\omega_\tau(e),\omega_{\tau\tau'}(e))=(a_e,b_e)\bgiven(\omega_\tau,\omega_{\tau\tau'})=(a,b)\text{ on }E\setminus\{e\}\big)>c.
\end{equation}

\noindent\textbf{``DLR'' property.}
Given a subgraph $G'=(V',E')$ of $G$ and boundary conditions $\eta_\tau,\eta_{\tau\tau'},\eta'_\tau,\eta'_{\tau\tau'}$ with $\eta_\tau=\eta'_\tau$ on $\bbE\setminus E'$ and $\eta_{\tau\tau'}=\eta'_{\tau\tau'}$ on $\bbE\setminus E'$, 
\begin{equation}
	\label{eq:smp}\tag{DLR}
	\atrc_{G;J,U}^{\eta_\tau,\eta_{\tau\tau'}}(\,\cdot \given
	(\omega_\tau,\omega_{\tau\tau'}) = (\eta'_\tau,\eta'_{\tau\tau'})\text{ on }\bbE\setminus E')=
	\atrc_{G';J,U}^{\eta'_\tau,\eta'_{\tau\tau'}}( \cdot ).
\end{equation}

\noindent\textbf{Stochastic domination and positive association.}
We first introduce these notions in a general setting.
Given a partially ordered set~\(\mathcal{S}\) (states) and some index set~$I$ (eg. a set of edges or vertices), consider~\({S}^I\): the set of functions from~\(I\) to~\(\mathcal{S}\), i.e. configurations on~\(I\) with values in~\(\mathcal{S}\).
A subset $A\subseteq\mathcal{S}^I$ is called \emph{increasing} if~\(\ind_A\) is increasing with respect to the pointwise order.
Given a $\sigma$-algebra $\mathcal{A}$ on \(\mathcal{S}^I\) and two probability measures $\mu$ and $\nu$ on $\mathcal{A}$, we say that~$\nu$ stochastically dominates $\mu$, and write $\mu\leq_{\mathrm{st}}\nu$ (or $\nu\geq_{\mathrm{st}}\mu$) if~$\mu(A)\leq\nu(A)$ for any increasing event $A\in\mathcal{A}$.
We say that~$\mu$ is \emph{positively associated} or satisfies the \emph{FKG property} (named after the seminal work of Fortuin, Kasteleyn and Ginibre~\cite{ForKasGin71}) if, for all increasing events $A,B\in\calA$, we have
\begin{equation}
	\label{eq:fkg-general}
	\tag{FKG}
	\mu(A\cap B)\geq \mu(A)\mu(B).
\end{equation}
If~$I$ is finite and~$\mu(\omega)>0$ for any $\omega\in\mathcal{S}^I$, we say that $\mu$ satisfies the \emph{strong FKG} property if
\begin{equation}
	\label{eq:strong-fkg}
	\tag{strong-FKG}
	\forall\,J\subset I\ \forall\,\omega_1\leq\omega_2:\quad \mu(\,\cdot\given\omega=\omega_1\text{ on }J)\leq_{\mathrm{st}}\mu(\,\cdot\given\omega=\omega_2\text{ on }J),
\end{equation}
where we wrote \(\omega=\omega_i\text{ on }J\) for the event that \(\omega(j)=\omega_i(j)\) for all \(j\in J\).
We note that this is equivalent to the \emph{FKG lattice condition}; see~\cite[Section~2]{Gri06} and~\cite[Section~4]{GeoHagMae01} for the background.

In order to incorporate~$\atrc_{G;J,U}^{\eta_\tau,\eta_{\tau\tau'}}$ into this general framework, note that this is a positive measure on~\(\{(0,0),(0,1),(1,1)\}^E\). Indeed, \(\omega_\tau\) and \(\omega_{\tau\tau'}\) that coincide respectively with \(\eta_\tau\) and \(\eta_{\tau\tau'}\) on \(\bbE^\bullet\setminus E\) can be viewed as elements of~\(\{0,1\}^E\), and moreover~\(\omega_\tau(e)\leq \omega_{\tau\tau'}(e)\) for every~\(e\in E\) by~\eqref{eq:atrc_def}.

\begin{lemma}[\cite{PfiVel97}]
\label{lem:atrc_strong_fkg}
For any~$U>J>0$, any subgraph $G=(V,E)$ of $\bbG$ and all boundary conditions~$\eta_\tau,\eta_{\tau\tau'}$, the measure $\atrc_{G;J,U}^{\eta_\tau,\eta_{\tau\tau'}}$ satisfies the strong FKG property.
\end{lemma}
This is an easy consequence of~\cite[Proposition~4.1]{PfiVel97} and its proof.

By Lemma~\ref{lem:atrc_strong_fkg} and~\eqref{eq:smp}, for any increasing sequence of subgraphs~$G_k \nearrow \bbG$, the measures $\atrc_{G_k;J,U}^{\atrcwired,\atrcwired}$ form a decreasing sequence in the sense of stochastic domination. Consequently, the weak limit exists and is independent~$(G_k)$.
We denote it by~$\atrc_{J,U}^{\atrcwired,\atrcwired}$.
Analogously, we define~$\atrc_{J,U}^{\atrcfree,\atrcfree}$ as the (increasing) limit of $\atrc_{G_k;J,U}^{\atrcfree,\atrcfree}$.

Furthermore, Lemma~\ref{lem:atrc_strong_fkg} and~\eqref{eq:smp} can be used to compare different boundary conditions:
if~$U>J>0$, $\eta_\tau\leq \eta'_\tau$ and~$\eta_{\tau\tau'}\leq \eta'_{\tau\tau'}$, then
\begin{equation}
	\label{eq:cbc} \tag{CBC}
	\atrc_{G;J,U}^{\eta_\tau,\eta_{\tau\tau'}} \leq_{\rm st} \atrc_{G;J,U}^{\eta'_\tau,\eta'_{\tau\tau'}}.
\end{equation}

\noindent\textbf{FK-type relations on the square lattice.}
Recall the infinite-volume AT expectation operators $\langle\,\cdot\,\rangle^{\mathrm{f},\mathrm{f}}$ and $\langle\,\cdot\,\rangle^{+,+}$ defined in Section~\ref{sec:ATdef} and the notation~$\xleftrightarrow{\omega}$ for a connectivity event introduced in Section~\ref{sec:notations}.
The essential feature of the ATRC is that its connection probabilities are related to the correlations in the AT model~\cite[Proposition~3.1]{PfiVel97}:
for~$U> J> 0$ and any ~$i,j\in \bbL_\bullet$,
\begin{equation}\label{eq:at_atrc_coupling}
	\langle\tau_i\tau_j\rangle^{+,+}=\atrc^{\atrcwired,\atrcwired}_{J,U}(i\xleftrightarrow{\omega_\tau}j),
	\qquad
	\langle\tau_i\tau'_i\tau_j\tau'_j\rangle^{+,+}=\atrc^{\atrcwired,\atrcwired}_{J,U}(i\overset{\omega_{\tau\tau'}}{\longleftrightarrow}j).
\end{equation}
The analogous relations for $\langle\cdot\rangle^{\mathrm{f},\mathrm{f}}$ and $\atrc_{J,U}^{\atrcfree,\atrcfree}$ are also valid.

\medskip

\noindent\textbf{Duality on the square lattice.} 
Recall the primal and dual square lattices $(\bbL_\bullet,\bbE^\bullet)$ and $(\bbL_\circ,\bbE^\circ)$ (Section~\ref{sec:notations}) and the notion of a dual percolation configuration~$\eta^*\in\{0,1\}^{\bbE^\circ}$ for each $\eta\in\{0,1\}^{\bbE^\bullet}$ (Section~\ref{sec:intro}).
Choose~\(J,U\) as in~\eqref{eq:parameters_bulk}.
Then, for any finite $E\subset\bbE^\bullet$ and any~$\eta_\tau,\eta_{\tau\tau'}\in\{0,1\}^{\bbE^\bullet}$,  by~\cite[Proposition~$3.2$]{PfiVel97},
\begin{equation}
	\label{eq:atrc_selfdual}
	(\omega_\tau,\omega_{\tau\tau'})\sim\atrc_{E;J,U}^{\eta_\tau,\eta_{\tau\tau'}}
	\quad\text{implies}\quad
	(\omega_{\tau\tau'}^*,\omega_\tau^*)\sim\atrc_{*E;J,U}^{\eta_{\tau\tau'}^*,\eta_\tau^*}.
\end{equation}
Notice the different order in the dual pair.

\subsection{Results on the ATRC}
In this section, we present our main results concerning the self-dual ATRC model on the square lattice $\bbL_\bullet$.
\medskip

\subsubsection*{Mixing properties.}
We establish exponential relaxation for the single edges, which together with~\eqref{eq:strong-fkg} implies \emph{exponential weak mixing} (see Section~\ref{sec:mixing}).

\begin{proposition}
	\label{prop:edge_relax_ATRC}
	Let \(0<J<U\) satisfy \(\sinh 2J=e^{-2U}\). There exists \(c>0\) such that, for any edge \(e\in\bbE^\bullet\) with \(0\in e\), and any \(n\geq 1\),
	\begin{equation*}
		\max_{\sigma\in \{\tau,\tau\tau'\}} \Big|\atrc_{\Lambda_n; J,U}^{1,1}\big(\omega_{\sigma}(e)=1\big) - \atrc_{\Lambda_n; J,U}^{0,0}\big(\omega_{\sigma}(e)=1\big)\Big|\leq e^{-cn}.
	\end{equation*}
\end{proposition}
As a consequence, the measures $\atrc_{G;J,U}^{\eta_\tau,\eta_{\tau\tau'}}$ converge (as $G\nearrow\bbL_\bullet$) to a limit measure~$\atrc_{J,U}$, which is independent of the choice of boundary conditions~$\eta_\tau,\eta_{\tau\tau'}$. Furthermore, the limit~$\atrc_{J,U}$ is the unique ATRC \emph{Gibbs measure}.
Moreover, applying the classical work of Alexander~\cite{Ale98}, we derive \emph{ratio} weak mixing.
\begin{theorem}
	\label{thm:ratio_weak_mixing_ATRC}
	Let $0<J<U$ satisfy $\sinh 2J=e^{-2U}$. There exist \(C\geq 0,c>0\) such that, for any finite subgraph $G=(V,E)$ of $\bbL_\bullet$, any \(F\subset E\), and any \(F\)-measurable event \(A\) having positive probability,
	\begin{equation*}
		\sup_{\eta_{\tau},\eta_{\tau\tau'},\eta_{\tau}',\eta_{\tau\tau'}'} \Bigg|\frac{\atrc_{G;J,U}^{\eta_{\tau},\eta_{\tau\tau'}}(A)}{\atrc_{G;J,U}^{\eta_{\tau}',\eta_{\tau\tau'}'}(A)}- 1\Bigg| \leq C\sum_{f\in F}\sum_{e\in E^c} e^{-c\,\rmd_{\infty}(e,f)},
	\end{equation*}whenever the right side is strictly less than \(1\).
\end{theorem}

The notion of ATRC Gibbs measures can be defined in an analogous way to the FK model; see~\cite[Chapter 4.4]{Gri06}. The following is a direct consequence of Theorem~\ref{thm:ratio_weak_mixing_ATRC}.

\begin{corollary}
	\label{cor:ATRC_unique}
	Let \(0<J<U\) satisfy \(\sinh 2J=e^{-2U}\). It holds that
	\begin{equation*}
	\atrc_{J,U}^{\atrcfree,\atrcfree}=\atrc_{J,U}^{\atrcwired,\atrcwired}=:\atrc_{J,U}.
	\end{equation*}
	In particular, \(\atrc_{J,U}\) is the unique ATRC Gibbs measure, and for any choice of boundary conditions \(\eta_\tau,\eta_{\tau\tau'}\), the measures \(\atrc_{G;J,U}^{\eta_\tau,\eta_{\tau\tau'}}\) converge to \(\atrc_{J,U}\) as \(G\nearrow\bbL_\bullet\).
\end{corollary}

We refer to Section~\ref{sec:strong_mixing}, for our results on \emph{strong} mixing properties of the ATRC.

\subsubsection*{Uniform exponential decay.}
In Section~\ref{sec:mixing}, we prove exponential weak mixing (Theorem~\ref{thm:weak_mixing_ATRC}).
It allows to apply the work of Alexander~\cite{Ale04} to derive the following uniform decay of connection probabilities:
\begin{theorem}
	\label{thm:Ale04_ATRC}
	Let $0<J<U$ satisfy $\sinh 2J=e^{-2U}$. There exists a constant \(c>0\) such that, for any finite simply lattice-connected \(E\subset \bbE^\bullet\), and any \(i,j\in \bbV_{E}\),
	\begin{equation*}
		\atrc_{E;J,U}^{1,1}\big(i\xleftrightarrow{\omega_\tau\cap E} j\big) \leq e^{-c\,\rmd_\infty(i,j)},
	\end{equation*}
	where \(\rmd_\infty\) is the distance induced by the $L^\infty$ norm.
\end{theorem}
\begin{proof}
	It is enough to check the conditions stated in~\cite[Theorem 1.1]{Ale04}: the push-forward of \(\atrc_{J,U}\) by \(\omega_{\tau}\) is translation invariant, has finite-energy for closing edges, exponential decay of connectivity probabilities~\cite[Proposition 1.1]{AouDobGla24}, and is exponentially weak mixing by Theorem~\ref{thm:weak_mixing_ATRC}. Finally, for \(n\) with \(E\subseteq\bbE_{\Lambda_{n-1}}\), we have by~\eqref{eq:smp} that
	\begin{equation*}
	\atrc_{J,U}(\omega_\tau\in\cdot\given \omega_\tau(e)=1\text{ for }e\in\bbE_{\Lambda_n}\setminus E)=\atrc_{E;J,U}^{1,1}.\qedhere
	\end{equation*}
\end{proof}

\subsubsection*{Ornstein--Zernike asymptotics.}
Recall the definition of the inverse correlation length \(\nu_{J,U}\) in the AT model in Section~\ref{sec:ATdef}. By the FK-type relations~\eqref{eq:at_atrc_coupling} and by Corollary~\ref{cor:ATRC_unique},
\begin{equation*}
\nu_{J,U}(x)=-\lim_{n\to\infty}\tfrac{1}{n}\ln \atrc_{J,U}(0\xleftrightarrow{\omega_\tau}\lfloor nx \rfloor).
\end{equation*}
The following theorem is the analogue of Theorem~\ref{thm:oz-at} for the ATRC model.

\begin{theorem}
\label{thm:oz-atrc}
Let $0<J<U$ satisfy $\sinh 2J=e^{-2U}$.
The inverse correlation length~$\nu_{J,U}$ is a norm on $\R^2$. Furthermore, uniformly in $|x|\to\infty$, 
\begin{equation*}
\atrc_{J,U}(0\xleftrightarrow{\omega_\tau} x)=\tfrac{g(x/|x|)}{\sqrt{|x|}}\,e^{-\nu(x)}\,(1+o(1)),
\end{equation*}
where $\nu=\nu_{J,U}$, and $g=g_{J,U}$ is a strictly positive analytic function on $\mathbb{S}^1$.
\end{theorem}

\section{Couplings and interfaces}
\label{sec:coupling:interfaces}

This section is concerned with the construction of a coupling of the FK-percolation and the Ashkin--Teller models, via the six-vertex model, and the derivation of its basic properties. 
The coupling of the FK and six-vertex measures is an adaptation of the Baxter--Kelland--Wu (BKW) coupling to the Dobrushin boundary conditions~\cite{BaxKelWu76}.
The relation of the six-vertex and AT measures has first been noticed in~\cite{Fan72} comparing their critical properties, and it was made explicit in~\cite{Fan72b,Weg72} on a level of partition functions.
We build on~\cite{GlaPel23}, where a coupling of the six-vertex model and a graphical representation of the AT model (a marginal of ATRC) was constructed.

\subsection{Different models and combinatorial mappings}
\label{sec:combinatorial_mappings}

This section provides an overview of the combinatorial objects that will be encountered, as well as a description of their relations.
We first discuss oriented loop configurations, which serve as an intermediate step in the BKW coupling of the FK-percolation and the six-vertex models.
This is followed by a description of two of the representations of the six-vertex model: edge orientations and spin configurations.
The height function representation appears only in Section~\ref{sec:mixing}, where its monotonicity properties are crucial to derive mixing and relaxation properties of the ATRC measures.

In Section~\ref{sec:intro}, we saw that percolation configurations are in bijection with (unoriented) loop configurations. To make the correspondence \(\omega\leftrightarrow \omega^* \leftrightarrow \ell=\loops(\omega)\) explicit, we can regard each of these models as an assignment of a local piece of drawing of an edge and two \emph{arcs} to lozenge tiles centred at the mid-edges as depicted in Fig.~\ref{fig:midEdgeTiles}. Clearly, retaining only either the primal or dual edges, or the arcs, provides complete information about all three.

\medskip

\noindent\textbf{Oriented loop configurations} are obtained from unoriented ones by assigning an orientation to each loop; formally, this is done by means of a sequence of independent uniform random variables on \([0,1]\) indexed by the set \(\mathcal{L}\) of all loops.
Alternatively, one can assign orientations to the loop arcs on each tile, subject to the constraint that the orientations of neighbouring arcs match. The eight local configurations that can occur at a tile are referred to as \emph{types}, see Fig.~\ref{fig:oriented_loop_arcs}.

\medskip

\noindent\textbf{The edge orientations of the six-vertex model}~\cite{Pau35,Rys63} are assignments of orientations to the edges in \(\bbE^\diamond\), obtained from oriented loop configurations via the natural  surjection; see Fig.~\ref{fig:oriented_loop_arcs}.
These edge orientations satisfy the {\em ice rule}: at every vertex, there are two incoming and two outgoing edges of~\(\bbE^\diamond\).
This constraint permits six possible local configurations at a tile, which are also called types; see Fig.~\ref{fig:oriented_loop_arcs}.
The local inverse operation can be considered as \emph{splitting} the oriented edges into two oriented loop arcs.
While tiles of types 1-4 permit a unique reconstruction of the loop arcs based on the edge orientations, there are two possibilities for tiles of types 5-6, giving the latter a special role. 

\begin{figure}
\centering
\ifpic
\begin{tikzpicture}
	\draw (-1.5,1.2) node[below]{type};
	\draw (0,1.2) node[below]{1};
	\draw (1.5,1.2) node[below]{2};
	\draw (3,1.2) node[below]{3};
	\draw (4.5,1.2) node[below]{4};
	\draw (6,1.2) node[below]{5A};
	\draw (7.5,1.2) node[below]{5B};
	\draw (9,1.2) node[below]{6A};
	\draw (10.5,1.2) node[below]{6B};

	\foreach \i in {0,1.5,3,4.5,6,7.5,9,10.5}{
		\draw (\i,-0.5)--(\i+0.5,0)--(\i,0.5)--(\i-0.5,0)--(\i,-0.5);
	}
	\foreach \i in {0,1.5,6,9}{
		\Hloop{\i}{0}
	}
	\foreach \i in {3,4.5,7.5,10.5}{
		\Vloop{\i}{0}
	}
		
	\RIGHTarrow{0}{0.5-1/(2*sqrt(2))}{blue}
	\RIGHTarrow{0}{1/(2*sqrt(2))-0.5}{blue}
	\LEFTarrow{1.5}{0.5-1/(2*sqrt(2))}{blue}
	\LEFTarrow{1.5}{1/(2*sqrt(2))-0.5}{blue}
	
	\UParrow{2.5+1/(2*sqrt(2))}{0}{blue}
	\UParrow{3.5-1/(2*sqrt(2))}{0}{blue}
	\DOWNarrow{4+1/(2*sqrt(2))}{0}{blue}
	\DOWNarrow{5-1/(2*sqrt(2))}{0}{blue}
	
	\LEFTarrow{6}{1/(2*sqrt(2))-0.5}{blue}
	\RIGHTarrow{6}{0.5-1/(2*sqrt(2))}{blue}	
	\DOWNarrow{7+1/(2*sqrt(2))}{0}{blue}
	\UParrow{8-1/(2*sqrt(2))}{0}{blue}
			
	\RIGHTarrow{9}{1/(2*sqrt(2))-0.5}{blue}
	\LEFTarrow{9}{0.5-1/(2*sqrt(2))}{blue}
	\UParrow{10+1/(2*sqrt(2))}{0}{blue}
	\DOWNarrow{11-1/(2*sqrt(2))}{0}{blue}	
	
	\foreach \i in {0,1.5,3,4.5,6.75,9.75}{
		\DrawTile{\i}{-2}{}
	}
	
	\foreach \i in {0,1.5,3,4.5}{
		\draw[->,thick] (\i,-0.6)--(\i,-1.4);
	}
	\draw[->,thick] (6,-0.6)--(6.7,-1.4);
	\draw[->,thick] (7.5,-0.6)--(6.8,-1.4);
	\draw[->,thick] (9,-0.6)--(9.7,-1.4);
	\draw[->,thick] (10.5,-0.6)--(9.8,-1.4);
	
	
	\draw (-1.5,-2.6) node[below]{type};
	\draw (0,-2.6) node[below]{1};
	\draw (1.5,-2.6) node[below]{2};
	\draw (3,-2.6) node[below]{3};
	\draw (4.5,-2.6) node[below]{4};
	\draw (6.75,-2.6) node[below]{5};
	\draw (9.75,-2.6) node[below]{6};
	
	\foreach \i in {0,3,9.75}{
		\draw[thick] (\i,-2)--({\i-0.1},{-2-0.1});
		\draw[thick, <-] (\i-0.1,-2-0.1)--({\i-0.25},{-2-0.25});
	}

	\foreach \i in {1.5,4.5,6.75}{
		\draw[thick] (\i-0.15,-2-0.15)--({\i-0.25},{-2-0.25});
		\draw[thick, ->] (\i,-2)--({\i-0.15},{-2-0.15});
	}
	
	\foreach \i in {0,4.5,6.75}{
		\draw[thick] (\i,-2)--({\i-0.1},{-2+0.1});		
		\draw[thick, <-] (\i-0.1,-2+0.1)--({\i-0.25},{-2+0.25});
	}

	\foreach \i in {1.5,3,9.75}{
		\draw[thick] (\i-0.15,-2+0.15)--({\i-0.25},{-2+0.25});
		\draw[thick, ->] (\i,-2)--({\i-0.15},{-2+0.15});
	}
	
	\foreach \i in {1.5,3,6.75}{
		\draw[thick] (\i,-2)--({\i+0.1},{-2-0.1});
		\draw[thick, <-] (\i+0.1,-2-0.1)--({\i+0.25},{-2-0.25});
	}

	\foreach \i in {0,4.5,9.75}{
		\draw[thick] (\i+0.15,-2-0.15)--({\i+0.25},{-2-0.25});
		\draw[thick, ->] (\i,-2)--({\i+0.15},{-2-0.15});
	}
	
	\foreach \i in {1.5,4.5,9.75}{
		\draw[thick] (\i,-2)--({\i+0.1},{-2+0.1});
		\draw[thick, <-] (\i+0.1,-2+0.1)--({\i+0.25},{-2+0.25});
	}

	\foreach \i in {0,3,6.75}{
		\draw[thick] (\i+0.15,-2+0.15)--({\i+0.25},{-2+0.25});
		\draw[thick, ->] (\i,-2)--({\i+0.15},{-2+0.15});
	}
\end{tikzpicture}
\fi
\caption{Tiles of the oriented loop model and their types and weights, and the mapping from oriented loop arcs to six-vertex edge orientations.}
\label{fig:oriented_loop_arcs}
\end{figure}
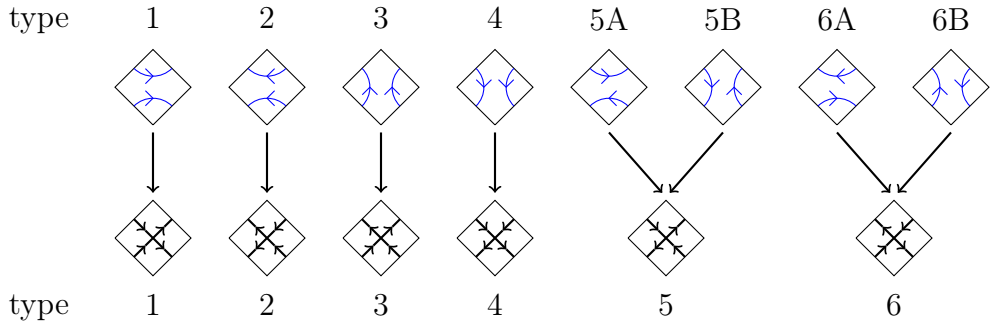
\medskip

\noindent\textbf{The six-vertex spin representation.}
The six-vertex model may be represented by pairs of spin-configurations \((\sigma_\bullet,\sigma_\circ)\in\{\pm1\}^{\bbL_\bullet}\times\{\pm1\}^{\bbL_\circ}\)~\cite{Wu71,KadWeg71,Lis22,GlaPel23}, obtained from the edge orientations via the following two-valued mapping.
Fix the value of \(\sigma_\bullet\) or \(\sigma_\circ\) at some arbitrary fixed vertex and proceed iteratively: given an edge \(e\in\bbE^\diamond\), denote the vertices that it separates by~$i\in\bbL_\bullet$ from a vertex~$u\in\bbL_\circ$; we impose~\(\sigma_\bullet(i)=\sigma_\circ(u)\) if~\(i\) is to the left side of \(e\) (with respect to its assigned orientation) and we impose~\(\sigma_\bullet(i)=-\sigma_\circ(u)\) otherwise; see Fig.~\ref{fig:six-vertex_arrows_spins_heights}.
The ice rule ensures that this mapping is well-defined.

This mapping is two-valued due to the liberty to choose the value of \(\sigma_\bullet\) or \(\sigma_\circ\) at one vertex, and the two images are related to each other by a global spin flip.
Furthermore, it is injective, meaning that the edge orientations can be reconstructed from the spins.
The type of a tile with respect to \((\sigma_\bullet,\sigma_\circ)\) is given by the type of the corresponding edge orientations; see Fig.~\ref{fig:six-vertex_arrows_spins_heights}.
It should also be noted that the ice rule can be translated as follows: for any tile \(t\in\bbL_\diamond\), either~\(\sigma_\bullet\) is constant on the endpoints of \(e_t\) or \(\sigma_\circ\) is constant on the endpoints of \(e_t^*\). Formally,
\begin{equation}\label{eq:ice-rule_spins}
\big(\sigma_\bullet(i)-\sigma_\bullet(j)\big)\big(\sigma_\circ(u)-\sigma_\circ(v)\big)=0\qquad\text{for any }t\in\bbL_\diamond\text{ with }e_t=ij,\,e_t^*=uv.
\end{equation}
In the context of spins, this property will henceforth be referred to as the ice rule. 

\begin{figure}
\centering
\includegraphics[scale=0.7]{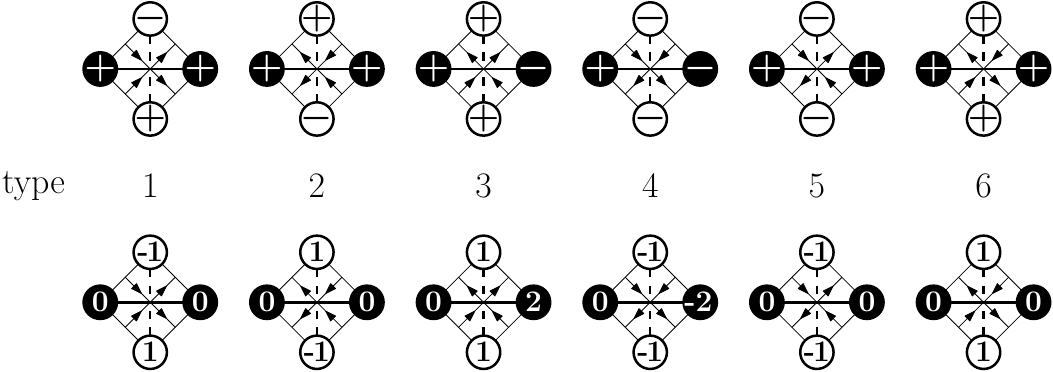}
\caption{The six-vertex types for all representations at a tile corresponding to a horizontal primal edge \(e\). Top: the spin at the left endpoint of \(e\) is fixed to be $+$. Bottom: the height at the left endpoint of \(e\) is fixed to be $0$.}
\label{fig:six-vertex_arrows_spins_heights}
\end{figure}

\medskip

\noindent\textbf{Baxter--Kelland--Wu correspondence~\cite{BaxKelWu76}.}
Given a measure on oriented loops, taking its pushforwards with respect to the above mappings, one obtains measures on the six-vertex edge orientations and spin configurations.

\subsection{Coupling under Dobrushin conditions}
\label{sec:coupling}

The idea is to take the FK measure with Dobrushin boundary conditions and construct from it first a measure on pairs of spin configurations on \(\bbL_\bullet\) and \(\bbL_\circ\) that satisfy the ice rule and then a measure on ATRC configurations. We will then identify these two measures as the six-vertex and the ATRC measures, respectively, under suitable versions of Dobrushin boundary conditions. This gives a coupling between the FK and a modified ATRC, which will allow us to transfer the study of the former to the study of the latter.
See also Section~\ref{sec:bkw} for the BKW coupling~\cite{BaxKelWu76} without Dobrushin boundary conditions and in the context of height functions.

Recall the subgraphs \(G_{n,m}=(V_{n,m},E_{n,m})\) of \((\bbL_{\bullet},\bbE^{\bullet} )\) given in Section~\ref{sec:notations}.
As \(n,m\) will be fixed in this section, we will omit them in the notation and simply write \(G=(V,E)\).
Consider~\(\fk_G^{1/0}\): the FK measure on~\(G\) under Dobrushin boundary conditions defined in Section~\ref{sec:intro}; see Fig.~\ref{fig:WiredFree_FK}).

\subsubsection*{Coupling measure.}
To couple \(\fk_G^{1/0}\) with both a six-vertex spin measure and a modified version of the ATRC measure, we augment our probability space to incorporate independent uniform \([0,1]\) random variables assigned to every loop and every tile.
Formally, let~\(\mathcal{L}\) be the set of all unoriented loops drawn on the tiles of~\(\bbL_{\diamond}\) (see the right of Fig.~\ref{fig:midEdgeTiles}).
Define~\(\Omega^{1/0} := \{0,1\}^{\bbE\bullet}\times[0,1]^{\mathcal{L}}\times[0,1]^{\bbL_{\diamond}}\) equipped with the product of Borel sigma algebras.
Define \(Q\) and \(Q'\) respectively as the product measures on~\([0,1]^{\mathcal{L}}\) and~\([0,1]^{\bbL_{\diamond}}\).
Finally, the coupling measure is defined by
\[
	\Psi_G^{1/0} := \fk_G^{1/0}\otimes Q\otimes Q'.
\]
In Lemmata~\ref{lem:6V_spins_to_01FK} and~\ref{lem:6V_spins_to_AT} below, we describe how to obtain the following two measures as marginals of~\(\Psi_G^{1/0}\): the six-vertex spin measure and the modified ATRC measure, both under suitable Dobrushin boundary conditions (Definitions~\ref{def:six-vertex} and~\ref{def:mATRC}).

\begin{figure}
	\centering
	\ifpic
	\begin{tikzpicture}[scale=0.9]
		\foreach \i\j in {-4/0.5, -4/1.5, -4/2.5, -4/3.5, -3/0.5, -3/1.5, -3/2.5, -3/3.5, -2/3.5, -1/3.5, 0/3.5, 1/3.5, 2/3.5, 3/0.5, 3/1.5, 3/2.5, 3/3.5, 4/0.5, 4/1.5, 4/2.5, 4/3.5}{
			\Vedge{\i}{\j}{very thick}
			\DrawTile{\i}{\j}{gray!50}
			\Vloop{\i}{\j}
		}
		\foreach \i\j in {-3.5/0, -3.5/1, -3.5/2, -3.5/3, -3.5/4, 3.5/0, 3.5/1, 3.5/2, 3.5/3, 3.5/4, -2.5/3, -2.5/4, -1.5/3, -1.5/4, -0.5/3, -0.5/4, 0.5/3, 0.5/4, 1.5/3, 1.5/4, 2.5/3, 2.5/4}{
			\Hedge{\i}{\j}{very thick}
			\DrawTile{\i}{\j}{gray!50}
			\Hloop{\i}{\j}
		}
		\foreach \i\j in {-3.5/-3, -3.5/-2, -3.5/-1, -2.5/-3, -2.5/-2, -2.5/-1, 2.5/-3, 2.5/-2, 2.5/-1, 3.5/-3, 3.5/-2, 3.5/-1, -1.5/-3, -0.5/-3, 1.5/-3, 0.5/-3}{
			\Vedge{\i}{\j}{dashed, very thick}
			\DrawTile{\i}{\j}{gray!50}
			\Vloop{\i}{\j}
		}
		\foreach \i\j in {-3/-3.5, -3/-2.5, -3/-1.5, -3/-0.5, 3/-3.5, 3/-2.5, 3/-1.5, 3/-0.5, -2/-3.5, -2/-2.5, -1/-3.5, -1/-2.5, 0/-3.5, 0/-2.5, 1/-3.5, 1/-2.5, 2/-3.5, 2/-2.5}{
			\Hedge{\i}{\j}{dashed, very thick}
			\DrawTile{\i}{\j}{gray!50}
			\Hloop{\i}{\j}
		}
		\foreach \i\j in {-3.5/1, -3.5/2, -3.5/3, -3.5/4, 3.5/1, 3.5/2, 3.5/3, 3.5/4, -2.5/4, -1.5/4, -0.5/4, 0.5/4, 1.5/4, 2.5/4, -3/-3.5, -3/-2.5, -3/-1.5, -3/-0.5, 3/-3.5, 3/-2.5, 3/-1.5, 3/-0.5, -2/-3.5, -1/-3.5, 0/-3.5, 1/-3.5, 2/-3.5}{
			\RIGHTarrow{\i}{\j-0.5+1/(2*sqrt(2))}{blue}
			\LEFTarrow{\i}{\j+0.5-1/(2*sqrt(2))}{blue}
		}
		\foreach \i\j in {-3.5/-3, -3.5/-2, -3.5/-1, -2.5/-3, 2.5/-3, 3.5/-3, 3.5/-2, 3.5/-1, -1.5/-3, -0.5/-3, 1.5/-3, 0.5/-3, -4/0.5, -4/1.5, -4/2.5, -4/3.5, -3/3.5, -2/3.5, -1/3.5, 0/3.5, 1/3.5, 2/3.5, 3/3.5, 4/0.5, 4/1.5, 4/2.5, 4/3.5}{
			\DOWNarrow{\i-0.5+1/(2*sqrt(2))}{\j}{blue}
			\UParrow{\i+0.5-1/(2*sqrt(2))}{\j}{blue}
		}
		
		\foreach \i\j in {-3.5/0,-2.5/3, -1.5/3, -0.5/3, 0.5/3, 1.5/3, 2.5/3, 3.5/0}{
			\LEFTarrow{\i}{\j+0.5-1/(2*sqrt(2))}{blue}
		}
		\foreach \i\j in {-2/-2.5, -1/-2.5, 0/-2.5, 1/-2.5, 2/-2.5}{
			\RIGHTarrow{\i}{\j-0.5+1/(2*sqrt(2))}{blue}
		}
		\foreach \i in {0,1,2}{
			\DOWNarrow{-3-0.5+1/(2*sqrt(2))}{\i+0.5}{blue}
			\UParrow{3+0.5-1/(2*sqrt(2))}{\i+0.5}{blue}
		}
		\DOWNarrow{-2.5-0.5+1/(2*sqrt(2))}{-2}{blue}
		\DOWNarrow{-2.5-0.5+1/(2*sqrt(2))}{-1}{blue}
		\UParrow{2.5+0.5-1/(2*sqrt(2))}{-2}{blue}
		\UParrow{2.5+0.5-1/(2*sqrt(2))}{-1}{blue}
		
		\LEFTarrow{-3.5}{-0.5+1/(2*sqrt(2))}{blue};
		\LEFTarrow{3.5}{-0.5+1/(2*sqrt(2))}{blue};
		\draw (-3,0) node[right]{$v_L$};
		\draw (3,0) node[left]{$v_R$};
		
		\foreach \i\j in {-4/-4, -4/-3, -4/-2, -4/-1, -4/0, -4/1, -4/2, -4/3, -4/4, -3/-4, -3/-3, -3/-2, -3/-1, -3/0, -3/1, -3/2, -3/3, -3/4, -2/-4, -2/-3, -2/-2, -2/-1, -2/0, -2/1, -2/2, -2/3, -2/4, -1/-4, -1/-3, -1/-2, -1/-1, -1/0, -1/1, -1/2, -1/3, -1/4, 0/-4, 0/-3, 0/-2, 0/-1, 0/0, 0/1, 0/2, 0/3, 0/4, 1/-4, 1/-3, 1/-2, 1/-1, 1/0, 1/1, 1/2, 1/3, 1/4, 2/-4, 2/-3, 2/-2, 2/-1, 2/0, 2/1, 2/2, 2/3, 2/4, 3/-4, 3/-3, 3/-2, 3/-1, 3/0, 3/1, 3/2, 3/3, 3/4, 4/-4, 4/-3, 4/-2, 4/-1, 4/0, 4/1, 4/2, 4/3, 4/4}{
			\filldraw[black] (\i,\j) circle (2pt);
		}
		\foreach \i\j in {-3.5/-3.5, -3.5/-2.5, -3.5/-1.5, -3.5/-0.5, -3.5/0.5, -3.5/1.5, -3.5/2.5, -3.5/3.5, -2.5/-3.5, -2.5/-2.5, -2.5/-1.5, -2.5/-0.5, -2.5/0.5, -2.5/1.5, -2.5/2.5, -2.5/3.5, -1.5/-3.5, -1.5/-2.5, -1.5/-1.5, -1.5/-0.5, -1.5/0.5, -1.5/1.5, -1.5/2.5, -1.5/3.5, -0.5/-3.5, -0.5/-2.5, -0.5/-1.5, -0.5/-0.5, -0.5/0.5, -0.5/1.5, -0.5/2.5, -0.5/3.5, 0.5/-3.5, 0.5/-2.5, 0.5/-1.5, 0.5/-0.5, 0.5/0.5, 0.5/1.5, 0.5/2.5, 0.5/3.5, 1.5/-3.5, 1.5/-2.5, 1.5/-1.5, 1.5/-0.5, 1.5/0.5, 1.5/1.5, 1.5/2.5, 1.5/3.5, 2.5/-3.5, 2.5/-2.5, 2.5/-1.5, 2.5/-0.5, 2.5/0.5, 2.5/1.5, 2.5/2.5, 2.5/3.5, 3.5/-3.5, 3.5/-2.5, 3.5/-1.5, 3.5/-0.5, 3.5/0.5, 3.5/1.5, 3.5/2.5, 3.5/3.5}{
			\filldraw[black, fill=white] (\i,\j) circle (2pt);
		}
		
		\draw (-4,1) node[left]{$\sigma_{\bullet} =1$};
		\filldraw[black, fill=white] (-4.5,0.5) circle (2pt);
		\draw (-4.5,0.5) node[left]{$\sigma_{\circ} =1$};
		\draw (-4,-1) node[left]{$\sigma_{\bullet} =1$};
		\draw (-3.5,-1.5) node[left]{$\sigma_{\circ} =-1$};

	\end{tikzpicture}
	\fi
	\caption{Boundary conditions on oriented loops and on six-vertex spins. The edges drawn are those induced by the oriented loop boundary conditions.}
	\label{fig:bc_FK_loops_spins}
\end{figure}
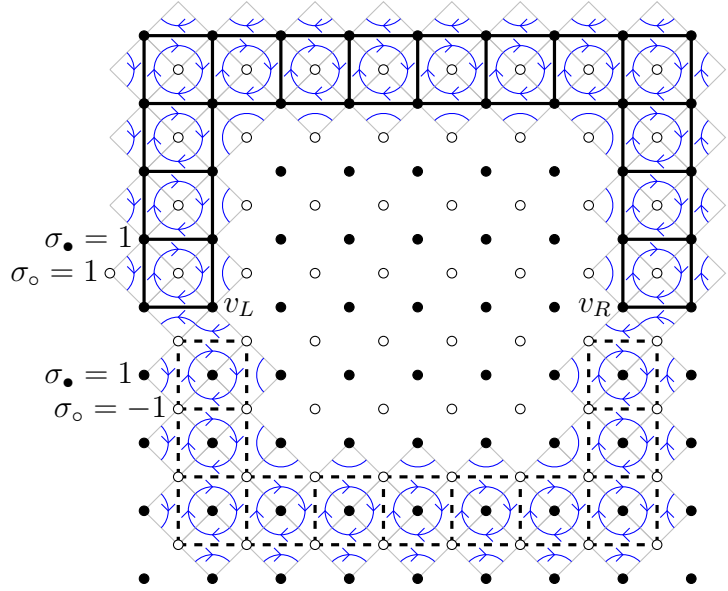

\subsubsection*{From FK to six-vertex.}
Recall the parameters~\(\svc\), \(\svcb\), and~\(\lambda\) from~\eqref{eq:parameters_bulk} and the standard rectangular domains from Section~\ref{sec:notations}. Below we omit~\(n,m\) everywhere.
\begin{definition}\label{def:six-vertex}
	The six-vertex spin model under Dobrushin conditions is a probability measure on \(\sigma=(\sigma_\bullet,\sigma_\circ)\in \{\pm 1\}^{\bbL_{\bullet}}\times \{\pm 1\}^{\bbL_{\circ}}\) defined by
	\begin{equation}
		\label{eq:def_6v-spin_dobrushin}
		\spin_\mathcal{D}^{+,+-}(\sigma) \propto
		\svc^{|T_{5,6}^{\rmi}(\sigma)|}\,\svcb^{|T_{5,6}^{\rmb}(\sigma)|}\,\mathds{1}_{\Sigma_\Lambda^{+}\times\Sigma_{\Lambda'}^{+-}}(\sigma)\,\mathds{1}_{\mathrm{ice}}(\sigma),
	\end{equation}
	where $\svc,\svcb$ are defined by~\eqref{eq:parameters_bulk}-\eqref{eq:parameters_bnd}, \(T_{5,6}^{\rmi}\) and \(T_{5,6}^{\rmb}\) are the sets of tiles of types 5-6 in \(A^\rmi\) and \(\partial A\), respectively (see Fig.~\ref{fig:six-vertex_arrows_spins_heights}), \(\mathds{1}_{\mathrm{ice}}\) is the indicator imposing the ice rule~\eqref{eq:ice-rule_spins}, and~\(\Sigma_{\Lambda}^{+}\) and~\(\Sigma_{\Lambda'}^{+-}\) are the boundary conditions defined by
	\begin{align*}
		\Sigma_{\Lambda}^{+} &:= \{\sigma_\bullet\in\{\pm1\}^{\bbL_\bullet}:\sigma_\bullet(i) = 1\,\forall i\in \bbL_\bullet\setminus\Lambda\},\\
		\Sigma_{\Lambda'}^{+-} &:= \{\sigma_\circ\in\{\pm1\}^{\bbL_\circ}:\sigma_\circ(u) = \mathds{1}_{\bbH^+}(u)-\mathds{1}_{\bbH^-}(u)\,\forall u\in \bbL_\circ\setminus\Lambda'\}.
	\end{align*}
\end{definition}
We now describe how to obtain~\(\spin_\mathcal{D}^{+,+-}\) as a marginal of the coupling measure~\(\Psi_G^{1/0}\).

\begin{lemma}
\label{lem:6V_spins_to_01FK}
	Let \((\omega,U,U')\) be distributed according to \(\Psi_G^{1/0}\) and let~\(\ell\) be the unoriented loop configuration associated to~$\omega$.
	Orient the loops of~\(\ell\) as follows (see Fig.~\ref{fig:bc_FK_loops_spins}):
	\begin{itemize}
		\item each loop~\(l\in \ell\) outside of~\(G\) (i.e. surrounding a vertex \((\bbL_\circ\cap\bbH^+)\setminus\Lambda'\) or in \((\bbL_\bullet\cap\bbH^-)\setminus\Lambda\) is oriented clockwise;
		\item the unique bi-infinite path is oriented from right to left;
		\item each loop~\(l\in \ell\) inside of~\(G\) (i.e. surrounding a vertex in \(\Lambda\cup\Lambda'\)) is oriented 
		clockwise if \(U_l<e^{\lambda}/\svc\) and counter-clockwise otherwise.
	\end{itemize}
	Denote by \(\ell_\shortrightarrow\) the obtained oriented loop configuration.
	Recall the combinatorial mappings introduced above and let \((\sigma_\bullet,\sigma_\circ)\) be the associated six-vertex spin configurations with \(\sigma_\bullet((n+1,0))=+1\).
	Then, the law of \((\sigma_\bullet,\sigma_\circ)\) is given by \(\spin_\mathcal{D}^{+,+-}\).
\end{lemma}

\begin{proof}
We follow the ideas of~\cite{BaxKelWu76}.
One has to examine which values of \((\omega, U)\) result in a given \((\sigma_\bullet,\sigma_\circ)\in \{\pm1\}^{\bbL_\bullet}\times\{\pm1\}^{\bbL_\circ}\).
The probability to obtain \((\sigma_\bullet,\sigma_\circ)\in\Sigma_\Lambda^{+}\times\Sigma_{\Lambda'}^{+-}\) is the sum of the probabilities of all oriented loop configurations \(\ell_\shortrightarrow\) that induce the edge orientations corresponding to \((\sigma_\bullet,\sigma_\circ)\). The probability of a given oriented loop configuration \(\ell_\shortrightarrow\) satisyfing the boundary conditions in Fig.~\ref{fig:bc_FK_loops_spins} is proportional to
\begin{equation}\label{eq:prf:bkw1}
\sqrt{q}^{\,|\loops(\ell_\shortrightarrow)|}\cdot\Big(\tfrac{e^\lambda}{e^\lambda+e^{-\lambda}}\Big)^{|\loops_{\circlearrowright}(\ell_\shortrightarrow)|- |\loops_{\circlearrowleft}(\ell_\shortrightarrow)|}
=e^{\lambda(|\loops_{\circlearrowright}(\ell_\shortrightarrow)|- |\loops_{\circlearrowleft}(\ell_\shortrightarrow)|)},
\end{equation}
where \(\loops_{\circlearrowright}(\ell_\shortrightarrow)\) and \(\loops_{\circlearrowleft}(\ell_\shortrightarrow)\) are respectively the sets of clockwise and counter-clockwise oriented loops in \(\ell_\shortrightarrow\) not imposed by boundary conditions, and \(\loops(\ell_\shortrightarrow)\) is their union.
Indeed, by~\ref{eq:FK_loop_expression}, the first factor on the left side is proportional to \(\fk_G^{1/0}(\omega(\ell))\), where \(\omega(\ell)\in\{0,1\}^{\bbE^\bullet}\) is the percolation configuration associated to the unoriented loop configuration \(\ell\) corresponding to \(\ell_\shortrightarrow\). The second factor comes from the values of the uniforms necessary to obtain correct orientations of loops in~\(\loops(\ell_\shortrightarrow)\). The equality holds since \(\sqrt{q}=e^\lambda+e^{-\lambda}\) due to the choice of \(\lambda\).

Notice that each loop which is oriented clockwise does 4 more right quarter-turns than left quarter-turns, that the converse holds for counter-clockwise oriented loops, and that the number of left and right quarter-turns in the left-right interface differ by a universal constant.
Thus, the expression on the right side of~\eqref{eq:prf:bkw1} is proportional to
\begin{equation*}
\exp(\lambda (\#_{\curvearrowright}(\ell_\shortrightarrow)- \#_{\curvearrowleft}(\ell_\shortrightarrow))/4),
\end{equation*}
where \(\#_{\curvearrowright}(\ell_\shortrightarrow)\) and \(\#_{\curvearrowleft}(\ell_\shortrightarrow)\) are respectively the number of right and left quarter-turns in \(\ell_\shortrightarrow\).
The key idea is to count these oriented loop arcs locally at each tile in \(A=A^\rmi\cup A^\rmb\).
Observe that, for tiles in \(A^\rmi\), types 5B,6A correspond to a pair of right-oriented loop arcs and types 5A,6B to a pair of left-oriented loop arcs, whereas types 1-4 correspond to one right-oriented and one left-oriented loop arc each.
Moreover, due to the boundary conditions (see Fig.~\ref{fig:bc_FK_loops_spins}), a tile in \(A^\rmb\) contains a right turn precisely if it is of type 5,6, and it contains a left turn otherwise.
We deduce that the probability of \(\ell_\shortrightarrow\) is proportional to
\begin{equation}\label{eq:prf:bkw2}
\prod_{t\in T_{5,6}(\ell_\shortrightarrow)\cap A^\rmi} \big(e^{\lambda/2}\,\mathds{1}_{T_{5B,6A}(\ell_\shortrightarrow)}(t) + e^{-\lambda/2}\,\mathds{1}_{T_{5A,6B}(\ell_\shortrightarrow)}(t)\big)
\cdot \big(e^{\lambda/2}\big)^{|T_{5,6}(\ell_\shortrightarrow)\cap \partial A|},
\end{equation}
where \(T_{5B,6A}(\ell_\shortrightarrow)\) and \(T_{5A,6B}(\ell_\shortrightarrow)\) are respectively the sets of tiles of types 5B and 6A and the set of tiles of types 5A and 6B in \(\ell_\shortrightarrow\), and \(T_{5,6}(\ell_\shortrightarrow)\) is their union.

Finally, fix a pair \((\sigma_\bullet,\sigma_\circ)\in\Sigma_\Lambda^{+}\times\Sigma_{\Lambda'}^{+-}\) that satisfies the ice-rule, and consider its associated edge orientations.
It remains to identify all oriented loop configurations \(\ell_\shortrightarrow\) that satisfy the boundary conditions in Fig.~\ref{fig:bc_FK_loops_spins} and that induce these edge orientations. Observe that the boundary conditions and the spins \((\sigma_\bullet,\sigma_\circ)\) uniquely determine the oriented loop arcs at tiles in \((\bbL_\diamond\setminus A)\cup A^\rmb\) and at tiles in \(A^\rmi\setminus T_{5,6}^\rmi(\sigma_\bullet,\sigma_\circ)\). For a tile in \(T_{5,6}^\rmi(\sigma_\bullet,\sigma_\circ)\), one can split the oriented edges either into a pair of right-oriented loops arcs (types 5B,6A) or into a pair of left-oriented loop arcs (types 5A,6B). Summing the probabilities~\eqref{eq:prf:bkw2} of all oriented loop configurations obtained in that way, we obtain that the probability of \((\sigma_\bullet,\sigma_\circ)\) is proportional to
\begin{equation*}
\big(e^{\lambda/2}+e^{-\lambda/2}\big)^{|T_{5,6}^\rmi(\sigma_\bullet,\sigma_\circ)|}\cdot\big(e^{\lambda/2}\big)^{|T_{5,6}^\rmb(\sigma_\bullet,\sigma_\circ)|}.
\end{equation*}
Recalling that \(\svc=e^{\lambda/2}+e^{-\lambda/2}\) and \(\svcb=e^{\lambda/2}\) finishes the proof. 
\end{proof}

\subsubsection*{From six-vertex to modified ATRC.}
We now adapt the coupling described in~\cite{GlaPel23} to the Dobrushin boundary conditions.
In particular, we need to define modified the ATRC measure.
Recall the parameters from~\eqref{eq:parameters_bulk} and the (augmented) rectangular domains from Section~\ref{sec:notations}. Again, we omit~\(n,m\) from the notation. 

\begin{definition}
	\label{def:mATRC}
	The modified ATRC model \(\matrc_{n,m} \equiv \matrc_K\) is a probability measure on \(\{0,1\}^{\bar{E}}\times\{0,1\}^{\bar{E}}\) defined by
	\begin{multline*}
		\matrc_{K}(\omega_\tau,\omega_{\tau\tau'})
		\propto
		\mathds{1}_{\omega_\tau\subseteq \omega_{\tau\tau'}}\mathds{1}_{\omega_{\tau\tau'}\setminus \omega_\tau\subseteq E}2^{|\omega_\tau\cap E|}\big(\tfrac{2}{\svcb-1} \big)^{|\omega_\tau\cap E_\rmb^+|}(\svc-2)^{|\omega_{\tau\tau'}\setminus \omega_\tau|}\\
		\cdot 2^{\kappa_{K}(\omega_\tau)}\prod_{\calC\in \clusterSet_{\Lambda}(\omega_{\tau\tau'})} \big(\mathds{1}_{\calC\subseteq\Lambda} + \svcb^{I(\calC)}\big),
	\end{multline*}	
	where~\(\clusters_{K}(\omega_\tau)\) is the number of clusters of~\(\omega_\tau\), \(\clusterSet_{\Lambda}(\omega_{\tau\tau'})\) is the set of clusters of~\(\omega_{\tau\tau'}\) that intersect~$\Lambda$ (both~\(\omega_\tau\) and~\(\omega_{\tau\tau'}\) are viewed as spanning subgraph of~$K$), $E_\rmb^+$ and $E_\rmb^-$ are defined in~\eqref{eq:def-e-b-pm}, and~\(I(\mathcal{C}):=|E_\rmb^-\cap\partialedge_{\bbL_\bullet} \mathcal{C}|\).
	We say that a cluster of~\(\omega_\tau\) or~\(\omega_{\tau\tau'}\) is an {\em inner cluster} if it is entirely contained in~$\Lambda$ and a {\em boundary cluster} otherwise. 
\end{definition}

We now describe how to obtain~\(\matrc_K\) as a marginal of the coupling measure~\(\Psi^{1/0}_G\). 

\begin{lemma}
\label{lem:6V_spins_to_AT}
	Let~\((\omega,U,U')\) be distributed according to \(\Psi^{1/0}\).
	Define~\((\sigma_\bullet,\sigma_\circ)\) as described in Lemma~\ref{lem:6V_spins_to_01FK}.
	Now define \(\omega_\tau,\omega_{\tau\tau'}\in \{0,1\}^{\bar{E}}\) as follows using the notation~\(e_t=ij\in \bar{E}\) and \(e_t^*=uv\) for each tile \(t\in A^\rmi\cup\partial^+A\):
	\begin{itemize}
		\item if \(\sigma_{\circ}(u)\neq \sigma_{\circ}(v)\), set \(\omega_\tau(e_t) = \omega_{\tau\tau'}(e_t) = 1\);
		\item if \(\sigma_{\bullet}(i) \neq \sigma_{\bullet}(j)\), set \(\omega_\tau(e_t) = \omega_{\tau\tau'}(e_t) = 0\) (never happens for $t\in\partial^+ A$);
		\item if both \(\sigma_{\circ}(u)= \sigma_{\circ}(v)\) and \(\sigma_{\bullet}(i) = \sigma_{\bullet}(j)\) hold (types 5-6), let
		\[(\omega_\tau(e_t),\omega_{\tau\tau'}(e_t))= \begin{cases}
			\mathds{1}_{[0,\frac{1}{\svc})}(U_t')\,{\cdot}\,(1,1)+\mathds{1}_{[\frac{1}{\svc},\frac{2}{\svc})}(U_t')\,{\cdot}\,(0,0)+\mathds{1}_{[\frac{2}{\svc},1]}(U_t')\,{\cdot}\,(0,1) & \text{if }t\in A^\rmi,\\
			\mathds{1}_{[0,\frac{1}{\svcb})}(U_t')\,{\cdot}\,(1,1)+\mathds{1}_{[\frac{1}{\svcb},1]}(U_t')\,{\cdot}\,(0,0)& \text{if }t\in\partial^+ A.
		\end{cases}
		\]
	\end{itemize}
	Then, the law of~\((\omega_\tau,\omega_{\tau\tau'})\) is~$\matrc_{K}(\, \cdot \, | \, v_L\xleftrightarrow{\omega_\tau} v_R)$.
	In particular, we have \(\omega_\tau\subseteq\omega_{\tau\tau'}\) and \(\omega_{\tau\tau'}\setminus\omega_\tau\subseteq \{e_t:t\in A^\rmi\}=E\); also~\(\sigma_\bullet\sim \omega_{\tau\tau'}\) and \(\sigma_\circ \sim \omega_\tau^*\).
\end{lemma}

\begin{remark}
The above sampling rule for \(t\in A^\rmi\) applied to a six-vertex spin or height measure without modified boundary weight~\(\svcb\) yields an ATRC measure, as defined in Secton~\ref{sec:atrc}. See Section~\ref{sec:sixv_at_duality_coupling} for details.
\end{remark}

\begin{proof}
	We build on~\cite[Proof of Lemma~7.1]{GlaPel23}.
	One has to examine which values of \((\sigma_{\bullet},\sigma_{\circ}, U)\) result in a given pair \(\omega_\tau,\omega_{\tau\tau'} \in \{0,1\}^{\bar{E}}\). Recall that \(\sigma_{\bullet} \sim \omega_{\tau\tau'}\) and~\(\sigma_{\circ}\sim \omega_\tau^*\).
	By~\eqref{eq:def_6v-spin_dobrushin}, the probability of a quadruplet \((\sigma_{\bullet},\sigma_{\circ},\omega_\tau,\omega_{\tau\tau'})\) is
	\begin{align}
		\label{eq:prf:6Vspins_to_AT:proba}
		\nonumber
		&\spin_{\mathcal{D}}^{+,+-}(\sigma_\bullet,\sigma_\circ)\,\mathds{1}_{\sigma_{\bullet}\sim \omega_{\tau\tau'}}\,\mathds{1}_{\sigma_{\circ}\sim \omega_\tau^*}\,\mathds{1}_{\omega_\tau\subseteq\omega_{\tau\tau'}}\,\mathds{1}_{\omega_{\tau\tau'}\setminus\omega_\tau\subseteq E}
		\\
		\nonumber 
		&\qquad\cdot\prod_{t\in T_{5,6}^{\rmi}}\big(\tfrac{1}{\svc}\,(\mathds{1}_{t\in\omega_\tau}+\mathds{1}_{t\in\omega_{\tau\tau'}^*})+\tfrac{\svc-2}{\svc}\,\mathds{1}_{t\in\omega_{\tau\tau'}\setminus\omega_\tau}\big)
		\prod_{t\in T_{5,6}^{\rmb,+}}\big(\tfrac{1}{\svcb}\,\mathds{1}_{t\in\omega_\tau}+\tfrac{\svcb-1}{\svcb}\,\mathds{1}_{t\in\omega_{\tau}^*}\big)
		\\
		\nonumber &\propto\mathds{1}_{\Sigma_\Lambda^+}(\sigma_{\bullet})\mathds{1}_{\Sigma_{\Lambda'}^{+-}}(\sigma_{\circ})\,\mathds{1}_{\sigma_{\bullet}\sim \omega_{\tau\tau'}}\,\mathds{1}_{\sigma_{\circ}\sim \omega_\tau^*}\,\mathds{1}_{\omega_\tau\subseteq\omega_{\tau\tau'}}\,\mathds{1}_{\omega_{\tau\tau'}\setminus\omega_\tau\subseteq E}\,\svcb^{|T_{5,6}^{\rmb,-}|}
		\\
		&\qquad\cdot\prod_{t\in T_{5,6}^{\rmi}}\big(\mathds{1}_{t\in\omega_\tau}+\mathds{1}_{t\in\omega_{\tau\tau'}^*}+(\svc-2)\mathds{1}_{t\in\omega_{\tau\tau'}\setminus\omega_\tau}\big)
		\prod_{t\in T_{5,6}^{\rmb,+}}\big(\mathds{1}_{t\in\omega_\tau}+(\svcb-1)\mathds{1}_{t\in\omega_{\tau}^*}\big),
	\end{align}
	where \(T_{5,6}^{\rmb,\pm}(\sigma_{\bullet},\sigma_{\circ})\) is the set of tiles of types 5-6 in \(\partial^\pm A\), and where we used the shorthand \(T_{5,6}^{\#}\equiv T_{5,6}^{\#}(\sigma_{\bullet},\sigma_{\circ})\) and the fact that, for~\(\sigma_{\bullet}\in\Sigma_\Lambda^+\) and~\(\sigma_{\circ}\in\Sigma_{\Lambda'}^{+-}\),
	\begin{equation*}
		\mathds{1}_{\mathrm{ice}}(\sigma_{\bullet},\sigma_{\circ})\,\mathds{1}_{\sigma_{\bullet}\sim \omega_{\tau\tau'}}\,\mathds{1}_{\sigma_{\circ}\sim \omega_\tau^*}\,\mathds{1}_{\omega_\tau\subseteq\omega_{\tau\tau'}} = \mathds{1}_{\sigma_{\bullet}\sim \omega_{\tau\tau'}}\,\mathds{1}_{\sigma_{\circ}\sim \omega_\tau^*}\,\mathds{1}_{\omega_\tau\subseteq\omega_{\tau\tau'}}.
	\end{equation*}
	Observe that, on the event \(\{\sigma_\bullet\sim\omega_{\tau\tau'},\,\sigma_\circ\sim\omega_\tau^*,\,\omega_\tau\subseteq\omega_{\tau\tau'}\}\),
	\begin{equation*}
		\mathds{1}_{T_{5,6}^{\rmi}}(t)=\mathds{1}_{\omega_\tau}(t)\mathds{1}_{\sigma_{\circ}\sim e_t^*}+\mathds{1}_{\omega_{\tau\tau'}^*}(t)\mathds{1}_{\sigma_{\bullet}\sim e_t}+\mathds{1}_{\omega_{\tau\tau'}\setminus\omega_\tau}(t).
	\end{equation*}
	Moreover, on the same event intersected with \(\{\sigma_\bullet\in\Sigma_\Lambda^+,\,\sigma_\circ\in\Sigma_{\Lambda'}^{+-},\,\omega_{\tau\tau'}\setminus\omega_\tau\subseteq E\}\),
	\begin{equation*}
		\mathds{1}_{T_{5,6}^{\rmb,-}}(t)=\mathds{1}_{\sigma_{\bullet}\sim e_t}
		\qquad\text{and}\qquad
		\mathds{1}_{T_{5,6}^{\rmb^+}}(t)=\mathds{1}_{\omega_\tau}(t)\mathds{1}_{\sigma_{\circ}\sim e_t^*}+\mathds{1}_{\omega_{\tau}^*}(t).
	\end{equation*}
	Using this,~\eqref{eq:prf:6Vspins_to_AT:proba} becomes 
	\begin{align*}
		&\mathds{1}_{\Sigma_\Lambda^+}(\sigma_{\bullet})\mathds{1}_{\Sigma_{\Lambda'}^{+-}}(\sigma_{\circ})\mathds{1}_{\sigma_{\bullet}\sim \omega_{\tau\tau'}}\mathds{1}_{\sigma_{\circ}\sim \omega_\tau^*}\mathds{1}_{\omega_\tau\subseteq\omega_{\tau\tau'}}\mathds{1}_{\omega_{\tau\tau'}\setminus\omega_\tau\subseteq E}\prod_{t\in \partial^- A}\svcb^{\mathds{1}_{T_{5,6}^{\rmb,-}}(t)}\\
		&\cdot\prod_{t\in A^{\rmi}}\big(\mathds{1}_{t\in\omega_\tau}+\mathds{1}_{t\in\omega_{\tau\tau'}^*}+(\svc-2)\mathds{1}_{t\in\omega_{\tau\tau'}\setminus\omega_\tau}\big)^{\mathds{1}_{T_{5,6}^{\rmi}}(t)}
		\prod_{t\in \partial^+ A}\big(\mathds{1}_{t\in\omega_\tau}+(\svcb-1)\mathds{1}_{t\in\omega_{\tau}^*}\big)^{\mathds{1}_{T_{5,6}^{\rmb,+}}(t)}
		\\
		=\ 
		&\mathds{1}_{\Sigma_\Lambda^+}(\sigma_{\bullet})\mathds{1}_{\Sigma_{\Lambda'}^{+-}}(\sigma_{\circ})\mathds{1}_{\sigma_{\bullet}\sim \omega_{\tau\tau'}}\mathds{1}_{\sigma_{\circ}\sim \omega_\tau^*}\mathds{1}_{\omega_\tau\subseteq\omega_{\tau\tau'}}\mathds{1}_{\omega_{\tau\tau'}\setminus\omega_\tau\subseteq E}\prod_{e\in E_\rmb^-}\svcb^{\mathds{1}_{\sigma_\bullet\sim e}}\\
		&\cdot(\svc-2)^{|\omega_{\tau\tau'}\setminus\omega_\tau|}		
		(\svcb-1)^{|E_\rmb^+\setminus\omega_{\tau}|}.
	\end{align*}
	Observe that \(\sigma_\circ\in\Sigma_{\Lambda'}^{+-}\) and \(\sigma_\circ\sim\omega_{\tau}^*\) imply that \(v_L\) and \(v_R\) are connected in \(\omega_\tau\) (see Fig.~\ref{fig:bc_FK_loops_spins}).
	To obtain the probability of a pair \((\omega_{\tau},\omega_{\tau\tau'})\) satisfying \(v_L\xleftrightarrow{\omega_{\tau}} v_R\), we need to sum the last expression over \((\sigma_\bullet,\sigma_\circ)\).
	
	The configurations \(\sigma_\circ\in\Sigma_{\Lambda'}^{+-}\) with \(\sigma_\circ\sim\omega_\tau^*\) are in bijective correspondence with assignments of $\pm1$ to the clusters of \(\omega_\tau^*\) in \(K^*\) that are contained in \(\Lambda'\), whence there exist \(2^{\clusters_{K^*}(\omega_\tau^*)-1}\) of them. By Euler's formula,
	\begin{equation*}
		\clusters_{K^*}(\omega_\tau^*)-1 = \clusters_{K}(\omega_\tau) + |\omega_{\tau}| - |\bbV_{\bar{E}}|.
	\end{equation*}In particular, the probability of a triplet \((\omega_{\tau},\omega_{\tau\tau'},\sigma_{\bullet})\) with \(v_L\xleftrightarrow{\omega_{\tau}} v_R\) is proportional to
	\begin{multline}
		\label{eq:prf:6Vspins_to_AT:proba2}
		\mathds{1}_{\Sigma_\Lambda^+}(\sigma_{\bullet})
		\mathds{1}_{\sigma_{\bullet}\sim \omega_{\tau\tau'}} \mathds{1}_{\omega_\tau\subseteq\omega_{\tau\tau'}}\mathds{1}_{\omega_{\tau\tau'}\setminus\omega_\tau\subseteq E}
		2^{\clusters_{K}(\omega_\tau)}2^{|\omega_{\tau}|}(\svc-2)^{|\omega_{\tau\tau'}\setminus\omega_\tau|}(\svcb-1)^{|E_\rmb^+\setminus\omega_{\tau}|}\prod_{e\in E_\rmb^-}\svcb^{\mathds{1}_{\sigma_\bullet\sim e}}.
	\end{multline}
	We now sum this expression over \(\sigma_{\bullet}\):
	since~\(\sigma_\bullet\sim\omega_{\tau\tau'}\), the value of~\(\sigma_\bullet\) is constant at each cluster of~\(\omega_{\tau\tau'}\); since~\(\sigma_\bullet\in\Sigma_\Lambda^+\), this value is fixed to be \(+1\) on boundary clusters.
	Denote the sets of inner and boundary clusters of~\(\omega_{\tau\tau'}\) by~\(\clusterSet^{\rmi}(\omega_{\tau\tau'})\) and~\(\clusterSet^{\rmb}(\omega_{\tau\tau'})\), respectively.
	If~\(e\in E_\rmb^-\setminus \omega_{\tau \tau'}\), then~\(\sigma_\bullet\sim e\) in precisely two cases: (i) if \(e\in\partialedge_{\bbL_\bullet}\calC\) for some~\(\calC\in \clusterSet^{\rmb}(\omega_{\tau\tau'})\); (ii) if \(e\in\partialedge_{\bbL_\bullet}\calC\) for some~\(\calC\in \clusterSet^{\rmi}(\omega_{\tau\tau'})\) on which \(\sigma_\bullet=+1\).
	Therefore,
	\begin{equation*}
		\sum_{\substack{\sigma_\bullet\in\Sigma_\Lambda^+\\ \sigma_{\bullet}\sim \omega_{\tau\tau'}}}\prod_{e\in E_\rmb^-}\svcb^{\mathds{1}_{\sigma_\bullet\sim e}}
		= \prod_{\calC\in \clusterSet^{\rmb}(\omega_{\tau\tau'})} \svcb^{|E_\rmb^-\cap\partialedge_{\bbL_\bullet} \calC|}\cdot \prod_{\calC\in \clusterSet^{\rmi}(\omega_{\tau\tau'})}\big(1+\svcb^{|E_\rmb^-\cap\partialedge_{\bbL_\bullet} \calC|}\big).
	\end{equation*}
	Substituting the last display in the sum of~\eqref{eq:prf:6Vspins_to_AT:proba2} over~\(\sigma_{\bullet}\), we get~$\matrc_K$.
\end{proof}

\subsection{Properties of the modified ATRC}
\label{sec:mATRC_properties}
In this section, we will derive basic properties of the modified ATRC measure that will be instrumental in its analysis in Section~\ref{sec:invariance_principle} and in the proof of Theorem~\ref{thm:order-disorder:FK_loop}.
We continue in the setting of the previous section.

\subsubsection*{Positive association.}
The motivation for sampling the modified ATRC only on~\(\bar{E}=E\cup E_\rmb^+\) (rather than on \(E\cup E_\rmb\)) is that this (marginal) distribution satisfies~\eqref{eq:strong-fkg}.
This allows to ``sandwich'' the associated modified ATRC measure between unmodified ATRC measures and deduce convergence to the unique infinite-volume \emph{Gibbs measure}; see Sections~\ref{sec:atrc}~and~\ref{sec:mixing}.
\begin{lemma}\label{lem:fkg_matrc_+}
The measure \(\matrc_{K}\) satisfies~\eqref{eq:strong-fkg}.
\end{lemma}
\begin{proof}	
	Recall that \(\svcb = e^{\lambda}\) with \(\lambda>0\).
	We verify the Holley criterion~\cite{Hol74}, see also~\cite[Section~4]{GeoHagMae01}. 
	To shorten notation, given \(a,b\in\{0,1\}^{\bar{E}}\) with \(a\subseteq b\) and \(a=b\) on \(E_b^+\), we write \(\matrc_{K}(a_e,b_e \given a_{e^c}, b_{e^c})\) for the conditional probability
	\begin{equation*}
	\matrc_{K}\big((\omega_\tau(e),\omega_{\tau\tau'}(e))=(a_e,b_e)\bgiven(\omega_\tau,\omega_{\tau\tau'})=(a,b)\text{ on }(\bar{E}\setminus\{e\}\big).
	\end{equation*}	
		
	{\em Case 1: \(e=ij\in E\).}
	Denote by \(C_i^a,C_j^a\) respectively the clusters of \(a_{e^c}\) containing~\(i,j\), and by \(\widetilde{C}_i^b,\widetilde{C}_j^b\) the clusters of~\(b_{e^c}\cup E_\rmb^+\) containing~\(i,j\). Observe that \(\widetilde{C}_i^b\neq\widetilde{C}_j^b\) implies that \(\widetilde{C}_i^b\subseteq\Lambda\) or \(\widetilde{C}_j^b\subseteq\Lambda\).
	Then, for any \(a,b\in\{0,1\}^{E\cup E_\rmb^+}\) with \(a\subseteq b\) and \(a=b\) on \(E_b^+\),
	\begin{equation}\label{eq:prf:fkg_matrc_+_proba1}
	\matrc_{K}(a_e,b_e \given a_{e^c}, b_{e^c}) = \big(2^{a_e}\tfrac{\svc}{\svc+1}\big)^{\mathds{1}_{C_i^a=C_j^a}}(\svc-2)^{b_e-a_e}\tfrac{f(b)^{b_e}}{1+(\svc-1)f(b)},
	\end{equation}
	where
	\begin{equation*}
	f(b)=\mathds{1}_{\widetilde{C}_i^b=\widetilde{C}_j^b}+\alpha\big(\mathds{1}_{\widetilde{C}_i^b\neq\widetilde{C}_j^b\not\subseteq\Lambda}+\mathds{1}_{\Lambda\not\supseteq\widetilde{C}_i^b\neq\widetilde{C}_j^b}\big)+\beta\mathds{1}_{\Lambda\supseteq\widetilde{C}_i^b\neq\widetilde{C}_j^b\subseteq\Lambda}
	\end{equation*}
	and
	\begin{equation*}
	\alpha=\frac{e^{\lambda I(\widetilde{C}_i^b)}}{1+e^{\lambda I(\widetilde{C}_i^b)}},\qquad \beta = \frac{1+e^{\lambda I(\widetilde{C}_i^b)+\lambda I(\widetilde{C}_j^b)}}{(1+e^{\lambda I(\widetilde{C}_i^b)})(1+e^{\lambda I(\widetilde{C}_j^b)})}.
	\end{equation*}
	Note first that \(\svc\geq 1\) (in fact \(\svc>2\)) and \(\mathds{1}_{C_i^a=C_j^a}\) is increasing in \(a\). Moreover, \(\beta\leq\alpha\leq 1\) and both \(\alpha\) and \(\beta\) are increasing in \(I(\widetilde{C}_i^b)\) and \(I(\widetilde{C}_j^b)\) and hence in \(b\). Therefore \(f\geq 1\) is increasing in \(b\). This implies that \(\matrc_{K}(1,1 \given a_{e^c}, b_{e^c})\) is increasing in \(a,b\) and \(\matrc_{K}(0,0 \given a_{e^c}, b_{e^c})\) is decreasing in \(a,b\).
	
	{\em Case 2: \(e=\{i,j\}\in E_b^+\).} If~\(a_e=b_e\), one has
	\begin{equation}\label{eq:prf:fkg_matrc_+_proba2}
		\matrc_{K}(a_e,b_e \given a_{e^c}, b_{e^c}) = \big(2^{a_e}\tfrac{\svcb}{\svcb+1}\big)^{\mathds{1}_{C_i^a=C_j^a}}\Big(\tfrac{f(b)}{\svcb-1}\Big)^{a_e}\tfrac{\svcb-1}{\svcb-1+f(b)},
	\end{equation}
	where \(f(b)\) is as above.
	Recall that \(\svcb=e^\lambda>1\) and that \(\mathds{1}_{C_i^a=C_j^a}\) and \(f\) are increasing in \(a\) and \(b\), respectively. This implies that \(\matrc_{K}(1,1 \given a_{e^c}, b_{e^c})\) is increasing in \(a,b\), and hence \(\matrc_{K}(0,0 \given a_{e^c}, b_{e^c})\) is decreasing in \(a,b\), and the proof is complete.
\end{proof}

\subsubsection*{Finite energy.} 
The explicit expressions~\eqref{eq:prf:fkg_matrc_+_proba1}-\eqref{eq:prf:fkg_matrc_+_proba2} of the conditional probabilities readily imply that~\(\matrc_{K}\) satisfies the finite-energy property:

\begin{lemma}\label{lem:fe_matrc_+}
The exists a constant \(c>0\) such that, for any \(e\in E\) and any \(a,b\in\{0,1\}^{E\cup E_\rmb^+}\) with \(a\subseteq b\) and \(b\setminus a\subseteq E\),
\begin{equation*}
\matrc_{K}\big((\omega_\tau(e),\omega_{\tau\tau'}(e))=(a_e,b_e)\bgiven(\omega_\tau,\omega_{\tau\tau'})=(a,b)\text{ on }\bar{E}\setminus\{e\}\big)>c.
\end{equation*}
The statement also holds for \(e\in E_\rmb^+\), provided that \(a_e=b_e\).
\end{lemma}

\subsubsection*{Decoupling property.}
The following decoupling property (and analogues of it), which also holds for the standard ATRC measures, will be crucial in the derivation of mixing properties in Sections~\ref{sec:mixing}-\ref{sec:strong_mixing} and the invariance principle in Section~\ref{sec:invariance_principle}.

Let us first introduce some terminology.
\begin{definition}
\label{def:within_circuits}
We identify circuits in $\bbL_\bullet$ and $\bbL_\circ$ with their sets of edges as well as with their planar embedding given by the union of line-segments between consecutive vertices. 
Given a simple circuit $\gamma$ in $\bbL_\bullet$ or $\bbL_\circ$, we say that a subset $R\subset\R^2$ is \emph{within} $\gamma$ if it is contained in the topological closure of the bounded connected component of $\R^2\setminus \gamma$. In this case,~\(\gamma\) is said to \emph{surround}~\(R\), or~\(R\) is surrounded by~\(\gamma\). 
\end{definition}
\begin{lemma}
\label{lem:decoupling_paths_mATRC}
Let \((\omega_\tau,\omega_{\tau\tau'})\) be distributed according to \(\matrc_K\). Let \(F_1\subset F_2\subset \bar{E}\), and let \(a,b\in\{0,1\}^{F_2\setminus F_1}\) with \(a\subseteq b\) and~\(b\setminus a\subseteq E\) be such that there exist circuits in \(a^*\) and in \(b\) that both surround the edges in \(F_1\). Then, conditionally on~\((\omega_\tau|_{F_2\setminus F_1},\omega_{\tau\tau'}|_{F_2\setminus F_1})=(a,b)\), the restrictions
\begin{equation*}
(\omega_\tau|_{F_1},\omega_{\tau\tau'}|_{F_1})\quad\text{and}\quad(\omega_\tau|_{\bar{E}\setminus F_2},\omega_{\tau\tau'}|_{\bar{E}\setminus F_2})
\end{equation*}
are independent.
\end{lemma}

\begin{proof}
Take any~\(a_1,b_1\in\{0,1\}^{F_1}\) and~\(a_2,b_2\in\{0,1\}^{\bar{E}\setminus F_2}\), and write
\begin{equation*}
\bar{a}=a_1\sqcup a\sqcup a_2\quad\text{and}\quad\bar{b}=b_1\sqcup b\sqcup b_2.
\end{equation*}
By Definition~\ref{def:mATRC}, 
\begin{align}
	&\matrc_{K}\big((\bar{a},\bar{b})\bgiven(\omega_\tau|_{F_2\setminus F_1},\omega_{\tau\tau'}|_{F_2\setminus F_1})=(a,b)\big)\label{eq:prf:decoupling_prop}\\
	&\propto \bigg(\prod_{i=1}^2\mathds{1}_{a_i\subseteq b_i}\mathds{1}_{b_i\setminus a_i\subseteq E} 2^{|a_i\cap E|}\big(\tfrac{2}{\svcb-1} \big)^{|a_i\cap E_\rmb^+|}(\svc-2)^{|b_i\setminus a_i|}\bigg) 2^{\kappa_{K}(\bar{a})}\prod_{\calC\in \clusterSet_{\Lambda}(\bar{b})} \big(\mathds{1}_{\calC\subseteq\Lambda} + \svcb^{I(\calC)}\big),	\nonumber
\end{align}
where the constant of proportionality depends on~\(a,b\). 

Let~\(\gamma'\subset a^*\subset\bbE^\circ\) be a circuit that surrounds~\(F_1\). 
Let~\(E_{\gamma'}^\mathrm{in}\subset\bar{E}\) be the set of edges within~\(\gamma'\), and set~\(E_{\gamma'}^\mathrm{ex}=\bar{E}\setminus (E_{\gamma'}^\mathrm{in}\cup*\gamma')\). For~\(\#\in\{\mathrm{in},\mathrm{ex}\}\), let~\(K_a^\#\) be the graph obtained from~\((\bbV_{E_{\gamma'}^\#},E_{\gamma'}^\#)\) by identifying endpoints of edges in~\(a\). As the circuit~\(\gamma'\) is open in~\(a^*\) and separates~\(K_a^\mathrm{in}\) from~\(K_a^\mathrm{ex}\), we clearly have
\begin{equation*}
\kappa_{K}(\bar{a})=\kappa_{K_a^\mathrm{in}}(a_1)+\kappa_{K_a^\mathrm{ex}}(a_2).
\end{equation*}
Similarly, let~\(\gamma\subset b\subset\bbE^\bullet\) be a circuit that surrounds~\(F_1\). Let~\(E_{\gamma}^\mathrm{in}\subset\bar{E}\) be the set of edges within~\(\gamma\), and set~\(E_{\gamma}^\mathrm{ex}=\bar{E}\setminus E_{\gamma}^\mathrm{in}\). For~\(\#\in\{\mathrm{in},\mathrm{ex}\}\), define~\(K_b^\#\) as the graph obtained from~\((\bbV_{E_{\gamma}^\#},E_{\gamma}^\#)\) by identifying endpoints of edges in~\(b\). 
Let~\(\clusterSet_{K_b^\mathrm{in}}(b_1)\) and~\(\clusterSet_{K_b^\mathrm{ex}}(b_2)\) be the sets of clusters of~\(b_1\) in~\(K_b^\mathrm{in}\) and~\(b_2\) in~\(K_b^\mathrm{ex}\), respectively. 
Then, every cluster in~\(\clusterSet_{\Lambda}(\bar{b})\) that does not contain~\(\bbV_\gamma\) belongs to either~\(\clusterSet_{K_b^\mathrm{in}}(b_1)\) or~\(\clusterSet_{K_b^\mathrm{ex}}(b_2)\).
Moreover, every cluster~\(\calC\in\clusterSet_{K_b^\mathrm{in}}(b_1)\) except that of the vertex corresponding to~\(\gamma\) satisfies~\(\calC\subseteq\Lambda\) and~\(I(\calC)=0\). Therefore,
\begin{equation*}
\prod_{\calC\in \clusterSet_{\Lambda}(\bar{b})} \big(\mathds{1}_{\calC\subseteq\Lambda} + \svcb^{I(\calC)}\big)=2^{|\clusterSet_{K_b^\mathrm{in}}(b_1)|-1}\prod_{\calC\in \clusterSet_{K_b^\mathrm{ex}}(b_2)} \big(\mathds{1}_{\calC\subseteq\Lambda} + \svcb^{I(\calC)}\big).
\end{equation*}
Therefore, the expression in~\eqref{eq:prf:decoupling_prop} factorises, and the proof is complete. 
\end{proof}

\subsection{Interface in the modified ATRC}
\label{sec:interface_matrc}
Recall the coupling measure \(\Psi^{1/0}\) of the six-vertex spin random variable \((\sigma_\bullet,\sigma_\circ)\sim\spin_{\mathcal{D}}^{+,+-}\) and the modified ATRC pair \((\omega_\tau,\omega_{\tau\tau'})\sim\matrc_{K}(\cdot\given v_L\xleftrightarrow{\omega_\tau}v_R)\) (Lemmata~\ref{lem:6V_spins_to_01FK} and~\ref{lem:6V_spins_to_AT}).

Since \(\sigma_\circ\) is constant on edges in \(\omega_\tau^*\), the Dobrushin boundary conditions (\(\sigma_\circ\in\Sigma_{\Lambda'}^{+-}\)) impose the existence of a path in~\(\omega_\tau\) that connects~\(v_L\) and~\(v_R\).

\begin{definition}
	Denote by~\(\calC=\calC_{v_L}\subseteq\bar{V}\) the cluster of \(v_L\) in \(\omega_\tau\); see Fig.~\ref{fig:interface_matrc}.
	Let~\(\gamma\subset\R^2\) be the unique closed curve formed by the line-segments between endpoints of edges in \(*\partialedge_{\bbL_\bullet} \calC\subset\bbE^\circ\) such that~\(\calC\) is contained in the bounded connected component of $\R^2\setminus\gamma$.
	Denote the closure of this component by~\(\mathcal{P}\), and define the {\em interface}~\(\Gamma=\Gamma_{\atrc}^{n,m}\) by
	\[
		\Gamma:=\mathcal{P}\cap\bbL_\bullet\subseteq \bar{V}.
	\]
\end{definition}

\begin{figure}
\includegraphics[scale=0.45,page=1]{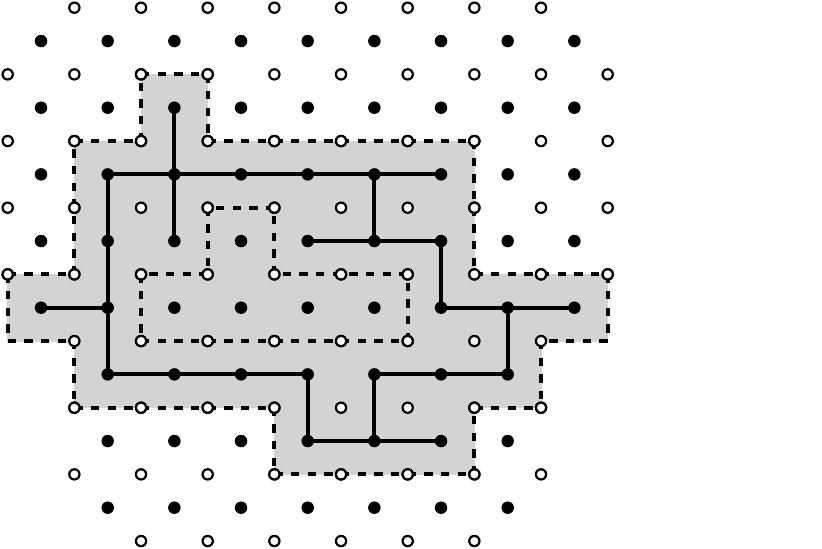}
\qquad
\includegraphics[scale=0.45,page=2]{at_interface_cluster}
\caption{Left: A realisation of $\calC=\calC_{v_L}$ (endpoints of solid edges) in $K$ for $n=m=3$, the dual $*\partialedge_{\bbL_\bullet}\calC$ of its edge-boundary in $\bbL_\bullet$ (dashed edges), and the surrounding polygon $\mathcal{P}$ (grey). Right: a path in $*\partialedge_{K}\calC$ connecting \((-n-\tfrac{3}{2},\tfrac{1}{2})\) and \((n+\tfrac{3}{2},\tfrac{1}{2})\) on which~\(\sigma_\circ=+1\) (red dashed edges), and a path in $*\partialedge_{K}\calC$ connecting \((-n-\tfrac{1}{2},-\tfrac{1}{2})\) and \((n+\tfrac{1}{2},-\tfrac{1}{2})\) on which~\(\sigma_\circ=-1\) (green dashed edges).}
\label{fig:interface_matrc}
\end{figure}

\section{Weak mixing in the ATRC}
\label{sec:mixing}

The main results proved in this section are the single-edge relaxation of the ATRC, Proposition~\ref{prop:edge_relax_ATRC}, and the \emph{ratio} weak mixing property for the ATRC, Theorem~\ref{thm:ratio_weak_mixing_ATRC}.

In Subsections~\ref{sec:height-func} and~\ref{sec:height-func-relax} we define the height function of the six-vertex model and use~\cite{GlaPel23} to show that it relaxes exponentially to its infinite-volume limit.
In Subsection~\ref{sec:atrc_relax}, we show that this implies exponential relaxation for the ATRC measure with ``\(\atrcfree,\atrcwired\)'' boundary conditions, stated in Proposition~\ref{prop:atrc01_relax}.
In Subsection~\ref{sec:proof_edge_relax_atrc}, we derive Proposition~\ref{prop:edge_relax_ATRC} from Proposition~\ref{prop:atrc01_relax} and an input from~\cite{AouDobGla24}, stated in Proposition~\ref{prop:exp_decay_omega_tau}. 
In Subsection~\ref{sec:ratio-mixing}, we show how to derive from~Proposition~\ref{prop:edge_relax_ATRC} the \emph{exponential} weak mixing property of the ATRC, stated in Theorem~\ref{thm:weak_mixing_ATRC}.
Finally, we prove Theorem~\ref{thm:ratio_weak_mixing_ATRC} using the classical work of Alexander~\cite{Ale98}, Theorem~\ref{thm:weak_mixing_ATRC} and the imput, Proposition~\ref{prop:exp_decay_omega_tau}, from~\cite{AouDobGla24}.

\subsection{The six-vertex height function}\label{sec:height-func}

The proof of Proposition~\ref{prop:edge_relax_ATRC} relies on exponential relaxation of the ATRC measures with ``$\atrcfree,\atrcwired$'' boundary conditions. The latter will be established via the height function representation of the six-vertex model and the BKW coupling~\cite{BaxKelWu76} with FK-percolation, while it also requires some input from~\cite{GlaPel23}. 
We first introduce the six-vertex height function and provide the necessary graph definitions.

\subsubsection{Six-vertex height function}
\label{sec:def_height_function_measures}
Recall the edge orientations and spin representation of the six-vertex model introduced in Section~\ref{sec:combinatorial_mappings}.
A six-vertex \emph{height function} is an assignment of integers, called \emph{heights}, to the vertices in both \(\bbL_\bullet\) and \(\bbL_\circ\). 
Given edge orientations that satisfy the ice rule, define \(h\) on \(\bbL_{\bullet}\cup \bbL_{\circ}\) as follows. Fix an integer height at some arbitrary fixed vertex. Then iteratively define the heights at other vertices by increasing the height by 1 when traversing an edge \(e\in\bbE^\diamond\) from its left to its right (with respect to its assigned orientation) and decreasing the height by 1 when traversing an arrow from its right to its left; see Fig.~\ref{fig:six-vertex_arrows_spins_heights}.
As with the spins, the procedure is self-consistent due to the ice rule.
Note that the heights on \(\bbL_{\bullet}\) and on \( \bbL_{\circ}\) automatically have different parity.
By convention, we set the parity on \(\bbL_{\bullet}\) to be even.
The gradient of \(h\) is in bijective correspondence with the edge orientations and hence with the spin representation, up to a global spin flip, as demonstrated by the following relation (see Fig.~\ref{fig:six-vertex_arrows_spins_heights}):
\begin{equation*}
h(u)-h(i)=\sigma_\bullet(i)\sigma_\circ(u)\qquad\text{for any }t\in\bbL_\diamond\text{ with }i\in e_t,\,u\in e_t^*.
\end{equation*}
On the other hand, up to a global spin flip, the spins are obtained from the height function by setting, for all \(i\in \bbL_{\bullet}\) and \(u\in \bbL_{\circ}\),
\begin{equation}\label{eq:spins_of_hf}
	\sigma_{\bullet}(i) = \begin{cases}
		+1 & \text{ if } h(i) = 0 \mod{4}\\
		-1 & \text{ if } h(i) = 2 \mod{4}
	\end{cases},\quad \sigma_{\circ}(u) = \begin{cases}
		+1 & \text{ if } h(u) = 1 \mod{4}\\
		-1 & \text{ if } h(u) = 3 \mod{4}
	\end{cases}.
\end{equation}
This motivates the following definition.
\begin{definition}
A function ${h:\bbL_\bullet\cup\bbL_\circ\to\Z}$ is called a (six-vertex) height function if it satisfies the following:
\begin{itemize}
	\item for any $t\in\bbL_\diamond$, $i\in t\cap\bbL_\bullet$, and~$u\in t\cap\bbL_\circ$, one has $|h(i)-h(u)|=1$,
	\item for any $i\in\bbL_\bullet$, one has $h(i)\in 2\Z$.
\end{itemize}
Denote the set of all height functions by $\Omega_{\mathrm{hf}}$. 
The type of a tile $t\in\bbL_\diamond$ in a height function $h$ is given by the type of its gradient function; see Fig.~\ref{fig:six-vertex_arrows_spins_heights}.
Fix $g\in\Omega_{\mathrm{hf}}$, parameters $\svc,\svcb>0$ and a finite subset $\Delta\subset\bbL_\bullet\cup\bbL_\circ$.
The set $A:=A_\Delta\subset\bbL_\diamond$ of tiles of $\Delta$ is given by the tiles with at least one corner in $\Delta$. The set $\partial A\subseteq A$ of boundary tiles of $\Delta$ is given by the tiles with precisely one corner in $\Delta$.
Let $\delta\subseteq\partial A$. 
The corresponding height function measure on $\Z^{\bbL_\bullet\cup\bbL_\circ}$ is defined by 
\begin{equation}\label{eq:hf_def}
	\hf_{\Delta,\delta;\svc,\svcb}^{g}(h)\propto\svc^{|T_{5,6}^{A\setminus \delta}(h)|}\,\svcb^{|T_{5,6}^{\delta}(h)|}\,\ind_{\Omega_\mathrm{hf}}(h)\,\ind_{h(x)=g(x)\,\forall x\in (\bbL_\bullet\cup\bbL_\circ)\setminus\Delta}\,,
\end{equation}
where $T_{5,6}^{A\setminus \delta}(h)$ (resp. $T_{5,6}^{\delta}(h)$) is the set of tiles $t\in A\setminus \delta$ (resp. $t\in \delta$) of type 5-6 in $h$.
\end{definition}
If $g$ is constant $n$ on $\bbL_\bullet$ and constant $m$ on $\bbL_\circ$ ($n$ even and $m=n\pm1$), we simply write $\hf_{\Delta,\delta;\svc,\svcb}^{n,m}$.
When $\delta=\partial A$, we omit $\delta$ from the subscript. When $\delta=\varnothing$, we omit $\delta$ and $\svcb$ from the subscript.

\begin{figure}
\includegraphics[scale=0.45,page=1]{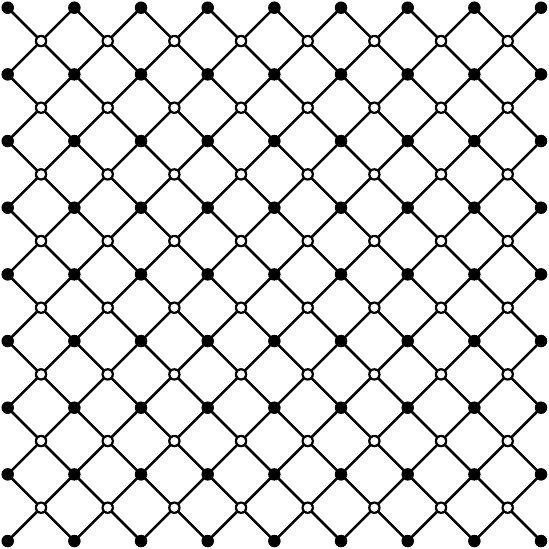}
\quad
\includegraphics[scale=0.45,page=2]{rotated_lattice_domains}
\quad
\includegraphics[scale=0.45,page=3]{rotated_lattice_domains}
\caption{Left: part of the rotated square lattice $\bbL$. Its faces are the tiles in $\bbL_\diamond$. Center: part of the augmented lattice $\overline{\bbL}$. Right: $\bbL$-domains given by the vertices strictly within simple circuits in $\overline{\bbL},\,\bbL_\bullet,\,\bbL_\circ$, respectively. The lower left and right $\bbL$-domains are even and odd, respectively.}
\label{fig:rotated_lattice}
\end{figure}

\subsubsection{Rotated lattice and domains}
Consider the rotated square lattice $\bbL$ with vertex-set $\bbL_\bullet\cup\bbL_\circ$ and edges between nearest neighbours, that is, between vertices of Euclidean distance $1/\sqrt{2}$.
Given $\Delta\subseteq\bbL$, we write $\Delta_\bullet=\Delta\cap\bbL_\bullet$ and $\Delta_\circ=\Delta\cap\bbL_\circ$.
The augmented graph $\overline{\bbL}$ has the same vertex-set as $\bbL$ and all edges of $\bbL,\,\bbL_\bullet$ and $\bbL_\circ$, see Fig.~\ref{fig:rotated_lattice}.
We restrict the notion of simple circuits in $\overline{\bbL}$ to those that do not traverse both $e$ and $e^*$ for any $e\in\bbE^\bullet$, so that they can be embedded in $\R^2$.
We identify such circuits with their planar embedding. 

\begin{definition}[$\bbL$-domains]
A finite subset $\mathcal{D}\subset\bbL$ is called an $\bbL$-domain if there exists a simple circuit $C$ in $\overline{\bbL}$ such that $\mathcal{D}$ is given by the vertices of $\bbL$ strictly within $C$, that is, in the bounded connected component of $\R^2\setminus C$.
It is called \emph{even} (respectively, \emph{odd}) if $\partialex_{\bbL}\mathcal{D}\subset\bbL_\bullet$ (respectively, $\partialex_{\bbL}\mathcal{D}\subset\bbL_\circ$); see Fig.~\ref{fig:rotated_lattice}.
\end{definition}

\begin{definition}[$\bbL_\bullet$- and $\bbL_\circ$-domains]
A subgraph $G=(V,E)$ of $\bbL_\bullet$ is called an $\bbL_\bullet$-domain of the \emph{first kind} if $E=\bbE_V$ and $V$ is given by the vertices within (see Definition~\ref{def:within_circuits}) a simple circuit in $\bbL_\bullet$.
We define $\bbL_\circ$-domains of the first kind in the same manner.
We say that a subgraph $G=(V,E)$ of $\bbL_\bullet$ is a domain of the \emph{second kind} if its dual $(\bbV_{*E},*E)$ is an $\bbL_\circ$-domain of the first kind.
It is called an $\bbL_\bullet$-domain if it is an $\bbL_\bullet$-domain of the first or second kind.
We define $\bbL_\circ$-domains (of the second kind) analogously; see Fig.~\ref{fig:domains}.
\end{definition}
Given a subgraph $G=(V,E)$ of $\bbL_\bullet$ with dual $G^*=(\bbV_{*E},*E)$, define $\Delta_G$ as the set of vertices of degree $4$ in $G$ and $G^*$, that is, 
\begin{equation*}
\Delta_G:=\{i\in V:|\partialedge_{G}\{i\}|=4\}\cup\{u\in\bbV_{*E}: |\partialedge_{G^*}\{u\}|=4\}.
\end{equation*}
Observe that, if $G$ is an $\bbL_\bullet$-domain, then $\mathcal{D}_G:=\Delta_G$ is an $\bbL$-domain. In this case, any tile $t\in\bbL_\diamond$ intersects $\mathcal{D}_G$ precisely if its associated edge $e_t\in\bbE^\bullet$ belongs to $E$, that is, $\{e_t:t\in A_{\mathcal{D}_G}\}=E$; see Section~\ref{sec:def_height_function_measures} and Fig.~\ref{fig:domains}.

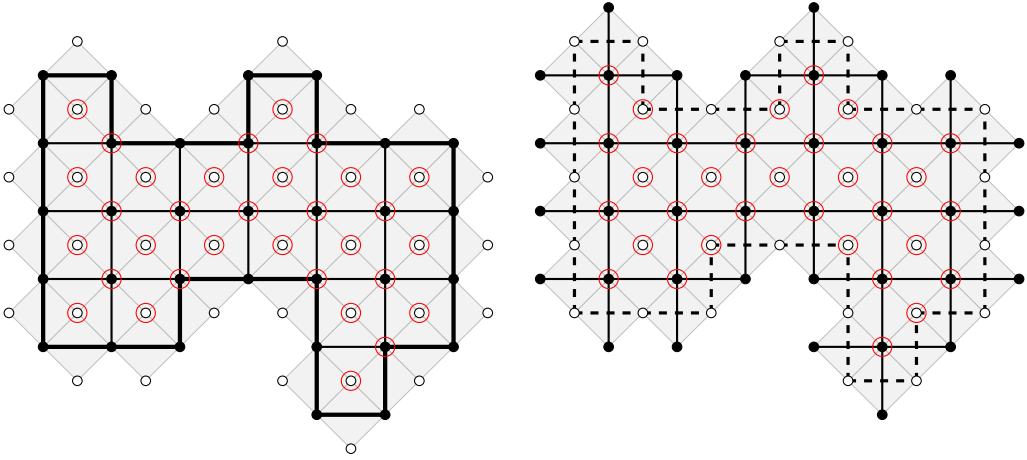
\begin{figure}
\begin{tikzpicture}[scale=0.9]
		\foreach \i\j in {0/0, 1/0, 2/1, 3/1, 4/-1, 5/0, 5/3, 4/3, 3/4, 2/3, 1/3, 0/4}{
			\filldraw[ultra thin, lightgray, fill=lightgray!20] ({\i},\j)--(\i+0.5,{\j+0.5})--({\i+1},\j)--(\i+0.5,{\j-0.5})--({\i},\j);
			\Hedge{\i+0.5}{\j}{ultra thick};
		}

		\foreach \i\j in {2/0, 4/0, 4/-1, 5/-1, 6/0, 6/1, 6/2, 4/3, 3/3, 1/3, 0/0, 0/1, 0/2, 0/3}{
			\filldraw[ultra thin, lightgray, fill=lightgray!20] ({\i-0.5},\j+0.5)--(\i,{\j+1})--({\i+0.5},\j+0.5)--(\i,{\j})--({\i-0.5},\j+0.5);
			\Vedge{\i}{\j+0.5}{ultra thick};
		}

		\foreach \i\j in {4/0, 0/1, 1/1, 4/1, 5/1, 0/2, 1/2, 2/2, 3/2, 4/2, 5/2, 0/3, 3/3}{
			\filldraw[ultra thin, lightgray, fill=lightgray!20] ({\i},\j)--(\i+0.5,{\j+0.5})--({\i+1},\j)--(\i+0.5,{\j-0.5})--({\i},\j);
			\Hedge{\i+0.5}{\j}{thick};
		}
		
		\foreach \i\j in {1/0, 5/0, 1/1, 2/1, 3/1, 4/1, 5/1, 1/2, 2/2, 3/2, 4/2, 5/2}{
			\filldraw[ultra thin, lightgray, fill=lightgray!20] ({\i-0.5},\j+0.5)--(\i,{\j+1})--({\i+0.5},\j+0.5)--(\i,{\j})--({\i-0.5},\j+0.5);
			\Vedge{\i}{\j+0.5}{thick};
		}
		
		\foreach \i\j in {4/-1, 5/-1, 0/0, 1/0, 2/0, 4/0, 5/0, 6/0, 0/1, 1/1, 2/1, 3/1, 4/1, 5/1, 6/1, 0/2, 1/2, 2/2, 3/2, 4/2, 5/2, 6/2, 0/3, 1/3, 2/3, 3/3, 4/3, 5/3, 6/3, 0/4, 1/4, 3/4, 4/4}{
			\filldraw[black] (\i,\j) circle (2pt);
		}
		
		\foreach \i\j in {4/-2, 0/-1, 1/-1, 3/-1, 4/-1, 5/-1, -1/0, 0/0, 1/0, 2/0, 3/0, 4/0, 5/0, 6/0, -1/1, 0/1, 1/1, 2/1, 3/1, 4/1, 5/1, 6/1, -1/2, 0/2, 1/2, 2/2, 3/2, 4/2, 5/2, 6/2, -1/3, 0/3, 1/3, 2/3, 3/3, 4/3, 5/3, 0/4, 3/4}{
			\filldraw[black, fill=white] (\i+0.5,\j+0.5) circle (2pt);
		}
		
		\foreach \i\j in {5/0, 1/1, 2/1, 4/1, 5/1, 1/2, 2/2, 3/2, 4/2, 5/2, 1/3, 3/3, 4/3}{
			\draw[red] (\i,\j) circle(4pt);
		}
		
		\foreach \i\j in {4/-1, 0/0, 1/0, 4/0, 5/0, 0/1, 1/1, 2/1, 3/1, 4/1, 5/1, 0/2, 1/2, 2/2, 3/2, 4/2, 5/2, 0/3, 3/3}{
			\draw[red] (\i+0.5,\j+0.5) circle(4pt);
		}
\end{tikzpicture}
\quad
\begin{tikzpicture}[scale=0.9]
		\foreach \i\j in {0/0, 1/0, 2/1, 3/1, 4/-1, 5/0, 5/3, 4/3, 3/4, 2/3, 1/3, 0/4}{
			\filldraw[ultra thin, lightgray, fill=lightgray!20] ({\i},\j)--(\i+0.5,{\j+0.5})--({\i+1},\j)--(\i+0.5,{\j-0.5})--({\i},\j);
			\Hedge{\i+0.5}{\j}{very thick, dashed};
		}

		\foreach \i\j in {2/0, 4/0, 4/-1, 5/-1, 6/0, 6/1, 6/2, 4/3, 3/3, 1/3, 0/0, 0/1, 0/2, 0/3}{
			\filldraw[ultra thin, lightgray, fill=lightgray!20] ({\i-0.5},\j+0.5)--(\i,{\j+1})--({\i+0.5},\j+0.5)--(\i,{\j})--({\i-0.5},\j+0.5);
			\Vedge{\i}{\j+0.5}{very thick, dashed};
		}

		\foreach \i\j in {1/0, 5/0, 1/1, 2/1, 3/1, 4/1, 5/1, 1/2, 2/2, 3/2, 4/2, 5/2}{
			\filldraw[ultra thin, lightgray, fill=lightgray!20] ({\i-0.5},\j+0.5)--(\i,{\j+1})--({\i+0.5},\j+0.5)--(\i,{\j})--({\i-0.5},\j+0.5);
			\Hedge{\i}{\j+0.5}{thick};
		}
		
		\foreach \i\j in {2/0, 4/0, 4/-1, 5/-1, 6/0, 6/1, 6/2, 4/3, 3/3, 1/3, 0/0, 0/1, 0/2, 0/3}{
			\Hedge{\i}{\j+0.5}{thick};
		}		

		\foreach \i\j in {4/0, 0/1, 1/1, 4/1, 5/1, 0/2, 1/2, 2/2, 3/2, 4/2, 5/2, 0/3, 3/3}{
			\filldraw[ultra thin, lightgray, fill=lightgray!20] ({\i},\j)--(\i+0.5,{\j+0.5})--({\i+1},\j)--(\i+0.5,{\j-0.5})--({\i},\j);
			\Vedge{\i+0.5}{\j}{thick};
		}
		
		\foreach \i\j in {0/0, 1/0, 2/1, 3/1, 4/-1, 5/0, 5/3, 4/3, 3/4, 2/3, 1/3, 0/4}{
			\Vedge{\i+0.5}{\j}{thick};
		}

		\foreach \i\j in {4/-2, 0/-1, 1/-1, 3/-1, 4/-1, 5/-1, -1/0, 0/0, 1/0, 2/0, 3/0, 4/0, 5/0, 6/0, -1/1, 0/1, 1/1, 2/1, 3/1, 4/1, 5/1, 6/1, -1/2, 0/2, 1/2, 2/2, 3/2, 4/2, 5/2, 6/2, -1/3, 0/3, 1/3, 2/3, 3/3, 4/3, 5/3, 0/4, 3/4}{
			\filldraw[black] (\i+0.5,\j+0.5) circle (2pt);
		}	
						
		\foreach \i\j in {4/-1, 5/-1, 0/0, 1/0, 2/0, 4/0, 5/0, 6/0, 0/1, 1/1, 2/1, 3/1, 4/1, 5/1, 6/1, 0/2, 1/2, 2/2, 3/2, 4/2, 5/2, 6/2, 0/3, 1/3, 2/3, 3/3, 4/3, 5/3, 6/3, 0/4, 1/4, 3/4, 4/4}{
			\filldraw[black, fill=white] (\i,\j) circle (2pt);
		}

		\foreach \i\j in {4/-1, 0/0, 1/0, 4/0, 5/0, 0/1, 1/1, 2/1, 3/1, 4/1, 5/1, 0/2, 1/2, 2/2, 3/2, 4/2, 5/2, 0/3, 3/3}{
			\draw[red] (\i+0.5,\j+0.5) circle(4pt);
		}
		
		\foreach \i\j in {5/0, 1/1, 2/1, 4/1, 5/1, 1/2, 2/2, 3/2, 4/2, 5/2, 1/3, 3/3, 4/3}{
			\draw[red] (\i,\j) circle(4pt);
		}		
		\draw[white] (0,-2) circle(2pt);
\end{tikzpicture}
\caption{Left: a circuit in $\bbL_\bullet$ (thick black edges) and the corresponding $\bbL_\bullet$-domain of the first kind (thin and thick black edges).
Right: a circuit in $\bbL_\circ$ (dashed edges) obtained by shifting the left one by $(1/2,1/2)$ and the corresponding $\bbL_\bullet$-domain of the second kind  (black edges).
The tiles corresponding to the edges of the domains are shaded grey, and the vertices of the corresponding $\bbL$-domains are surrounded by red circles.}
\label{fig:domains}
\end{figure}

\subsubsection{Duality coupling with the AT model}
\label{sec:sixv_at_duality_coupling}
The ATRC measures can be sampled locally from the HF measures, which is the content of the following lemma.
We build on~\cite{GlaPel23}, where a marginal of the ATRC measure was sampled. 
Given $\svc>2$, a height function $h\in\Z^{\bbL}$, an edge $e=ij\in \bbE^\bullet$ with $e^*=uv$, and $x\in [0,1]$, define $\mathcal{X}_\svc(h,e,x)\in\{(0,0),(0,1),(1,1)\}$ as follows:
\begin{equation}\label{eq:sample_atrc_hf}
\begin{aligned}
&\bullet\text{ if }h(u)\neq h(v),\text{ set }\mathcal{X}_\svc(h,e,x)=(1,1),\\
&\bullet\text{ if }h(i)\neq h(j),\text{ set }\mathcal{X}_\svc(h,e,x)=(0,0),\\
&\bullet\text{ if }h(u)=h(v)\text{ and }h(i)=h(j),\text{ set }
\\
&\qquad\mathcal{X}_\svc(h,e,x)=\mathds{1}_{[0,\frac{1}{\svc})}(x)\cdot(1,1)+\mathds{1}_{[\frac{1}{\svc},\frac{2}{\svc})}(x)\cdot(0,0)+\mathds{1}_{[\frac{2}{\svc},1]}(x)\cdot(0,1).
\end{aligned}
\end{equation}
Observe that, in order to define $\mathcal{X}_\svc(h,e,x)$, it suffices to know the values of the corresponding six-vertex spins $\sigma_\bullet(i),\sigma_\bullet(j),\sigma_\circ(u),\sigma_\circ(v)$, defined in~\eqref{eq:spins_of_hf}.

\begin{lemma}\label{lem:sample_atrc_hf}
Let $0<J<U$ satisfy $\sinh 2J=e^{-2U}$ and take~$\svc:=\coth 2J$.
Let $G=(V,E)$ be an $\bbL_\bullet$-domain, and let $\mathcal{D}:=\mathcal{D}_G$ be its associated $\bbL$-domain.
Let $h$ be distributed according to $\hf_{\mathcal{D};\svc}^{0,1}$, and let $(U_e)_{e\in E}$ be a sequence of i.i.d. uniform random variables on $[0,1]$, independent of $h$.
Define $\omega_\tau,\omega_{\tau\tau'}\in \{0,1\}^{\bbE^\bullet}$ by setting, for~$e\in\bbE^\bullet$,
\begin{equation}\label{eq:sample_pair_atrc_from_hf}
(\omega_\tau(e),\omega_{\tau\tau'}(e)):=
\begin{cases}
(0,1) & \text{if }e\in\bbE^\bullet\setminus E,\\
\mathcal{X}_\svc(h,e,U_e) & \text{if }e\in E.
\end{cases}
\end{equation}
Then, the law of $(\omega_\tau,\omega_{\tau\tau'})$ is given by $\atrc_{G;J,U}^{\atrcfree,\atrcwired}$.
\end{lemma}

\begin{proof}
Denote the underlying probability measure by $\mathbf{P}$.
Recall that a tile $t\in\bbL_\diamond$ intersects $\mathcal{D}$ precisely if its associated edge $e_t\in\bbE^\bullet$ belongs to $E$, see Fig.~\ref{fig:domains}.
Let $(\sigma_\bullet,\sigma_\circ)$ with values in $\{\pm 1\}^{\bbL_\bullet}\times\{\pm 1\}^{\bbL_\circ}$ be the spin configurations obtained from $h$ by~\eqref{eq:spins_of_hf}.
For ease of notation, we use the same symbols for the random variables and their realisations.
For $\sigma_\bullet\in\{\pm 1\}^{\bbL_\bullet},\,\sigma_\circ\in\{\pm 1\}^{\bbL_\circ},\,\omega\in\{0,1\}^{\bbE^\bullet}$ and $\#\in\{0,1\}$, define
\begin{equation*}
\Sigma_{\mathcal{D}_\bullet}^+(\sigma_\bullet):=\ind_{\sigma_\bullet\equiv 1 \text{ on } \bbL_\bullet\setminus\mathcal{D}_\bullet},\quad
\Sigma_{\mathcal{D}_\circ}^+(\sigma_\circ):=\ind_{\sigma_\circ\equiv 1 \text{ on } \bbL_\circ\setminus\mathcal{D}_\circ},\quad
\Omega_E^{\#}(\omega)=\ind_{\omega\equiv\# \text{ on } \bbE^\bullet\setminus E}.
\end{equation*}
Then, the law of the quadruple~$(\sigma_\bullet,\sigma_\circ,\omega_\tau,\omega_{\tau\tau'})$ can be written as follows:
\begin{multline*}
	\mathbf{P}(\sigma_\bullet,\sigma_\circ,\omega_\tau,\omega_{\tau\tau'}) \propto
	\ind_{\mathrm{ice}}(\sigma_\bullet,\sigma_\circ)\,\Sigma_{\mathcal{D}_\bullet}^+(\sigma_\bullet)\,\Sigma_{\mathcal{D}_\circ}^+(\sigma_\circ)\,\ind_{\sigma_\bullet\sim\omega_{\tau\tau'}}\,\ind_{\sigma_\circ\sim\omega_\tau^*}\\\cdot(\svc-2)^{\abs{(\omega_{\tau\tau'}\setminus\omega_\tau)\cap E}}\,\ind_{\omega_\tau\subseteq\omega_{\tau\tau'}}\,\Omega_E^0(\omega_\tau)\,\Omega_E^1(\omega_{\tau\tau'}).
\end{multline*}
Observe that $\sigma_\bullet\sim\omega_{\tau\tau'},\,\sigma_\circ\sim\omega_\tau^*$ and $\omega_\tau\subseteq\omega_{\tau\tau'}$ imply that $(\sigma_\bullet,\sigma_\circ)$ satisfies the ice rule. 
Summing over $\sigma_\bullet$ and $\sigma_\circ$, we obtain the law of~$(\omega_\tau,\omega_{\tau\tau'})$:
\begin{equation*}
\mathbf{P}(\omega_\tau,\omega_{\tau\tau'})\propto
2^{\clusters_V(\omega_{\tau\tau'})+\clusters_{V'}(\omega_\tau^*)}\,\Omega_E^0(\omega_\tau)\,\Omega_E^1(\omega_{\tau\tau'})\,\ind_{\omega_\tau\subseteq\omega_{\tau\tau'}}\,(\svc-2)^{\abs{(\omega_{\tau\tau'}\setminus\omega_\tau)\cap E}}, 
\end{equation*}
where $V':=\bbV_{*E}$, and where we used that there exist $2^{\clusters_V(\omega_{\tau\tau'})-1}$ spin configurations $\sigma_\bullet\sim\omega_{\tau\tau'}$ with $\Sigma_{\mathcal{D}_\bullet}^+(\sigma_\bullet)=1$, and similarly for $\sigma_\circ\sim\omega_\tau^*$.
By planar duality, $\clusters_{V'}(\omega_\tau^*)=\clusters_{V}(\omega_\tau)+\abs{\omega_\tau\cap E}+\mathrm{const}(G)$.
Since $\rmw_\tau,\rmw_{\tau\tau'}$ from~\eqref{eq:atrc_weights} satisfy $\rmw_\tau=2$ and $\rmw_{\tau\tau'}=\svc-2$, the proof is complete.
\end{proof}

\subsubsection{Baxter--Kelland--Wu coupling}\label{sec:bkw}

In Section~\ref{sec:coupling}, we described the BKW coupling by sampling the FK-percolation from the six-vertex model.
We now focus on the reverse direction of the BKW.
Fix $q>4$ and $\beta = \beta_c(q) = \ln(1+\sqrt{q})$, and let $\lambda,\svc,\svcb$ be as in~\eqref{eq:parameters_bulk} and~\eqref{eq:parameters_bnd}.
Recall the loop representation of FK-percolation and the representations of the six-vertex model introduced in Section~\ref{sec:combinatorial_mappings}.

Let $\mathcal{D}$ be an even $\bbL$-domain, and set
\begin{equation}\label{eq:bkw_even_graph}
E:=*\bbE_{\mathcal{D}_\circ}\subset\bbE^\bullet,\quad
\Lambda:=\bbV_E\subset\bbL_\bullet,\quad
\text{and}\quad G:=(\Lambda,E).
\end{equation}
The following result~\cite{BaxKelWu76} is classical. It may be proved through the arguments presented in the proof of Lemma~\ref{lem:6V_spins_to_01FK}; see~\cite[Theorem 7]{GlaPel23} for a proof in a similar setting.
We will construct a height function from a loop configuration by increasing or decreasing the height whenever we cross a loop, which is the same as orienting the loops as in Section~\ref{sec:coupling}. We chose the first option for the sake of brevity.  
\begin{proposition}\label{prop:bkw_cb}
Let $\omega$ be a random element of~$\{0,1\}^{\bbE^\bullet}$ distributed according to $\fk_{G;\beta,q}^\fkwired$. Consider the corresponding loop configuration $\loops(\omega)$, and define a height function $h$ as follows:
\begin{itemize}
\item[\textbf{H1}] Set $h(i)=0$ for $i\in\bbL_\bullet\setminus\mathcal{D}$, and $h(u)=1$ for $u\in\bbL_\circ\setminus\mathcal{D}$.
\item[\textbf{H2}] Assign constant values to clusters of $\omega$ and $\omega^*$ in $\mathcal{D}$ by \textbf{decreasing} the height by 1 with probability $e^{\lambda}/\sqrt{q}$ and \textbf{increasing} the height by 1 with probability $e^{-\lambda}/\sqrt{q}$ when crossing a loop from outside, independently for every loop. 
\end{itemize}
The height function $h$ is distributed according to $\hf_{\mathcal{D};\svc,\svcb}^{0,1}$.
\end{proposition}
The above procedure also works for odd $\bbL$-domains with the difference that one has to take $\omega\sim\fk_{\mathcal{D}_\bullet;p,q}^\fkfree$, and \textbf{H2} must be replaced by 
\begin{itemize}
\item[\textbf{H2'}] Assign constant values to clusters of $\omega$ and $\omega^*$ in $\mathcal{D}$ by \textbf{decreasing} the height by 1 with probability $e^{\lambda}/\sqrt{q}$ and \textbf{increasing} the height by 1 with probability $e^{-\lambda}/\sqrt{q}$ when crossing a loop from outside, independently for every loop. 
\end{itemize}

\subsubsection{Input from~\cite{GlaPel23}}
For the whole section, fix $q>4$ and $\beta = \beta_c(q) = \ln(1+\sqrt{q})$, and let $\lambda,\svc,\svcb$ be as in~\eqref{eq:parameters_bulk} and~\eqref{eq:parameters_bnd}.
The following proposition is a consequence of~\cite[Proposition 6.1 and Lemma 6.2]{GlaPel23} and their proofs.
\begin{proposition}\label{prop:hf_cb_conv}
For any sequence of even or odd $\bbL$-domains $\mathcal{D}_k\nearrow\bbL$, the measures $\hf_{\mathcal{D}_k;\svc,\svcb}^{0,1}$ converge to some $\hf_\svc^{0,1}$ which is independent of the sequence $(\mathcal{D}_k)$. Moreover, the limiting measure $\hf_\svc^{0,1}$ can be constructed in either of the following two ways:
\begin{enumerate}[label=(\roman*)]
\item Sample $\omega\in\{0,1\}^{\bbE^\bullet}$ according to $\fk_{\beta,q}^{\fkwired}$. Set $h=0$ on the unique infinite cluster of $\omega$, and sample $h$ elsewhere according to \textbf{\emph{H2}} in Section~\ref{sec:bkw}.
\item Sample $\omega\in\{0,1\}^{\bbE^\bullet}$ according to $\fk_{\beta,q}^{\fkfree}$. Set $h=1$ on the unique infinite cluster of $\omega^*$, and sample $h$ elsewhere according to \textbf{\emph{H2'}} in Section~\ref{sec:bkw}.
\end{enumerate}
\end{proposition}

The following lemma is a slight generalisation of~\cite[Eq. (28)]{GlaPel23} and can be proved in exactly the same manner. 

\begin{lemma}\label{lem:hf_st_dom}
Let $\mathcal{D}$ be an $\bbL$-domain, and let $\delta$ be a set of boundary tiles of $\mathcal{D}$. Define $\mathcal{D}^\mathrm{even}=\mathcal{D}\setminus(\partialin\mathcal{D})_\bullet$ and $\mathcal{D}^\mathrm{odd}=\mathcal{D}\setminus(\partialin\mathcal{D})_\circ$, where the boundaries are taken in $\bbL$. Then, $\mathcal{D}^\mathrm{even}$ and $\mathcal{D}^\mathrm{odd}$ are disjoint unions of even and odd domains in $\bbL$, respectively. Moreover, the following stochastic ordering of measures holds:
\[
\hf_{\mathcal{D}^{\mathrm{even}};\svc,\svcb}^{0,1}\leq_{\mathrm{st}}\hf_{\mathcal{D},\delta;\svc,\svcb}^{0,1}\leq_{\mathrm{st}}\hf_{\mathcal{D}^{\mathrm{odd}};\svc,\svcb}^{0,1}.
\]
\end{lemma}

This lemma readily implies that $\hf_{\mathcal{D}_k,\delta_k;\svc,\svcb}^{0,1}$ (in particular $\hf_{\mathcal{D}_k;\svc}^{0,1}$) converges to the same limit, no matter which sequences of~$\mathcal{D}_k$ and~$\delta_k$ we chose.

\subsection{Relaxation of height function measures}\label{sec:height-func-relax}
The aim of this section is to prove a relaxation statement for the six-vertex height function measures with modified weight $\svcb$ on arbitrary boundary tiles, in particular for the measure with unmodified weight $\svc$ on all tiles.
For the whole section, fix $q>4$ and $\beta = \beta_c(q) = \ln(1+\sqrt{q})$, and let $\lambda,\svc,\svcb$ be as in~\eqref{eq:parameters_bulk} and \eqref{eq:parameters_bnd}.

Given a measure $\hf$ on $\Z^\bbL$ and $\Delta\subset\bbL$, define $\hf|_\Delta$ as its marginal on $\Z^\Delta$.
\begin{proposition}\label{prop:hf_exp_relax}
There exist constants $c,\alpha>0$ such that, for every finite $\Delta\subset\bbL$, every $\bbL$-domain $\mathcal{D}$ that contains $\Delta$, and every set $\delta\subset\bbL_\diamond$ of boundary tiles of $\mathcal{D}$,
\begin{equation*}
\tvd\left(\hf_{\mathcal{D},\delta;\svc,\svcb}^{0,1}\vert _{\Delta},\hf_{\svc}^{0,1}\vert_{\Delta})\right)
< c\,\abs{\Delta}\,(\diameter (\Delta)+\rmd_\infty(\Delta,\mathcal{D}^c))^2\,e^{-\alpha\,\rmd_\infty(\Delta,\mathcal{D}^c)}.
\end{equation*}
\end{proposition}

The statement on even or odd $\bbL$-domains and with modified weight $\svcb$ on all boundary tiles can be proven in the same way as~\cite[Proposition 4.2]{AouDobGla24}, which is slightly less general. We provide only the statement. 

\begin{lemma}\label{lem:hf_cb_exp_relax}
There exist constants $c,\alpha>0$ such that, for every finite $\Delta\subset\bbL$ and every even or odd $\bbL$-domain $\mathcal{D}$ that contains $\Delta$,
\begin{equation*}
    \tvd\left(\hf_{\mathcal{D};\svc,\svcb}^{0,1}\vert _{\Delta},\hf_{\svc}^{0,1}\vert _{\Delta}\right)< c\,(\diameter (\Delta)+\rmd_\infty(\Delta,\mathcal{D}^c))\,e^{-\alpha\,\rmd_\infty(\Delta,\mathcal{D}^c)}.
\end{equation*}
\end{lemma}

Before deducing Proposition~\ref{prop:hf_exp_relax} from the above and the stochastic ordering of height function measures, Lemma~\ref{lem:hf_st_dom}, we need a general lemma.

\begin{lemma}{\cite[Lemma 2.16]{HarSpi22}}\label{lem:stoch_dom_tv_bound}
Let $\mu$ and $\nu$ be two probability measures on a finite \emph{totally ordered} set $S$ with $\mu\leq_{\mathrm{st}}\nu$. Let $\mathbf{P}$ be a monotone coupling of $X\sim\mu$ and $Y\sim\nu$. Then,
\[
\mathbf{P}(X\neq Y)\leq (\abs{S}-1)\,\tvd(\mu,\nu).
\]
\end{lemma}

\begin{proof}
Identify $S$ with $\{0,\dots,\abs{S}-1\}$.
Let $\mathbf{P}_{\mathrm{op}}$ be an optimal coupling of $X$ and $Y$, that is, $\mathbf{P}_{\mathrm{op}}(X\neq Y)=\tvd(\mu,\nu)$. Since $\mathbf{P}(X\leq Y)=1$, we have
\begin{align*}
\mathbf{P}(X\neq Y)=\mathbf{P}(Y-X\geq 1)&\leq\mathbf{E}[Y-X]\\
&=\mathbf{E}_{\mathrm{op}}[Y-X]\leq (\abs{S}-1)\,\mathbf{P}_{\mathrm{op}}(X\neq Y),
\end{align*}
where we applied Markov's inequality.
\end{proof}

\begin{proof}[Proof of Proposition~\ref{prop:hf_exp_relax}]
Let $\Delta\subset\bbL$ be finite, let $\mathcal{D}$ be an $\bbL$-domain, and let $\delta\subset\bbL_\diamond$ be a set of boundary tiles of $\mathcal{D}$.
Assume without loss of generality that $\Delta\subseteq\mathcal{D}\setminus\partialin\mathcal{D}$. 
Let $\mathcal{D}^{\mathrm{even}}$ and $\mathcal{D}^{\mathrm{odd}}$ be as in Lemma~\ref{lem:hf_st_dom}, and recall the stochastic domination statement therein. Consider a monotone coupling $\mathbf{P}$ of $h^-\sim\hf_{\mathcal{D}^{\mathrm{even}};\svc,\svcb}^{0,1},\,h\sim\hf_{\mathcal{D},\delta;\svc,\svcb}^{0,1}$ and $h^+\sim\hf_{\mathcal{D}^{\mathrm{odd}};\svc,\svcb}^{0,1}$, that is,
\[
	\mathbf{P}(h^-\leq h\leq h^+)=1.
\]

Take any vertex $v\in\Delta$ and note that, deterministically, the height functions at $v$ take values in $[-m,m]$, where $m:=2(\diameter (\Delta)+\rmd_\infty(\Delta,\mathcal{D}^c))$.
By Lemma~\ref{lem:stoch_dom_tv_bound}, we have
\begin{align}
\mathbf{P}(h^-(v)\neq h(v))&\leq \mathbf{P}(h^-(v)\neq h^+(v))\leq 2m\,\tvd\big(\hf_{\mathcal{D}^{\mathrm{even}};\svc,\svcb}^{0,1}\vert_\Delta,\hf_{\mathcal{D}^{\mathrm{odd}};\svc,\svcb}^{0,1}\vert_\Delta\big)\label{eq:d_tv_estimate}\\
&\leq 2m\,\big(\tvd\big(\hf_{\mathcal{D}^{\mathrm{even}};\svc,\svcb}^{0,1}\vert_\Delta,\hf_{\svc}^{0,1}\vert_\Delta\big)+\tvd\big(\hf_{\svc}^{0,1}\vert_\Delta,\hf_{\mathcal{D}^{\mathrm{odd}};\svc,\svcb}^{0,1}\vert_\Delta\big)\big).\nonumber
\end{align}
Now, $\mathcal{D}^{\mathrm{even}}$ is a disjoint union of even $\bbL$-domains $\mathcal{D}_1,\dots,\mathcal{D}_n$ with $\partialex\mathcal{D}_i\cap\mathcal{D}_j=\varnothing$ for any $i\neq j$. There exists an $i$ with $\Delta\subseteq\mathcal{D}_i$ and $\rmd_\infty(\Delta,\mathcal{D}_i^c)\geq\rmd_\infty(\Delta,\mathcal{D}^c)-1$. Moreover, the definition of the height function measures~\eqref{eq:hf_def} readily implies that, for any event $A\subset\Z^\bbL$ depending only on values in $\Delta$, we have $\hf_{\mathcal{D}^{\mathrm{even}};\svc,\svcb}^{0,1}(A)=\hf_{\mathcal{D}_i;\svc,\svcb}^{0,1}(A)$. Similar reasoning applies to $\mathcal{D}^{\mathrm{odd}}$. By Lemma~\ref{lem:hf_cb_exp_relax}, there exist $c,\alpha>0$, such that the right-hand side of~\eqref{eq:d_tv_estimate} is bounded by $c\,m^2\,e^{-\alpha\,\rmd_\infty(\Delta,\mathcal{D}^c)}$.
The union bound gives
\[
\tvd\left(\hf_{\mathcal{D}^{\mathrm{even}};\svc,\svcb}^{0,1}\vert _{\Delta},\hf_{\mathcal{D},\delta;\svc,\svcb}^{0,1}\vert_{\Delta}\right)\leq\mathbf{P}(h^-\vert_\Delta\neq h\vert_\Delta)\leq \sum_{v\in\Delta}\mathbf{P}(h^-(v)\neq h(v)).
\]
Another application of the triangle inequality and of Lemma~\ref{lem:hf_cb_exp_relax} finishes the proof.
\end{proof}

\begin{remark}\label{rem:opt_hf_exp_relax}
Using that $h^-,h^+$ have uniformly bounded second moment (localisation) and adapting Lemma~\ref{lem:stoch_dom_tv_bound} allow to remove the square of $(\diameter (\Delta)+\rmd_\infty(\Delta,\mathcal{D}^c))^2$ in the bound in Proposition~\ref{prop:hf_exp_relax}.
\end{remark}

\subsection{Relaxation of ATRC measures}
\label{sec:atrc_relax}
The proof of Proposition~\ref{prop:edge_relax_ATRC} is based on Proposition~\ref{prop:exp_decay_omega_tau} and another one below, which is a consequence of Proposition~\ref{prop:hf_exp_relax}.

The following proposition establishes exponential relaxation of the ATRC measures on domains and with ``$\atrcfree,\atrcwired$'' boundary conditions. Once Proposition~\ref{prop:edge_relax_ATRC} is proven, it can be concluded that the limit coincides with the unique ATRC Gibbs measure. Recall the definition~\eqref{eq:sample_atrc_hf} of $\mathcal{X}_\svc$.

\begin{proposition}\label{prop:atrc01_relax}
Let $0<J<U$ satisfy $\sinh 2J=e^{-2U}$. There exists a measure $\atrc_{J,U}$ on $\{0,1\}^{\bbE^\bullet}\times\{0,1\}^{\bbE^\bullet}$ and constants $c,\alpha>0$ such that, for any $\bbL_\bullet$-domain $G=(V,E)$ and any $F\subseteq E$, 
\begin{multline*}
\tvd\left(\atrc_{G;J,U}^{\atrcfree,\atrcwired}[(\omega_\tau|_F,\omega_{\tau\tau'}|_F)\in\cdot\,],\atrc_{J,U}[(\omega_\tau|_F,\omega_{\tau\tau'}|_F)\in\cdot\,]\right)\\
<c\,\abs{\bbV_F}\,(\diameter (\bbV_F)+\rmd_\infty (\bbV_F,V^c))^2\,e^{-\alpha\,\rmd_\infty(\bbV_F,V^c)},
\end{multline*}
where \(\diameter\) denotes the diameter with respect to \(\rmd_\infty\).
Furthermore, the measure $\atrc_{J,U}$ is constructed as follows.
Let $\svc=\coth 2J$, let $h$ be distributed according to $\hf_\svc^{0,1}$, and let $(U_e)_{e\in\bbE^\bullet}$ be a sequence of i.i.d. uniform random variables on $[0,1]$, independent of $h$. 
The measure $\atrc_{J,U}$ is given by the law of $(\omega_\tau,\omega_{\tau\tau'})\in\{0,1\}^{\bbE^\bullet}\times\{0,1\}^{\bbE^\bullet}$ defined by 
\begin{equation*}
(\omega_\tau(e),\omega_{\tau\tau'}(e)):=\mathcal{X}_\svc(h,e,U_e),\ e\in\bbE^\bullet.
\end{equation*}
\end{proposition}

\begin{remark}
It is possible to improve the bound in Proposition~\ref{prop:atrc01_relax} and remove the square in $(\diameter (\bbV_F)+\rmd_\infty (\bbV_F,V^c))^2$, see Remark~\ref{rem:opt_hf_exp_relax}.
\end{remark}

The above proposition follows from the analogous statement for the height function measures, Proposition~\ref{prop:hf_exp_relax}, and the fact that the ATRC measures can be sampled locally from these measures, which is the content of Lemma~\ref{lem:sample_atrc_hf}.

\begin{proof}[Proof of Proposition~\ref{prop:atrc01_relax}]
Let $0<J<U$ satisfy $\sinh 2J=e^{-2U}$ and $\svc = \coth 2J$.
Let $G=(V,E)$ be an $\bbL_\bullet$-domain and $\mathcal{D}:=\mathcal{D}_G$ be its associated $\bbL$-domain.
Recall that a tile $t\in\bbL_\diamond$ intersects $\mathcal{D}$ precisely if its associated primal edge $e_t\in\bbE^\bullet$ belongs to $E$, see Fig.~\ref{fig:domains}.
Let $F\subset E$, and set $\Delta:=\bbV_F\cup\bbV_{*F}$.

Take $h\sim \hf_{\mathcal{D};\svc}^{0,1}$ and $h'\sim\hf_{\svc}^{0,1}$, and consider an optimal coupling of $h|_\Delta$ and $h'|_\Delta$. Let $(U_e)_{e\in\bbE^\bullet}$ be a sequence of i.i.d. uniform random variables on $[0,1]$, independent of $(h,h')$. Denote the corresponding probability measure by $\mathbf{P}$.
Define $\omega_\tau,\omega_{\tau\tau'}\in \{0,1\}^{\bbE^\bullet}$ from $h$ and $(U_e)$ as in~\eqref{eq:sample_pair_atrc_from_hf}. Define $\omega'_\tau,\omega'_{\tau\tau'}$ from $h'$ and $(U_e)$ by setting $(\omega'_\tau(e),\omega'_{\tau\tau'}(e))=\mathcal{X}_\svc(h',e,U_e)$ for $e\in\bbE^\bullet$.
Then, clearly
\begin{equation*}
(\omega_\tau\vert_F,\omega_{\tau\tau'}\vert_F)\neq(\omega'_\tau\vert_F,\omega'_{\tau\tau'}\vert_F)\quad\text{implies}\quad h\vert_\Delta\neq h'\vert_\Delta.
\end{equation*}
Therefore, by the definition of an optimal coupling and by Proposition~\ref{prop:hf_exp_relax},
\begin{align*}
\mathbf{P}\big((\omega_\tau\vert_F,\omega_{\tau\tau'}\vert_F)\neq(\omega'_\tau\vert_F,\omega'_{\tau\tau'}\vert_F)\big)&<c\,\abs{\Delta}\,(\diameter (\Delta)+\rmd_\infty(\Delta,\mathcal{D}^c))^2\,e^{-\alpha\,\rmd_\infty(\Delta,\mathcal{D}^c)}\\
&<c'\,\abs{F}\,(\diameter (\bbV_F)+\rmd_\infty (\bbV_F,V^c))^2\,e^{-\alpha\,\rmd_\infty(\bbV_F,V^c)},
\end{align*}
for some $c,c',\alpha>0$.
By Lemma~\ref{lem:sample_atrc_hf}, the law of $(\omega_\tau,\omega_{\tau\tau'})$ is given by $\atrc_{G;J,U}^{\atrcfree,\atrcwired}$. Finally, denote the law of $(\omega'_\tau,\omega'_{\tau\tau'})$ by $\atrc_{J,U}$, and the proof is complete.
\end{proof}

\subsection{Proof of Proposition~\ref{prop:edge_relax_ATRC}}
\label{sec:proof_edge_relax_atrc}
Finally, we derive Proposition~\ref{prop:edge_relax_ATRC} from Proposition~\ref{prop:atrc01_relax} and the following one that we import from~\cite{AouDobGla24}.

\begin{proposition}{\cite[Proposition~1.1]{AouDobGla24}}\label{prop:exp_decay_omega_tau}
Let $0<J<U$ satisfy $\sinh 2J=e^{-2U}$. There exists $c>0$ such that, for every $n\geq 1$,
\begin{equation*}
\atrc_{\Lambda_n;J,U}^{\atrcwired,\atrcwired}(0\xleftrightarrow{\omega_\tau}\partialex\Lambda_n)<e^{-c n}.
\end{equation*}
\end{proposition}

\begin{proof}[Proof of Proposition~\ref{prop:edge_relax_ATRC}]
Let $\sigma\in \{\tau,\tau\tau'\}$. Assume without loss of generality that $n=2k$ is even (the other case is treated similarly). We proceed in two steps.

\noindent\textbf{Step 1.} There exists $\alpha>0$ such that 
	\[
		\atrc_{\Lambda_n;J,U}^{\atrcwired,\atrcwired}[\omega_\sigma(e)]-\atrc_{\Lambda_n;J,U}^{\atrcfree,\atrcwired}[\omega_\sigma(e)]\leq e^{-\alpha n}.
	\]

Let $(\omega_\tau,\omega_{\tau\tau'})$ be distributed according to $\atrc_{\Lambda_n;J,U}^{\atrcwired,\atrcwired}$.
Define $\mathcal{C}$ to be the outermost (dual) circuit in $\omega_\tau^*$ that surrounds $\Lambda_k$ and is contained in $\Lambda_{2k}$ if it exists, otherwise set $\mathcal{C}=\varnothing$.
By exponential decay of connection probabilities in $\omega_\tau$, Proposition~\ref{prop:exp_decay_omega_tau}, there exist $c,\alpha>0$ such that 
\begin{align*}
\atrc_{\Lambda_n;J,U}^{\atrcwired,\atrcwired}[\omega_\sigma(e)]&\leq \atrc_{\Lambda_n;J,U}^{\atrcwired,\atrcwired}\left[\omega_\sigma(e)\,\big|\,\mathcal{C}\neq\varnothing\right]+\atrc_{\Lambda_n;J,U}^{\atrcwired,\atrcwired}[\Lambda_k\xleftrightarrow{\omega_\tau}\partialin \Lambda_{2k}]\\
&\leq \atrc_{\Lambda_n;J,U}^{\atrcwired,\atrcwired}\left[\omega_\sigma(e)\,\big|\,\mathcal{C}\neq\varnothing\right]+ cke^{-\alpha k},
\end{align*}
where we also used the strong FKG and DLR properties of the ATRC measures (Lemma~\ref{lem:atrc_strong_fkg} and~\eqref{eq:smp} in Section~\ref{sec:atrc}).
Now, again by the strong FKG and~\eqref{eq:smp} and by exponential relaxation, Proposition~\ref{prop:atrc01_relax}, there exist $c',\alpha'>0$ such that
\begin{align*}
\atrc_{\Lambda_n;J,U}^{\atrcwired,\atrcwired}&\left[\omega_\sigma(e)\,\big|\,\mathcal{C}\neq\varnothing\right]\\
&=\sum_{C}\atrc_{\Lambda_n;J,U}^{\atrcwired,\atrcwired}\left[\omega_\sigma(e)\,\big|\,\mathcal{C}=C\right]\atrc_{\Lambda_n;J,U}^{\atrcwired,\atrcwired}\left[\mathcal{C}=C\,\big|\,\mathcal{C}\neq\varnothing\right]\\
&\leq \sum_{C}\atrc_{G_C;J,U}^{\atrcfree,\atrcwired}[\omega_\sigma(e)]\ \atrc_{\Lambda_n;J,U}^{\atrcwired,\atrcwired}\left[\mathcal{C}=C\,\big|\,\mathcal{C}\neq\varnothing\right]\\
&\leq \atrc_{\Lambda_n;J,U}^{\atrcfree,\atrcwired}\left[\omega_\sigma(e)\right]+c'k^2e^{-\alpha' k},
\end{align*}
where the summation is over all realisations $C$ of $\mathcal{C}$, and where $G_C$ is the largest $\bbL_\bullet$-domain of the first kind within $C$ that contains $\Lambda_k$. This proves Step 1.

\noindent\textbf{Step 2.} There exists $\alpha>0$ such that
	\[
		\atrc_{\Lambda_n;J,U}^{\atrcfree,\atrcwired}[\omega_\sigma(e)]-\atrc_{\Lambda_n;J,U}^{\atrcfree,\atrcfree}[\omega_\sigma(e)]\leq e^{-\alpha n}.
	\]

We use the same strategy as in Step 1. Indeed, let $(\tilde{\omega}_\tau,\tilde{\omega}_{\tau\tau'})$ be distributed according to $\atrc_{\Lambda_n;J,U}^{\atrcfree,\atrcfree}$. Since, by self-duality~\eqref{eq:atrc_selfdual}, $\tilde{\omega}_{\tau\tau'}^*$ has the same law as $\omega_\tau$ (but on the dual graph of $\Lambda_n$) from Step 1, Proposition~\ref{prop:exp_decay_omega_tau} allows to find a circuit in $\tilde{\omega}_{\tau\tau'}$ that surrounds $\Lambda_k$ and is contained in $\Lambda_{2k}$. For a realisation $C$ of an outermost such circuit, consider the largest $\bbL_\bullet$-domain of the second kind within $C$ that contains $\Lambda_k$. The remainder of the argument is analogous to Step 1.
\end{proof}

\subsection{Ratio weak mixing: proof of Theorem~\ref{thm:ratio_weak_mixing_ATRC}}
\label{sec:ratio-mixing}

We first derive the \emph{exponential} weak mixing property from the single edge relaxation, Proposition~\ref{prop:edge_relax_ATRC}.
Given $F\subset\bbE^\bullet$ and a measure $\atrc$ on $\{0,1\}^{\bbE^\bullet}\times\{0,1\}^{\bbE^\bullet}$, we write $\atrc|_F$ for its marginal on $\{0,1\}^F\times\{0,1\}^F$.

\begin{theorem}
	\label{thm:weak_mixing_ATRC}
	Let $0<J<U$ satisfy $\sinh 2J=e^{-2U}$. There exists \(c>0\) such that, for any finite subgraph $G=(V,E)$ of $\bbL_\bullet$ and any \(F\subset E\),
	\begin{equation*}
		\sup_{\eta_{\tau},\eta_{\tau\tau'},\eta_{\tau}',\eta_{\tau\tau'}'} \tvd\Big(\atrc_{G;J,U}^{\eta_{\tau},\eta_{\tau\tau'}}|_{F},\atrc_{G;J,U}^{\eta_{\tau}',\eta_{\tau\tau'}'}|_{F}\Big) \leq 2\sum_{e\in F} e^{-c\,\rmd_{\infty}(e,V^c)},
	\end{equation*}where \(\rmd_{\infty}\) is the distance induced by the \(L^{\infty}\) norm.
\end{theorem}

As a consequence, the measures $\atrc_{G;J,U}^{\eta_\tau,\eta_{\tau\tau'}}$ converge (as $G\nearrow\bbL_\bullet$) to a limit measure~$\atrc_{J,U}$, which is independent of the choice of boundary conditions~$\eta_\tau,\eta_{\tau\tau'}$. Furthermore, the limit~$\atrc_{J,U}$ is the unique ATRC \emph{Gibbs} measure.

\begin{proof}
	By the strong FKG property, Lemma~\ref{lem:atrc_strong_fkg}, the supremum is attained for \(\eta_{\tau}=\eta_{\tau\tau'} = 1\) and \(\eta_{\tau}'=\eta_{\tau\tau'}' = 0\). Let \(\Psi\) be a monotone coupling of \(\atrc_{G;J,U}^{1,1}\) and \(\atrc_{G;J,U}^{0,0}\), and let \((X,Y) =\big((X_{\tau},X_{\tau\tau'}),(Y_{\tau},Y_{\tau\tau'})\big)\sim \Psi\). Then,
	\begin{align*}
		\tvd\Big(\atrc_{G;J,U}^{1,1}|_{F},\atrc_{G;J,U}^{0,0}|_{F}\Big)
		&\leq
		\Psi(X|_{F}\neq Y|_{F})\\
		&\leq
		\Psi(X_{\tau}|_{F}\neq Y_{\tau}|_{F})+\Psi(X_{\tau\tau'}|_{F}\neq Y_{\tau\tau'}|_{F}).
	\end{align*}Now, for \(\sigma\in \{\tau,\tau\tau'\}\),
	\begin{align*}
		\Psi(X_{\sigma}|_{F}\neq Y_{\sigma}|_{F})
		&\leq
		\sum_{e\in F} \Psi(X_{\sigma}(e)\neq Y_{\sigma}(e))
		=
		\sum_{e\in F} \Psi(X_{\sigma}(e)> Y_{\sigma}(e))\\
		&=
		\sum_{e\in F} \left(\Psi(X_{\sigma}(e)=1)- \Psi(Y_{\sigma}(e)=1) \right)
		\leq
		\sum_{e\in F} e^{-c\rmd_{\infty}(e,V^c)}
	\end{align*}for some \(c>0\) by Proposition~\ref{prop:edge_relax_ATRC}, where we also used strong positive association and the DLR property of the ATRC measures.
\end{proof}

A less trivial consequence is the \emph{ratio} weak mixing property, Theorem~\ref{thm:ratio_weak_mixing_ATRC}. This follows from the study performed in~\cite{Ale98}, relying on Theorem~\ref{thm:weak_mixing_ATRC} and Proposition~\ref{prop:exp_decay_omega_tau}. 

\begin{proof}[Proof of Theorem~\ref{thm:ratio_weak_mixing_ATRC}]
This is a direct application of~\cite[Section 5]{Ale98}. The reasoning there requires two properties of the measure. The first is exponential weak mixing, which is the content of Theorem~\ref{thm:weak_mixing_ATRC}. The second is the admittance of exponentially bounded controlling regions in the sense of~\cite{Ale98}, and we argue its validity as follows.

Let \((\omega_\tau,\omega_{\tau\tau'})\sim\atrc_{G;J,U}^{\eta_{\tau},\eta_{\tau\tau'}}\), fix \(F\subset E\), and define
\begin{equation*}
H(F)=\{e\in E\setminus F: e\xleftrightarrow{\omega_\tau}\bbV_F\text{ or }e^*\xleftrightarrow{\omega_{\tau\tau'}^*}\bbV_{*F}\}.
\end{equation*} 
Conditionally on the states of the edges in \(H(F)\), the values of the field in \(F\) and \(E\setminus(F\cup H(F))\) are independent. Moreover, the probability that \(e\in E\setminus F\) belongs to \(H(F)\) decays exponentially in \(\rmd_{\infty}(e,\bbV_F)\) by uniform exponential decay in \(\omega_{\tau}\) and in the dual of \(\omega_{\tau\tau'}\). The former is the content of Proposition~\ref{prop:exp_decay_omega_tau} and the latter follows from self-duality.
One can then apply~\cite[Section 5]{Ale98} to obtain the desired claim.
\end{proof}

\section{Strong mixing}
\label{sec:strong_mixing}

In this section, we use~\cite{Ott25} to push the mixing properties derived in the previous section to strong mixing properties for finite volume ATRC measures.

\subsection{Setup}
\label{subsec:strMix:setup}

For \(k\geq 1\), let \(\Gamma_k = ((2k+1)\Z)^2\).
Given~\(x\in \Z^2\) or~\(V\subset \Z^2\), denote~\([x]_k:= x+\{-k,\dots,k\}^2\) or~\([V]_k:= \cup_{x\in V} (x+\{-k,\dots,k\}^2)\) respectively.

\noindent{\bf Path decoupling property.}
Let \(k\geq 1\). For a simple closed nearest-neighbour path \(\gamma\) on \(\Gamma_k\), denote by \(F_{\mathrm{in}}(\gamma)\), \(F_{\mathrm{out}}(\gamma)\) the sets of edges in \(\bbE^{\bullet}\) that are, respectively, in the finite connected component of \(\bbE^{\bullet}\setminus \bbE_{[\gamma]_k}\), in the infinite connected component of \(\bbE^{\bullet} \setminus \bbE_{[\gamma]_k}\).

Let \((a,b)\in \{(0,0),(0,1),(1,1)\}^{\bbE^{\bullet}}\), and \(F\subset \bbE^{\bullet}\) finite.
Say that \(\atrc_{F}^{a,b}\) has the \emph{\(k\)-decoupling path property} if for any simple closed nearest-neighbour path \(\gamma\) in \(\Gamma_k\), and any configuration \((\xi,\xi')\in \{(0,0),(0,1),(1,1)\}^{F\cap \bbE_{[\gamma]_k}}\) such that
\begin{itemize}
	\item the dual of the configuration \(\xi_{F}0_{F^c}\) (\(\xi\) on \(F\) and constant \(0\) on \(F^c\)) contains a simple closed path of open dual edges surrounding \([\mathring{\gamma}]_k\), where \(\mathring{\gamma}\) is the set of sites in \(\Gamma_k\) surrounded by \(\gamma\),
	\item the configuration \(\xi_{F}' 1_{F^c}\) (\(\xi'\) on \(F\) and constant \(1\) on \(F^c\)) contains a simple closed path of open edges surrounding \([\mathring{\gamma}]_k\),
\end{itemize}one has that if
\begin{equation*}
	(\omega_{\tau},\omega_{\tau\tau'})\sim \atrc_{F}^{a,b}\big(\ \bgiven \omega_{\tau}(e) = \xi(e),\, \omega_{\tau\tau'}(e) = \xi'(e)\ \forall e\in \bbE_{[\gamma]_k}\big),
\end{equation*}then the restriction of \((\omega_{\tau},\omega_{\tau\tau'})\) to \(F\cap F_{\mathrm{in}}(\gamma)\) is independent from the restriction of \((\omega_{\tau},\omega_{\tau\tau'})\) to \(F\cap F_{\mathrm{out}}(\gamma)\). A configuration \((\xi,\xi')\) as above is said to contain \emph{a pair of decoupling paths for \(\gamma\) in \(F\)}.

One has this property in, for example, the case of constant boundary conditions \((0,1)\), or in the case of \(F\) simply connected and constant \((0,0)\) or constant \((1,1)\) boundary conditions. More generally, to check if the property holds, it suffices to check that cluster count is factorized by the presence of the required paths in \(\bbE_{[\gamma]_k}\), as in Lemma~\ref{lem:decoupling_paths_mATRC}.

\noindent{\bf ATRC as a model on sites.}
The setup of~\cite{Ott25} is a collection of site models indexed by finite volumes and some index set \(T\). Cast \(\atrc\) as a model on sites by taking the spin space at site \(x\) to be
\begin{equation*}
	\Omega_x = \{(0,0),(0,1),(1,1)\}^{\{\{x,x+\rme_1\}, \{x,x+\rme_2\}\}}.
\end{equation*}I.e. store the value of the field at a given edge in its left/bottom endpoint. Let
\begin{equation*}
	\psi: \{(0,0),(0,1),(1,1)\}^{\bbE^{\bullet}} \to \bigtimes_{x\in \Z^2} \Omega_x
\end{equation*}be the natural bijection induced by the above. Let
\begin{equation*}
	\bbE_{BL}(\Lambda) = \big\{\{x,x+\rme_1\},\{x,x+\rme_2\}:\ x\in \Lambda\big\}.
\end{equation*}

\noindent{\bf Bloc Percolation.}
For \(p,q\in [0,1]\), let \(X_i,Y_i\), \(i\in \Z^2\) be an independent family of Bernoulli random variables with parameters
\begin{equation*}
	P(X_i = 1) = p,\quad P(Y_i = 1) = q \quad \forall i\in \Z^2.
\end{equation*}For \(k\geq 1\), \(F\subset \bbE^{\bullet}\), \(x\in \Z^2\), define \(i_x\in \Gamma_k\) to be the unique point with \(x\in [i_x]_k\), and
\begin{equation*}
	\eta_x =
	\begin{cases}
		X_{i_x} & \text{ if } \bbE_{BL}([i_x]_{3k+1})\subset F,\\
		Y_{i_x} & \text{ if } x\in \bbV_{F},\text{ and } \bbE_{BL}([i_x]_{3k+1})\cap F^c \neq \varnothing,\\
		0 & \text{ else}.
	\end{cases}
\end{equation*}Let \(P_{F;k,p,q}\) be the law of \((\eta_x)_{x\in \Z^2}\).

\subsection{Strong mixing: main Theorem}
\label{subsec:strMix:MainThm}

Then, the main result of the section is the next Theorem. To get it, we will apply the main theorem of~\cite{Ott25} (see Section~\ref{subsec:strMix:HypVerif}). For \(F_-,F_+\subset F\subset \bbE^{\bullet}\), \((a,b)\in \{(0,0),(0,1),(1,1)\}^{\bbE^{\bullet}}\), define
\begin{equation}
	\label{eq:def:atrc_cond_versions}
	\nu_{F,F_-,F_+}^{a,b}(\cdot) = \atrc_{F}^{a,b}\big(\cdot \given \omega_{\tau}(e) = 0 \ \forall e\in F_-,\ \omega_{\tau\tau'}(e) = 1 \ \forall e\in F_+ \big).
\end{equation}
\begin{theorem}
	\label{thm:strong_mixing_atrc}
	Let \(0\leq J<U\) be such that \(\sinh(2J) = e^{-2U}\). There is \(l\geq 1\) such that: for any \(p>0\), there are \(L<\infty\) and \(q<1\) such that for any \(F\subset \bbE^{\bullet}\), any \((a,b)\in \{(0,0),(0,1),(1,1)\}^{\bbE^{\bullet}}\) such that \(\atrc_F^{a,b}\) has the \(l\)-decoupling path property, any \(F_-,F_+\subset \partialin F\), any \(F_1,F_2\subset F\) disjoint, and any \((f,g),(f',g')\in \{(0,0),(0,1),(1,1)\}^{F_1}\) having positive probability under \(\nu_{F,F_-,F_+}^{a,b}\) (defined in~\eqref{eq:def:atrc_cond_versions})
	\begin{multline*}
		\tvd\big(\nu_{F,F_-,F_+}^{a,b}\big(W|_{F_2}\in \cdot \bgiven W|_{F_1} = (f,g) \big), \nu_{F,F_-,F_+}^{a,b}\big( W|_{F_2}\in\cdot \bgiven W|_{F_1} = (f',g') \big) \big)
		\\
		\leq
		P_{F\setminus (F_+\cup F_-); L,p,q}\big( [\bbV_{F_1}]_L\leftrightarrow_* [\bbV_{F_2}]_L \big),
	\end{multline*}where \(W = (\omega_{\tau},\omega_{\tau\tau'})\), and \(\leftrightarrow_*\) means \(*\)-connections (connections with \(\norm{\ }_{\infty}\) nearest-neighbours).
\end{theorem}

\begin{remark}
	\label{rem:oneDperco}
	We will systematically apply this result in cases where \(\{x\in \bbV_{F}:\ \bbE_{BL}([i_x]_{3L+1})\cap F^c \neq \varnothing\}\) is a very elongated one dimensional object, so exponential decay of the connection probability can be obtained by a coarse-graining argument as in~\cite[Lemma 3.2]{OttVel18}.
\end{remark}

\subsection{Verifications of the mixing and decoupling hypotheses}
\label{subsec:strMix:HypVerif}

In this section, we apply~\cite{Ott25} to obtain Theorem~\ref{thm:strong_mixing_atrc}. The setup of~\cite{Ott25} is a family of site models satisfying a collection of hypotheses, denoted Mix1, Mix2, Mar1, Mar2, Mar3. We cast our model in that setup, and check that these hypotheses hold. Theorem~\ref{thm:strong_mixing_atrc} then follows from~\cite[Theorem 1.1]{Ott25}

\noindent\textbf{Family of site measures.}
Let \(l\geq 1\) be some number to be fixed large later (it will depend only on \(J,U\)). In~\cite{Ott25}, a family \((\nu_{\Lambda}^t)_{\Lambda\subset \Z^2, t\in T}\) is considered, with \(T\) some index set. We take \(T\) to be the set of quintuplets \((F,a,b,F_-,F_+)\) such that \((a,b)\in \{(0,0),(0,1),(1,1)\}^{\bbE^{\bullet}}\), \(F\subset \bbE^{\bullet}\) finite such that \(\atrc_F^{a,b}\) has the \(l\)-decoupling path property, and \(F_-,F_+\subset \partialin F\).

The family \((\nu_{\Lambda}^t)_{\Lambda\subset \Z^2, t\in T}\) is then given by
\begin{multline}
	\label{eq:atrc_cond_versions}
	\nu_{\Lambda}^{a,b,F,F_+,F_-}(\cdot) :=
	\\
	\begin{cases}
		\atrc_{F}^{a,b}\big(\psi^{-1}(\cdot) \bgiven \omega_{\tau}(e) = 0\,\forall e\in F_-,\ \omega_{\tau\tau'}(e) = 1\, \forall e\in F_+\big) & \text{ if } \Lambda = \bbV_{F},\\
		\atrc_{\bbE_{\Lambda}}^{0,1}\big(\psi^{-1}(\cdot)\big) & \text{ else }.
	\end{cases}
\end{multline}Note the small abuse of notation: we used the same notation as in~\eqref{eq:def:atrc_cond_versions} for the measure on sites rather than on edges. We are only interested in the first case, the second is only there to fit the setup of~\cite{Ott25}.

\noindent\textbf{Mixing hypotheses.}
The hypotheses Mix1 and Mix2 are exponential mixing hypotheses which are weak versions of the ratio weak mixing property (Theorem~\ref{thm:ratio_weak_mixing_ATRC}).

\noindent\textbf{Markov-type hypotheses. }
To avoid lengthy displays, introduce the shorthand
\begin{equation*}
	F' =
	\begin{cases}
		F & \text{ if } \bbV_F = \Lambda,\\
		\bbE_{\Lambda} & \text{ else}.
	\end{cases}
\end{equation*}
Hypotheses Mar1, Mar2, Mar3 ask for
\begin{itemize}
	\item a family of local events, \(\mathrm{Ma}_i\), \(i\in \Gamma_l\), supported on \(\bbE_{BL}([i]_{3l+1})\);
	\item a family of local events, \(\mathrm{Ma}_{i;\Lambda}^t\), \(i\in \Gamma_l\), \(t=(F,a,b,F_-,F_+)\in T\), \(\Lambda\subset \Z^2\), supported on \(\bbE_{BL}([i]_{3l+1})\cap F'\);
	\item configurations \(\mathrm{a}_{i}\in \{(0,0),(0,1),(1,1)\}^{\bbE_{BL}([i]_l)}\), \(i\in \Gamma_l\);
	\item configurations \(\mathrm{a}_{i,\Lambda}^t\in \{(0,0),(0,1),(1,1)\}^{\bbE_{BL}([i]_l) \cap F'}\), \(i\in \Gamma_l\).
\end{itemize}These should satisfy:
\begin{enumerate}
	\item if \(\gamma\) is a simple closed nearest-neighbour path in \(\Gamma_l\), and \((\xi,\xi')\) is a configuration on \(F'\cap \bbE_{[\gamma]_{3l+1}}\) such that for any \(i\in \gamma\) with \(\bbE_{[i]_{5l+2}}\subset F'\) one has
	\begin{equation*}
		(\xi,\xi')|_{\bbE_{[i]_{3l+1}}} \in \mathrm{Ma}_i,
	\end{equation*}and for any \(i\in \gamma\) with \(\bbE_{[i]_{5l+2}}\not\subset F'\), \(\bbE_{[i]_{3l+1}}\cap F'\neq \varnothing\), one has
	\begin{equation*}
		(\xi,\xi')|_{\bbE_{[i]_{3l+1}}\cap F' } \in \mathrm{Ma}_{i;\Lambda}^t,
	\end{equation*}then \((\xi,\xi')\) contains a pair of decoupling paths for \(\gamma\) in \(F'\). This is Mar1.
	\item The probability of \(\mathrm{Ma}_i\) conditionally on the configuration outside of \(\bbE_{[i]_{3l+1}}\) is lower bounded by \(p_0<1\) fixed universal, uniformly over the configuration outside of \(\bbE_{[i]_{3l+1}}\). This is Mar2.
	\item If \(w\in \mathrm{Ma}_i\), and \(j\in \Gamma_l\) with \(\norm{i-j}_{\infty} \leq 2l+1\), then the configuration obtained by swapping the value of \(w|_{\bbE_{BL}([j]_{l}) }\) to \(\mathrm{a}_{j}\) is still in \(\mathrm{Ma}_i\), and similarly for \(\mathrm{a}_{i,\Lambda}^t\). Moreover, the probability of \(\mathrm{a}_{i}\), \(\mathrm{a}_{i,\Lambda}^t\) conditionally on the configuration outside of their respective supports is greater than \(\theta>0\) for some \(\theta >0\) uniform over the the configuration outside of the support. This is Mar3.
\end{enumerate}

Take the events/configurations (\(l\) is taken large enough):
\begin{itemize}
	\item \(\mathrm{Ma}_i\) is the event that \(\bbE_{[i]_{3l/2}}\setminus \bbE_{[i]_{l+1}}\) contains both a simple closed path of edges open in \(\omega_{\tau\tau'}\) surrounding \([i]_{l+1}\), and a simple closed path of dual edges open in \(\omega_{\tau}^*\) surrounding \([i]_{l+1}\).
	\item \(\mathrm{a}_{i,\Lambda}^t\), and \(\mathrm{a}_{i}\) are the constant \((0,1)\) configurations.
	\item \(\mathrm{Ma}_{i,\Lambda}^t \) contains only the constant \((0,1)\) configuration.
\end{itemize}See Fig.~\ref{Fig:StrMix:path_and_decoupling_event}. Mar1 follows by Jordan's curve Theorem. Mar2 follows (once \(l\) is fixed large enough) from exponential mixing (Theorem~\ref{thm:weak_mixing_ATRC}) and uniform exponential decay of connectivities in \(\omega_{\tau}\), and \(\omega_{\tau\tau'}^*\) (Theorem~\ref{thm:Ale04_ATRC}). Finally, Mar3 follows from finite energy~\eqref{eq:fe_atrc}.

\begin{figure}
	\centering
	\includegraphics[scale=0.7]{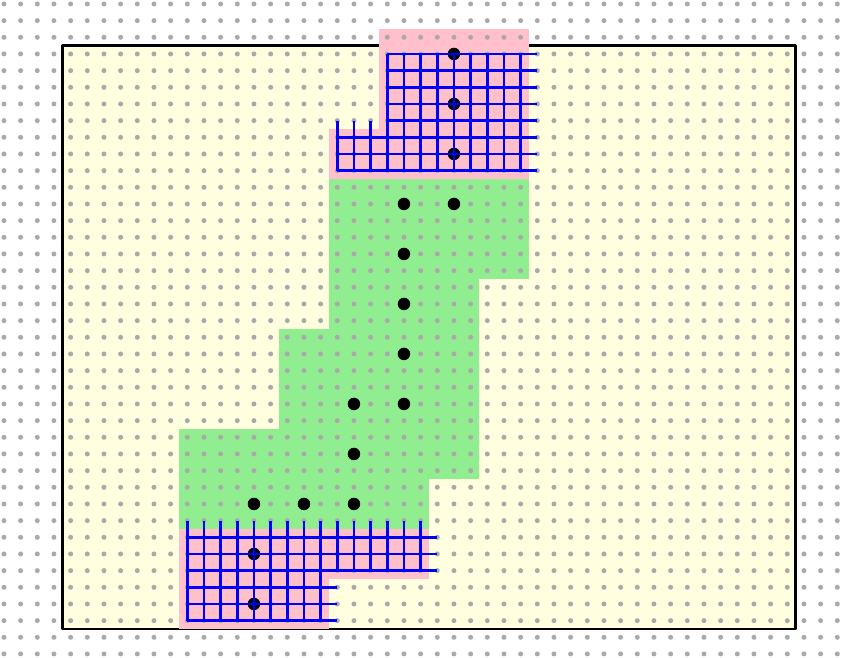}
	\hspace*{1cm}
	\includegraphics[scale=1]{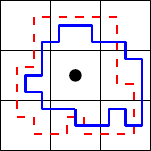}
	\caption{Left: a path in \(\Gamma_1\) (black disks), the blue edges have their state forced to be \((0,1)\), the black disks in a green square imply the realization of \(\mathrm{Ma}_i\) at that place. Right: a realization of \(\mathrm{Ma}_i\) (blue edges are open in \(\omega_{\tau\tau'}\), while the dashed red path represent the dual of edges closed in \(\omega_{\tau}\)).}%
	\label{Fig:StrMix:path_and_decoupling_event}%
\end{figure}

\section{Random Walk picture in the ATRC}
\label{sec:ATRC:RW_infinite_volume}

We start this section by introducing the objects necessary to the development of the ``Ornstein-Zernike theory'' (renewal picture of connection probabilities). We follow the general strategy used in~\cite{CamIof02, CamIofVel03, CamIofVel08, OttVel18, AouOttVel24}, most results will be imported when their proof is a repetition of existing arguments.
To keep notations readable, we will use the following short-hand throughout this section:
\begin{equation*}
	\Phi \equiv \atrc_{J,U}, \qquad \Phi_{\Lambda}^{a,b} \equiv \atrc_{\Lambda;J,U}^{a,b},
\end{equation*}
where~$\atrc_{J,U}$ is the unique infinite-volume measure for the~$\atrc$ model (recall Proposition~\ref{prop:atrc01_relax}).
The theorems that will be used in other sections will be stated with the notation matching the rest of the paper. Moreover, again to lighten notations,
\begin{equation*}
	\text{connections are in } \omega_{\tau} \text{ unless explicitly stated otherwise.}
\end{equation*}
Define \(\calC\) as the cluster of \(0\) in \(\omega_{\tau}\).

\subsection{Decay rate, Cones, and Diamonds}
\label{subsec:OZ_def}

Start by introducing some objects.

\noindent{\bf Norm induced by the decay rates and associated sets.}
For \(s\in \bbS^1\), define
\begin{equation}
	\label{eq:decay_rate}
	\nu(s) = -\lim_{n\to\infty} \tfrac{1}{n}\ln \Phi(0\leftrightarrow ns).
\end{equation}
Existence of the limit follows in the standard fashion by Fekete's Lemma, using sub-additivity which follows from the FKG inequality.
We now extend~$\nu$ to~$\R^2$ by positive homogeneity of order one: for any~$s\in \bbS^1$ and~$r\geq 0$, define
\[
	\nu(r s) := r \nu(s).
\]
Using existence of the limit and the FKG inequality, one can show that~$\nu$ is a non-degenerate norm as soon as it is non-zero: 
$\nu$ is positive homogeneous by definition; 
exponential decay of connection probabilities in~$\Phi$ (the ATRC model) implies positivity; 
the FKG inequality for~$\Phi$ implies the triangular inequality for~$\nu$. See~\cite[Section 2]{Ale01} for details.
Clearly, \(\nu\) inherits the symmetries of the lattice, that is axial and diagonal reflections and rotations by \(\pi/2\). 
From these symmetry considerations, one has that for any \(s\in \bbS^1\),
\begin{equation*}
	\tfrac{1}{\sqrt{2}}\nu(\rme_1)\leq \nu(s)\leq \sqrt{2}\nu(\rme_1),
\end{equation*}with \(\rme_1=(1,0)\). Moreover, one has
\begin{equation}
	\label{eq:perco_order_one_decay_rate}
	\Phi(0\leftrightarrow x) = e^{-\nu(x)(1+o(1))}.
\end{equation}
As \(\nu\) is a norm, there are two convex sets naturally associated to it: the equi-decay set \(\calU\), and the convex set of which \(\nu\) is the support function, \(\calW\).
\begin{equation}
	\label{eq:equi_decay_Wulff}
	\calU = \{x\in \R^d:\ \nu(x)\leq 1\},\qquad 
	\calW = \bigcap_{s\in \bbS^1}\{x\in \R^d:\ x\cdot s\leq \nu(s)\}.
\end{equation}
It is easy to see that \(\calU\) and \(\calW\) are dual to each other, that is: \(\calW\) is the equi-decay set of the norm dual to \(\nu\), and the norm dual to \(\nu\) is the support function of \(\calU\).
We say that $(s,t)\in \bbS^1\times \partial\calW$ is a \emph{dual pair} if \(\nu(s) = t\cdot s\).
From the definitions, one has, for any~$x\in \R^d$,
\begin{equation}
	\nu(x) = \max_{t\in \partial \calW} t\cdot x.
\end{equation}
We refer to~\cite{Roc70} for details on convex duality. From the last display, it is also easy to see that \(\calW\) is the closure of the convergence domain of
\begin{equation}
	\label{eq:generating_fct_connections}
	\bbG(t) = \sum_{x\in \Z^2}\Phi(0\leftrightarrow x) e^{t\cdot x}.
\end{equation}

\noindent{\bf Cones and Diamonds.}
Let \(t\in \partial \calW, \delta\in (0,1)\). Let us introduce the geometric objects used in the interface study. We first define the cones and the associated diamonds (see Fig.~\ref{Fig:OZ_cones_diamonds}):
\begin{gather*}
	\fcone_{t,\delta} := \big\{ x\in \R^2:\ x\cdot  t\geq (1-\delta)\nu(x) \big\}, \quad \bcone_{t,\delta} := -\fcone_{t,\delta},\\
	\diam_{t,\delta}(u,v) := (u+\fcone_{t,\delta})\cap (v+\bcone_{t,\delta}).
\end{gather*}

\begin{figure}
	\includegraphics[scale=0.8]{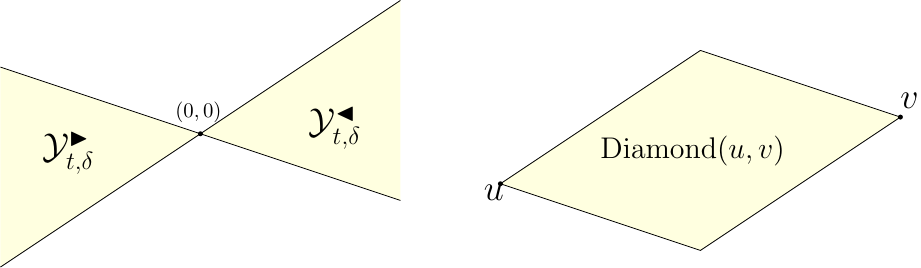}
	\caption{Forward and backward cones, and the associated diamond.}
	\label{Fig:OZ_cones_diamonds}
\end{figure}

As \(\delta\) goes to \(1\), the cone \(\fcone_{t,\delta}\) converges to the half-plane that contains~$t$ and whose boundary is orthogonal to~$t$. As \(\delta\) goes to \(0\), the cone \(\fcone_{t,\delta}\) converges to the convex cone generated by the directions dual to \(t\).
The latter set is a line when \(\nu\) is strictly convex.

Let \(V\subset \Z^2\). We will say that \(V\) is:
\begin{itemize}
	\item \((t,\delta)\)-\emph{forward-confined} if there exists \(u\in V\) such that \(V\subset u+\fcone_{t,\delta}\). When it exists, such a \(u\) is unique; we denote it by \(\fend({V})\).
	\item \((t,\delta)\)-\emph{backward-confined} if there exists \(v\in V\) such that \(V\subset v+\bcone_{t,\delta}\). When it exists, such a \(v\) is unique; we denote it by \(\bend({V})\).
	\item \((t,\delta)\)-\emph{diamond-confined} if it is both forward- and backward-confined.
\end{itemize}

We will say that \(v\in V\) is a \emph{\((t,\delta)\)-cone-point} of \(V\) if
\begin{equation*}
	V\subset v+ (\bcone_{t,\delta}\cup \fcone_{t,\delta}).
\end{equation*}
Define~\(\CPts_{t,\delta}(V)\) as the set of cone-points of \(V\). When speaking about cone-points of graphs, we mean cone-points of their vertex set.

We call a graph with a distinguished vertex a \emph{marked graph}. The distinguished vertex is denoted \(v^*\). Define the following objects (see Fig.~\ref{Fig:OZ_confined_graphs_displacements}):
\begin{itemize}
	\item The sets of confined pieces are the following sets of finite connected subgraphs of \((\Z^2,\bbE)\):
	\begin{align*}
		\SetRootMarkBackCont(t,\delta) &= \{(\gamma,0) \text{ marked backward-confined } \},\\
		\SetRootMarkForwCont(t,\delta) &= \{(\gamma,v^*) \text{ marked forward-confined with } \fend(\gamma) =0 \},\\
		\SetRootDiaCont(t,\delta) &= \{\gamma \text{ diamond-confined with } \fend(\gamma) =0 \}.
	\end{align*}
	To fix ideas we shall, unless stated otherwise, think of $\SetRootDiaCont$ as of a subset of $\SetRootMarkBackCont$, that is, by default the vertex $\fend (\gamma) = 0$ is marked for any $\gamma\in \SetRootDiaCont$. 
	Note that $\SetRootDiaCont$ can alternatively be viewed as subset of~\(\SetRootMarkForwCont\) by marking~\(\bend(\gamma)\). We will usually omit the marked vertex from the notation as it will be clear from the context.
	
	\item The displacement along a piece:
	\begin{equation}
		\label{eq:displacement}
		\displace(\gamma) :=
		\begin{cases}
			\bend(\gamma)	& \text{ if } \gamma\in
			\SetRootMarkBackCont,\ \text{in particular, if $\gamma\in \SetRootDiaCont$,} \\
			v^* 			& \text{ if } \gamma\in\SetRootMarkForwCont.
		\end{cases}
	\end{equation}
	\item The \emph{concatenation} operation (see Fig.~\ref{Fig:OZ_concatenation_displacements}): for \(\gamma_1\in \SetRootMarkBackCont\) and \(\gamma_2\in \SetRootMarkForwCont\) define the concatenation of $\gamma_2$ to $\gamma_1$ as
	\begin{equation*}
		\gamma_1\concatenate \gamma_2 = \gamma_1\cup (\displace(\gamma_1) + \gamma_2).
	\end{equation*}
	The concatenation of two graphs in \(\SetRootDiaCont\) is an element of \(\SetRootDiaCont\), the concatenation of
	a graph in \(\SetRootDiaCont\) to an element of \(\SetRootMarkBackCont\) is an element of \(\SetRootMarkBackCont\), and the concatenation of a element in \(\SetRootMarkForwCont\) to an graph in \(\SetRootDiaCont\) is an element of \(\SetRootMarkForwCont\). The displacement along a concatenation is the sum of the displacements along the pieces.
\end{itemize}

\begin{figure}
	\includegraphics[scale=0.8]{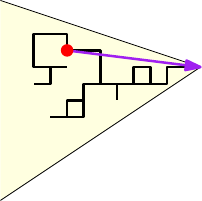}
	\hspace*{1cm}
	\includegraphics[scale=0.8]{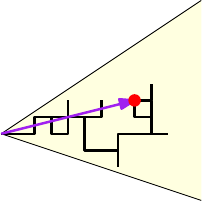}
	\hspace*{1cm}
	\includegraphics[scale=0.8]{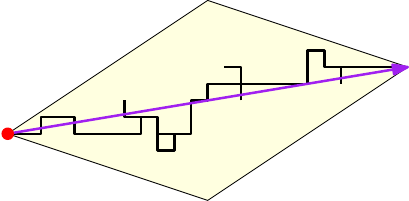}
	\caption{Backward, forward, and diamond confined marked graphs,with their displacement.}
	\label{Fig:OZ_confined_graphs_displacements}
\end{figure}

\begin{figure}
	\includegraphics[scale=0.8]{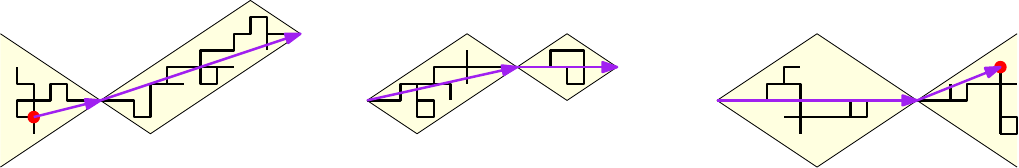}
	\caption{Concatenation of confined graphs. Red dots are the marked vertices of the marked graphs.}
	\label{Fig:OZ_concatenation_displacements}
\end{figure}

These sets can be seen as equivalence classes of general marked forward, backward, or diamond confined graphs modulo translations. A general such graph can then be recovered uniquely from an element of $\SetRootMarkBackCont$, $\SetRootMarkForwCont$, or~$\SetRootDiaCont$ by specifying the translation vector.

\subsection{Main result of the section}

Our main goal is to prove the following ``coupling with random walk'' in infinite volume for long connections in the ATRC model.
Recall that~$\rme_1= (1,0), \rme_2=(0,1)$ are the canonical basis vectors and~$\calC$ denotes the cluster of~$0$ in~$\omega_\tau$.

\begin{theorem}
	\label{thm:OZ_for_ATRC_infinite_vol}
	Let \(0\leq J<U\) be such that \(\sinh(2J) = e^{-2U}\). Let \(s_0\in \bbS^1\), \(t\in \partial\calW\) be dual to \(s_0\) and \(\delta\in (0,1)\) be such that the interior of \(\fcone_{t,\delta}\) contains an element of $\{\pm \rme_1,\pm\rme_2\}$. Then, there exist constants~\(\rmC,C_1,C_2\geq 0, c_1,c_2>0, \epsilon>0\) and probability measures \(p_L\), \(p_R\), and~\(p\) on \(\SetRootMarkBackCont(t,\delta)\), \(\SetRootMarkForwCont(t,\delta)\), and~\(\SetRootDiaCont(t,\delta)\) respectively such that
	\begin{enumerate}
		\item \emph{renewal structure:} for any \(x\in \Z^2\) such that \(x\cdot s_0 \geq (1-\epsilon)|x|\), and any \(f\) real valued function of the cluster of \(0\),
		\begin{multline}
			\Big| \rmC\sum_{\gamma_{L},\gamma_R}p_L(\gamma_L)p_R(\gamma_R)\sum_{k\geq 0}\sum_{\gamma_1,\dots,\gamma_k}  f(\bar{\gamma})\mathds{1}_{\displace(\bar{\gamma})= x} \prod_{i=1}^k p(\gamma_k) \\
			- e^{t\cdot x}\atrc_{J,U}\big(f(\calC)\mathds{1}_{0\xleftrightarrow{\omega_{\tau}} x}\big)\Big|\leq C_1\norm{f}_{\infty}e^{-c_1|x|},
		\end{multline}where \(\bar{\gamma} = \gamma_L\concatenate \gamma_1\concatenate\dots\concatenate \gamma_k\concatenate \gamma_R\), and the sums are over \(\gamma_L\in \SetRootMarkBackCont(t,\delta)\), \(\gamma_R\in \SetRootMarkForwCont(t,\delta)\), and \(\gamma_1,\dots,\gamma_k\in \SetRootDiaCont(t,\delta)\);
		\item \emph{exponential tails and ``finite energy'' for the steps:} for \linebreak
		\((p_*, D_*)\in \{(p_L,\SetRootMarkBackCont(t,\delta)),(p_R,\SetRootMarkForwCont(t,\delta)),(p,\SetRootDiaCont(t,\delta))\}\),
		\begin{equation}
			\sum_{\gamma \in D_*} p_*(\gamma) \mathds{1}_{\norm{\displace(\gamma)} \geq \ell}
			\leq C_2 e^{-c_2\ell} \quad \forall \ell\geq 0,
		\end{equation}
		\begin{equation}
			p_*\big(\gamma\big) \geq \alpha^{|\gamma|+1}\quad \forall \gamma\in D_*:\ \CPts_{t,\delta}(\gamma)\cap \diam_{t,\delta}(0,\displace(\gamma)) \subset \{0,\displace(\gamma)\},
		\end{equation}
		for some \(\alpha>0\) depending on \(J,U\) only;
		\item \emph{mean value of a step is proportional to \(s_0\):} there exists \(a>0\) such that
		\begin{equation}
			\sum_{\gamma\in \SetRootDiaCont(t,\delta)} p(\gamma)\displace(\gamma) = a s_0.
		\end{equation}
	\end{enumerate}
	
	Moreover, \(\partial\calU\), \(\partial \calW\) are analytic manifolds and the norm \(\nu\) is uniformly strictly convex, that is \(\calU\) is strictly convex and \(\partial\calU\) has uniformly lower bounded curvature.
	In particular, each direction \(s\in \bbS^1\) has a unique dual vector~\(t_s\in \partial\calW\) and there exist \(\kappa>0\) such that the following sharp triangle inequality holds:
	\begin{equation*}
		\nu(x)+\nu(y) -\nu(x+y) \geq  \kappa(|x|+|y|-|x+y|).
	\end{equation*}
\end{theorem}

The rest of the section is devoted to the proof of Theorem~\ref{thm:OZ_for_ATRC_infinite_vol}. Subsections~\ref{subsec:OZ_CG},~\ref{subsec:OZ_Ene_Entro}, and~\ref{sec:importations} are preparations for the core of the proof. The proof itself is then divided into two main steps, presented respectively in Subsections~\ref{subsec:CPs_pre_renewal} and~\ref{subsec:mixing_weights_renewal}.

\begin{remark}
	We use the value of \(J,U\) through only two inputs: the exponential decay in \(\omega_{\tau}\) and \(\omega_{\tau\tau'}^*\), and the mixing of Theorems~\ref{thm:ratio_weak_mixing_ATRC} and~\ref{thm:strong_mixing_atrc} (which follows from edge relaxation, Proposition~\ref{prop:edge_relax_ATRC}, and exponential decay in \(\omega_{\tau}\) and \(\omega_{\tau\tau'}^*\)). In particular, if one can extend these properties beyond the self-dual line (part of which is done in~\cite{AouDobGla24}), Theorem~\ref{thm:OZ_for_ATRC_infinite_vol} extends directly.
\end{remark}

\subsection{The coarse-graining procedure}
\label{subsec:OZ_CG}

The first step is to analyse the typical geometry of long connections. The analysis of~\cite{CamIofVel08,AouOttVel24} is based on a coarse-graining of the cluster of \(0\). 
Recall the equi-decay set~$\calU$ defined in~\eqref{eq:equi_decay_Wulff}.
Introduce the cells: for \(K,k\geq 1\) and \(A\subset \Z^2\),
\begin{gather*}
	[A]_{k} = \bigcup_{x\in A} (x+\{-k,\dots,k\}^2),\\
	\Delta_K = K\calU\cap \Z^2,\quad \Delta_K' = [\Delta_{K}]_{\ln(K)^2}.
\end{gather*}
The scale parameter \(K\) will be picked large enough in the course of the proof.

We then coarse-grain the cluster of \(0\) using the same algorithm as in~\cite{CamIofVel08}. Denote \(\Delta \equiv \Delta_K\) and \(\Delta' \equiv \Delta_K'\). For \(C\) a realization of \(\calC\) define \(\Skel(C)\) via the next algorithm (see also Fig.~\ref{Fig:OZ_CG}).

\begin{algorithm}[H]
	\label{alg:CGpercolation}
	Set \(v_0=0\), \(\Skel_V=\{v_0\}\), \(\Skel_E=\varnothing\), \(V= \Delta'\), \(i=1\)\;
	\While{\(A=\big\{ z\in\partialex V:\ z\xleftrightarrow{(z+\Delta)\setminus V }\partialin(z+\Delta)\big\}\neq \varnothing \)}{
		Set \(v_{i}= \min A \)\;
		Let \(v^*\) be the smallest \(v \in \Skel_V\) such that \(v_i\in \partialin (\Delta + v^*)\)\;
		Update \(\Skel_V=\Skel_V\cup\{v_{i}\}\), \(\Skel_E=\Skel_E\cup\{\{v^*,v_i\}\}\), \(V= V\cup(v_i + \Delta')\), \(i=i+1\)\;
	}
	\Return \((\Skel_V,\Skel_E)\)\;
	\caption{Coarse graining of a cluster containing \(0\).}
\end{algorithm}

The output \(\Skel(C) = (\Skel_V(C),\Skel_E(C))\) of Algorithm~\ref{alg:CGpercolation} is a tree with \(\Skel_V(C)\subset \Z^2\); we use the shorthand \(|\Skel(C)| := |\Skel_V(C)|\) for its size.

\begin{figure}
	\includegraphics[scale=0.9]{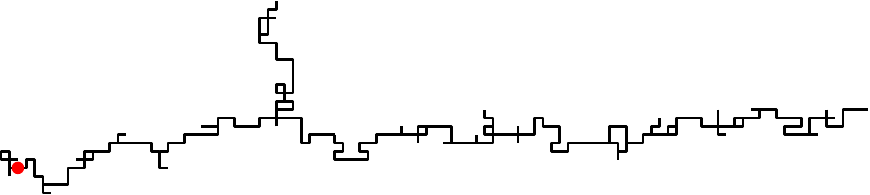}
	\includegraphics[scale=0.9]{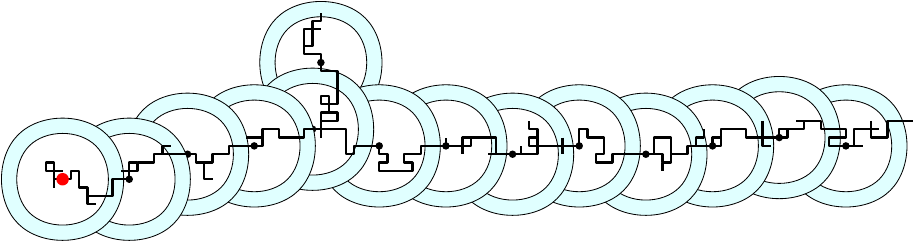}
	\includegraphics[scale=0.9]{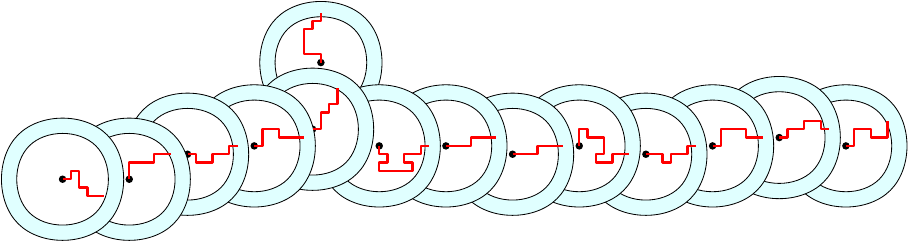}
	\includegraphics[scale=0.9]{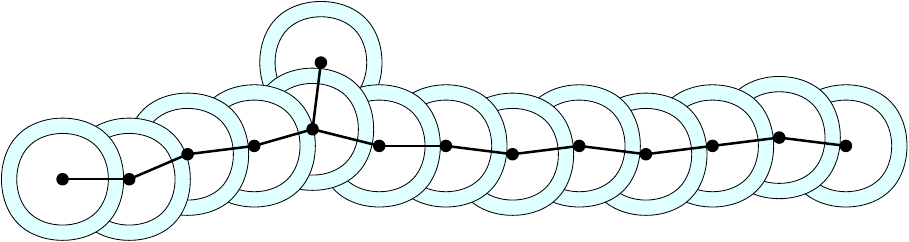}
	\caption{From top to bottom: a realization of \(\calC\) (\(0\) is in red); the cells of the associated coarse-graining; the connections implied by the presence of the cells; the skeleton (in black).}
	\label{Fig:OZ_CG}
\end{figure}

\subsection{Energy-Entropy estimates}
\label{subsec:OZ_Ene_Entro}

The next step is to establish energy bounds which control the probability to see a given tree as the skeleton of a cluster.
\begin{lemma}
	\label{lem:finite_volume_exp_decay}
	For any \(0\in A\subset \Z^2\) and \(K\geq 1\),
	\begin{equation}
		\label{eq:connection_decay_bulk}
		\Phi_{[\Delta_{K}\cap A]_{\ln(K)^2}}^{1,1}(0 \xleftrightarrow{\Delta_K\cap A} \partialin\Delta_K) \leq e^{-K(1+o_K(1))},
	\end{equation}where \(o_K(1)\) is a quantity that goes to \(0\) as \(K\) goes to \(\infty\).
\end{lemma}
Note that the wanted probabilities are zero whenever \(A\cap \Delta_K\) does not contain a path going from \(0\) to \(\partialin \Delta_K\).
\begin{proof}
	By the definition of \(\calU\), the ratio weak mixing property (Theorem~\ref{thm:ratio_weak_mixing_ATRC}) and monotonicity:
	\begin{align*}
		\Phi_{[\Delta_{K}\cap A]_{\ln(K)^2}}^{1,1}(0 \xleftrightarrow{\Delta_K\cap A} \partialin \Delta_K) &
		\leq (1+ CK^2e^{-c' \ln(K)^2})\Phi_{[\Delta_{K}\cap A]_{\ln(K)^2}}^{0,0}(0 \xleftrightarrow{\Delta_K\cap A} \partialin \Delta_K)\\ 
		&\leq 2 \Phi(0 \xleftrightarrow{\Delta_K} \partialin \Delta_K)\\
		&\leq C\sum_{x\in \partialin \Delta_K} e^{-\nu(x)(1+o_K(1))} = C|\partialin \Delta_K|e^{-K(1+o_K(1))}
	\end{align*}as soon as \(K\) is large enough.
\end{proof}

As a direct consequence of Lemma~\ref{lem:finite_volume_exp_decay} and of the definition of the coarse-graining procedure, one gets the following:
\begin{lemma}[Energy bound]
	\label{lem:energy_of_a_skeleton}
	There exists \(K_0\geq 0\) such that, for any \(K\geq K_0\),
	and any \(\calT \subset \Z^2\),
	\begin{equation*}
		\Phi\big(\Skel_V(\calC) = \calT \big) \leq e^{-K|\calT|(1+o_K(1))},
	\end{equation*}where the \(o_K\) is uniform over \(\calT\)
\end{lemma}

Indeed, every new vertex of the tree away from the boundary induces a connection of the form~\eqref{eq:connection_decay_bulk} in the complement of the neighbourhood of the previously explored vertices (See Fig.~\ref{Fig:OZ_CG}). We do not provide further details,
see for example~\cite[Display (2.2)]{CamIofVel08} for the implementation of the bound.

Recall that~$\Skel(C)$ is a tree rooted at~$0$.
Denote by~$\mathrm{Tree}_N$ the set of all possible values of~$\Skel(C)$ if~$|\Skel(C)|=N$.
We now state a general combinatorial lemma that bounds the size of~$\mathrm{Tree}_N$.

\begin{lemma}[Entropy bound]
	\label{lem:entropy_of_skeleton_sets}
	There exists a universal \(c>0\) such that
	\begin{equation*}
		|\mathrm{Tree}_N| \leq e^{c\ln(K)N} = K^{cN}.
	\end{equation*}
\end{lemma}

This Lemma follows from the fact that, for some \(C\geq 0\), the size of~$\mathrm{Tree}_N$ is smaller or equal to the number of \(N\)-vertex connected sub-trees of the \(CK^2\)-regular tree containing \(0\).
The latter is bounded by~\(e^{c\ln(K)N}\) for some \(c>0\) by Kesten's argument~\cite[page 85]{Kes82}.

\subsection{Input from~\cite{CamIofVel08,CamIofVel03}: skeleton and cluster cone-points}
\label{sec:importations}
 
The main result that we import from~\cite{CamIofVel08,CamIofVel03} is~\cite[Theorem 2.1]{CamIofVel08} that describes a typical geometry of skeletons (\cite{CamIofVel08} builds on~\cite{CamIofVel03}).
The result in~\cite{CamIofVel08} is stated for the FK-percolation, but the proof is general and relies only on Lemmata~\ref{lem:energy_of_a_skeleton} and~\ref{lem:entropy_of_skeleton_sets}.

\begin{lemma}[Cone-points of a skeleton]
	\label{lem:spine_skeleton_CPts}
	Let \((\Omega,\calF,P)\) be a probability space, \(v\in \Z^2\), and let \(\calC\) be a random finite connected subset of \(\Z^2\) containing \(v\) defined on \((\Omega,\calF,P)\). Suppose that for any tree \(\calT\) in the image of \(\Skel\), \(K\geq 0\),
	\begin{equation*}
		P(\Skel(\calC-v) = \calT)\leq e^{-K |\calT|(1+o_K(1))}.
	\end{equation*}
	Then, for any \(\delta>0\), there exist \(c_1,c_2>0\), \(K_0\geq 0\) such that, for any \(K\geq K_0\), \(t\in \partial\calW\), \(w\in \Z^2\),
	\begin{equation*}
		e^{t\cdot (w-v)}P\big(w\in \calC, |\CPts_{t,\delta}(\Skel(\calC-v))|\leq c_1 |w-v|/K\big)\leq e^{-c_2|w-v|}.
	\end{equation*}
	Moreover, by monotonicity, \(c_1,c_2,K_0\) can be taken uniform over \(\delta\geq \delta_0>0\).
\end{lemma}

\begin{remark}
	\label{rem:out_of_the_cone_connections}
	Note that the Lemma holds directly when \(w\) is not in \(v+\fcone_{t,\delta}\): then \(P(w\in \calC)\leq e^{-\nu(w-v)(1+o(1))}\) (as the event \(w\in \calC\) implies that \(|\Skel(\calC)|\geq \nu(w-v)/K\)) but \(t\cdot(w-v) -\nu(w-v) \leq -\delta\nu(w-v)\) by definition of \(\fcone_{t,\delta}\).
\end{remark}

The second result we import is a simple but notationally heavy use of finite energy. One can find two different implementations of this argument in~\cite[Section 2.9]{CamIofVel08}, and~\cite[Section 6.1]{AouOttVel24}. Introduce a small variation on the notion of cone-points which will be convenient later (one could work directly with cone-points, but the equations become a bit heavier).
Recall \(\rme_1=(1,0)\), \(\rme_2=(0,1)\).

\begin{definition}[Regular Cone-points]
	Let \(C=(V,E)\) be a connected subgraph of \((\Z^2,\bbE)\). Say that \(v\in V\) is a \emph{regular \((t,\delta)\)-cone-point} of \(C\) (see Fig.~\ref{Fig:OZ_RegCPts}) if it is a \((t,\delta)\)-cone-point of \(V\) and, if \(\{\pm \rme_1\}\cap \fcone_{t,\delta} \neq \varnothing\), \(\{v, v+\rme_1\},\{v,v-\rme_1\} \in E\) and \(\{v, v+\rme_2\},\{v,v-\rme_2\} \notin E\).
	Define~\(\rCPts_{t,\delta}(C)\) as the set of \((t,\delta)\)-regular cone-points of \(C\).
\end{definition}
Note than when \(\fcone_{t,\delta}\) contains exactly one element of \(\{\pm\rme_1,\pm\rme_2\}\), all cone-points are necessarily regular.

\begin{figure}
	\includegraphics{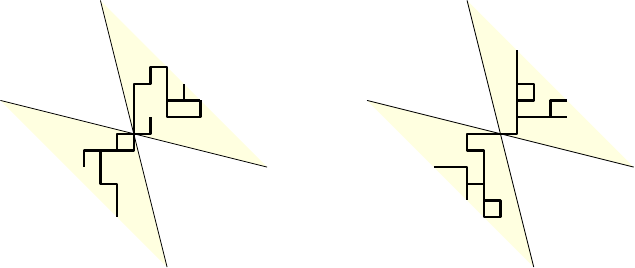}
	\caption{Left: a cone-point that is \emph{not} regular. Right: a regular cone-point.}
	\label{Fig:OZ_RegCPts}
\end{figure}

The idea for going from Lemma~\ref{lem:spine_skeleton_CPts} to the next lemma is simple: when \(v\leftrightarrow w\), up to an exponentially small error, there must be at least \(c|v-w|/K\) cone-points of the skeleton by Lemma~\ref{lem:spine_skeleton_CPts}.
Up to anther exponentially small error, a positive fraction of these cone-points must then be regular cone-points of~$\calC$ by uniform finite energy.

\begin{lemma}[Cluster Cone-points]
	\label{lem:Cluster_CPts}
	Let~\((\Omega,\calF,P)\) be probability space, \(v\in \Z^2\), and \(\omega:\Omega\to \{0,1\}^{\bbE}\) be a bond percolation random variable.
	The cluster of \(v\) in \(\omega\) is denoted by~\(\calC = \calC(\omega)\).
	Suppose that
	\begin{itemize}
		\item \(\omega\) has uniform finite energy (for opening and closing edges): for \(e\in \bbE\), let \(\calF_{e^{c}}\) be the sigma-algebra generated by \((\omega_f)_{f\neq e}\). Then, there exists \(\epsilon\in (0,1)\) such that for every \(e\in \bbE\),
		\begin{equation*}
			\epsilon <E(\omega_e\given \calF_{e^c}) < 1-\epsilon,
		\end{equation*}\(P\)-almost surely.
		\item the conclusion of Lemma~\ref{lem:spine_skeleton_CPts} holds for \(\calC\).
	\end{itemize}
	Then, for any \(\delta>0\), there exist \(c_1',c_2'>0\), \(L_0\geq 1\) such that, for any \(t\in \partial\calW\) such that the interior of \(\fcone_{t,\delta}\) contains an element of $\{\pm\rme_1, \pm \rme_2\}$, and any 
	\(w\in \Z^2\cap \fcone_{t,\delta}\) with \(|w-v|\geq L_0\),
	\begin{equation*}
		e^{t\cdot(w-v)}P\big(w\in \calC, \ |\rCPts_{t,\delta}(\calC)|\leq c_1' |w-v|\big) \leq e^{-c_2' |v-w|}.
	\end{equation*}By monotonicity, \(c_1',c_2',L_0\) can be taken uniform over \(\delta\geq \delta_0>0\).
\end{lemma}
See~\cite[Section 2.9]{CamIofVel08} or~\cite[Section 6.1]{AouOttVel24} for two different proofs of Lemma~\ref{lem:Cluster_CPts} via local surgery arguments.
From these two lemmas we can deduce our main cone-point estimate.

\begin{theorem}
	\label{thm:cone_points}
	Let \(\delta\in (0,1)\). There exist \(C\geq 0,c_1,c_2>0\) such that for any \(t\in \partial\calW\) such that the interior of \(\fcone_{t,\delta}\) contains an element of \(\{\pm\rme_1,\pm \rme_2\}\), and any \(x\in \Z^2\),
	\begin{equation*}
		e^{t\cdot x}\Phi\big(0\leftrightarrow x, |\rCPts_{t,\delta}(\calC)|\leq c_1 |x|\big)\leq Ce^{-c_2|x|}.
	\end{equation*}
\end{theorem}
\begin{proof}
	When \(x\notin \fcone_{t,\delta}\), the claim is trivial, see Remark~\ref{rem:out_of_the_cone_connections}. For \(x\in \fcone_{t,\delta}\), the result follows from Lemmata~\ref{lem:spine_skeleton_CPts}, and~\ref{lem:Cluster_CPts}: Lemma~\ref{lem:energy_of_a_skeleton} provides the required bound on the probability of a given skeleton, and the model has finite energy, which are the needed hypotheses for Lemmata~\ref{lem:spine_skeleton_CPts}, and~\ref{lem:Cluster_CPts}.
\end{proof}

This concludes our preparations. We are now ready to attack the proof of Theorem~\ref{thm:OZ_for_ATRC_infinite_vol}.

\subsection{Proof of Theorem~\ref{thm:OZ_for_ATRC_infinite_vol} part I: Cone-points and pre-renewal structure}
\label{subsec:CPs_pre_renewal}

We now fix \(s_0\in \bbS^1\), \(t_0\in \partial\calW\) dual to \(s_0\), \(\delta\in (0,1)\) such that the interior of \(\fcone_{t_0,\delta}\) contains and element of \(\{ \pm\rme_1, \pm \rme_2\}\), and we set
\begin{equation}
	\label{eq:t_0_fixing}
	\begin{gathered}
		\rCPts \equiv \rCPts_{t_0,\delta},\quad \fcone \equiv \fcone_{t_0,\delta}, \quad \bcone \equiv \bcone_{t_0,\delta},
		\\
		\SetRootMarkBackCont \equiv \SetRootMarkBackCont(t_0,\delta),\quad 
		\SetRootMarkForwCont \equiv \SetRootMarkForwCont(t_0,\delta),\quad
		\SetRootDiaCont \equiv \SetRootDiaCont(t_0,\delta).
	\end{gathered}
\end{equation}

From Theorem~\ref{thm:cone_points}, typical long clusters have many cone-points: having a non-linear (in \(|x|\)),  number of cone-points is exponentially more unlikely than \(\{x\in \calC\}\). The idea is now to write a realization of \(\calC\) containing many cone-points as a concatenation of ``irreducible graphs'' by splitting it at its regular cone-points. This will lead to a structure that, graphically, looks like a renewal structure. This is what we will do in this subsection. The last step will then be to extract a \emph{real} renewal structure (at the level of the measure) from this graphical one. This will be Subsection~\ref{subsec:mixing_weights_renewal}. To this end, introduce the notion of \emph{irreducible graphs}. Let \(\rme = \{\pm\rme_1\}\cap \fcone\) if the intersection is not empty, and \(\rme= \{\pm\rme_2\}\cap\fcone\) else. Say that
\begin{itemize}
	\item A marked backward-confined graph \((\gamma_L,v^*)\) is \emph{irreducible} if the diamond \linebreak \(\diam(v^*, \bend(\gamma_L))\) does not contain a regular cone-point of \(\gamma_L\), and \(\bend(\gamma_L)\) is a regular cone-point of the graph \(\gamma_L\cup \{\bend(\gamma_L), \bend(\gamma_L)+\rme\}\) (see Fig.~\ref{Fig:irr_examples}).
	\item A marked forward-confined graph \((\gamma_R,v^*)\) is \emph{irreducible} if the diamond \linebreak \(\diam(\fend(\gamma_R), v^*)\) does not contain a regular cone-point of \(\gamma_R\), and \(\fend(\gamma_R)\) is a regular cone-point of the graph \(\gamma_R\cup \{\fend(\gamma_R), \fend(\gamma_R)-\rme\}\).
	\item A diamond-confined graph \(\gamma\) is \emph{irreducible} if it does not contain a regular cone-point, and if \(\fend(\gamma),\bend(\gamma)\) are regular cone-points of \(\gamma\cup \{\{\fend(\gamma), \fend(\gamma)-\rme\},\{\bend(\gamma), \bend(\gamma) +\rme\}\}\) (see Fig.~\ref{Fig:irr_examples}).
\end{itemize}
In words: irreducible graphs are those that cannot be written as the concatenation of two non-trivial graphs, with the concatenation point being a regular cone-point. We will denote the sets of irreducible marked graphs in, respectively, \(\SetRootMarkBackCont\), \(\SetRootMarkForwCont\), \(\SetRootDiaCont\) by
\begin{equation*}
	\SetRootMarkBackCont^{\irr},\quad 
	\SetRootMarkForwCont^{\irr},\quad
	\SetRootDiaCont^{\irr}.
\end{equation*}

\begin{figure}
	\centering
	\includegraphics[scale=0.6]{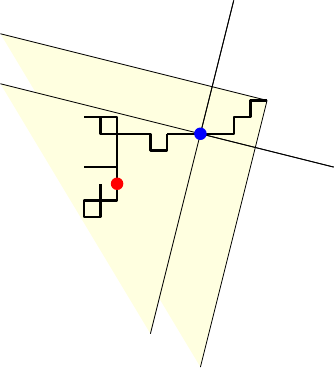}
	\hfill
	\includegraphics[scale=0.9]{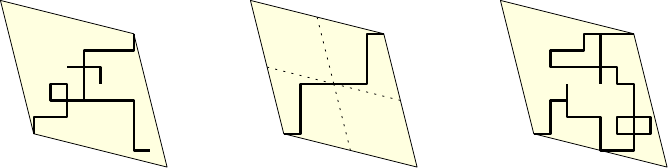}
	\caption{From left to right. 1) Backward confined graph which is not irreducible: the graph can be written as the concatenation of two graphs so that the concatenation point is a regular cone-point. 2) The graph is not irreducible as its endpoints will not generate regular cone-points when concatenated. 3) The graph is not irreducible as it contains a regular cone-point. 4) An irreducible graph.}
	\label{Fig:irr_examples}
\end{figure}

For \(x\in \fcone\cap \Z^2\), and \(C\ni x\) a realization of \(\calC\) containing at least two regular cone points \(v_1,v_2\) with \(0\in v_i+\bcone\), \(x\in v_i+ \fcone\), we can introduce the splitting into irreducible components:
\begin{equation*}
	C = \eta_L\sqcup \eta_1\sqcup\dots\sqcup\eta_M \sqcup \eta_R,
\end{equation*}where \(M\geq 1\), \((\eta_L,0)\), \((\eta_R,x)\), \(\eta_1,\dots,\eta_M\) are all irreducible, confined, (marked) graphs, and \(\sqcup\) means disjoint union of edges; note that there are sites overlap at the cone-points. Now, as mentioned in the end of Section~\ref{subsec:OZ_def}, there is a bijection between pairs \((\tilde{\gamma},v)\in \SetRootDiaCont \times \Z^2\) and diamond-confined connected graphs: translate \(\tilde{\gamma}\) by \(v\) to obtain the graph \(\gamma = v+\tilde{\gamma}\). Similar considerations hold for marked forward/backward confined connected graphs. In particular, for \(\eta_i\) in the above decomposition, there is a unique \(w_i\in \Z^2\) and a unique \(\tilde{\eta}_i\in \SetRootDiaCont^{\irr}\) such that \(\eta_i = w_i+ \tilde{\eta}_i\). Similarly for \(\eta_L,\eta_R\). As the marked point of \(\eta_L\) is \(0\), one has directly
\begin{equation*}
	w_L=0,
	\quad
	w_1 = \displace(\tilde{\eta}_L),
	\quad
	w_2 = \displace(\tilde{\eta}_L\concatenate\tilde{\eta}_1), \dots
	\quad
	w_R = \displace(\tilde{\eta}_L\concatenate\tilde{\eta}_1\concatenate\dots\concatenate \tilde{\eta}_M),
\end{equation*}
and the equivalent writing of \(C\):
\begin{equation*}
	\eta_L\sqcup \eta_1\sqcup\dots\sqcup\eta_M \sqcup \eta_R = \tilde{\eta}_L\concatenate \tilde{\eta}_1\concatenate\dots\concatenate\tilde{\eta}_M \concatenate \tilde{\eta}_R,
	\qquad
	\displace(\tilde{\eta}_L\concatenate\tilde{\eta}_1\concatenate\dots \concatenate \tilde{\eta}_R) = x,
\end{equation*}
with \(\tilde{\eta}_1,\dots,\tilde{\eta}_M\in \SetRootDiaCont^{\irr}\), \(\tilde{\eta}_L\in \SetRootMarkBackCont^{\irr}\), and \(\tilde{\eta}_R\in \SetRootMarkForwCont^{\irr}\).

By our definition of ``regular cone points'', each \(\tilde{\eta}\in \SetRootDiaCont^{\irr}\) contains a unique edge incident to \(\fend(\tilde{\eta})\), and the same for \(\bend(\tilde{\eta})\). Denote these edges by \(\fend_e(\tilde{\eta}), \bend_e(\tilde{\eta})\).
For \(\tilde{\eta},\tilde{\eta}_L,\tilde{\eta}_R\) as above, and \(v\in \Z^2\), introduce the percolation events \(A_v(\tilde{\eta})\) stating that the following occurs:
\begin{itemize}
	\item the edges in \((v+\tilde{\eta}\setminus \{\fend_e(\tilde{\eta}), \bend_e(\tilde{\eta})\})\) are open in \(\omega_{\tau}\),
	\item \(v+ \fend_e(\tilde{\eta})\), and \(v+\bend_e(\tilde{\eta})\) are open in \(\omega_{\tau\tau'}\) and closed in \(\omega_{\tau}\), and
	\item \(\partialex (v+\tilde{\eta})\) are closed in \(\omega_{\tau}\).
\end{itemize}Define the edges \(\bend_e(\tilde{\eta}_L)\), \(\fend_e(\tilde{\eta}_R)\), and the events \(A_v(\tilde{\eta}_L)\), \(A_v(\tilde{\eta}_R)\) similarly. See Fig.~\ref{Fig:Irred_decomp_Perco_events}.

\begin{figure}
	\includegraphics{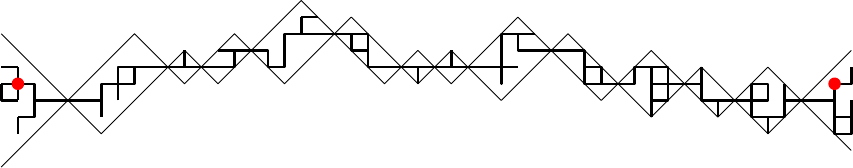}
	\includegraphics{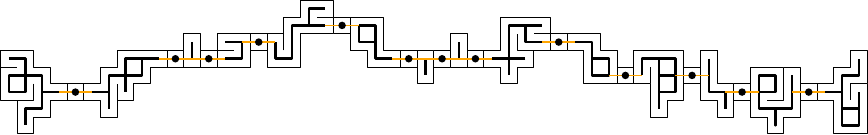}
	\caption{Top: A cluster with its irreducible decomposition. Bottom: the corresponding percolation events. Bold black edges are open in \(\omega_{\tau}\) (hence also in \(\omega_{\tau\tau'}\)); golden edges are closed in \(\omega_{\tau}\) but open in \(\omega_{\tau\tau'}\); thin lines are the dual of edges that are closed in \(\omega_{\tau}\), but for which the state of \(\omega_{\tau\tau'}\) is not prescribed.}
	\label{Fig:Irred_decomp_Perco_events}
\end{figure}

\begin{claim}
	\label{claim:chain_of_events}
	Define \(c_{J,U} := \big(\frac{\sinh(2J)}{e^{-2J}-e^{-2U}}\big)^{2}\). Then, one has
	\begin{equation*}
		\Phi(\calC= \eta_L\sqcup \eta_1\sqcup\dots\sqcup\eta_M \sqcup \eta_R)
		=
		c_{J,U}^{(M+1)}\Phi\big(A_{w_L}(\tilde{\eta}_L), A_{w_1}(\tilde{\eta}_1),\dots, A_{w_R}(\tilde{\eta}_R)\big).
	\end{equation*}
\end{claim}
\begin{proof}
	Under \(\calC= \eta_L\sqcup \eta_1\sqcup\dots\sqcup\eta_M \sqcup \eta_R\), closing any edge of the form \(\bend_e,\fend_e\) cuts the cluster into two parts.
	Thus, changing the state of such an edge in\(\omega_{\tau}, \omega_{\tau\tau'}\) from~$(1,1)$ to~$(0,1)$ brings a factor
	\begin{equation*}
		\frac{\rmw_{\tau}}{2\rmw_{\tau\tau'}} = \frac{2e^{2U}\sinh(2J)}{2(e^{2(U-J)}-1)},
	\end{equation*}by definition of the measure: closing \(e\) in \(\omega_{\tau}\) creates a cluster in \(\omega_{\tau}\), hence the factor \(2\) in the denominator, and swaps the value of \((\omega_{\tau}(e),\omega_{\tau\tau}'(e))\) from \((1,1)\), which has weight \(\rmw_{\tau}\), to \((0,1)\), which has weight \(\rmw_{\tau\tau'}\). The claim follows from the definition of \(\rmw_{\tau},\rmw_{\tau\tau'}\), see~\eqref{eq:atrc_weights}, the fact that there are \(2(M+1)\) edges of the form \(\bend_e,\fend_e\) (\(2\) per regular cone-point, \(M+1\) regular cone-points), and that the event obtained from \(\{\calC= \eta_L\sqcup \eta_1\sqcup\dots\sqcup\eta_M \sqcup \eta_R\} \) by closing the edges of the form \(\bend_e,\fend_e\) is precisely \(A_{w_L}(\tilde{\eta}_L)\cap A_{w_1}(\tilde{\eta}_1)\cap \dots \cap A_{w_R}(\tilde{\eta}_R)\).
\end{proof}

For \(\tilde{\eta}_L\in \SetRootMarkBackCont^{\irr}\), and \(\tilde{\eta}\in \SetRootDiaCont^{\irr}\cup \SetRootMarkForwCont^{\irr}\), introduce the conditional weights
\begin{equation}
	\label{eq:def:dep_kernels}
	\begin{gathered}
		q_L(\tilde{\eta}_L) := \Phi\big(A_{w_L}(\tilde{\eta}_L)\big),
		\\
		q(\tilde{\eta} \given \tilde{\eta}_L, \tilde{\eta}_1,\dots,\tilde{\eta}_k) := c_{J,U}\Phi\big( A_{\displace(\bar{\eta})}(\tilde{\eta}) \bgiven A_{w_L}(\tilde{\eta}_L), A_{w_1}(\tilde{\eta}_1),\dots, A_{w_k}(\tilde{\eta}_k)\big),
	\end{gathered}
\end{equation}
where \(\bar{\eta} = \tilde{\eta}_L\concatenate \tilde{\eta}_1\concatenate\dots\concatenate\tilde{\eta}_k\).
By Claim~\ref{claim:chain_of_events}, we get the following decomposition:
\begin{multline}
	\label{eq:condition_proba_decomposition}
	\Phi(\calC = \tilde{\eta}_L\concatenate \tilde{\eta}_1\concatenate\dots\concatenate\tilde{\eta}_M \concatenate \tilde{\eta}_R)
	\\
	=
	q_L(\tilde{\eta}_L)q( \tilde{\eta}_1\given \tilde{\eta}_L)q(\tilde{\eta}_2\given \tilde{\eta}_L, \tilde{\eta}_1)\dots q(\tilde{\eta}_M\given \tilde{\eta}_L,\dots,\tilde{\eta}_{M-1}).
\end{multline}
Let us describe informally the second part of the proof.
So far the probability of a given cluster linking~$0$ to~$x$ is decomposed as a product of conditional kernels (which are \emph{not} probability kernels). The idea is now to represent this as a mixture of \emph{independent} (factorized) kernels defined on diamond-confined graphs.
Then we will suitably normalize them by a factor \(e^{t_0\cdot \displace(\gamma)}\) to obtain probability kernels.

\subsection{Proof of Theorem~\ref{thm:OZ_for_ATRC_infinite_vol} part II: Mixing of weights and renewal structure}
\label{subsec:mixing_weights_renewal}

We will use the same procedure as~\cite[Section 7]{AouOttVel24} with the tricks from~\cite[Appendix C]{OttVel18} to compensate for the lack of monotonicity in the conditional kernels. We only describe the needed inputs, and will refer to~\cite[Section 7]{AouOttVel24} once we arrive at a stage where the remaining arguments is a copy-pasting of~\cite[Section 7]{AouOttVel24}.

The first step is to prove that the conditional kernels of~\eqref{eq:condition_proba_decomposition} have good mixing properties. This is the content of the next lemma.
\begin{lemma}
	\label{lem:mixing_conditonal_chain_infinite_volume}
	With the notations and definitions of Subsection~\ref{subsec:CPs_pre_renewal}, there exist \(\rho>0, C\geq 0, c>0\), such that
	\begin{itemize}
		\item for every \(\tilde{\eta}\in \SetRootMarkBackCont^{\irr}\cup \SetRootMarkForwCont^{\irr}\), one has that
		\begin{equation}
			\label{eq:fin_ene_conditonal_chain_infinite_volume}
			\inf_{n\geq 0}\inf_{\tilde{\eta}_0\in \SetRootMarkBackCont^{\irr}} \inf_{\tilde{\eta}_1,\dots,\tilde{\eta}_n\in \SetRootDiaCont^{\irr}} q(\tilde{\eta}\given \tilde{\eta}_0,\dots , \tilde{\eta}_n) \geq \rho^{|\tilde{\eta}|};
		\end{equation}
		\item for any \(n,k,k'\geq 0\), any \(\tilde{\eta}_0,\tilde{\eta}_0'\in \SetRootMarkBackCont^{\irr}\), any \(\tilde{\eta}\in \SetRootMarkForwCont^{\irr}\cup \SetRootDiaCont^{\irr}\), and any \linebreak \(\tilde{\eta}_1,\dots,\tilde{\eta}_{k+n}, \tilde{\eta}_{1}',\dots,\tilde{\eta}_{k'}\in \SetRootDiaCont^{\irr}\)
		\begin{equation}
			\label{eq:mixing_conditonal_chain_infinite_volume}
			\Big|\frac{q\big(\tilde{\eta} \bgiven \tilde{\eta}_0, \tilde{\eta}_1, \dots, \tilde{\eta}_{k},\tilde{\eta}_{k+1},\dots, \tilde{\eta}_{k+n}\big)}{q\big(\tilde{\eta} \bgiven \tilde{\eta}_0', \tilde{\eta}_1', \dots, \tilde{\eta}_{k'}',\tilde{\eta}_{k+1},\dots, \tilde{\eta}_{k+n}\big)} - 1\Big|\leq Ce^{-cn}.
		\end{equation}
	\end{itemize}
\end{lemma}
\begin{proof}
	The first point is by definition of \(q\), see~\eqref{eq:def:dep_kernels}, and finite energy for \(\Phi\): the support of the event \( A_{\displace(\bar{\eta})}(\tilde{\eta})\) in~\eqref{eq:def:dep_kernels} is of size at most \( 4|\tilde{\eta}|\). Focus on the second point. We implicitly always work in a large finite volume, \(\Lambda_N = \{-N,\dots, N\}^2\), with \(0,1\) boundary conditions and take limits (everything being uniform over the large enough volume). Let \(\tilde{\eta}_{k'+i}' \equiv \tilde{\eta}_{k+i}\) for \(i=1,\dots n\). For \(l\geq 0\), let
	\begin{equation*}
		v_l = \displace(\tilde{\eta}_l),
		\quad 
		v_l'= \displace(\tilde{\eta}_l'),
		\quad
		w_l = \displace(\tilde{\eta}_0\concatenate \tilde{\eta}_1 \dots\concatenate \tilde{\eta}_{l}),
		\quad
		w_l'=\displace(\tilde{\eta}_0'\concatenate \tilde{\eta}_1' \dots\concatenate \tilde{\eta}_{l}').
	\end{equation*}Let also
	\begin{equation*}
		u_l = w_{k+l}-w_k = \displace(\tilde{\eta}_{k+1}\concatenate\dots\concatenate \tilde{\eta}_{k+l}) = w_{k'+l}'-w_{k'}'.
	\end{equation*}
	For \(v\in \Z^2\), introduce the translations of the events corresponding to a given chain of irreducible graphs:
	\begin{gather*}
		B_{L,v} = A_v(\tilde{\eta}_0)\cap\bigcap_{i=1}^{k} A_{v+w_{i-1}}(\tilde{\eta}_i),
		\quad
		B_{L,v}' = A_v(\tilde{\eta}_0')\cap\bigcap_{i=1}^{k'} A_{v+w_{i-1}'}(\tilde{\eta}_i'),
		\\
		B_{v} = \bigcap_{l=1}^{n} A_{v+u_l}(\tilde{\eta}_{k+l}),
		\quad
		B_{R,v} = A_{v+u_n}(\tilde{\eta}).
	\end{gather*}
	Now, that by definition of \(q\) (see~\eqref{eq:def:dep_kernels}),
	\begin{multline*}
		\frac{q\big(\tilde{\eta} \bgiven \tilde{\eta}_0, \tilde{\eta}_1, \dots, \tilde{\eta}_{k},\tilde{\eta}_{k+1},\dots, \tilde{\eta}_{k+n}\big)}{q\big(\tilde{\eta} \bgiven \tilde{\eta}_0', \tilde{\eta}_1', \dots, \tilde{\eta}_{k'}',\tilde{\eta}_{k+1},\dots, \tilde{\eta}_{k+n}\big)}
		\\=
		\frac{\Phi\big( A_{w_{k+n}}(\tilde{\eta}) \bgiven A_{0}(\tilde{\eta}_0), A_{w_0}(\tilde{\eta}_1),\dots, A_{w_{k-1}}(\tilde{\eta}_k)\big)}{\Phi\big( A_{w_{k'+n}'}(\tilde{\eta}) \bgiven A_{0}(\tilde{\eta}_0), A_{w_0'}(\tilde{\eta}_1),\dots, A_{w_{n+k'-1}'}(\tilde{\eta}_{n+k'}')\big)}
		=
		\frac{\Phi\big(B_{R,w_k} \bgiven B_{L,0}\cap B_{w_k}\big)}{\Phi\big(B_{R,w_k'} \bgiven B_{L,0}'\cap B_{w_{k'}'}\big)}.
	\end{multline*}
	
	Then, using translation invariance of \(\Phi\),
	\begin{multline}
		\label{eq:ratio_cond_proba_to_ratio_events}
		\frac{\Phi\big(B_{R,w_k} \bgiven B_{L,0}\cap B_{w_k}\big)}{\Phi\big(B_{R,w_k'} \bgiven B_{L,0}'\cap B_{w_{k'}'}\big)}
		=
		\frac{\Phi\big(B_{R,0} \bgiven B_{L,-w_k}\cap B_{0}\big)}{\Phi\big(B_{R,0} \bgiven B_{L,-w_{k'}'}'\cap B_{0}\big)}
		\\
		=
		\frac{\Phi\big(B_{R,0} \cap B_{L,-w_k} \bgiven B_{0}\big)}{\Phi\big(B_{R,0} \bgiven B_{0}\big) \Phi\big(B_{L,-w_k} \bgiven B_{0}\big)}\frac{\Phi\big(B_{R,0} \bgiven B_{0}\big)\Phi\big(B_{L,-w_{k'}'}'\bgiven  B_{0}\big)}{\Phi\big(B_{L,-w_{k'}'}'\cap B_{R,0} \bgiven  B_{0}\big)}
		.
	\end{multline}
	Now, \(B_{R,0}\) is supported on the edges with at least one endpoint in \(u_n  + \fcone\), denoted \(E_{\triangleleft}(u_n)\), whilst \(B_{L,-w_k}, B_{L,-w_{k'}'}'\) are supported on the edges with at least one endpoint in \(- \fcone\), denoted \(E_{\triangleright}\). Looking at~\eqref{eq:ratio_cond_proba_to_ratio_events},~\eqref{eq:mixing_conditonal_chain_infinite_volume} will follow from
	\begin{itemize}
		\item[(1)] a uniform upper and lower bounds on~\eqref{eq:ratio_cond_proba_to_ratio_events},
		\item[(2)] a suitable form of ratio mixing for \(P(\cdot ) := \Phi(\cdot \given B_0)\).
	\end{itemize}
	Indeed, (2) deals with the case \(n\) large enough, while small values of~\(n\) are covered by (1). We start with (1).
	\begin{claim}
		\label{claim:mixing_weights:claim_bounded}
		There exists \(c\geq 1\) such that for any \(n\geq 0\), any \(v\in \Z^2\), any \(\tilde{\eta}_1, \dots , \tilde{\eta}_n\in \SetRootDiaCont^{\irr}\)
		with \(\displace(\tilde{\eta}_1\concatenate \dots \concatenate \tilde{\eta}_n) = v\), any \(F_{\triangleright}\subset E_{\triangleright}\), \(F_{\triangleleft}\subset E_{\triangleleft}(v)\) finite, and any \(\alpha\in \{(0,0),(0,1),(1,1)\}^{F_{\triangleleft}}\),
		\begin{equation*}
			\frac{1}{c}\leq \inf_{\xi,\xi'} \frac{P(Y_{F_{\triangleleft}} = \alpha \given Y_{F_{\triangleright}} = \xi)}{P(Y_{F_{\triangleleft}} = \alpha \given  Y_{F_{\triangleright}} = \xi')}\leq \sup_{\xi,\xi'} \frac{P(Y_{F_{\triangleleft}} = \alpha \given Y_{F_{\triangleright}} = \xi)}{P(Y_{F_{\triangleleft}} = \alpha \given  Y_{F_{\triangleright}} = \xi')} \leq c,
		\end{equation*}where the sup/inf are over \(\xi,\xi'\) having positive probability, \(P = \Phi(\ \given B_0)\), and \(Y_{F} = (\omega_{\tau}|_F,\omega_{\tau\tau'}|_F)\).
	\end{claim}
	\begin{proof}
		By symmetry of the expression, it is sufficient to prove the upper bound. Let \(M\geq 1\) be a large enough integer and \(\Lambda_M = \{-M,\dots, M\}^2\), \(E_M = \bbE_{\Lambda_M}\). Let \(\xi,\xi'\) have positive \(P\)-probability.
		We now considering a decomposition
		\begin{gather*}
			B_0 = D_{B0}\cap D_{BR},
			\quad
			\{Y_{F_{\triangleleft}} = \alpha\} = D_{R0}\cap D_{R},
			\\
			\{ Y_{F_{\triangleright}} = \xi\} = D_L\cap D_{L0},
			\quad
			\{Y_{F_{\triangleright}} = \xi'\} = D_L'\cap D_{L0}',
		\end{gather*}
		where the events \(D_L\) and~\(D_L'\) are supported on \(E_{\triangleright}\setminus \bbE_{\Lambda_M}\), the events \(D_{RB}\) and~\(D_{R}\) are supported on \(E_{\triangleleft}\setminus \bbE_{\Lambda_M}\), and the events \(D_{L0}\), \(D_{L0}'\), \(D_{B0}\), and~\(D_{R0}\) are supported on \(\bbE_{\Lambda_M}\).
		
		Now, by finite energy applied to lower bound the probability of the events with supported in \(\bbE_{\Lambda_M}\) conditionally on the others, one has that for some \(c>0\),
		\begin{multline*}
			\frac{P(Y_{F_{\triangleleft}} = \alpha \given Y_{F_{\triangleright}} = \xi)}{P(Y_{F_{\triangleleft}} = \alpha \given  Y_{F_{\triangleright}} = \xi')}
			\\
			=
			\frac{\Phi(D_{R0}\cap D_R \cap D_{L0}\cap D_L\cap D_{B0}\cap D_{BR})\Phi(D_{L0}'\cap D_L'\cap D_{B0}\cap D_{BR})}{\Phi(D_{L0}\cap D_L\cap D_{B0}\cap D_{BR}) \Phi(D_{R0}\cap D_R \cap D_{L0}'\cap D_L'\cap D_{B0}\cap D_{BR})}
			\\
			\leq
			\frac{\Phi(D_R \cap D_L\cap D_{BR})\Phi( D_L'\cap D_{BR})}{\Phi( D_L\cap D_{BR}) \Phi( D_R \cap D_L'\cap D_{BR})}e^{c M^2}.
		\end{multline*}But, by exponential ratio weak mixing (Theorem~\ref{thm:ratio_weak_mixing_ATRC}), for \(M\) large enough depending only on the angular aperture of \(\fcone\),
		\begin{equation*}
			\frac{\Phi(D_L\cap D_R\cap D_{BR})}{\Phi(D_L'\cap D_R\cap D_{BR})} \leq \frac{\Phi(D_L)\Phi(D_R\cap D_{BR})}{\Phi(D_L')\Phi(D_R\cap D_{BR})} \frac{1 + e^{-cM}}{1 - e^{-cM}}\leq 2\frac{\Phi(D_L)}{\Phi(D_L')},
		\end{equation*}and
		\begin{equation*}
			\frac{\Phi(D_L'\cap D_{BR})}{\Phi(D_L\cap D_{BR})} \leq \frac{\Phi(D_L')\Phi(D_{BR})}{\Phi(D_L)\Phi(D_{BR})} \frac{1 + e^{-cM}}{1 - e^{-cM}}\leq 2\frac{\Phi(D_L')}{\Phi(D_L)}.
		\end{equation*}
		This concludes the proof of the claim.
	\end{proof}
	
	Divide the proof of ratio mixing into two claims. We start by proving a mixing bound for \(P\). We will use the following observation: let \(\gamma\) be the simple closed dual path surrounding the connected graph \(\tilde{\eta}_1\concatenate \dots\concatenate\tilde{\eta}_n\).
	Let \(*\gamma\) be the set of primal edges that are crossed by \(\gamma\), and let \(V_{\gamma}\) be the set of sites surrounded by \(\gamma\). Then,
	\begin{equation}
		\label{eq:cond_diamond_chain_to_atrc}
		P|_{\bbE\setminus \bbE_{V_{\gamma}}} = \atrc_{\bbE\setminus \bbE_{V_{\gamma}}}^{1,1}\big(\ \bgiven \omega_{\tau}|_{*\gamma} = 0 \big) = \atrc_{\bbE\setminus \bbE_{V_{\gamma}}}^{0,1}\big(\ \bgiven \omega_{\tau}|_{*\gamma} = 0 \big),
	\end{equation}where \(\omega_{\tau}|_{*\gamma}\) is the restriction of \(\omega_{\tau}\) to \(*\gamma\), see Fig.~\ref{Fig:Effective_bc_chain}.

	\begin{figure}
		\includegraphics[scale=1]{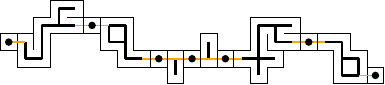}
		\hspace*{1cm}
		\includegraphics[scale=1]{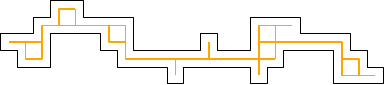}
		\caption{Black edges have state \((1,1)\), gold ones have state \((0,1)\), thin lines are dual to edges where \(\omega_{\tau} \) is \(0\). Left: the percolation event one conditions on. Right: the boundary conditions it effectively imposes.}
		\label{Fig:Effective_bc_chain}
	\end{figure}
	
	For \(\tilde{\eta}_1,\dots,\tilde{\eta}_n\), \(u_1,\dots, u_n\) as before, let \(E_{\epsilon,\times}(\tilde{\eta}_1,\dots,\tilde{\eta}_n)\) be the edges with both endpoints in \( \R^2 \setminus \big((u_{\lceil n/3\rceil}+\bcone_{t_0,\delta+\epsilon})\cup (u_{\lfloor 2n/3\rfloor}+\fcone_{t_0,\delta+\epsilon})\big) \), where~\(t_0,\delta\) are fixed in~\eqref{eq:t_0_fixing}.
	
	\begin{claim}
		\label{prf:ratio_mix_weights:str_mix_diam_meas}
		For any \(\epsilon>0\), there exist \(n_0\geq 1\), \(C\geq 0,c>0\) such that, for any \(n\geq n_0\), any \(\tilde{\eta}_1,\dots,\tilde{\eta}_n\in \SetRootDiaCont^{\irr}\), any finite sets of edges \(F_{\triangleright}\subset E_{\triangleright}\), \( F_{\triangleleft}\subset E_{\triangleleft}\), \(F_{\times}\subset E_{\epsilon,\times}(\tilde{\eta}_1,\dots,\tilde{\eta}_n)\), any \(F_1\in \{F_{\triangleright}, F_{\triangleleft}, F_{\times}\}\), and any \(\xi,\xi' \in \{(0,0),(0,1),(1,1)\}^{F_2 }\), where \(F_2=(F_{\triangleright}\cup F_{\triangleleft}\cup F_{\times})\setminus F\),
		\begin{equation*}
			\tvd\big(P(X_{F_1}\in \cdot \given X_{F_2} = \xi), P(X_{F_1}\in \cdot \given X_{F_2} = \xi')\big) \leq Ce^{-cn},
		\end{equation*}where \(P,X\) are defined as in Claim~\ref{claim:mixing_weights:claim_bounded}.
	\end{claim}
	\begin{proof}
		The idea is simple but the details of the percolation estimate are a bit tedious to write down, so we only sketch the percolation argument. The formal implementation is similar to~\cite[Lemma 3.2]{OttVel18}. Looking at~\eqref{eq:cond_diamond_chain_to_atrc}, \(P|_{\bbE\setminus \bbE_{V_{\gamma}}}\) has the \(l\)-path decoupling property of Section~\ref{subsec:strMix:setup}. One can therefore use Theorem~\ref{thm:strong_mixing_atrc} to obtain the claim. Indeed, one can bound the wanted total variation distance with a connection probability in a Bernoulli block percolation model. The model one compares with is at worse of the type depicted on Fig.~\ref{Fig:Mixing_infVol_weight_perco_chain}.
		\begin{figure}
			\includegraphics[scale=0.7]{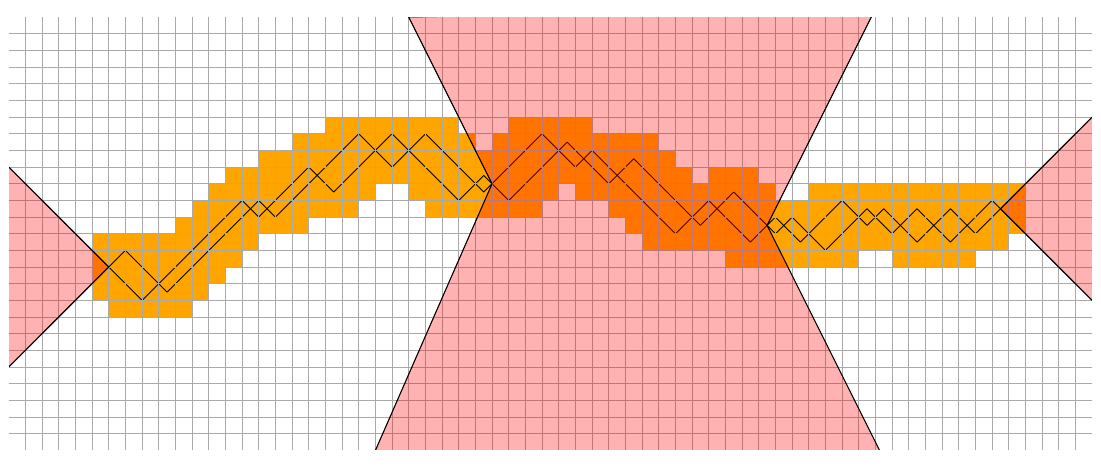}
			\caption{The grey grid represent the blocks of the block-percolation of Theorem~\ref{thm:strong_mixing_atrc}. The orange squares have a strictly smaller than one probability to be open. The white squares have a probability as small as wanted to be open. The red regions represent \(E_{\triangleright}, E_{\epsilon,\times}, E_{\triangleleft}\).}
			\label{Fig:Mixing_infVol_weight_perco_chain}
		\end{figure}
		The \(n/3\) cone-points between the different regions each provides an opportunity to disconnect \(F_1\) from \(F_2\): each induces a bottleneck in the orange chain illustrated on Fig.~\ref{Fig:Mixing_infVol_weight_perco}, and each bottleneck gives a positive probability to disconnect \(F_1\) from \(F_2\). The choice of the larger angular opening for the central region (in Figure~\ref{Fig:Mixing_infVol_weight_perco_chain}), \(\delta +\epsilon\) instead of \(\delta\), is a simple way to get rid of connections between the left/right red regions of Figure~\ref{Fig:Mixing_infVol_weight_perco_chain} and the central one occurring very high/low in the \(\rme_2\)-direction.
		\begin{figure}
			\includegraphics[scale=0.5]{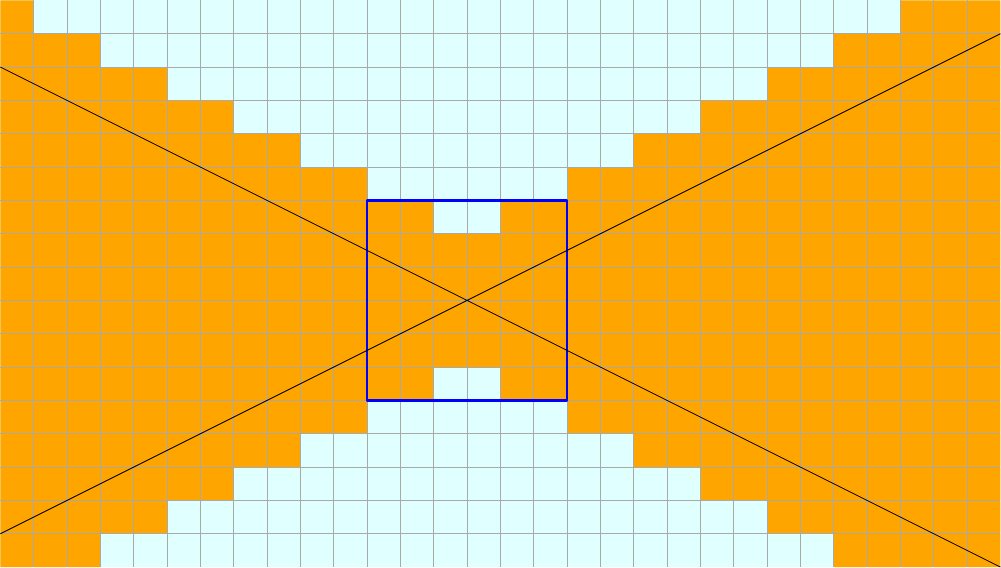}
			\caption{The grey grid represent the blocks of the block-percolation of Theorem~\ref{thm:strong_mixing_atrc}. The orange squares have a strictly smaller than one probability to be open. The light blue squares have a probability as small as wanted to be open. By exponential decay in the blue blocks, a positive density of bottleneck have the property that the two orange sides of the blue square are not connected in the blue blocks. Then, with positive probability, the blocks in the square are closed, cutting the orange chain at that bottleneck.}
			\label{Fig:Mixing_infVol_weight_perco}
		\end{figure}
	\end{proof}
	We then turn this into ratio mixing using Appendix~\ref{lem:app:mixing_to_ratioMixing}.
	\begin{claim}
		\label{prf:ratio_mix_weights:ratio_mix_cone_meas}
		There exists \(C\geq 0 ,c>0, n_0\geq 1\) such that for any \(n\geq n_0\), any \(v\in \Z^2\), any \(\tilde{\eta}_1, \dots , \tilde{\eta}_n\in \SetRootDiaCont^{\irr}\)
		with \(\displace(\tilde{\eta}_1\concatenate \dots \concatenate \tilde{\eta}_n) = v\), any \(F_{\triangleright}\subset E_{\triangleright}\), \(F_{\triangleleft}\subset E_{\triangleleft}(v)\) finite, and any \(\alpha\in \{(0,0),(0,1),(1,1)\}^{F_{\triangleleft}}\),
		\begin{equation*}
			\sup_{\xi,\xi'} \Big|\frac{P(X_{F_{\triangleleft}} = \alpha \given X_{F_{\triangleright}} = \xi)}{P(X_{F_{\triangleleft}} = \alpha \given  X_{F_{\triangleright}} = \xi')}-1\Big|\leq Ce^{-cn},
		\end{equation*}where the sup is over \(\xi,\xi'\) having positive probability, \(P,X\) are defined as in Claims~\ref{claim:mixing_weights:claim_bounded}, and~\ref{prf:ratio_mix_weights:str_mix_diam_meas}.
	\end{claim}
	\begin{proof}
		The proof will be an application of Lemma~\ref{lem:app:mixing_to_ratioMixing}. Recall~\eqref{eq:cond_diamond_chain_to_atrc}.
		
		Let \(1-\delta>\epsilon>0\), and let \(E_{\times} = E_{\epsilon,\times}(\tilde{\eta}_1,\dots,\tilde{\eta}_n)\). We apply Lemma~\ref{lem:app:mixing_to_ratioMixing} with
		\begin{itemize}
			\item \(\Omega_1 = \{(0,0),(0,1),(1,1)\}^{F_{\triangleleft}}\), \(\Omega_2 = \{(0,0),(0,1),(1,1)\}^{E_{\times}}\),
			\item \(\mu = P(\ \given X_{F_{\triangleright}} = \xi)|_{F_{\triangleleft}\sqcup E_{\times}}\), \(\nu = P(\ \given X_{F_{\triangleright}} = \xi')|_{F_{\triangleleft}\sqcup E_{\times}}\),
			\item \(D\subset \Omega_2\) is the event that there is an open path in \(\omega_{\tau\tau'}\) from the top boundary of \(\Lambda_N\) (the implicit finite volume we work in) to its bottom boundary staying in \(E_{\times}\), and that there is a dual path of open edges in \(\omega_{\tau}^*\) from the top boundary of \(\Lambda_N\) to its bottom boundary using only edges dual to edges in \(E_{\times}\).
		\end{itemize}
		The hypotheses of Lemma~\ref{lem:app:mixing_to_ratioMixing} hold with \(\epsilon = e^{-cn}\) when \(n\) is large enough by Claim~\ref{prf:ratio_mix_weights:str_mix_diam_meas}, and the uniform exponential decay of \(\omega_{\tau\tau'}^*,\omega_{\tau}\) (Theorem~\ref{thm:Ale04_ATRC}).
	\end{proof}
	This concludes the proof of~\eqref{eq:mixing_conditonal_chain_infinite_volume}.
\end{proof}

Lemma\ref{lem:mixing_conditonal_chain_infinite_volume} is the needed input to put us in the framework of~\cite[Section 7]{AouOttVel24}. Let us describe the needed modifications which all take place in the beginning of the argument. The main difference is that our conditional weights \(q\) are not monotonic. We thus need to define the \(p_k\)'s used there in a slightly different way.

For \(k> l\geq 1\), \(\tilde{\eta}_0\in \SetRootMarkBackCont^{\irr}\), \(\tilde{\eta}_1,\dots,\tilde{\eta}_{k-1} \in \SetRootDiaCont^{\irr}\), and \(\tilde{\eta}\in \SetRootDiaCont^{\irr}\cup\SetRootMarkForwCont^{\irr}\), set
\begin{gather*}
	a_{0}(\tilde{\eta}\given \tilde{\eta}_0,\dots, \tilde{\eta}_{k-1})\equiv a_0(\tilde{\eta}) = \inf_{n\geq 0}\inf_{\zeta_0\in \SetRootMarkBackCont^{\irr}}\inf_{\zeta_1,\dots,\zeta_n\in \SetRootDiaCont^{\irr}} q\big(\tilde{\eta} \bgiven \zeta_0,\zeta_1, \dots , \zeta_n\big),
	\\
	a_k(\tilde{\eta}\given \tilde{\eta}_0,\dots, \tilde{\eta}_{k-1}) = q(\tilde{\eta}\given \tilde{\eta}_0,\dots, \tilde{\eta}_{k-1}),
\end{gather*}
and
\begin{multline*}
	a_{l}(\tilde{\eta} \given \tilde{\eta}_0,\dots, \tilde{\eta}_{k-1}) \equiv a_{l}(\tilde{\eta} \given \tilde{\eta}_{k-l},\dots, \tilde{\eta}_{k-1}) \\
	= \inf_{n\geq 0}\inf_{\zeta_0\in \SetRootMarkBackCont^{\irr}}\inf_{\zeta_1,\dots,\zeta_n\in \SetRootDiaCont^{\irr}} q\big(\tilde{\eta} \bgiven \zeta_0, \dots , \zeta_n , \tilde{\eta}_{k-l},\dots, \tilde{\eta}_{k-1}\big)
\end{multline*}

In worlds: \(a_k\) records the minimal mass given by a conditional probability to an irreducible graph for a fixed frozen ``recent'' past, and any ``less recent'' possible past. Introduce then the mass increments (\(p_k\) has the same definition domain as \(a_k\)) 
\begin{equation*}
	p_{0} = a_0,\quad p_{k} = a_{k}-a_{k-1}.
\end{equation*}
Note that by definition of the \(a_k\)'s, the \(p_k\)'s are always non-negative. This property is what compensates the lack of property P6: monotonicity of \(q(\tilde{\eta}\given \tilde{\eta}_0,\dots,\tilde{\eta}_{k-1})\). One can now duplicate~\cite[Section 7]{AouOttVel24} to conclude the proof of Theorem~\ref{thm:OZ_for_ATRC_infinite_vol} with the following adaptations: use the \(p_k\)'s defined above in place of the ones defined in~\cite{AouOttVel24}, and use Lemma~\ref{lem:mixing_conditonal_chain_infinite_volume} as a replacement for~\cite[properties P5, P7]{AouOttVel24}. Indeed, these are the only properties used in the argument.

\section{Invariance principle for the modified ATRC cluster}
\label{sec:invariance_principle}

\subsection{Notations and main result of the section}

Recall~\(\matrc_{n,m}(\ \given v_L\leftrightarrow v_R)\) defined in Section~\ref{sec:coupling:interfaces}.
As in Section~\ref{sec:ATRC:RW_infinite_volume}, we will describe the geometry of the cluster of \(v_L\) under~\(\matrc_{n,m}(\ \given v_L\leftrightarrow v_R)\) using a coupling with a random walk bridge in such a way that the cluster is included in diamonds with endpoints at the random walk steps.
In this section, we will need to track certain constants that will play the role of minimal/maximal scale at which some estimates hold. Constants like \(C,c,c'\) will be defined only locally, whilst numbered constants will have fixed value re-used in several places. All constants can depend on \(J,U\) but never on \(n\).

To be able to re-use our result in future work~\cite{DobGlaOtt26}, we state a slightly more general result than what we need here. The reader can think of \(\Phi_n\) as \(\matrc_{n/2,m_n}\). Moreover, for notational convenience, we translated everything so that \(v_L = 0\). In the setup below, one can think of \(v_n\) as \(v_R-v_L\). Let \(\FVBox_{n}\) be a sequence of rectangles of the form:
\begin{equation*}
	\FVBox_{n} = \{0,\dots, n\}\times\{-m_n',\dots,m_n\}
\end{equation*}where \(m_n,m_n'\) are non-decreasing sequences whose properties are part of the main theorem. Also introduce
\begin{align*}
	\bar{\FVBox}_{n} &= \{-1,\dots, n+1\}\times\{-m_n'-1,\dots,m_n+1\},
	\\
	\FVBox_{n}^{'} &= \{1,\dots, n-1\}\times\{1-m_n',\dots,m_n - 1\}.
\end{align*}
Finally, introduce the family of rectangles: for \(k,l\geq 0\), \(n\geq 2k\),
\begin{equation}
	\label{eq:def:rectangle}
	\rectangle_{k,l}^n = \{k,\dots,n-k\}\times\{-l,\dots,l\}.
\end{equation}
We will be considering families of measures \(\Phi_{n}\) with the following properties:
\begin{enumerate}[label=(\Roman*), ref=\theenumii (\Roman*)]
	\item \label{hyp_fin_vol_connection:support} \(\Phi_{n}\) has support contained in \((\{0,1\}^2)^{\rmE_{n}}\) with \(\bbE_{\FVBox_{n}'} \subset \rmE_{n}\subset \bbE_{\bar{\FVBox}_{n}}\).
	\item \label{hyp_fin_vol_connection:FKG} \(\Phi_{n}\) is FKG-lattice.
	\item \label{hyp_fin_vol_connection:stochastic_sandwich} 
	\(\Phi_{n}(\cdot \given 1_{\rmE_{n}\setminus \bbE_{\FVBox_{n}'}})\) restricted to \(\bbE_{\FVBox_{n}'}\) is equal to \(\atrc_{\bbE_{\FVBox_{n}'}}^{1,1}\), and \(\Phi_{n}(\cdot \given 0_{\rmE_{n}\setminus \bbE_{\FVBox_{n}'}})\) restricted to \(\bbE_{\FVBox_{n}'}\) is equal to \(\atrc_{\bbE_{\FVBox_{n}'}}^{0,0}\), where \(1_F\) and~\(0_F\) are the events that all edges in \(F\) have state \((1,1)\) and~\((0,0)\) respectively. 
	In particular, as \(\Phi_n\) is FKG-lattice,
	\begin{equation}
		\label{eq:stochastic_sandwich}
		\atrc_{\bbE_{\FVBox_{n}'}}^{0,0} \preccurlyeq \Phi_n|_{\bbE_{\FVBox_{n}'}} \preccurlyeq \atrc_{\bbE_{\FVBox_{n}'}}^{1,1}
	\end{equation}
	\item \label{hyp_fin_vol_connection:finite_energy} \(\Phi_{n}\) has the following finite energy property: there exists \(\feCst\in (0,1)\), independent of \(n\) such that, for any \(e\in \rmE_{n}\),
	\begin{equation}
	\label{eq:fin_vol_FE}
	\begin{gathered}
		\Phi_{n}((\omega_{\tau}(e),\omega_{\tau\tau'}(e)) = (1,1) \given \calF_{\rmE_{n}\setminus e}) \geq \feCst \ \Phi_n\text{-almost surely},
		\\
		\Phi_{n}((\omega_{\tau}(e),\omega_{\tau\tau'}(e)) = (0,0) \given \calF_{\rmE_{n}\setminus e}) \geq \feCst \ \Phi_n\text{-almost surely},
	\end{gathered}
	\end{equation}where \((\omega_{\tau},\omega_{\tau\tau'})\sim \Phi_n\), and \(\calF_F\) is the sigma-algebra generated by \((\omega_{\tau}(e),\omega_{\tau\tau'}(e))_{e\in F}\).
	\item \label{hyp_fin_vol_connection:edge_state_swap} \(\Phi_{n}\) has the double-edge-state-swap property: there is \(\stateSwapCst>0\) such that for any \(n\), any \(e=\{x,y\}, f=\{i,j\}\in \bbE_{\calR_n'}\), and any configurations \(a,b\in \{0,1\}^{\rmE_n}\) with \(\Phi_{n}(\{(a,b)\})>0\) such that \(i\in\calC_x(a_{ef})\subset \calR_n'\setminus \partialin\calR_n'\), \(j,y\notin \calC_x(a_{ef})\),
	\begin{equation}
		\label{eq:Phi_n_edge_state_swap}
		\frac{\Phi_n(\{(a^{ef},b)\})}{\Phi_n(\{(a_{ef},b)\})} = \stateSwapCst,
	\end{equation}where \(a_{ef}(e') = a^{ef}(e') = a(e')\) if \(e'\notin \{e,f\}\), and \(a_{ef}(e) = a_{ef}(f) =0\), \(a^{ef}(e) = a^{ef}(f) = 1\). 
\end{enumerate}
Note that by Hypotheses~\ref{hyp_fin_vol_connection:stochastic_sandwich}, \(\Phi_n\xrightarrow{n\to\infty} \atrc_{J,U}\) the unique infinite volume ATRC measure as soon as \(m_n,m_n'\xrightarrow{n\to\infty} \infty\), as we work with \(J,U\) on the self dual curve, where there is a unique ATRC measure. We therefore denote
\begin{equation*}
	\Phi = \atrc_{J,U}.
\end{equation*}
Let finally \(v_n\) be a sequence of points such that:
\begin{equation}
	\label{eq:def:FV_v_n}
	\norm{v_n-(n,0)}_{\infty} \leq 2,
	\quad
	\Phi_n(0\xleftrightarrow{\omega_{\tau}} v_n )>0.
\end{equation}Let \((\omega_{\tau},\omega_{\tau\tau'})\sim \Phi_n\). We will be studying the cluster of \(0\) in \(\omega_{\tau}\) under \(\Phi_n(\cdot \given 0 \xleftrightarrow{\omega_{\tau}} v_n)\), so we will use the shorter notations
\begin{equation*}
	\calC_x = \text{ connected component of } x \text{ in } \omega_{\tau},
	\quad
	\calC \equiv \calC_0.
\end{equation*}
We will use the probabilistic description of long clusters from Theorem~\ref{thm:OZ_for_ATRC_infinite_vol} with a suitable choice of cones. First, note that by reflection symmetry the (unique by Theorem~\ref{thm:OZ_for_ATRC_infinite_vol}) \(t\in \partial \calW\) dual to \(\rme_1\) is proportional to \(\rme_1\). Moreover, it satisfies
\begin{equation}
	t\cdot \rme_1 = \nu(\rme_1) \equiv \nu_1.
\end{equation}
So the only \(t\in \partial \calW\) dual to \(\rme_1\) is \(\nu_1\rme_1\). Then, by strict convexity and symmetry, the cones \(\fcone_{\nu_1\rme_1,\delta}\) are invariant under reflection through \(\{x:\ x_2=0\}\), and have an angular aperture continuous in \(\delta\), strictly increasing, which converges to \(0\) as \(\delta\to 0\) and to \(\pi\) as \(\delta\to 1\). 
For~\(\theta\in (0,\pi/2)\), define the symmetric cones with angular aperture \(2\theta\):
\begin{align*}
	\fcone_{\theta} = \{x\in \R^2:\ \tan(\theta) x_1 \geq |x_2| \}	
	\quad\text{and}\quad
	\bcone_{\theta} = -\fcone_{\theta}.
\end{align*}
Then, for any \(\theta\in (0,\pi/2)\), there is a unique \(\delta\in (0,1)\) such that \(\fcone_{\theta} = \fcone_{\nu_1\rme_1,\delta}\).
As in Section~\ref{sec:ATRC:RW_infinite_volume}, we say that \(x\) is a \(\theta\)-cone-point of \(A\) if \(A\subset x+(\fcone_{\theta}\cup \bcone_{\theta})\), and write
\begin{align*}
	\diam_{\theta}(u,v) &= (u+\fcone_{\theta}) \cap (v+\bcone_{\theta}),
	\\
	\CPts_{\theta}(A) &= \{x\in A:\ x\text{ is a } \theta\text{-cone-point of } A\}.
\end{align*}
We will mostly use the value \(\theta = \pi/4 \), but we will need the value \(\pi/8\) for some intermediate results. We therefore set \(\theta = \pi/4\) by default, and will explicitly specify the use of different values.
\begin{equation*}
	\text{When omitted from the notation, }\theta\text{ is set to be }\pi/4.
\end{equation*}
Denote \(\SetRootMarkBackCont,\SetRootMarkForwCont,\SetRootDiaCont \) the sets of backward-confined, forward-confined, diamond-confined graphs (see Section~\ref{subsec:OZ_def}) for \(\theta = \pi/4\). Let then \(p_L,p_R,p\) be the measures given by Theorem~\ref{thm:OZ_for_ATRC_infinite_vol} for \((\delta,\nu_1\rme_1)\) with \(\fcone = \fcone_{\nu_1\rme_1,\delta}\). Denote \(\OZmeas = p^{\N}\) the random chain probability measure, \(\OZWalkmeas = (p\circ \displace^{-1})^{\otimes\N}\) the random walk increment probability measure. Denote \(\OZDecompmeas\) the positive measure (\emph{not} probability measure) on pairs length + sequences of graphs given by
\begin{equation}
	\label{eq:def:OZDecompmeas}
	\OZDecompmeas(M;\gamma_L,\gamma_1,\dots,\gamma_M,\gamma_R) = \rmC p_L(\gamma_L)p_R(\gamma_R)\prod_{i=1}^M p(\gamma_i),
\end{equation}
where \(\rmC\) is the constant given by Theorem~\ref{thm:OZ_for_ATRC_infinite_vol}. Integrating with respect to~$\OZDecompmeas$, we get
\begin{multline}
	\label{eq:def:measure_sequence}
	\int f(M;\gamma_L,\gamma_1,\dots,\gamma_M,\gamma_R) \, d\OZDecompmeas
	\\=
	\rmC\sum_{M\geq 0}\sum_{\gamma_L\in \SetRootMarkBackCont}\sum_{\gamma_R\in \SetRootMarkForwCont}\sum_{\gamma_1,\dots,\gamma_M\in \SetRootDiaCont} f(M;\gamma_L,\dots,\gamma_R)p_L(\gamma_L)p_R(\gamma_R)\prod_{i=1}^M p(\gamma_i).
\end{multline}
In particular, the above is just a finite sum whenever~\(f\) is supported on sequences of graphs with a bounded displacement of \(\gamma_L\concatenate\gamma_1\concatenate\dots\concatenate\gamma_M\concatenate\gamma_R\).
Let \((\gamma_i)_{i\geq 1} \sim \OZmeas\). For any~$v\in\Z^2$, define
\begin{equation}
	\label{eq:def:hitting_times}
	T_v = \inf\{k\geq 1:\ \sum_{i=1}^k \displace(\gamma_i)= v\},
\end{equation}where the inf of an empty set is set to be \(+\infty\) by convention. Let \(\OZmeas_v = \OZmeas(\cdot \given T_v<\infty)\), and \(\OZWalkmeas_v = \OZWalkmeas(\cdot \given T_v<\infty)\). 
Note that \(T_v\) depends only on the sequence of displacements, it is thus also well defined under \(\OZWalkmeas\).
We will frequently use the following facts about \(\OZWalkmeas\) without further details.
\begin{enumerate}
	\item First, there exist \(c_-,c_+>0\) such that for any \(n\geq 1\), \(v=(v_1,v_2)\) with \(n\leq v_1\leq 2n\), \(|v_2|\leq \sqrt{n}\),
	\begin{equation*}
		\frac{c_-}{\sqrt{n}} \leq \OZmeas(T_v<\infty) \leq \frac{c_+}{\sqrt{n}}.
	\end{equation*}This follows from the same computation as the OZ asymptotics of Theorem~\ref{thm:oz-atrc}, see for example~\cite[Section 8.1]{AouOttVel24}.
	\item Second, for any \(\epsilon>0\), there exist \(c>0,L_0\) such that for any \(k\geq 1\), \(L\geq L_0\),
	\begin{equation*}
		\OZmeas\Big(\sum_{i=k}^{k+L-1} \displace(\gamma_i)\cdot \rme_1 \notin [(\mu_1-\epsilon)L, (\mu_1+\epsilon)L]\Big) \leq e^{-cL},
	\end{equation*}where \(\mu_1 = \sum_{\gamma\in \SetRootDiaCont} p(\gamma) \displace(\gamma)\cdot \rme_1 >0\). This is a standard large deviation estimate for sums of independent random variables with exponential tails.
	\item Finally, a direct consequence of the two points above is that there exist \(n_0\geq 1\), \(c_0,c,\rho>0\), such that for any \(n\geq n_0\), \(v=(n,v_2)\) with \(|v_2|\leq \sqrt{n}\), and any \(n\geq k_n \geq c_0\ln(n)\),
	\begin{gather*}
		\OZmeas_v\big(\exists i\in \{1,\dots, n-k_n\}:\, |\{j:\, i\leq \displace(\gamma_1\concatenate\dots \concatenate \gamma_j)\cdot \rme_1 \leq i + k_n \}| < \rho k_n\big) \leq e^{-ck_n},
		\\
		\OZmeas_v\big(\exists i\in \{1,\dots, T_v\}:\, |\displace(\gamma_i)| \geq k_n\big) \leq e^{-ck_n}.
	\end{gather*}Note that the first estimate is stronger than the second, but we state the second explicitly for clarity.
\end{enumerate}

The following ``representation as a mixture of random walk bridges'' is the main result of this section.

\begin{theorem}
	\label{thm:main_Inv_principle_coupling_with_RW}
	Let \(0<J<U\) satisfy \(\sinh 2J=e^{-2U}\). Suppose \((\Phi_n)_{n\geq 1}\) is a family of measures satisfying Hypotheses \ref{hyp_fin_vol_connection:support}, \ref{hyp_fin_vol_connection:FKG}, \ref{hyp_fin_vol_connection:stochastic_sandwich}, \ref{hyp_fin_vol_connection:finite_energy}, and \ref{hyp_fin_vol_connection:edge_state_swap}. 
	Then, there exist \(c>0, C_0\geq 0, c_0\geq 0,c_1\geq 0\) such that the following holds: for any increasing sequence \(\frac{\sqrt{n}}{c_0} \geq \scale_n\geq c_0\ln(n)\), there is \(n_0\geq 1\) such that for any \(n\geq n_0\), \(m_n,m_n'\geq C_0 n\), there exists a probability measure \(\MixMeas_{n}\) supported on pairs~\((\zeta_L,\zeta_R)\in \SetRootMarkBackCont\times \SetRootMarkForwCont\) with displacement sup-norm at most \(c_1\scale_n^2\) such that, for \(v(\zeta_L,\zeta_R) := v_n-\displace(\zeta_R)-\displace(\zeta_L)\) and any \(f\) function of \(\calC\),
	\begin{multline*}
		\Big|\sum_{\zeta_L,\zeta_R} \MixMeas_{n}(\zeta_L,\zeta_R) \OZmeas_{v(\zeta_L,\zeta_R)} \big[ f(\zeta_L\concatenate \gamma_1\concatenate\dots\concatenate \gamma_M\concatenate \zeta_R)\big]
		\\
		-
		\Phi_n\big[ f(\calC) \bgiven v_n \in \calC\big]\Big|
		\leq
		\norm{f}_{\infty}e^{-c\scale_n}.
	\end{multline*}
\end{theorem}

Note that this is a total variation distance bound, so by existence of maximal coupling, this theorem gives the existence of a coupling between the law of \(\calC\) under \(\Phi_n\big(~\cdot~\bgiven v_n \in \calC\big)\) and a product law on chains of diamond confined graphs, such that \(\calC\) coincides with the chain with a probability going to \(1\) as \(n\to\infty\). The law of the sequence of cone-points of this chain is then a random walk bridge, and the distance (e.g. Hausdorff) between the linear interpolation of this bridge and the actual cluster is at most a constant times the norm of the largest step.
By a general result of Kovchegov~\cite{Kov04}, this linear interpolation converges to a Brownian Bridge under diffusive scaling and the maximal step size is \(C\ln^2(n)\) with probability going to \(1\) as \(n\to\infty\) (see Lemma~\ref{lem:InvPrinc_mATRC}). This therefore implies an invariance principle for the upper and lower envelopes of \(\calC\).

\subsection{Implications for \(\matrc\)}
\label{subsec:InvPrinc:matrc_statement}

In this paper, we are ultimately interested in a particular application of Theorem~\ref{thm:main_Inv_principle_coupling_with_RW}. We apply it to
\begin{itemize}
	\item \(m_n=m_n'= Cn\) for \(C\) a large enough integer,
	\item \(\Phi_{2n}\) is the translate of \(\matrc_{n,m_n}\) by \(v_L\),
	\item \(\rmE_{2n} = \bar{E}_{n,m_n} -v_L\) is the translation of the augmented rectangular domain of Section~\ref{sec:notations}, see also Figure~\ref{Fig:rmE_n_for_matrc},
	\item \(v_{2n} = v_R-v_L\).
\end{itemize}

\begin{figure}
	\includegraphics[scale=0.7]{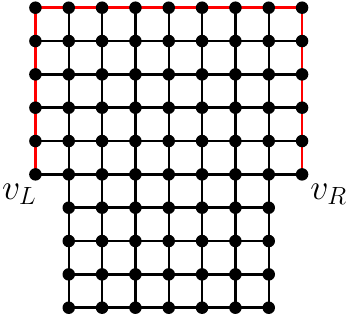}
	\caption{\(\rmE_n\) in the application to \(\matrc\).}
	\label{Fig:rmE_n_for_matrc}
\end{figure}

Hypotheses~\ref{hyp_fin_vol_connection:FKG} and~\ref{hyp_fin_vol_connection:stochastic_sandwich} follow from Lemma~\ref{lem:atrc_strong_fkg} and the explicit expression for the density of \(\matrc_{n,m_n}\), see Definition~\ref{def:mATRC}. Hypotheses~\ref{hyp_fin_vol_connection:finite_energy} follows from Lemma~\ref{lem:fe_matrc_+}. Hypotheses~\ref{hyp_fin_vol_connection:support} and~\ref{hyp_fin_vol_connection:edge_state_swap} follow from the explicit expression for the density of \(\matrc_{n,m_n}\), see again Definition~\ref{def:mATRC}.

\subsection{Good clusters and overlook of the proof}
\label{subsec:InvPrinc:good_cluster}

As the proof of Theorem~\ref{thm:main_Inv_principle_coupling_with_RW} is fairly long and spans over the rest of the section, we start by introducing the main ideas of the proof. The first step is to prove that under \(\Phi_n\big(~\cdot~\bgiven v_n \in \calC\big)\), the cluster~\(\calC\) has a nice structure with very high probability. Let us describe what we mean by nice. Introduce
\begin{equation*}
	\slab_{l,r} = [l,r]\times \R,
	\quad
	\slabCP_{\theta,l,r}(C) = \slab_{l,r} \cap \CPts_{\theta}(C),
	\quad
	\slabCP_{l,r} \equiv \slabCP_{\pi/4,l,r}(C).
\end{equation*}
We introduce the set of \emph{good clusters}: \(C\in \goodCl_n(\scale_n)\) if all of the following are fulfilled for some small enough~$\rho>0$ depending only on \(U,J\):
\begin{enumerate}
	\item \(v_n\in C\);
	\item \(|\slabCP_{2,\scale_n-1}(C)| \geq \rho \scale_n\), \(|\slabCP_{n-\scale_n+1,n-2}(C)|\geq \rho \scale_n\);
	\item \(|\slabCP_{\scale_n+1,2\scale_n-1}(C)| > 0\), \(|\slabCP_{n-2\scale_n+1,n-\scale_n-1}(C)| >0\);
	\item \(|\slabCP_{2\scale_n+1,3\scale_n-1}(C)|  \geq \rho \scale_n\), \(|\slabCP_{n-3\scale_n+1,n-2\scale_n-1}(C)| \geq \rho \scale_n\);
	\item \(|\slabCP_{3\scale_n+1,4\scale_n-1}(C)| > 0\), \(|\slabCP_{n-4\scale_n+1,n-3\scale_n-1}(C)| > 0\).
\end{enumerate}
The value of~$\rho$ is given by Lemmas~\ref{lem:InvPrinc_crossing_clusters_are_nice}, and~\ref{lem:InvPrinc_cluster_is_nice}.
We introduce the following notation (see Figure~\ref{Fig:InvPrinc_goodCluster_Decomp}):
\begin{itemize}
	\item if \(\slabCP_{3\scale_n+1,4\scale_n-1}(\calC)\neq \varnothing\), then denote its leftmost point by~\(W_L\), and otherwise set~\(W_L = \dagger\);
	\item if \(\slabCP_{n-4\scale_n+1,n-3\scale_n-1}(\calC) \neq \varnothing\), then denote its rightmost point by~\(W_R\), and otherwise set~\(W_R = \dagger\);
	\item \(\calK_L := \calC \cap (W_L+\bcone)\) if \(W_L\neq \dagger\), and \(\calK_L = \dagger\) otherwise;
	\item \(\calK_R := -W_R + \calC \cap (W_R+\bcone)\) if \(W_R\neq \dagger\), and \(\calK_R = \dagger\) otherwise;
	\item \(\calK := -W_L + \calC \cap \diam(W_L,W_R)\) if \(W_L\neq \dagger\) and \(W_R\neq \dagger\), and \(\calK=\dagger\) otherwise.
\end{itemize}
In particular, on the event \(\{W_L\neq \dagger,W_R\neq \dagger\}\), \(\calC = \calK_L\concatenate \calK\concatenate \calK_R\).

The first step of the proof of Theorem~\ref{thm:main_Inv_principle_coupling_with_RW} is to show that \(\calC\in \goodCl_n(\scale_n)\) with high probability; this is the content of Section~\ref{subsec:InvPrinc:phi_n_goodCl}.
Then, on the event \(\calC\in \goodCl_n(\scale_n)\), we will do a sampling as follows: first sample \(\calK_L,\calK_R\) compatible with \(\goodCl_n(\scale_n)\). Then, sample \(\calK\) conditionally on \(\calK_L,\calK_R\). Mixing properties of \(\Phi_n\), proved in Section~\ref{subsec:InvPrinc:phi_n_properties}, will then allow us to show that for these realizations of \(\calK_L,\calK_R\), the density of \(\calK\) conditional on \(\calK_L,\calK_R\) is close to the density of an infinite volume cluster connecting \(0\) to \(v_n\) under \(\atrc\) conditional on \(\calK_L,\calK_R\). This is done in Section~\ref{subsec:InvPrinc:density_swap}. But this last object is naturally coupled to a random walk bridge by Theorem~\ref{thm:OZ_for_ATRC_infinite_vol}, which will allow us to conclude the proof of Theorem~\ref{thm:main_Inv_principle_coupling_with_RW} in Section~\ref{subsec:InvPrinc:prf_coupling_RWB}. Section~\ref{subsec:InvPrinc:BBconv} then recalls the setup needed to go from Theorem~\ref{thm:main_Inv_principle_coupling_with_RW} to the invariance principle.

\begin{figure}
	\includegraphics[scale=0.73]{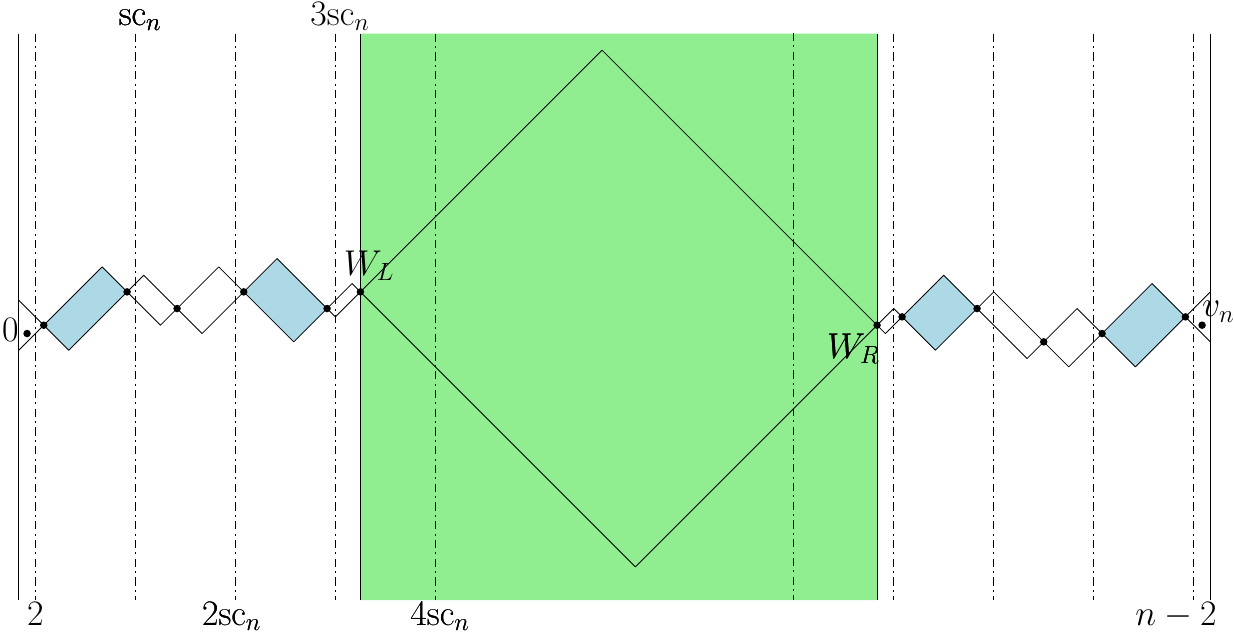}
	\caption{If \(\calC\in \goodCl_n(\scale_n)\), the blue regions all contain at least \(\rho \scale_n\) cone-points. The green region represents \(\Delta(W_L,W_R)\).}
	\label{Fig:InvPrinc_goodCluster_Decomp}
\end{figure}

Finally, introduce a few additional objects that will be needed for the control of the conditional densities of \(\calK\). For \(u,v\in \Z^2\) with \(u_1< v_1\), let, see Figure~\ref{Fig:InvPrinc_goodCluster_Decomp},
\begin{multline}
	\label{eq:def:Delta_u_v}
	\Delta(u,v)
	=
	\{u_1,\dots,v_1 \}
	\\
	\times \{ \min(u_2,v_2)- (v_1-u_1+2), \dots, \max(u_2,v_2)+(v_1-u_1+2)\}.
\end{multline}We will then need the \emph{envelope} of cone-confined graphs. The \emph{outer boundary} and \emph{filling} of \(V\subset \Z^2\) finite connected (for nearest-neighbour connectivity) are, respectively, the set of edges \(\OutBnd(V)\), the set of sites \(\Fill(V)\) obtained as follows: let \(A_0,\dots,A_m\) be the connected components of \(\Z^2\setminus V\) with \(A_0\) the only infinite connected component. Then, \(\OutBnd(V) = \partialedge A_0\), and \(\Fill(V) = \Z^2\setminus A_0\). Define then the envelopes.
\begin{itemize}
	\item For \(\gamma_L\) a connected graph, \(v\in \gamma_L\) a vertex with \(\gamma_L\subset v+\bcone\), define
	\begin{equation*}
		\Env_L(\gamma_L) = \OutBnd(\gamma_L\setminus v).
	\end{equation*}Note that \(v\) is uniquely defined when it exists.
	\item For \( \gamma_R\) a connected graph, \(v\in \gamma_R\) a vertex with \(\gamma_R\subset v+\fcone\), define
	\begin{equation*}
		\Env_R(\gamma_R) = \OutBnd(\gamma_R\setminus v)
	\end{equation*}Note that \(v\) is uniquely defined when it exists.
	\item For \(\gamma\) a connected graph, \(u,v\in \gamma\) with \(\gamma\subset \diam(u,v)\), define
	\begin{gather*}
		\Env(\gamma) = \OutBnd(\gamma\setminus \{u,v\}),
		\\
		\Env^+(\gamma) = \{e\in \Env(\gamma):\ e \text{ strictly above }p_{\gamma}\},
		\\
		\Env^-(\gamma) = \{e\in \Env(\gamma):\ e \text{ strictly below }p_{\gamma}\},
	\end{gather*}where \(p_{\gamma}\) is any bi-infinite self-avoiding path obtained by gluing \((\dots, u-2\rme_1, u-\rme_1,u)\), any self-avoiding path from \(u\) to \(v\) inside \(\gamma\), and \((v,v+\rme_1,v+\rme_2,\dots)\). In particular, \(\Env^+(\gamma)\cup \Env^-(\gamma) = \Env(\gamma)\setminus \{\{u,u+\rme_1\},\{v,v-\rme_1\}\}\).
\end{itemize}
Then, introduce the sets of good forward, backward, and diamond contained marked graphs that can be realizations of \(\calK_L\), \(\calK_R\), and \(\calK\) respectively. For \(v\in \Z^2\),
\begin{equation}
	\label{eq:def:Set_pieces_decomp_FVcluster}
	\begin{aligned}
		\SetRootMarkBackContFV(\scale_n) &= \big\{(\gamma_L,0):\ 0\in \gamma_L,\ \widetilde{\Phi}_n\big(W_L=\displace(\gamma_L),\, \calK_L=\gamma_L\big) >0\big\},
		\\
		\SetRootMarkForwContFV(\scale_n) &= \big\{(\gamma_R,v):\ v\in \gamma_R,\ \widetilde{\Phi}_n\big(W_R = v_n-v,\, \calK_R=\gamma_R \big) >0\big\},
		\\
		\SetRootDiaCont_{v} &= \big\{\gamma\in \SetRootDiaCont:\ \displace(\gamma) = v\big\},
	\end{aligned}
\end{equation}where \(\widetilde{\Phi}_n = \Phi_n(\,\cdot\,\given \calC\in \goodCl_n(\scale_n) )\).
Finally, introduce the events linked with the occurrence of cone confined graphs. 
These are finite volume versions of the events \(A_v(\eta)\) introduced in Section~\ref{subsec:CPs_pre_renewal}. For \(F\subset \bbE\), \(w\in \Z^2\), define the following:
\begin{itemize}
	\item For \(0\in \gamma_L\), \(\gamma_L\) connected such that \(\gamma_L\subset v+\bcone\) for a (uniquely defined) \(v\in \gamma_L\), let \(A_{w,F}^L(\gamma_L)\) be the event
	\begin{equation*}
		A_{w,F}^L(\gamma_L) = \bigcap_{e\in \gamma_L\cap F}\big\{\omega_{\tau}(w+e)= \omega_{\tau\tau'}(w+e) = 1\big\} \cap \bigcap_{e\in (\partialedge \gamma_L\cap F)\setminus \{v,v+\rme_1\}} \big\{\omega_{\tau}(w+e)= 0\big\}.
	\end{equation*}
	\item For \(0,v\in \gamma_R\), \(\gamma_R\) connected such that \(\gamma_R\subset \fcone\), let \(A_{w,F}^R(\gamma_R)\) be the event
	\begin{equation*}
		A_{w,F}^R(\gamma_R) = \bigcap_{e\in \gamma_R\cap F}\big\{\omega_{\tau}(w+e)= \omega_{\tau\tau'}(w+e) = 1\big\} \cap \bigcap_{e\in (\partialedge \gamma_R\cap F)\setminus \{0,-\rme_1\}} \big\{\omega_{\tau}(w+e)= 0\big\}.
	\end{equation*}
	\item For \(0,v\in \gamma\), \(\gamma\) connected such that \(\gamma\subset \diam(0,v)\), let \(A_{w,F}^0(\gamma)\) be the event
	\begin{equation*}
		A_{w,F}^0(\gamma) = \bigcap_{e\in \gamma\cap F}\big\{\omega_{\tau}(w+e)= \omega_{\tau\tau'}(w+e) = 1\big\} \cap \bigcap_{e\in (\partialedge \gamma\cap F)\setminus \mathrm{bnd}_{\gamma}} \big\{\omega_{\tau}(w+e)= 0\big\},
	\end{equation*}where \(\mathrm{bnd}_{\gamma} = \big\{\{0,-\rme_1\}, \{0,\rme_2\}, \{0,-\rme_2\},\{v,v+\rme_1\}, \{v,v+\rme_2\}, \{v,v-\rme_2\}\big\}\).
\end{itemize}
When \(F\) is omitted from the notation, it is set to be \(\bbE_{\Z^2}\).

\subsection{Basic properties of the measure}
\label{subsec:InvPrinc:phi_n_properties}

\begin{lemma}
	\label{lem:InvPrinc:exp_dec}
	Let \(0<J<U\) satisfy \(\sinh 2J=e^{-2U}\). Suppose the sequence \(\Phi_n\) satisfies Hypotheses \ref{hyp_fin_vol_connection:support}, \ref{hyp_fin_vol_connection:FKG}, \ref{hyp_fin_vol_connection:stochastic_sandwich}, \ref{hyp_fin_vol_connection:finite_energy}, and \ref{hyp_fin_vol_connection:edge_state_swap}. Then, there exists \(c>0\) depending on \(\feCst,J,U\) such that for any \(n,m_n,m_n'\geq 1\), any simply connected \(F\subset \rmE_n\), and any \(x,y\in \bbV_{F}\),
	\begin{equation}
		\sup_{\eta}\Phi_n\big(x\xleftrightarrow{F} y \bgiven \omega_{\tau}(e) = \eta(e) \, \forall e \in F^c\big) \leq e^{-c|x-y|}.
		\label{eq:expDecPhi_n}
	\end{equation}
	In particular, there exist integers \(n_0\geq 1, \Cl[ImpCst]{cst:maxHeightCluster}\geq 0\) such that, for any \(m_n,m_n'\geq \Cr{cst:maxHeightCluster} n\), and \(n\geq n_0\),
	\begin{equation}
		\Phi_n (\exists x\in \calC :\ |x_2|\geq \Cr{cst:maxHeightCluster} n) \leq e^{-3\nu_1 n}.
		\label{eq:linear_max_height}
	\end{equation}
\end{lemma}
\begin{proof}
	Note that~\eqref{eq:linear_max_height} follows from~\eqref{eq:expDecPhi_n} and a union bound: taking \(\Cr{cst:maxHeightCluster}\geq \frac{100}{c}\nu_1 \), we have (for \(n\) large enough)
	\begin{equation*}
		\Phi_n (\exists x\in \calC :\ |x_2|\geq \Cr{cst:maxHeightCluster} n)
		\leq
		\sum_{x: \norm{x}_{\infty} \geq \Cr{cst:maxHeightCluster} n}e^{-c|x|}
		\leq
		\sum_{k \geq \Cr{cst:maxHeightCluster} n} 4k e^{-ck/2}
		\leq
		e^{-25\nu_1n}.
	\end{equation*}
	Thus, we focus on~\eqref{eq:expDecPhi_n}.
	By Hypotheses~\ref{hyp_fin_vol_connection:stochastic_sandwich}, for any \(F\subset \bbE_{\FVBox'_n}\),
	\begin{equation*}
		\Phi_n\big(~\cdot~\bgiven \omega_{\tau}(e) = \eta(e) \, \forall e \in F^c\big) \preccurlyeq \atrc_{F}^{1,1} = \Phi(~\cdot~\given \omega_{\tau}(e) =1 \, \forall e\in F^c).
	\end{equation*}
	The right-hand side decays exponentially by Theorem~\ref{thm:Ale04_ATRC}, thus implying~\eqref{eq:expDecPhi_n} whenever \(F\subset \bbE_{\FVBox'_n}\).
	To extend~\eqref{eq:expDecPhi_n} to the case \(F\not\subset \bbE_{\FVBox'_n}\), we use one-dimensionality of~$\bbE_{\bar{\FVBox}_n}\setminus\bbE_{\FVBox'_n}$ and the finite energy property of~$\Phi_n$ (Hypotheses~\ref{hyp_fin_vol_connection:finite_energy}).
	We only sketch the argument. One can first sample the configuration on \(\bbE_{\FVBox'_n}\). By the previous observation, all clusters there have exponential tails. Then, the edges in~$\rmE_n\cap (\bbE_{\bar{\FVBox}_n}\setminus\bbE_{\FVBox'_n})$ do not allow to merge many small clusters into something much bigger because each of these edges is closed with a uniformly positive probability.
	This was made formal in a classical work of Alexander~\cite{Ale04}.
	Though our setting does not fit in~\cite[Theorem~1.1]{Ale04}, the arguments in~\cite[Section 2]{Ale04} do apply {\em mutatis mutandis}.
\end{proof}

Using Theorem~\ref{thm:ratio_weak_mixing_ATRC}, we obtain a fast relaxation of \(\Phi_n\) towards \(\Phi\).

\begin{lemma}
	\label{lem:InvPrinc:relax_ratio}
	Let \(0<J<U\) satisfy \(\sinh 2J=e^{-2U}\). Suppose the sequence \(\Phi_n\) satisfies Hypotheses \ref{hyp_fin_vol_connection:support}, \ref{hyp_fin_vol_connection:FKG}, \ref{hyp_fin_vol_connection:stochastic_sandwich}, \ref{hyp_fin_vol_connection:finite_energy}, and \ref{hyp_fin_vol_connection:edge_state_swap}.
	Then, there exist \(c>0,C\geq 0\) such that, for every \(n,m_n,m_n'\geq 2\), any \(F\subset \bbE_{\FVBox_{n}'}\), and any event \(A\) supported on \(F\) with \(\Phi(A)>0\),
	\begin{equation*}
		\Big|\frac{\Phi_m(A)}{\Phi(A)} -1 \Big| \leq C\sum_{e\in F} e^{-c\rmd_{\infty}(e,(\FVBox_{n}')^c)},
	\end{equation*}
	as soon as the R.H.S. is strictly less than \(0.1\). In particular, there exists \(\Cl[ImpCst]{cst:minMixScale}\geq 0\) such that, for any \(a> 0\), \(m_n,m_n'\geq an + \Cr{cst:minMixScale}\ln(n)\), one can find \(n_0\geq 1\) such that for any \(n\geq n_0\) and any event \(B\) supported on \(\bbE_{\rectangle_{\Cr{cst:minMixScale}\ln(n), an}^n}\), one has
	\begin{equation*}
		\tfrac{1}{2}\leq \frac{\Phi_n(B)}{\Phi(B)} \leq 2.
	\end{equation*}
\end{lemma}
\begin{proof}
	This follows from Hypotheses~\ref{hyp_fin_vol_connection:stochastic_sandwich}, the exponential ratio weak mixing of \(\atrc\) (Theorem~\ref{thm:ratio_weak_mixing_ATRC}), and Lemma~\ref{lem:ratio_FKG_measures}: by Theorem~\ref{thm:ratio_weak_mixing_ATRC}, for any \(A\) supported on \(F\),
	\begin{equation*}
		\Big|\frac{\atrc^{1,1}_{\bbE_{\FVBox_{n}'}}(A)}{\atrc^{0,0}_{\bbE_{\FVBox_{n}'}}(A)} -1\Big|
		\leq
		C\sum_{e\in F}\sum_{e'\in \bbE_{\FVBox_{n}'}^c} e^{-c\rmd_{\infty}(e,e')}
		\leq
		C'\sum_{e\in F} e^{-c\rmd_{\infty}(e,(\FVBox_{n}')^c)}.
	\end{equation*} 
	For \(F= \bbE_{\rectangle_{\Cr{cst:minMixScale}\ln(n), an}^n} \), the last sum can then be made as small as wanted by taking \(\rectangle_{\Cr{cst:minMixScale}\ln(n), an}^n\) at distance at least \(c'\ln(n)\) from \((\FVBox_{n}')^c\) for \(c'\) large enough, which amounts to take \(\Cr{cst:minMixScale}\) large enough. 
	The claim then follows from Lemma~\ref{lem:ratio_FKG_measures} (\(0.1\) is any small enough constant to control the ratios in the conclusion of Lemma~\ref{lem:ratio_FKG_measures}).
\end{proof}

Our last ``preparatory Lemma'' is a sharp-up-to-constant lower bound on the probability of the event on which we want to condition.
\begin{lemma}
	\label{lem:InvPrinc:weak_LB}
	Let \(0<J<U\) satisfy \(\sinh 2J=e^{-2U}\). Suppose the sequence \(\Phi_n\) satisfies Hypotheses \ref{hyp_fin_vol_connection:support}, \ref{hyp_fin_vol_connection:FKG}, \ref{hyp_fin_vol_connection:stochastic_sandwich}, \ref{hyp_fin_vol_connection:finite_energy}, and \ref{hyp_fin_vol_connection:edge_state_swap}.
	Then, there exists \(c\geq 0\) such that, for any \(a>0\), the following holds.
	There exists \(n_0\geq 1\) such that, for any \(n\geq n_0\), \(m_n,m_n'\geq a n + \Cr{cst:minMixScale}\ln(n)\),
	\begin{equation*}
		\Phi_{n}(v_n\in \calC) \geq \frac{c}{\sqrt{n}}e^{-n\nu_1}.
	\end{equation*}
\end{lemma}
\begin{proof}
	Let \(x_N=(N,0)\), \(y_N=(n-N,0)\) with \(N\) to be fixed large enough later. Let
	\begin{equation*}
		V_{N} = \Z^2\cap [N,n-N]\times [-an,an] \cap (x_N+\fcone) \cap (y_N+\bcone),
		\quad
		F_{N} = \bbE_{V_{N}}.
	\end{equation*}Then, by inclusion of events and FKG inequality,
	\begin{equation}
		\label{eq:splitting_restricted_connection}
		\begin{aligned}
			\Phi_{n}(v_n\in \calC)
			&\geq
			\Phi_{n}(0\leftrightarrow x_{2N})\Phi_{n}(v_n\leftrightarrow y_{2N})\Phi_{n}(x_{2N}\xleftrightarrow{F_N} y_{2N})
			\\
			&\geq
			\Big(1-C\sum_{e\in F_N} e^{-c\rmd_{\infty}(e, (\FVBox_{n}')^c)} \Big) \feCst^{4N} \Phi(x_{2N}\xleftrightarrow{F_N} y_{2N}),
		\end{aligned}
	\end{equation}
	where we used Lemma~\ref{lem:InvPrinc:relax_ratio}, and Hypotheses~\ref{hyp_fin_vol_connection:finite_energy} in the second line. Now, choosing \(N\) large enough as a function of \(C\) and~\(c\), we get~\(C\sum_{e\in F_N} e^{-c\rmd_{\infty}(e, (\FVBox_{n}')^c)} \leq \frac{1}{2}\) for any \(n\) large enough.
	
	To conclude, we can then use Theorem~\ref{thm:OZ_for_ATRC_infinite_vol} to lower bound the remaining probability. Let \(x_N=S_0, S_1,S_2,\dots\) be a random walk with step distribution \(\OZWalkmeas\). Using the finite energy of the probability measures given by Theorem~\ref{thm:OZ_for_ATRC_infinite_vol}, \(\Phi(x_N\xleftrightarrow{F_N} y_N)\) is greater than \(C'e^{-\nu_1(n-4N)}\) times
	\begin{multline*}
		\OZWalkmeas\Big(T_{y_{2N}-x_{2N}}<\infty,\ \bigcup_{i=1}^{T_{y_{2N}-x_{2N}}} \diam(S_{i-1},S_i) \\
		\subset [N,n-N]\times [-an,an] \cap \big(x_N+\fcone\big) \cap \big(y_N+\bcone\big)\Big),
	\end{multline*}where \(S_0=x_N\), \(S_i = X_i+ S_{i-1}\), \((X_i)_{i\geq 1}\sim \OZWalkmeas\), and \(T\) is 
	defined in~\eqref{eq:def:hitting_times}. The steps \(X_i\)'s are supported on the whole of \(\fcone\cap \Z^2\), have exponential tails, and mean proportional to \(\rme_1\).
	Then, the Local Limit Theorem applies and a simple computation implies that the probability of \(T_{y_{2N}-x_{2N}}<\infty\) is greater than \(C/\sqrt{n}\); see eg.~\cite{AouOttVel24} between eq.~$(20)$ and~$(22)$. Adding the constraint that the diamond envelope has the wanted containment property only costs a multiplicative constant. This is a straightforward consequence of the exponential tails of the steps. The reader interested by the precise implementation of this last point can have a look at~\cite[Lemmas 2.6,2.7]{OttVel19} for the implementation of a very similar (though more involved) claim.
\end{proof}

\subsection{Typical geometry}
\label{subsec:InvPrinc:phi_n_goodCl}

We will go through an intermediate set of edges before trying to control \(\calC\). Introduce \(\omega_{\tau}|_{k,l}^n \equiv \omega_{\tau}|_{\bbE_{\rectangle_{k,l}^n}}\) the restriction of \(\omega_{\tau}\) to the edges with both endpoints in \(\rectangle_{k,l}^n\), and the set of edges in \(\bbE_{\rectangle_{k,l}}\) that are connected to a left-right crossing of \(\rectangle_{k,l}\):
\begin{multline}
	\label{eq:def:crossing_clusters}
	\crossCl^{k,l}
	\equiv
	\crossCl^{k,l,n}(\omega_{\tau})
	=
	\big\{e\in \bbE_{\rectangle_{k,l}^n}:\ \omega_{\tau}(e)=1,\ e\xleftrightarrow{\omega_{\tau}|_{k,l}^n} \{k\}\times \{-l,\dots,l\},
	\\ e\xleftrightarrow{\omega_{\tau}|_{k,l}^n} \{n-k\}\times \{-l,\dots,l\}\big\}.
\end{multline}Also, introduce the event \(\doubleCross_{k,l}\) that \(\crossCl^{k,l}\) contains at least two disjoint paths from \(\{k\}\times \Z\) to \(\{n-k\}\times \Z\).

To get control over the geometry of \(\calC\) under \(\Phi_n(\cdot \given v_n\in \calC)\), we will first get control over the geometry of \(\crossCl\) and then prove that \(\calC\) and \(\crossCl\) are actually very close.

The next Lemma is our only use (in Section~\ref{sec:invariance_principle}) of an angular aperture different than \(\pi/4\) for cone-points. The angle is therefore explicitly mentioned in the notation.
\begin{lemma}
	\label{lem:InvPrinc_crossing_clusters_are_nice}
	Let \(0<J<U\) satisfy \(\sinh 2J=e^{-2U}\). Suppose the sequence \(\Phi_n\) satisfies Hypotheses \ref{hyp_fin_vol_connection:support}, \ref{hyp_fin_vol_connection:FKG}, \ref{hyp_fin_vol_connection:stochastic_sandwich}, \ref{hyp_fin_vol_connection:finite_energy}, and \ref{hyp_fin_vol_connection:edge_state_swap}. Let \(\Cr{cst:maxHeightCluster}\) be given by Lemma~\ref{lem:InvPrinc:exp_dec}, and \(\Cr{cst:minMixScale}\) be given by Lemma~\ref{lem:InvPrinc:relax_ratio}.
	Then, there exist \(n_0\geq 1\), \(\epsilon_0>0\), \(\rho>0\), \(K_0\geq 0\), \(c>0,C\geq 0\) such that for any \(n\geq n_0\), \( \epsilon_0 n \geq k_n\geq \Cr{cst:minMixScale}\ln(n)\), \(m_n,m_n'\geq 2\Cr{cst:maxHeightCluster} n+k_n\), one has the following bounds.
	\begin{enumerate}
		\item \label{item:lem_crossing:double_cross} \(\Phi_n\big(\doubleCross_{k_n,2\Cr{cst:maxHeightCluster}n}\big) \leq e^{-\nu_1n - cn}\).
		\item \label{item:lem_crossing:inclusion_Cross_in_Cl} \(\Phi_n\big(\varnothing\neq \crossCl^{k_n,2\Cr{cst:maxHeightCluster}n} \not\subset \calC,\ v_n\in \calC\big) \leq e^{-\nu_1 n- cn}\).
		\item \label{item:lem_crossing:inclusion_Cl_in_LargeCross} For any \(K\geq K_0\),
		\begin{multline*}
			\Phi_n\big(\varnothing\neq \crossCl^{k_n,2\Cr{cst:maxHeightCluster}n}\subset \bbE_{\rectangle_{k_n,\Cr{cst:maxHeightCluster}n}^n},\ \exists x\in\bbV_{\rmE_n}:\ x \leftrightarrow \crossCl^{k_n,2\Cr{cst:maxHeightCluster}n},
			\\
			\rmd_{\infty}\big(x,\crossCl^{k_n,2\Cr{cst:maxHeightCluster}n} \cap \{k_n,n-k_n\}\times \Z\big) \geq Kk_n \big)
			\leq
			e^{-\nu_1 n - cKk_n}.
		\end{multline*}
		\item \label{item:lem_crossing:slabCPTs} For any \(K\geq K_0\), and any \(i\in \{k_n,\dots, n-(K+1)k_n\}\),
		\begin{equation*}
			\Phi_n\big(\varnothing\neq \crossCl^{k_n,2\Cr{cst:maxHeightCluster}n}\subset \bbE_{\rectangle_{k_n,\Cr{cst:maxHeightCluster}n}^n},\ |\slabCP_{\pi/8,i,i+Kk_n}(\crossCl^{k_n,2\Cr{cst:maxHeightCluster}n})| < \rho K k_n \big) \leq e^{-\nu_1 n - c K k_n}.
		\end{equation*}
	\end{enumerate}
\end{lemma}
\begin{proof}
	Proceed in increasing order.
	
	\smallskip
	
	\noindent\textbf{Start with item~\ref{item:lem_crossing:double_cross}.}
	The proof goes as follows: if there are at least two crossings of \(\rectangle_{k_n,2\Cr{cst:maxHeightCluster}n}^n\), one can condition on the lowest. Then, the probability of a second crossing given the lowest one is at most \((4\Cr{cst:maxHeightCluster}n)^2e^{-c (n-2k_n)}\) uniformly over the lowest crossing by a union bound and Theorem~\ref{thm:Ale04_ATRC}. Indeed, one can dominate the measure conditioned on the lowest crossing by an ATRC measure on the part of \(\rectangle_{k_n,2\Cr{cst:maxHeightCluster}n}^n\) above the lowest crossing with \(1,1\) boundary conditions. The later has exponential decay of connectivities by Theorem~\ref{thm:Ale04_ATRC}. Re-summing over the lowest crossing, we obtain as upper bound \((4\Cr{cst:maxHeightCluster}n)^2e^{-c (n-2k_n)}\) times the probability that \(\rectangle_{k_n,2\Cr{cst:maxHeightCluster}n}^n\) is crossed. Using Lemma~\ref{lem:InvPrinc:relax_ratio}, a union bound, and Theorem~\ref{thm:OZ_for_ATRC_infinite_vol}, one can upper bound this probability by \(2(4\Cr{cst:maxHeightCluster}n)^2e^{-\nu_1(n-2k_n)}\) for any \(n\) large enough. This gives the first point once \(\epsilon_0>0\) is taken small enough and \(n_0\) large enough.
	
	\smallskip
	
	\noindent\textbf{Continue with item~\ref{item:lem_crossing:inclusion_Cross_in_Cl}.}
	Assume \(\crossCl^{k_n,2\Cr{cst:maxHeightCluster}n}\neq \varnothing\) and~\(v_n\in \calC\). 
	If~\(\crossCl^{k_n,2\Cr{cst:maxHeightCluster}n}\not \subset \calC\), then at least one of the two following events occurs:
	\begin{itemize}
		\item There exists \(x\in \bbV_{\rmE_n}\) with \(|x_2|\geq \Cr{cst:maxHeightCluster}n\) such that \(x\in \calC\).
		\item \(\doubleCross_{k_n,2\Cr{cst:maxHeightCluster}n}\) occurs.
	\end{itemize}
	The first event has probability at most \(e^{-3\nu_1n}\) by Lemma~\ref{lem:InvPrinc:exp_dec}, and the second has probability at most \(e^{-\nu_1n - cn}\) by item~\ref{item:lem_crossing:double_cross}.
	
	\smallskip
	
	\noindent\textbf{Turn to item~\ref{item:lem_crossing:inclusion_Cl_in_LargeCross}.}
	By item~\ref{item:lem_crossing:double_cross}, we can restrict to the case were the event \(\doubleCross_{k_n,2\Cr{cst:maxHeightCluster}n}\) does not occur. In particular, we can assume that \(\crossCl^{k_n,2\Cr{cst:maxHeightCluster}n}\) is connected. Now, observe that, as \(\crossCl^{k_n,2\Cr{cst:maxHeightCluster}n}\subset \bbE_{\rectangle_{k_n,\Cr{cst:maxHeightCluster}n}^n}\), if \(x\) is connected to \(\crossCl^{k_n,2\Cr{cst:maxHeightCluster}n}\) but is not part of it, \(x\) has to be connected to \(\crossCl^{k_n,2\Cr{cst:maxHeightCluster}n}\cap \{k_n,n-k_n\}\times \{-\Cr{cst:maxHeightCluster}n,\dots,\Cr{cst:maxHeightCluster}n\}\), see Figure~\ref{Fig:InvPrinc_cluster_cross_CPts}.
	Now, by Lemma~\ref{lem:InvPrinc:exp_dec}, the probability that \(x\) is connected to \(\crossCl^{k_n,2\Cr{cst:maxHeightCluster}n}\) conditionally on \(\crossCl^{k_n,2\Cr{cst:maxHeightCluster}n}\) is at most \(\exp(-c\rmd_{\infty}(x, \crossCl^{k_n,2\Cr{cst:maxHeightCluster}n}\cap \{k_n,n-k_n\}\times \Z))\). The claim follows from a union bound over \(x\) conditionally on \(\crossCl^{k_n,2\Cr{cst:maxHeightCluster}n}\).
	
	\smallskip
	
	\noindent\textbf{Finish with item~\ref{item:lem_crossing:slabCPTs}.}
	Under the event \(\varnothing\neq \crossCl^{k_n,2\Cr{cst:maxHeightCluster}n}\subset \bbE_{\rectangle_{k_n,\Cr{cst:maxHeightCluster}n}^n}\), there exist \(x\in \{k_n\}\times \{-\Cr{cst:maxHeightCluster}n,\dots, \Cr{cst:maxHeightCluster}n\}\), and \(y\in \{n-k_n\}\times \{-\Cr{cst:maxHeightCluster}n,\dots, \Cr{cst:maxHeightCluster}n\}\) such that \(x\) is connected to \(y\) in \(\rectangle_{k_n,\Cr{cst:maxHeightCluster}n}^n\). For \(\epsilon>0\), let 
	\begin{equation*}
		A_{\epsilon}^n = \big\{\exists x,y\in \rectangle_{k_n,\Cr{cst:maxHeightCluster}n}^n:\ x_1=k_n,\: y_1=n-k_n,\: |x_2-y_2|\geq  \epsilon n,\: x\xleftrightarrow{\rectangle_{k_n,\Cr{cst:maxHeightCluster}n}^n} y \big\}.
	\end{equation*}
	Now, by a union bound, Theorem~\ref{thm:oz-atrc}, and Lemma~\ref{lem:InvPrinc:relax_ratio}, one has that for any \(\epsilon>0\) and \(n\) large enough,
	\begin{align*}
		\Phi_n\big(A_{\epsilon}^n\big)
		&\leq
		2\sum_{x: x_1=k_n, |x_2|\leq \Cr{cst:maxHeightCluster}n} \sum_{y: y_1=n-k_n, |y_2|\leq \Cr{cst:maxHeightCluster}n}\mathds{1}_{|x_2-y_2|\geq \epsilon n}\Phi(x\leftrightarrow y)
		\\
		&\leq
		4\Cr{cst:maxHeightCluster}n \sum_{y: y_1=n-2k_n, |y_2|\leq 2\Cr{cst:maxHeightCluster}n}\mathds{1}_{|y_2|\geq \epsilon n} e^{-\nu(y)}
		\\
		&\leq
		4\Cr{cst:maxHeightCluster}n \sum_{\epsilon n \leq |k|\leq 2\Cr{cst:maxHeightCluster}n} e^{-\nu_1(n-2k_n)- c|k|}
		\leq
		(4\Cr{cst:maxHeightCluster}n)^2 e^{-\nu_1(n-2k_n)- c\epsilon n},
	\end{align*}for some \(c>0\) as \(\nu\) is a uniformly convex norm. The above shows that we can fix any \(\epsilon>0\) and restrict the study of the wanted probability to \((A_{\epsilon}^n)^c\) at the price of taking \(\epsilon_0\) small enough. Now, take \(\epsilon>0\) small enough so that we can apply Theorem~\ref{thm:OZ_for_ATRC_infinite_vol} with \(s_0=\rme_1\) to connections between \(x,y\) with \(x_1<y_1\), \(|x_2-y_2|\leq \epsilon (y_1-x_1)\).
	Moreover, by item~\ref{item:lem_crossing:double_cross}, we can work under \((\doubleCross_{k_n,2\Cr{cst:maxHeightCluster}n})^c\), so we can assume \(\crossCl^{k_n,2\Cr{cst:maxHeightCluster}n}\) is connected.
	Now, remains to study the event of item~\ref{item:lem_crossing:slabCPTs} under the event \((\doubleCross_{k_n,2\Cr{cst:maxHeightCluster}n})^c \cap (A_{\epsilon}^n)^c\). In particular, that there are \(x,y\in \rectangle_{k_n,\Cr{cst:maxHeightCluster}n}\) with \(x_1=k_n\), \(y_1=n-k_n\), \(|y_2-x_2|\leq \epsilon n\), \(x\) connected to \(y\) in \(\rectangle_{k_n,\Cr{cst:maxHeightCluster}n}\), and \(\crossCl^{k_n,2\Cr{cst:maxHeightCluster}n}\) is the cluster of \(x,y\) in \(\rectangle_{k_n,2\Cr{cst:maxHeightCluster}n}^n\). We can thus make a union bound over such \(x,y\)'s, and use Lemma~\ref{lem:InvPrinc:relax_ratio} to replace \(\Phi_n\) by \(\Phi\) to obtain that the probability of the intersection of the event in item~\ref{item:lem_crossing:slabCPTs} with \((\doubleCross_{k_n,2\Cr{cst:maxHeightCluster}n})^c \cap (A_{\epsilon}^n)^c\) is less or equal to
	\begin{multline*}
		2\sum_{x,y} \Phi\big(y\in \calC_{x}^{\restrict},\ |\slabCP_{\pi/8,i,i+Kk_n}(\calC_{x}^{\restrict})| < \rho K k_n\big)
		\\
		\leq
		2\sum_{x:x_1=n-2k_n, |x_2|\leq \epsilon n} \Phi\big(x\in \calC,\ |\slabCP_{\pi/8,i,i+Kk_n}(\calC)| < \rho K k_n\big)
	\end{multline*}where \(i\) is as in item~\ref{item:lem_crossing:slabCPTs}, the first sum is over \(x,y\) with \(x_1=k_n\), \(y_1=n-k_n\), \(|x_2-y_2|\leq \epsilon n\), \(|x_2|,|y_2|\leq \Cr{cst:maxHeightCluster}n\), and \(\calC_{x}^{\restrict}\) is the cluster of \(x\) in the restriction of \(\omega_{\tau}\) to \(\bbE_{\rectangle_{k_n,2\Cr{cst:maxHeightCluster}n}^n}\). The inequality uses that, under the event \(x\leftrightarrow y\), \(\slabCP_{\pi/8,i,i+Kk_n}(\calC_{x}^{\restrict})\subset \slabCP_{\pi/8,i,i+Kk_n}(\calC_{x}) \), with \(\calC_x\) the cluster of \(x\), and translation invariance of \(\Phi\). Now, using Theorem~\ref{thm:OZ_for_ATRC_infinite_vol}, the exponential tails of the measures \(p\circ \displace^{-1}\), \(p_L\circ \displace^{-1}\), \(p_R\circ \displace^{-1}\) there, and classical large deviation bounds for deviation of sums of i.i.d. random vectors with exponential tails away from their mean, one directly obtains that
	\begin{equation*}
		\sum_{x:x_1=n-2k_n, |x_2|\leq \epsilon n} \Phi\big(x\in \calC,\ |\slabCP_{\pi/8,i,i+Kk_n}(\calC)| < \rho K k_n\big)
		\leq
		e^{-\nu_1(n-2k_n)}e^{-cKk_n},
	\end{equation*}for some \(c>0\), and all \(n\) large enough, as soon as \(K\) is larger than some fixed constant.
\end{proof}

\begin{figure}
	\includegraphics[scale=0.8]{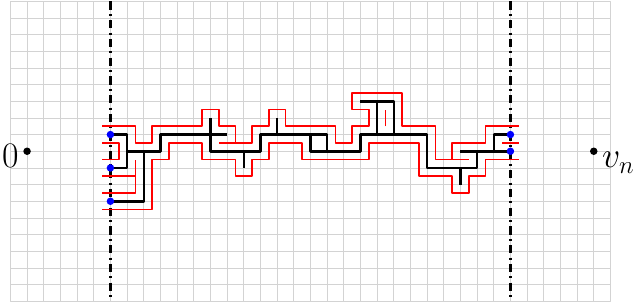}
	\caption{The open cluster \(\crossCl \) (in black) and the edges forced to be closed by its presence (dual of the red). The blue points are where connections to \(\crossCl \) have to arrive.}
	\label{Fig:InvPrinc_cluster_cross_CPts}
\end{figure}

We then push the properties of \(\crossCl^{k_n,2\Cr{cst:maxHeightCluster}n}\) to \(\calC\) using a deterministic observation about cone-points.
\begin{figure}
	\includegraphics[scale=0.9]{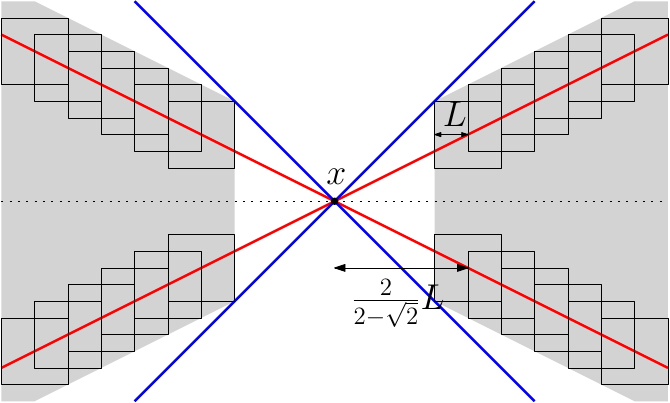}
	\caption{Red: the cone-point for \(\theta = \pi/8\). Blue: the cone-point for \(\theta = \pi/4\). The addition of the grey area is compensated by the increase of the cone opening.}
	\label{Fig:InvPrinc_CPts_enlargement}
\end{figure}
\begin{lemma}
	\label{lem:InvPrinc_cluster_is_nice}
	Let \(0<J<U\) satisfy \(\sinh 2J=e^{-2U}\). Suppose the sequence \(\Phi_n\) satisfies Hypotheses \ref{hyp_fin_vol_connection:support}, \ref{hyp_fin_vol_connection:FKG}, \ref{hyp_fin_vol_connection:stochastic_sandwich}, \ref{hyp_fin_vol_connection:finite_energy}, and \ref{hyp_fin_vol_connection:edge_state_swap}.
	Then, there exist \(n_0\geq 1\), \(\Cl[ImpCst]{cst:scaleGCl}>0\), \(\Cl[ImpCst]{cst:scaleMinBox}\geq 5\), \(c>0\), such that for any \(n\geq n_0\), \(\frac{n}{\Cr{cst:scaleGCl}} \geq \scale_n \geq \Cr{cst:scaleGCl}\ln(n)\), \(m_n,m_n'\geq \Cr{cst:scaleMinBox} n\), one has
	\begin{equation*}
		\Phi_n\big(\calC\notin \goodCl_n(\scale_n) ,\ v_n\in \calC \big) 
		\leq
		e^{-\nu_1 n- c\scale_n}.
	\end{equation*}
\end{lemma}
\begin{proof}
	Let \(\Cr{cst:maxHeightCluster},\Cr{cst:minMixScale}\) be the constants given by Lemmas~\ref{lem:InvPrinc:exp_dec}, and~\ref{lem:InvPrinc:relax_ratio} respectively. The deterministic observation is that for any \(L\geq 1\), and \(V\) connected, any \(x \in \CPts_{\pi/8}(V)\) also satisfies \(x\in \CPts_{\pi/4}(\lrangle{V}_L)\), where \(\lrangle{V}_L\) is given by
	\begin{equation*}
		\lrangle{V}_L = V\cup \bigcup_{y\in V: y_1< x_1 - \frac{2}{2-\sqrt{2}}L }(y+[-L,L]^2) \cup \bigcup_{y\in V: y_1> x_1 + \frac{2}{2-\sqrt{2}}L }(y+[-L,L]^2)
	\end{equation*}where we used \(\tan(\pi/8) = \sqrt{2}-1\) and \(\tan(\pi/4) = 1\). See Figure~\ref{Fig:InvPrinc_CPts_enlargement}. We will use Lemma~\ref{lem:InvPrinc_crossing_clusters_are_nice} to say that \(\calC\) is close to some crossing cluster (included in a slight enlargement), and that the said crossing cluster has all the cone-points required to be a good cluster, \emph{but for a smaller cone opening}. Then, we increase the cone opening to take care of the slight enlargement of the crossing cluster needed to control \(\calC\).
	
	Introduce
	\begin{equation*}
		k_n = \lfloor \scale_n/r \rfloor.
	\end{equation*}for some \(r>0\) to be chosen large later.
	
	Let
	\begin{equation*}
		\crossCl^{\partial} = \crossCl^{k_n,2\Cr{cst:maxHeightCluster}n} \cap \{k_n,n-k_n\}\times \Z.
	\end{equation*}
	Let \( K,K' \geq 1\) be some integers to be fixed large later, and \(\rho>0\) is small to be fixed small later. Introduce the events
	\begin{itemize}
		\item \(A_1(K)\) is the event that \(\crossCl^{k_n,2\Cr{cst:maxHeightCluster}n} \subset \calC\) and that
		\begin{equation*}
			\calC\subset \crossCl^{k_n,2\Cr{cst:maxHeightCluster}n} \cup \bigcup_{x\in \crossCl^{\partial}} (x+[-Kk_n,Kk_n]^2).
		\end{equation*}
		\item \(A_2(\rho,K')\) is the event that for every \(i\in \{k_n,\dots,n-(K'+1)k_n\}\),\linebreak \(|\slabCP_{\pi/8,i,i+K'k_n}(\crossCl^{k_n,2\Cr{cst:maxHeightCluster}n})|\geq \rho k_n\).
	\end{itemize}
	By Lemma~\ref{lem:InvPrinc_crossing_clusters_are_nice}, if we fix \(\rho >0\) small enough, there exist \(c>0, K_0\geq 0, n_0\geq 0\) such that for any \(K,K'\geq K_0\), and any \(n\geq n_0\),
	\begin{equation*}
		\Phi_n\big(A_1(K)^c,\ v_n\in \calC\big) \leq e^{-\nu_1n -c K k_n},
		\quad
		\Phi_n\big(A_2(\rho,K')^c,\ v_n\in \calC\big) \leq e^{-\nu_1n -c K' k_n},
	\end{equation*}as long as \(\epsilon_0 n\geq k_n \geq \Cr{cst:minMixScale} \ln(n)\) with \(\epsilon_0>0\) fixed given by Lemma~\ref{lem:InvPrinc_crossing_clusters_are_nice}, and \(m_n,m_n'\geq 2\Cr{cst:maxHeightCluster} n + k_n\).
	Now, the deterministic observation implies that on the event \(A_1(K)\cap \{v_n\in \calC\}\),
	\begin{equation*}
		\slabCP_{\pi/8, a_K k_n,n-a_K k_n }(\crossCl^{k_n,2\Cr{cst:maxHeightCluster}n}) \subset \slabCP_{\pi/4, a_K k_n,n-a_K k_n }(\calC)
	\end{equation*}where \(a_K = \frac{2K}{2-\sqrt{ 2}} +1\). The conclusion follows once we fix \(K,K'\) large enough, and take \(r = 10a_K\) (\(10\) is an arbitrary number strictly larger than \(2\)), as the cone-points required to be in \(\goodCl_n(\scale_n)\) are included in those guaranteed by the occurrence of \(A_2(\rho,K')\). The constant \(\Cr{cst:scaleGCl}\) is taken large enough so that \(k_n = \scale_n/r\) can fit the required constraints, and the constant \(\Cr{cst:scaleMinBox}\) is taken to be any constant such that \(\Cr{cst:scaleMinBox}n \geq 2\Cr{cst:maxHeightCluster} n + k_n = 2\Cr{cst:maxHeightCluster} n + \scale_n/r \) for all \(n\) large enough and all \(\frac{n}{\Cr{cst:scaleGCl}}\geq \scale_n \geq \Cr{cst:scaleGCl} \ln(n)\).
\end{proof}

\subsection{Density swapping}
\label{subsec:InvPrinc:density_swap}

We now study the density of \(\Phi_n\circ \calC^{-1}\) on the event \(\calC\in \goodCl_n(\scale_n)\) for \(\scale_n\) as in Lemma~\ref{lem:InvPrinc_cluster_is_nice}. Recall the notations and definitions introduced in Section~\ref{subsec:InvPrinc:good_cluster} and the edge-set~$\rmE_n$ on which~$\Phi_n$ is defined.

\begin{lemma}
	\label{lem:InvPrinc:density_swapping}
	Let \(0<J<U\) satisfy \(\sinh 2J=e^{-2U}\). Suppose the sequence \(\Phi_n\) satisfies Hypotheses \ref{hyp_fin_vol_connection:support}, \ref{hyp_fin_vol_connection:FKG}, \ref{hyp_fin_vol_connection:stochastic_sandwich}, \ref{hyp_fin_vol_connection:finite_energy}, and \ref{hyp_fin_vol_connection:edge_state_swap}. Let \(\Cr{cst:scaleGCl},\Cr{cst:scaleMinBox}\) be given by Lemma~\ref{lem:InvPrinc_cluster_is_nice}. Then, there are \(n_0\geq 1\), \(c>0\), such that for any \(n\geq n_0\), \(m_n,m_n'\geq \Cr{cst:scaleMinBox} n\), \(\frac{n}{\Cr{cst:scaleGCl}} \geq \scale_n \geq \Cr{cst:scaleGCl} \ln(n)\), one has that for every \((\gamma_L,0)\in \SetRootMarkBackContFV(\scale_n)\), every \((\gamma_R,v) \in \SetRootMarkForwContFV(\scale_n)\), and every \(\gamma\) connected containing~\(0\) and~\(v_n-v-u\), where \(u\equiv \displace(\gamma_L)\), with \(\gamma\subset \diam(0,v_n-v-u)\),
	\begin{equation*}
		\Big|\frac{\Phi_n\big( A_{u}^0(\gamma)\given A_{0,\rmE_n}^L(\gamma_L),\, A_{v_n-v,\rmE_n}^R(\gamma_R),\, \calC\in \goodCl_n(\scale_n)\big)}{\Phi\big( A_{u}^0(\gamma)\given A_{0}^L(\gamma_L),\, A_{v_n-v}^R(\gamma_R),\, \calC\in \goodCl_n(\scale_n)\big)} -1 \Big|
		\leq 
		e^{-c\scale_n}.
	\end{equation*}
\end{lemma}

\begin{proof}
	To shorten notations, let \(w_L=\displace(\gamma_L)\), \(w_R = v_n-v\). We will use several times that \(A_{u}^0(\gamma)\) is supported on edges with both endpoints in \(\Delta(w_L,w_R)\).
	First, notice that, as \((\gamma_L,0)\in \SetRootMarkBackContFV(\scale_n)\), and \((\gamma_R,v) \in \SetRootMarkForwContFV(\scale_n)\), there is a leftmost cone-point of \(\gamma_L\) in the slab \(\slab_{2,\scale_n-1}\), denote it \(w_{L}^1\), a leftmost cone-point of \(\gamma_L\) in the slab \(\slab_{\scale_n+1,2\scale_n-1}\), denote it \(w_{L}^{2}\), and a leftmost cone-point of \(\gamma_L\) in the slab \(\slab_{2\scale_n+1,3\scale_n-1}\), denote it \(w_{L}^{3}\).
	In the same fashion, there is a rightmost cone-point of \(v_n-v+ \gamma_R\) in the slab \(\slab_{n-\scale_n+1,n-2}\), denote it \(w_{R}^{1}\), a rightmost cone-point of \(v_n-v+\gamma_R\) in the slab \(\slab_{n-2\scale_n+1,n-\scale_n-1}\), denote it \(w_{R}^{2}\), and a rightmost cone-point of \(v_n-v+\gamma_R\) in the slab \(\slab_{n-3\scale_n+1,n-2\scale_n-1}\), denote it \(w_{R}^{3}\). See Figure~\ref{Fig:InvPrinc_Decomp_Annuli}.
	
	Let then
	\begin{equation*}
		F_n = \bbE_{\{w_L^1\cdot \rme_1,\dots,w_R^1\cdot \rme_1\}\times \{-5n,\dots, 5n\}},
	\end{equation*}
	and introduce the sets, see Figure~\ref{Fig:InvPrinc_Decomp_Annuli},
	\begin{equation*}
		E_1 = \bbE_{\Delta(w_L, w_R)},
		\quad
		E_2 = \bbE_{\{w_L^2\cdot \rme_1,\dots,w_R^2\cdot\rme_1\}\times \{-4n,\dots,4n\}\setminus \{w_L^3\cdot \rme_1,\dots,w_R^3\cdot \rme_1\}\times \{-3n,\dots,3n\}}.
	\end{equation*}Note that \(\Delta(w_L, w_R) \subset \{w_L\cdot \rme_1,\dots, w_R\cdot \rme_1\}\times \{-2n,\dots,2n\}\) by definition, and that \(m_n,m_n'\geq 5n\) by choice of \(\Cr{cst:scaleMinBox}\).
	
	Introduce the following modifications \(\tilde{A}_{0,\rmE_n}^L(\gamma_L),\, \tilde{A}_{w_R,\rmE_n}^R(\gamma_R),\, \tilde{A}_{0}^L(\gamma_L),\, \tilde{A}_{w_R}^R(\gamma_R)\) of \(A_{0,\rmE_n}^L(\gamma_L),\, A_{w_R,\rmE_n}^R(\gamma_R),\, A_{0}^L(\gamma_L),\, A_{w_R}^R(\gamma_R)\): ask that \(\{w_L^1,w_L^1+\rme_1\}, \{w_R^1,w_R^1-\rme_1\}\) are closed in \(\omega_{\tau}\) rather than open. See Figure~\ref{Fig:InvPrinc_Decomp_Annuli}.
	
	\begin{figure}
		\includegraphics[scale=0.75]{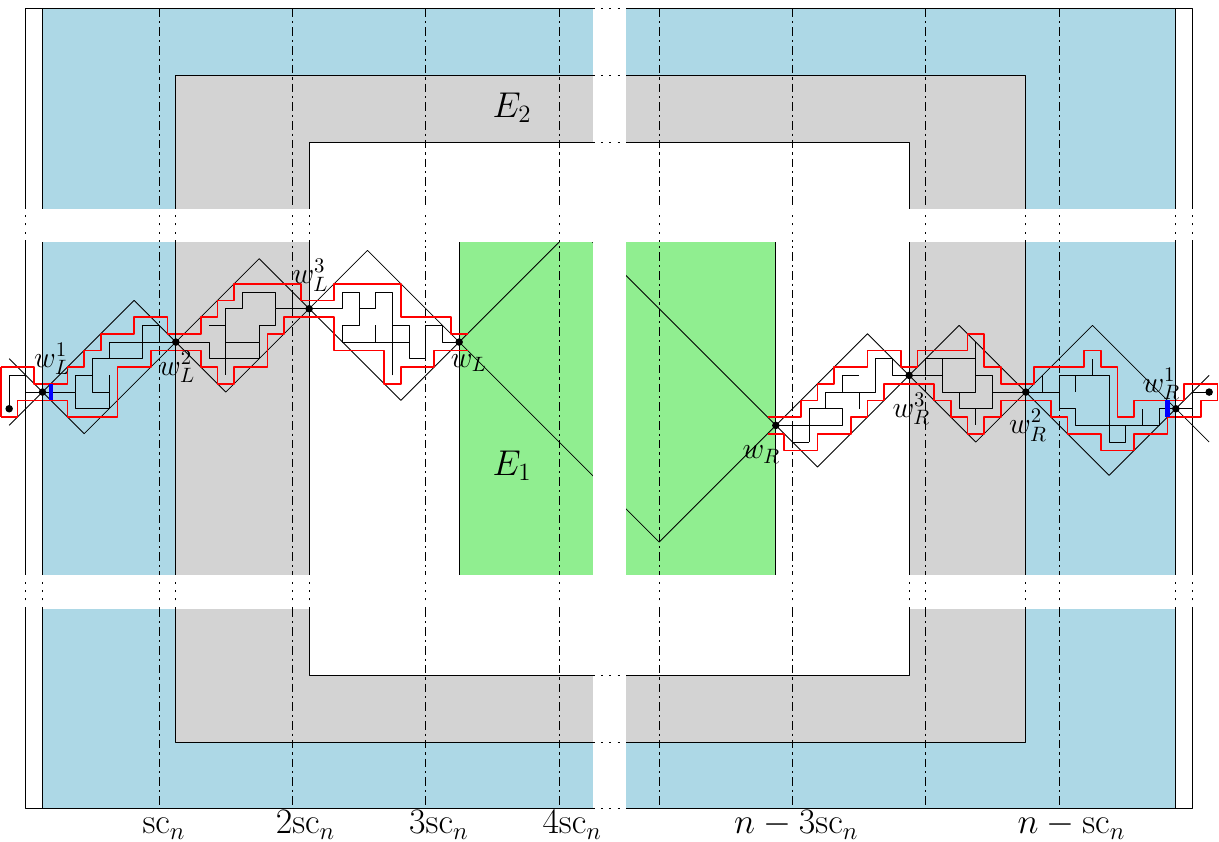}
		\caption{For \(\gamma_L\in \SetRootMarkBackContFV(\scale_n)\), \(\gamma_R\in \SetRootMarkForwContFV(\scale_n)\), \(\gamma_L,\gamma_R\) both contain at least \(\rho \scale_n\) cone-points in the blue and white annuli. The thick blue edges are the ones which states are changed in \(\tilde{A}_v(\gamma_L,\gamma_R)\). The red paths are the edges dual to \(\Env_L(\gamma_L)\), and to \(\Env_R(v_n-v+\gamma_R)\). A zoom on the relevant part of the box is made, white strips represent parts of the system that are not represented.}
		\label{Fig:InvPrinc_Decomp_Annuli}
	\end{figure}
	
	We first observe that by Hypotheses~\ref{hyp_fin_vol_connection:edge_state_swap}, and the same observation as in Claim~\ref{claim:chain_of_events},
	\begin{align}\label{eq:InvPrinc:prf_lem_density_swapp:edge_swap}
		\frac{\Phi_n\big( A_{w_L}^0(\gamma),\, \tilde{A}_{0,\rmE_n}^L(\gamma_L),\, \tilde{A}_{w_R,\rmE_n}^R(\gamma_R)\big)}{\Phi_n\big( A_{u}^0(\gamma),\, A_{0,\rmE_n}^L(\gamma_L),\, A_{w_R,\rmE_n}^R(\gamma_R)\big)}
		&=
		\stateSwapCst,
		\\
		\frac{\Phi\big( A_{u}^0(\gamma),\, \tilde{A}_{0}^L(\gamma_L),\, \tilde{A}_{w_R}^R(\gamma_R)\big)}{\Phi\big( A_{u}^0(\gamma),\, A_{0}^L(\gamma_L),\, A_{w_R}^R(\gamma_R)\big)}
		&=
		\Big(\frac{\sinh(2J)}{e^{-2J}-e^{-2U}}\Big)^2 =: c_{J,U}.\nonumber
	\end{align}
	Thus, as
	\begin{multline*}
		\Phi_n\big( \calC\in \goodCl_n(\scale_n),\, A_{0,\rmE_n}^L(\gamma_L),\, A_{w_R,\rmE_n}^R(\gamma_R)\big)
		\\=
		\sum_{\gamma'\in \SetRootDiaCont_{w_R-w_L}}\Phi_n\big( A_{w_L}^0(\gamma'),\, A_{0,\rmE_n}^L(\gamma_L),\, A_{w_R,\rmE_n}^R(\gamma_R)\big),
	\end{multline*}we have that
	\begin{multline}
		\label{eq:InvPrinc:prf_lem_density_swapp:cond_proba_expansion}
		\Phi_n\big( A_{w_L}^0(\gamma) \bgiven \calC\in \goodCl_n(\scale_n),\, A_{0,\rmE_n}^L(\gamma_L),\, A_{w_R,\rmE_n}^R(\gamma_R)\big)
		\\=
		\frac{\Phi_n\big( A_{w_L}^0(\gamma)\bgiven \tilde{A}_{0,\rmE_n}^L(\gamma_L),\, \tilde{A}_{w_R,\rmE_n}^R(\gamma_R)\big)}{\sum_{\gamma'\in \SetRootDiaCont_{w_R-w_L}}\Phi_n\big( A_{w_L}^0(\gamma') \bgiven \tilde{A}_{0,\rmE_n}^L(\gamma_L),\, \tilde{A}_{w_R,\rmE_n}^R(\gamma_R)\big)},
	\end{multline}and similarly for \(\Phi\). The Lemma will follow from the above observations and the next claim.
	
	\begin{claim}
		\label{claim:InvPrinc:prf_lem_density_swapp}
		There exist \(c>0\) and~\(n_0\geq 1\) such that, for any~\(n\geq n_0\) and any event \(B\) supported on edges with both endpoints in \(\Delta(w_L,w_R)\),
		\begin{equation*}
			\Big|\frac{\Phi_n\big( B\given \tilde{A}_{0,\rmE_n}^L(\gamma_L),\, \tilde{A}_{w_R,\rmE_n}^R(\gamma_R)\big)}{\Phi\big( B \given \tilde{A}_{0}^L(\gamma_L),\, \tilde{A}_{w_R}^R(\gamma_R)\big)} -1\Big|
			\leq
			e^{-c\scale_n}.
		\end{equation*}
	\end{claim}
	\begin{proof}
		We use Lemma~\ref{lem:app:mixing_to_ratioMixing} with \(\epsilon = e^{-c\scale_n}\) and the following inputs:
		\begin{itemize}
			\item \(\mu\) is the restriction of~\(\Phi_n\big( \cdot \given \tilde{A}_{0,\rmE_n}^L(\gamma_L),\, \tilde{A}_{w_R,\rmE_n}^R(\gamma_R)\big) \) to \(E_1\cup E_2\), \(\mu_1\) is the restriction of \(\mu\) to \(E_1\), and \(\mu_2\) is the restriction of \(\mu\) to \(E_2\).
			\item \(\nu\) is the restriction of~\(\Phi\big( \cdot \given \tilde{A}_{0}^L(\gamma_L),\, \tilde{A}_{w_R}^R(\gamma_R)\big)\) to \(E_1\cup E_2\), \(\nu_1\) is the restriction of \(\nu\) to \(E_1\), and \(\nu_2\) is the restriction of \(\nu\) to \(E_2\).
			\item \(D\) is the event that \(E_2\) contains a circuit of primal edges open in \(\omega_{\tau\tau'}\), and a circuit of dual edges open in \(\omega_{\tau}^*\) which both surround \(E_1\).
		\end{itemize}
		We need to show that the Hypotheses of Lemma~\ref{lem:app:mixing_to_ratioMixing} are fulfilled. Introduce the shorthands
		\begin{equation*}
			\tilde{A}_F(\gamma_L,\gamma_R) \equiv \tilde{A}_{0,F}^L(\gamma_L) \cap \tilde{A}_{w_R,F}^R(\gamma_R),
			\quad
			\tilde{A}(\gamma_L,\gamma_R) \equiv \tilde{A}_{0}^L(\gamma_L) \cap \tilde{A}_{w_R}^R(\gamma_R).
		\end{equation*}Let
		\begin{equation*}
			\tilde{\gamma}_L = \gamma_L\cap \diam(w_L^1,w_L),
			\quad
			\tilde{\gamma}_R = (\gamma_R + w_R)\cap \diam(w_R,w_R^1),
		\end{equation*}and
		\begin{equation*}
			\calE = \Env^+(\tilde{\gamma}_L) \cup \Env^-(\tilde{\gamma}_L) \cup \Env^+(\tilde{\gamma}_R) \cup \Env^-(\tilde{\gamma}_R) \cup \big\{\{w_L^1,w_L^1+\rme_1\},\{w_R^1,w_R^1-\rme_1\}\big\}.
		\end{equation*}
		Let \(N = 10^{10^{m_n+m_n'+n}}\) and set
		\begin{equation*}
			V_N = \{-N,\dots,N\}^2 \setminus (\Fill(\tilde{\gamma}_L)\cup \Fill(\tilde{\gamma}_R)),
			\quad
			Q = \atrc_{\bbE_{V_N}}^{1,1}\big(\cdot \bgiven \omega_{\tau}(e) = 0\, \forall e \in \calE\big).
		\end{equation*}
		As we will see below, any large enough integer as function of \(m_n,m_n',n\) would do.
		
		The proofs of Hypotheses 1 and 2 of Lemma~\ref{lem:app:mixing_to_ratioMixing} are very close to the proof of Claim~\ref{prf:ratio_mix_weights:str_mix_diam_meas}. We therefore refer to it after reducing the wanted claim to a statement analogous to the one present in the proof of Claim~\ref{prf:ratio_mix_weights:str_mix_diam_meas}.
		
		\smallskip
		
		\noindent\textbf{Hypotheses 1 of Lemma~\ref{lem:app:mixing_to_ratioMixing} holds:}
		let \(a,b,a',b'\in \{0,1\}^{E_1}\) with \(a\leq b\), \(a'\leq b'\). Use the shorthand \(F_n'= F_n\setminus E_1\) Then, let \(\pi\) be a monotone coupling of \(\atrc^{1,1}_{F_n'}\big(\cdot \given \tilde{A}_{F_n'}(\gamma_L,\gamma_R)\big)\), and \(\atrc^{0,0}_{F_n'}\big(\cdot \given \tilde{A}_{F_n'}(\gamma_L,\gamma_R)\big)\), and \((X,Y)\sim \pi\). Then, by Hypotheses~\ref{hyp_fin_vol_connection:FKG},~\ref{hyp_fin_vol_connection:stochastic_sandwich}, the measures \(\Phi_n\big( \cdot \given \tilde{A}_{\rmE_n}(\gamma_L,\gamma_R),\, \omega_{\tau}|_{E_1} = a,\, \omega_{\tau\tau'}|_{E_1} = b\big)\),\linebreak
		\(\Phi\big( \cdot \given \tilde{A}(\gamma_L,\gamma_R),\, \omega_{\tau}|_{E_1} = a,\, \omega_{\tau\tau'}|_{E_1} = b\big)\), as well as their versions with \(a',b'\) replacing \(a,b\), are stochastically dominated by \(\atrc^{1,1}_{F_n'}\big(\cdot \given \tilde{A}_{F_n'}(\gamma_L,\gamma_R)\big)\), and stochastically dominate \(\atrc^{0,0}_{F_n'}\big(\cdot \given \tilde{A}_{F_n'}(\gamma_L,\gamma_R)\big)\). Thus, the total variation distance between \(\Phi_n\big( \cdot \given \tilde{A}_{\rmE_n}(\gamma_L,\gamma_R),\, \omega_{\tau}|_{E_1} = a,\, \omega_{\tau\tau'}|_{E_1} = b\big)\) and \(\Phi_n\big( \cdot \given \tilde{A}_{\rmE_n}(\gamma_L,\gamma_R),\, \omega_{\tau}|_{E_1} = a',\, \omega_{\tau\tau'}|_{E_1} = b'\big)\) is less than \(\pi(X|_{E_2}\neq Y|_{E_2})\). The same holds for \(\Phi\). Then, by a union bound an monotonicity of \(\pi\),
		\begin{align*}
			&\pi(X|_{E_2}\neq Y|_{E_2})
			\\
			&\leq
			\sum_{e\in E_2} \Big(
			\atrc^{1,1}_{F_n'}\big(\omega_{\tau}(e)=1 \given \tilde{A}_{F_n'}(\gamma_L,\gamma_R)\big)
			- \atrc^{0,0}_{F_n'}\big(\omega_{\tau}(e)=1 \given \tilde{A}_{F_n'}(\gamma_L,\gamma_R)\big)
			\\&\qquad 
			+ \atrc^{1,1}_{F_n'}\big(\omega_{\tau\tau'}(e)=1 \given \tilde{A}_{F_n'}(\gamma_L,\gamma_R)\big)
			- \atrc^{0,0}_{F_n'}\big(\omega_{\tau\tau'}(e)=1 \given \tilde{A}_{F_n'}(\gamma_L,\gamma_R)\big) \Big).
		\end{align*}
		
		We then will use Theorem~\ref{thm:strong_mixing_atrc} to bound the last differences. Letting \(\calE'\) be the union of \(E_1\) with the set of edges with at least one endpoint in \(\{x\in \bbV_{F_n}^c:\ \rmd_{\infty}(x,\bbV_{F_n}) = 1 \}\), we have that (recall that \(Q = \atrc_{\bbE_{V_N}}^{1,1}\big(\cdot \bgiven \omega_{\tau}(e) = 0\, \forall e \in \calE\big)\)) for any event \(B\) supported on \(E_2\),
		\begin{gather*}
			\atrc^{1,1}_{F_n'}\big(B \given \tilde{A}_{F_n}(\gamma_L,\gamma_R)\big)
			=
			Q\big(B \bgiven \omega_{\tau}(e') = \omega_{\tau\tau'}(e') = 1\, \forall e'\in \calE'\big),
			\\
			\atrc^{0,0}_{F_n'}\big(B \given \tilde{A}_{F_n}(\gamma_L,\gamma_R)\big)
			=
			Q\big(B \bgiven \omega_{\tau}(e') = \omega_{\tau\tau'}(e') = 0\, \forall e'\in \calE'\big).
		\end{gather*}
		Theorem~\ref{thm:strong_mixing_atrc} can then be applied to \(Q\) to give that the differences are less than the probability in an inhomogeneous bloc percolation that \(e\) is connected to \(\calE'\), exactly as in the proof of Claim~\ref{prf:ratio_mix_weights:str_mix_diam_meas}. This probability is then less than \(e^{-c\scale_n}\) as both \(\gamma_L\) and \(w_R+\gamma_R\) contain at least \(\rho\scale_n\) cone points in the both slabs \(\slab_{w_L^1\cdot \rme_1,w_L^2\cdot \rme_1}, \slab_{w_L^3\cdot \rme_1,w_L\cdot \rme_1}\), and \(\slab_{w_R^2\cdot \rme_1,w_R^1\cdot \rme_1}, \slab_{w_R\cdot \rme_1,w_R^3\cdot \rme_1}\) respectively. The argument is then the same as in the proof of Claim~\ref{prf:ratio_mix_weights:str_mix_diam_meas}.
		
		\smallskip
		
		\noindent\textbf{Hypotheses 2 of Lemma~\ref{lem:app:mixing_to_ratioMixing} holds:} let \(\pi\) be a monotone coupling of \(\atrc^{1,1}_{F_n'}\big(\cdot \given \tilde{A}_{F_n'}(\gamma_L,\gamma_R)\big)\), and \(\atrc^{0,0}_{F_n'}\big(\cdot \given \tilde{A}_{F_n'}(\gamma_L,\gamma_R)\big)\), and \((X,Y)\sim \pi\). Then, by Hypotheses~\ref{hyp_fin_vol_connection:FKG},~\ref{hyp_fin_vol_connection:stochastic_sandwich}, the measures \(\Phi_n\big( \cdot \given \tilde{A}_{\rmE_n}(\gamma_L,\gamma_R)\big)\), and \(\Phi\big( \cdot \given \tilde{A}(\gamma_L,\gamma_R)\big)\) are stochastically dominated by \(\atrc^{1,1}_{F_n'}\big(\cdot \given \tilde{A}_{F_n'}(\gamma_L,\gamma_R)\big)\), and stochastically dominate \(\atrc^{0,0}_{F_n'}\big(\cdot \given \tilde{A}_{F_n'}(\gamma_L,\gamma_R)\big)\). Thus, \(\tvd(\mu_2,\nu_2) \leq \pi(X|_{E_2}\neq Y|_{E_2})\). The rest is the same as in the previous point.
		
		\smallskip
		
		\noindent\textbf{Hypotheses 3 of Lemma~\ref{lem:app:mixing_to_ratioMixing} holds:}
		let \((a,b)\in D\) compatible with \(\tilde{A}_{\rmE_n}(\gamma_L,\gamma_R)\). We have that by Hypotheses~\ref{hyp_fin_vol_connection:FKG}, Hypotheses~\ref{hyp_fin_vol_connection:stochastic_sandwich}, and the lattice-FKG property of ATRC, the measures \(\Phi_n\big( \cdot \given \tilde{A}_{\rmE_n}(\gamma_L,\gamma_R),\, (\omega_{\tau}, \omega_{\tau\tau'})|_{E_2} = (a,b)\big)\), and\linebreak \(\Phi\big( \cdot \given \tilde{A}(\gamma_L,\gamma_R),\, (\omega_{\tau},\omega_{\tau\tau'})|_{E_2} = (a,b)\big)\) are stochastically dominated by \linebreak \(\atrc^{1,1}_{F_n}\big(\cdot \given \tilde{A}_{F_n}(\gamma_L,\gamma_R),\, (\omega_{\tau},\omega_{\tau\tau'})|_{E_2} = (a,b)\big)\), and stochastically dominate \(\atrc^{0,0}_{F_n}\big(\cdot \given \tilde{A}_{F_n}(\gamma_L,\gamma_R),\, (\omega_{\tau},\omega_{\tau\tau'})|_{E_2} = (a,b)\big)\). The path decoupling property of \(\atrc\), see Lemma~\ref{lem:decoupling_paths_mATRC} and the two paragraphs before it, imply that these two ATRC measures agree on \(E_1\). So, the two measures restricted to \(E_1\) are equal, which is the ``conditional equality'' part of the third Hypotheses of Lemma~\ref{lem:app:mixing_to_ratioMixing}. The fact that \(D\) has probability at least \(1-e^{-c\scale_n}\) under \(\mu_2,\nu_2\) follows from Lemma~\ref{lem:InvPrinc:exp_dec} for \(\mu_2\), and Theorem~\ref{thm:Ale04_ATRC} for \(\nu_2\).
	\end{proof}
	
	By Claim~\ref{claim:InvPrinc:prf_lem_density_swapp}, and~\eqref{eq:InvPrinc:prf_lem_density_swapp:cond_proba_expansion}, we get, with \(\epsilon_n \equiv e^{-c\scale_n}\),
	\begin{align*}
		&\Phi_n\big( A_{u}^0(\gamma)\given A_{0,\rmE_n}^L(\gamma_L),\, A_{w_R,\rmE_n}^R(\gamma_R),\, \calC\in \goodCl_n(\scale_n)\big)
		\\
		&\qquad=
		\frac{\Phi_n\big( A_{w_L}^0(\gamma)\given \tilde{A}_{0,\rmE_n}^L(\gamma_L),\, \tilde{A}_{w_R,\rmE_n}^R(\gamma_R)\big)}{\sum_{\gamma'\in \SetRootDiaCont_{w_R-w_L}}\Phi_n\big( A_{w_L}^0(\gamma') \given \tilde{A}_{0,\rmE_n}^L(\gamma_L),\, \tilde{A}_{w_R,\rmE_n}^R(\gamma_R)\big)}
		\\
		&\qquad\leq
		\frac{1+\epsilon_n}{1-\epsilon_n}\frac{\Phi\big( A_{w_L}^0(\gamma)\given \tilde{A}_{0}^L(\gamma_L),\, \tilde{A}_{w_R}^R(\gamma_R)\big)}{\sum_{\gamma'\in \SetRootDiaCont_{w_R-w_L}}\Phi\big( A_{w_L}^0(\gamma') \given \tilde{A}_{0}^L(\gamma_L),\, \tilde{A}_{w_R}^R(\gamma_R)\big)}
		\\
		&\qquad=
		\frac{1+\epsilon_n}{1-\epsilon_n}\Phi\big( A_{u}^0(\gamma)\given A_{0}^L(\gamma_L),\, A_{w_R}^R(\gamma_R),\, \calC\in \goodCl_n(\scale_n)\big),
	\end{align*}which is half of the wanted claim. The other half follows the same computation with \(\leq\) replaced by \(\geq\), and \(\frac{1+\epsilon_n}{1-\epsilon_n}\) replaced by \(\frac{1-\epsilon_n}{1+\epsilon_n}\).
\end{proof}

From this Lemma and trivial algebra, we obtain the following: for \(n\) large enough, \(\scale_n\) as in Lemma~\ref{lem:InvPrinc:density_swapping}, one has that for any \(A\subset \goodCl_n(\scale_n)\),
\begin{equation}
	\label{eq:InvPrinc:Phi_n_to_cond_Phi}
	\begin{aligned}
		&\Phi_n\big(\calC\in A\given \calC\in \goodCl_n\big)
		\\
		&\qquad\lesseqgtr(1\pm e^{-c\scale_n})\sum_{\gamma_L\in \SetRootMarkBackContFV}\sum_{\gamma_R\in \SetRootMarkForwContFV}
		\Phi_n\big(A_{0,\rmE_n}^L(\gamma_L),\, A_{v_n-\displace(\gamma_R),\rmE_n}^R(\gamma_R) \bgiven \calC\in \goodCl_n\big)
		\\
		&\phantom{\lesseqgtr(1\pm e^{-c\scale_n})\sum_{\gamma_L\in \SetRootMarkBackContFV}\sum_{\gamma_R\in \SetRootMarkForwContFV}}\cdot \Phi\big( \calC \in A \bgiven \calC\in \goodCl_n,\, A_{0}^L(\gamma_L),\, A_{v_n-\displace(\gamma_R)}^R(\gamma_R) \big)
	\end{aligned}
\end{equation}where \(\SetRootMarkBackContFV\equiv \SetRootMarkBackContFV(\scale_n)\), \(\SetRootMarkForwContFV\equiv \SetRootMarkForwContFV(\scale_n)\), and \(\goodCl_n\equiv \goodCl_n(\scale_n)\).

\subsection{Proof of Theorem~\ref{thm:main_Inv_principle_coupling_with_RW}}
\label{subsec:InvPrinc:prf_coupling_RWB}

The proof of Theorem~\ref{thm:main_Inv_principle_coupling_with_RW} will build around~\eqref{eq:InvPrinc:Phi_n_to_cond_Phi}, and the OZ decomposition of~\eqref{eq:def:measure_sequence}. For \(\gamma_L\in \SetRootMarkBackCont\), \(\gamma_R\in\SetRootMarkForwCont\), and \(C\) a realization of \(\calC\) under \(\Phi\) containing~\(0\) and~\(v_n\), denote
\begin{itemize}
	\item \(C\sim \gamma_L\) if \(C\cap (\displace(\gamma_L) + \bcone) = \gamma_L\) and \(\displace(\gamma_L)\in \CPts(C)\),
	\item \(C\sim \gamma_R\) if \(C\cap (v_n-\displace(\gamma_R) + \fcone) = \gamma_R\) and \(v_n - \displace(\gamma_R)\in \CPts(C)\).
\end{itemize}
Moreover, we will use the shorthand for sequences:
\begin{equation*}
	(x_i,x_{i+1},\dots,x_j) \equiv x_{i}^j.
\end{equation*}

For \(\eta_L\in\SetRootMarkBackCont,\, \eta_R\in \SetRootMarkForwCont\), introduce the probability measure \(\OZDecompmeas_{\eta_L,\eta_R}^{v_n}\) on the same space as \(\OZDecompmeas\) via: for any event \(B\),
\begin{equation}
	\label{eq:InvPrinc:def:OZDecCond}
	\OZDecompmeas_{\eta_L,\eta_R}^{v_n}(B)
	\coloneqq
	\frac{\int d\OZDecompmeas(M;\gamma_0^{M+1}) \mathds{1}_{\gamma_0^{M+1}\in B} \mathds{1}_{\displace(\bar{\gamma}) = v_n}\mathds{1}_{\bar{\gamma}\sim \eta_L} \mathds{1}_{\bar{\gamma}\sim \eta_R} }{\int d\OZDecompmeas(M;\gamma_0^{M+1}) \mathds{1}_{\displace(\bar{\gamma}) = v_n}\mathds{1}_{\bar{\gamma}\sim \eta_L} \mathds{1}_{\bar{\gamma}\sim \eta_R}},
\end{equation}where \(\OZDecompmeas\) is the measure introduced in~\eqref{eq:def:OZDecompmeas}, and
where \(\bar{\gamma} = \gamma_0\concatenate \dots \concatenate \gamma_{M+1}\). Also, write
\begin{equation*}
	\OZDecompmeas_{\eta_L,\eta_R}^{v_n} \circ \bar{\gamma}^{-1}
\end{equation*}the measure induced on connected subset of \(\Z^2\) containing \(0\) by concatenation of the graph sequence sampled under \(\OZDecompmeas_{\eta_L,\eta_R}^{v_n}\).
Then, note that for \(\gamma_L\in \SetRootMarkBackContFV(\scale_n)\), \(\gamma_R\in \SetRootMarkForwContFV(\scale_n)\),
\begin{align*}
	\{\calC\in \goodCl_n(\scale_n)\}\cap A_{0}^L(\gamma_L)\cap A_{v_n-\displace(\gamma_R)}^R(\gamma_R)
	&=
	\bigcup_{\gamma\in \SetRootDiaCont_{v_n-\displace(\gamma_L)-\displace(\gamma_R)}}\{\calC = \gamma_L\concatenate \gamma\concatenate\gamma_R\}
	\\
	&=
	\big\{\calC\in \goodCl_n(\scale_n),\, \calC\sim \gamma_L,\, \calC\sim \gamma_R\big\}.
\end{align*}

\begin{lemma}
	\label{lem:InvPrinc:cond_Phi_to_OZDecCond}
	Let \(0<J<U\) satisfy \(\sinh 2J=e^{-2U}\). Then, there exist \(n_0\geq 1\), \(c>0\), and~\(\epsilon_0>0\), such that for any \(n\geq n_0\), and any \(\eta_L\in \SetRootMarkBackCont,\eta_R\in \SetRootMarkForwCont\) with \(\norm{\displace(\eta_L)}_{\infty},\norm{\displace(\eta_R)}_{\infty} \leq \epsilon_0 \sqrt{n}\),
	\begin{equation*}
		\tvd\Big( \Phi\big( \calC\in \cdot \bgiven \goodCl_n(\scale_n),\, \calC\sim \eta_L,\, \calC\sim \eta_R \big),\, \OZDecompmeas_{\eta_L,\eta_R}^{v_n}\circ \bar{\gamma}^{-1} \Big)
		\leq
		e^{-cn}.
	\end{equation*}
\end{lemma}
\begin{proof}
	First, by Theorem~\ref{thm:OZ_for_ATRC_infinite_vol}, there exists \(c>0\) such that for any \(n\) large enough, and any set \(B\) of realisations of \(\calC\),
	\begin{equation}
		\label{eq:prf_lem_cond_Phi_to_OZDecCond:OZ_coulping}
		\Big|e^{\nu_1 v_n\cdot \rme_1}\Phi\big( \calC\in B,\, v_n\in \calC\big)
		-
		\int d\OZDecompmeas(M;\gamma_{0}^{M+1}) \mathds{1}_{\displace(\bar{\gamma}) = v_n} \mathds{1}_{\bar{\gamma}\in B}\Big|
		\leq
		e^{-cn}
	\end{equation}where \(\bar{\gamma} = \gamma_0\concatenate\dots\concatenate \gamma_{M+1}\). Then, by definition of \(\OZDecompmeas,\OZWalkmeas\), see~\eqref{eq:def:measure_sequence}, and the ``finite energy'' properties of the probability measures \(p_L,p_R,p\), see Theorem~\ref{thm:OZ_for_ATRC_infinite_vol}, for some \(c'>0\), one has
	\begin{equation*}
		\int d\OZDecompmeas(M;\gamma_0^{M+1}) \mathds{1}_{\displace(\bar{\gamma}) = v_n}\mathds{1}_{\bar{\gamma}\sim \eta_L} \mathds{1}_{\bar{\gamma}\sim \eta_R}
		\geq
		\rmC e^{-c'(|\eta_L|+|\eta_R|)} \OZWalkmeas(T_{v_n-\displace(\eta_L)-\displace(\eta_R)}<\infty).
	\end{equation*}Finally, by a direct random walk computation using the Local Limit Theorem in dimension 2, see for example~\cite[Section 8.1]{AouOttVel24}, and the fact that \(\norm{\displace(\eta_L)+\displace(\eta_R)}_{\infty} \leq 2\epsilon_0\sqrt{n}\),
	\begin{equation*}
		\OZWalkmeas(T_{v_n-\displace(\eta_L)-\displace(\eta_R)}<\infty)
		\geq
		\tfrac{C}{\sqrt{n}},
	\end{equation*}for some \(C>0\). Combining, we get, for \(n\) large enough,
	\begin{equation*}
		\int d\OZDecompmeas(M;\gamma_0^{M+1}) \mathds{1}_{\displace(\bar{\gamma}) = v_n}\mathds{1}_{\bar{\gamma}\sim \eta_L} \mathds{1}_{\bar{\gamma}\sim \eta_R}
		\geq
		e^{-c' \epsilon_0^2 n}
	\end{equation*}for some \(c'>0\).
	Now, by~\eqref{eq:prf_lem_cond_Phi_to_OZDecCond:OZ_coulping}, for any \(B\),
	\begin{multline*}
		\Phi\big( \calC\in B \bgiven \goodCl_n(\scale_n),\, \calC \sim \eta_L,\, \calC \sim \eta_R \big)
		\\
		\lesseqgtr
		\frac{\pm e^{-cn} + \int d\OZDecompmeas(M;\gamma_0^{M+1}) \mathds{1}_{\bar{\gamma}\in B} \mathds{1}_{\displace(\bar{\gamma}) = v_n}\mathds{1}_{\bar{\gamma}\sim \eta_L} \mathds{1}_{\bar{\gamma}\sim \eta_R} }{\mp e^{-cn} + \int d\OZDecompmeas(M;\gamma_0^{M+1}) \mathds{1}_{\displace(\bar{\gamma}) = v_n}\mathds{1}_{\bar{\gamma}\sim \eta_L} \mathds{1}_{\bar{\gamma}\sim \eta_R}}.
	\end{multline*}
	So, if \(\epsilon_0>0\) is small enough, for any \(n\) large enough,
	\begin{align*}
		(1-e^{-cn/2})\OZDecompmeas_{\eta_L,\eta_R}^{v_n}(B) - e^{-cn/2}
		&\leq
		\Phi\big( \calC\in B \bgiven \goodCl_n(\scale_n),\, \calC \sim\eta_L,\, \calC \sim\eta_R \big)
		\\
		&\leq
		e^{-cn/2} + (1+e^{-cn/2})\OZDecompmeas_{\eta_L,\eta_R}^{v_n}(B),
	\end{align*}
	which is the wanted claim.
\end{proof}

For \(\eta_L\in \SetRootMarkBackCont,\eta_R\in \SetRootMarkForwCont\), and \(\gamma_0^{M+1}\sim \OZDecompmeas_{\eta_L,\eta_R}^{v_n}\), define the random variables
\begin{align*}
	\rmT_{\eta_L} &= 
	\begin{cases}
		\min\{k\geq 0:\ \eta_L\subset \gamma_0\concatenate\dots\concatenate\gamma_k \} & \text{ if }  \gamma_0\concatenate\dots\concatenate \gamma_{M+1} \sim \eta_L
		\\
		\dagger & \text{ else}
	\end{cases},
	\\
	\rmT_{\eta_R} &= \begin{cases}
		\max\{0\leq k\leq M+1:\ \eta_R\subset \gamma_k\concatenate\dots\concatenate\gamma_{M+1} \} & \text{ if }  \gamma_0\concatenate\dots\concatenate \gamma_{M+1} \sim \eta_R
		\\
		\dagger & \text{ else}
	\end{cases},
\end{align*}
and the events for \(K\geq 0\),
\begin{align*}
	\Forget_{\eta_L,\eta_R}^{v_n}(K)
	=
	&\big\{\rmT_{\eta_L}\neq \dagger,\, \norm{\displace(\gamma_0\concatenate\dots\concatenate \gamma_{\rmT_{\eta_L}})}_{\infty} \leq K \big\}
	\\
	\cap 
	&\big\{\rmT_{\eta_R}\neq \dagger,\, \norm{\displace(\gamma_{\rmT_{\eta_R}}\concatenate\dots\concatenate \gamma_{M+1})}_{\infty} \leq K\big\}.
\end{align*}
	
Now, for \(K\geq \max(\norm{\displace(\eta_L)}_{\infty},\norm{\displace(\eta_R)}_{\infty})\), and \(n\) large enough, note that for any function \(f\)
\begin{multline}
	\label{eq:InvPrinc:OZDec_forget_bnd_pieces}
	\OZDecompmeas_{\eta_L,\eta_R}^{v_n}\big( f(\bar{\gamma}_0^{M+1}) \,;\, \Forget_{\eta_L,\eta_R}^{v_n}(K) \big)=
	\\
	\sum_{\substack{\gamma_L\in\SetRootMarkBackCont\\ \norm{\displace(\gamma_L)}_{\infty}\leq K }}
	\sum_{\substack{\gamma_R\in\SetRootMarkForwCont\\ \norm{\displace(\gamma_R)}_{\infty}\leq K }}
	\sum_{k,l\geq 0} \OZDecompmeas_{\eta_L,\eta_R}^{v_n}\big( (\rmT_{\eta_L},\rmT_{\eta_R}) = (k,l),\, \bar{\gamma}_0^{k} = \gamma_L,\, \bar{\gamma}_{M+1-l}^{M+1} = \gamma_R\big)
	\\
	\cdot\OZmeas_{v_n-\displace(\gamma_L)-\displace(\gamma_R)}\big( f(\gamma_L\concatenate\tilde{\gamma}_1\concatenate\dots\concatenate \tilde{\gamma}_M \concatenate \gamma_{R})\big),
\end{multline}where \(\tilde{\gamma}_1^{M}\sim \OZmeas_{v_n-\displace(\gamma_L)-\displace(\gamma_R)}\), and \(\bar{x}_i^j = x_i\concatenate\dots\concatenate x_j\). Note that, after conditioning on \(\Forget_{\eta_L,\eta_R}^{v_n}(K)\), this is precisely a ``bridge decomposition'' for the probability measure \(\OZDecompmeas_{\eta_L,\eta_R}^{v_n}( \cdot \given \Forget_{\eta_L,\eta_R}^{v_n}(K) )\), which is what we are after. To conclude the proof of Theorem~\ref{thm:main_Inv_principle_coupling_with_RW}, the idea is to prove that \(\Forget_{\eta_L,\eta_R}^{v_n}(C\scale_n^2)\) has very large probability once \(C\) is taken large enough and \(\scale_n\) is as in the statement of Theorem~\ref{thm:main_Inv_principle_coupling_with_RW}. This is the content of the next Lemma.

\begin{lemma}
	\label{lem:InvPrinc:OZDecCond_to_bridgeMixture}
	Let \(0<J<U\) satisfy \(\sinh 2J=e^{-2U}\). Then, there exist \(n_0\geq 1\), \(c>0\), and~\(\Cl[ImpCst]{cst:scaleForget}\geq 1\), such that for any \(n\geq n_0\), any \(\ln(n)\leq K\leq \sqrt{n}\), and any \(\eta_L\in \SetRootMarkBackCont,\eta_R\in \SetRootMarkForwCont\) with \(\norm{\displace(\eta_L)}_{\infty} \leq K\), \(\norm{\displace(\eta_R)}_{\infty} \leq K\),
	\begin{equation*}
		\OZDecompmeas_{\eta_L,\eta_R}^{v_n}\big( \Forget_{\eta_L,\eta_R}^{v_n}(\Cr{cst:scaleForget}K^2) \big)
		\geq
		1- e^{-cK^2}.
	\end{equation*}
\end{lemma}
\begin{proof}
	Recall~\eqref{eq:InvPrinc:def:OZDecCond}. We will use this expression with \(B = (\Forget_{\eta_L,\eta_R}^{v_n}(\Cr{cst:scaleForget}K^2))^c\) by upper bounding the numerator and lower bounding the denominator. We start with the latter task:
	\begin{multline*}
		\int d\OZDecompmeas(M;\gamma_0^{M+1}) \mathds{1}_{\displace(\bar{\gamma}) = v_n}\mathds{1}_{\bar{\gamma}\sim \eta_L} \mathds{1}_{\bar{\gamma}\sim \eta_R}
		\\
		\geq
		\rmC e^{-c'(|\eta_L|+|\eta_R|)} \OZWalkmeas(T_{v_n-\displace(\eta_L)-\displace(\eta_R)}<\infty)
		\geq
		\tfrac{C}{\sqrt{n}}e^{-c'K^2 },
	\end{multline*}for some \(C>0\), \(c'\geq 0\) as in the proof of Lemma~\ref{lem:InvPrinc:cond_Phi_to_OZDecCond}. We then turn to the upper bound on the numerator. Note that under \((\Forget_{\eta_L,\eta_R}^{v_n}(\Cr{cst:scaleForget}K^2))^c\), one of the elements of \(\gamma_0^{M+1} \) must have a displacement with sup-norm at least \(\Cr{cst:scaleForget}K^2 -K\) as \(\norm{\displace(\eta_L)}_{\infty}, \norm{\displace(\eta_R)}_{\infty} \leq K\). Moreover, as \(\norm{v_n-(n,0)}_{\infty}\leq 2\), and \(\displace(\gamma) \cdot \rme_1 \geq 1\) \(p\)-a.s., \(M\leq n+2\) with full measure. Thus, as \(p,p_L,p_R\) all have exponential tails,
	\begin{multline*}
		\int d\OZDecompmeas(M;\gamma_0^{M+1}) \mathds{1}_{(\Forget_{\eta_L,\eta_R}^{v_n}(\Cr{cst:scaleForget}K^2))^c}(\bar{\gamma}) \mathds{1}_{\displace(\bar{\gamma}) = v_n}\mathds{1}_{\bar{\gamma}\sim \eta_L} \mathds{1}_{\bar{\gamma}\sim \eta_R}
		\\
		\leq
		(n+4) e^{-c(\Cr{cst:scaleForget}K^2 -K)}
		\leq
		e^{-c\Cr{cst:scaleForget} K^2}
	\end{multline*}as soon as \(\Cr{cst:scaleForget}\) is taken large enough. Thus, combining everything,
	\begin{equation*}
		\OZDecompmeas_{\eta_L,\eta_R}^{v_n}\big( (\Forget_{\eta_L,\eta_R}^{v_n}(\Cr{cst:scaleForget}K^2))^c \big)
		\leq
		\tfrac{\sqrt{n}}{C}e^{c'K^2 }e^{-c\Cr{cst:scaleForget} K^2}.
	\end{equation*}Taking \(\Cr{cst:scaleForget}\) large enough gives the wanted claim.
\end{proof}

We are now ready to define \(\MixMeas_n\). Rather than writing the explicit expression, we describe how to sample from it. Let \(\max(\Cr{cst:scaleGCl}, \Cr{cst:minMixScale})\ln(n) \leq \scale_n \leq \min(\epsilon_0,1) \sqrt{n}\) with \(\epsilon_0\) given by Lemma~\ref{lem:InvPrinc:cond_Phi_to_OZDecCond}, and assume \(n\) large enough and \(m_n,m_n'\geq \Cr{cst:scaleMinBox} n\). For short, let \(\goodCl_n \equiv \goodCl_n(\scale_n)\).
\begin{itemize}
	\item First, sample \(\calK_{L},\calK_R\) under \(\Phi_n(\cdot \given \calC\in \goodCl_n)\);
	recall Section~\ref{subsec:InvPrinc:good_cluster} for their definition). They satisfy \(\norm{\displace(\calK_L)}_{\infty}\leq 5\scale_n\) a.s., and the same for \(\calK_R\), by construction.
	\item Then, sample \(\gamma_0^{M+1}\) under \(\OZDecompmeas_{\calK_L,\calK_R}^{v_n}(\cdot \given \Forget_{\calK_L,\calK_R}^{v_n}(\Cr{cst:scaleForget} \scale_n^2))\) and set \(\calK_L'= \gamma_0\concatenate\dots\concatenate\gamma_{\rmT_{\calK_L}}\), \(\calK_R'=\gamma_{\rmT_{\calK_R}} \concatenate \dots\concatenate\gamma_{M+1}\). They satisfy \(\norm{\displace(\calK_L')}_{\infty} \leq \Cr{cst:scaleForget} \scale_n^2\) by construction.
\end{itemize}\(\MixMeas_n\) is then the law of \((\calK_L',\calK_R')\).

Theorem~\ref{thm:main_Inv_principle_coupling_with_RW} now follows from Lemmas~\ref{lem:InvPrinc_cluster_is_nice},~\ref{lem:InvPrinc:density_swapping},~\ref{lem:InvPrinc:cond_Phi_to_OZDecCond},~\ref{lem:InvPrinc:OZDecCond_to_bridgeMixture}, and~\eqref{eq:InvPrinc:OZDec_forget_bnd_pieces}, and some routine considerations. Indeed, first note that by Lemmas~\ref{lem:InvPrinc_cluster_is_nice}, and~\ref{lem:InvPrinc:weak_LB},
\begin{equation*}
	\Phi_{n}(\calC\notin \goodCl_n \given v_n\in \calC)
	=
	\frac{\Phi_{n}(\calC\notin \goodCl_n,\, v_n\in \calC)}{\Phi_{n}(v_n\in \calC)}
	\leq
	C\sqrt{n} e^{-\nu_1n -c\scale_n}e^{\nu_1n}
	\leq
	e^{-c\scale_n}
\end{equation*}for \(n\) large enough.
So, the total variation distance between \(\Phi_n(\cdot \given v_n\in \calC)\) and \(\Phi_n(\cdot \given \calC\in \goodCl_n)\) is at most \(e^{-c\scale_n}\), see Lemma~\ref{lem:tot_var_event_norm} if more details are needed. Then, by Lemma~\ref{lem:InvPrinc:density_swapping}, see~\eqref{eq:InvPrinc:Phi_n_to_cond_Phi}, the total variation distance between \(\Phi_n(\calC \in \cdot \given \calC\in \goodCl_n)\) and
\begin{equation*}
	\sum_{\eta_L,\eta_R} \Phi_n\big(\calK_L=\eta_L,\,\calK_R=\eta_R \given \calC\in \goodCl_n\big) \Phi\big(\calC\in \cdot \given \calC\in \goodCl_n,\, \eta_L\sim \calC,\eta_R\sim \calC \big)
\end{equation*}is at most \(e^{-c\scale_n}\). Now, from Lemma~\ref{lem:InvPrinc:cond_Phi_to_OZDecCond}, we get that this last measure is at total variation distance at most \(e^{-cn}\) from the push-forward of
\begin{equation*}
	\sum_{\eta_L,\eta_R} \Phi_n\big(\calK_L=\eta_L,\,\calK_R=\eta_R \given \calC\in \goodCl_n\big) \OZDecompmeas_{\eta_L,\eta_R}^{v_n}
\end{equation*}by \(\gamma_0^{M+1}\mapsto \gamma_0\concatenate\dots\concatenate \gamma_{M+1}\). Finally, by Lemma~\ref{lem:InvPrinc:OZDecCond_to_bridgeMixture}, and direct considerations on total variation distance, see again Lemma~\ref{lem:tot_var_event_norm} in case of need, this last measure is at total variation distance at most \(e^{-c\scale_n^2}\) from the pushforward of
\begin{equation*}
	\sum_{\eta_L,\eta_R} \Phi_n\big(\calK_L=\eta_L,\,\calK_R=\eta_R \given \calC\in \goodCl_n\big) \OZDecompmeas_{\eta_L,\eta_R}^{v_n}(\cdot \given \Forget_{\eta_L,\eta_R}(\Cr{cst:scaleForget} \scale_n^2)).
\end{equation*}Theorem~\ref{thm:main_Inv_principle_coupling_with_RW} thus follows from this and a look at~\eqref{eq:InvPrinc:OZDec_forget_bnd_pieces}.

\subsection{Convergence to the Brownian Bridge}
\label{subsec:InvPrinc:BBconv}

A general result of Kovchegov~\cite{Kov04} allows to deduce the invariance principle from Theorem~\ref{thm:main_Inv_principle_coupling_with_RW}.
We start by introducing some notation.
Recall the constant \(C_0\) and the laws \(p\), \(\overline{\MixMeas}_{n,m_n}\) from Theorem~\ref{thm:main_Inv_principle_coupling_with_RW}.
Let \(m_n\) be any sequence with \(m_n\geq C_0 n\), \(p_{*}\) be a probability measure on \(\Z^2\cap \fcone\) defined as the push-forward of \(p\) by \(\displace\), \(q_*^{n}\) be a probability measure on \(\big(\{-\lceil c_0 \ln^2(n)\rceil,\dots, \lceil c_0 \ln^2(n)\rceil\}^2\cap \fcone\big)^2\) defined as the push-forward of \(\overline{\MixMeas}_{n,m_n}\) by \(\displace\).

Define all random variables on a common probability space \((\Omega,\calF,\bbP)\). Let \(X_1,X_2,\dots\) be an i.i.d. sequence of random variables with law \(p_*\). Let
\begin{equation*}
	\alpha = E(X_1\cdot \rme_1),\quad \sigma^2_x = E((X_1\cdot \rme_1)^2)-\alpha^2, \quad \sigma^2_y = E((Y_1\cdot\rme_2)^2),\quad \chi= \frac{\sigma^2_y}{\alpha}.
\end{equation*}
Since~\(\Phi\) has a reflection symmetry, so do \(p\) and \(p_*\)).
Then, the mean vector and covariance matrix of \(p_*\) are given by
\begin{equation}
\label{eq:mean_Cov_matrix}
	\bar{\mu} = \begin{pmatrix}
		\alpha \\
		0
	\end{pmatrix},
	\quad
	\CovMat = \begin{pmatrix}
		\sigma^2_x & 0\\
		0 & \sigma^2_y
	\end{pmatrix}.
\end{equation}
Define
\begin{equation}
	\label{eq:def:directed_random_walk}
	S_0:=(0,0),\quad S_{k} := S_{k-1} + X_k.
\end{equation}
Let \((V,W)\sim q^n_*\) be independent of the \(X_i\)'s. Define
\begin{equation*}
	T_n := \min\big\{k\geq 1: \ S_k =  v_R-v_L-V-W\big\},
\end{equation*}
where \(T_n := \infty\) when the set is empty.
When \(T_n<\infty\), let~$Z_0:= (0,0)$, $Z_1:=V$, and
\[
	Z_k:=
	\begin{cases}
		V+S_{k-1}, & \text{ if } k=2,\dots T_n+1,\\
		Z_{T_n + 1}+W = v_R-v_L & \text{ if } k=T_n + 2.
	\end{cases}
\]
Further, for \(s\in [0,2n+2]\), define
\begin{equation*}
	\tau_-(s) := \max\{k\geq 0:\ Z_k\cdot\rme_1 \leq s \},\quad \tau_+(s) := \min\{k\geq 0:\ Z_k\cdot\rme_1 \geq s \}.
\end{equation*}
and the linear interpolation:
\begin{equation*}
	\LinInt_s(Z) = \frac{Z_{\tau_+(s)}\cdot \rme_1 - s}{(Z_{\tau_+(s)}-Z_{\tau_-(s)})\cdot \rme_1}  Z_{\tau_-(s)}\cdot \rme_2 + \frac{ s- Z_{\tau_-(s)}\cdot \rme_1 }{(Z_{\tau_+(s)}-Z_{\tau_-(s)})\cdot \rme_1} Z_{\tau_+(s)}\cdot \rme_2
\end{equation*}
If \(T_{n}=\infty\), set \(\LinInt_s(Z) :\equiv 0\). 
Define also a re-scaled version of \((\LinInt_s(Z))_{s\in [0,2n+2]}\):
\begin{equation*}
	\Gamma_t^n = \Gamma_t^n(Z) = \frac{1}{\sqrt{2n \chi}}\LinInt_{t(2n+2)}(Z),\quad t\in[0,1].
\end{equation*}

The next Lemma concludes our study of the invariance principle for the (modified) ATRC cluster.
\begin{lemma}
\label{lem:InvPrinc_mATRC}
	The sequence of random functions \((\Gamma_t^n)_{t\in [0,1]}\) under~$\bbP$ conditioned on~$T_n<\infty$, converges weakly to a standard Brownian Bridge. Moreover, with probability going to \(1\) as \(n\to \infty\), the maximal step of \(Z\) has norm less than \(C\ln^2(n)\) for some \(C\geq 0\). 
\end{lemma}
\begin{proof}
	\cite[Technical Theorem]{Kov04} proves a general result about convergence of directed random walk bridges in dimension~$d$ to the graph of a Brownian bridge in dimension~$d-1$.
	In our setup, this result gives the following.
	Consider a directed random walk with iid steps having exponential tails and mean equal to~$a\rme_1$ with~$a>0$.
	Now, look at a sequence of bridges for this walk from~$0$ to~$x_n\in \Z^2$.
	Then, for any~$\varepsilon >0$, the sequence of the bridges converges, under linear scaling in coordinate~$\rme_1$ and diffusive scaling in coordinate~$\rme_2$, to the graph of a~$1D$ Brownian bridge uniformly in~$\norm{x_n - (n,0)} \leq n^{1/2-\varepsilon}$.
	
	Our case satisfies these conditions: Theorem~\ref{thm:main_Inv_principle_coupling_with_RW} gives the wanted properties for the steps; 
	the walk starts at~$V$ and ends at~$W$, where~$\norm{V}$ and~$\norm{W-(2n,0)}$ are a.s. bounded by~$c\ln^2 n$.
	Thus, we get the convergence.
	The estimate on the maximal step size is a simple large deviation estimate combined with the local limit theorem.
\end{proof}

\begin{remark}
	\label{rem:diffusivity_constant}
	Note that the value of the diffusivity constant \(\chi\) can be identified with the curvature of \(\partial\calW\) at \(\nu(\rme_1)\rme_1\), exactly as in~\cite[Appendix B]{IofOttShlVel22}.
\end{remark}

\section{Proofs of Theorems~\ref{thm:potts}, \ref{thm:order-disorder:FK_loop},~\ref{thm:oz-at} and~\ref{thm:oz-atrc}}
\label{sec:proofs_thms}

Previous sections establish coupling between the FK-percolation and the (modified) ATRC models and several results for the latter: mixing properties, Ornstein--Zernike theory, invariance principle for the interface under suitable Dobrushin boundary conditions.
In order to transfer the invariance principle to the FK-percolation, we need to show proximity of interfaces in the two models.

Consider the coupling measure \(\Psi^{1/0}\) constructed in Section~\ref{sec:coupling:interfaces}.
A crucial feature of this coupling is that the respective interfaces stay close to each other with high probability. To measure the distance between interfaces, we will work with the \emph{one-sided Hausdorff distance} defined by
\begin{equation*}
\rmd_{\mathrm{H}}(R,S):=\sup_{x\in R}\inf_{y\in S}\rmd_\infty(x,y),\qquad R,S\subset\R^2.
\end{equation*}
Given a bi-infinite connected set \(C\subset\bbL_\bullet\cap  (\R\times [-c,c])\), we say that a subset \(R\subset\R^2\) is (weakly) \emph{above} \(C\) if it is contained in (the closure of) the connected component of the point~\((0,c+1)\) in \(\R^2\setminus C\), where we identify \(C\) with the union of line segments between the endpoints of the edges in \(\bbE_C\). We say that \(R\) is (weakly) \emph{below} \(C\) if it is contained in (the closure of) the connected component of the point~\((0,-c-1)\) in \(\R^2\setminus C\). We make the analogous definitions for finite connected sets by extending them to bi-finite connected sets by attaching left-infinite horizontal lines to all the leftmost points of the finite set and right-infinite horizontal lines to all the rightmost points; and, analogously, for connected subsets \(C\subset\bbL_\circ\).
Recall the upper and lower envelopes~\(\Gamma_{\fk}^{\pm,n}\) in \(G_n\) defined in Section~\ref{sec:intro}.
Define~\(\Gamma_{\fk}^{\pm,n,m}\) in~\(G=G_{n,m}\) analogously.
Recall also the definition of~\(\Gamma_{\atrc}^{n,m}\) in Section~\ref{sec:interface_matrc}.

\begin{lemma}\label{lem:closeness_interfaces_fk_matrc}
There exist constants \(C,c>0\) such that, for any \(n,m,k\geq 1\), 
\begin{equation*}
\Psi^{1/0}(\rmd_{\mathrm{H}}(\Gamma_{\fk}^{\pm,n,m},\Gamma_{\atrc}^{n,m})>k)
\leq Cnme^{-ck}.
\end{equation*}
\end{lemma}
\begin{proof}
Let us fix~\(n,m\geq 1\) and omit them from the notation. 
We present the argument for \(\Gamma_{\fk}^{+}\). The statement for \(\Gamma_{\fk}^{-}\) is proved in an analogous fashion. 
Consider the coupling \(\Psi^{1/0}\) of \(\omega\sim\fk_{G}^{1/0},\ (\sigma_\bullet,\sigma_\circ)\sim\spin_{\mathcal{D}}^{+,+-}\) and \((\omega_{\tau},\omega_{\tau\tau'})\sim\matrc_K\) described in Section~\ref{sec:coupling}. 
Let \(\calC'=\calC'_{v_L'}\) be the cluster of~\(v'_L:=(-n-\tfrac{1}{2},-\tfrac{1}{2})\) in~\(\omega^*\); see Fig.~\ref{fig:prox_int}.
\begin{figure}
\includegraphics[scale=0.42,page=1]{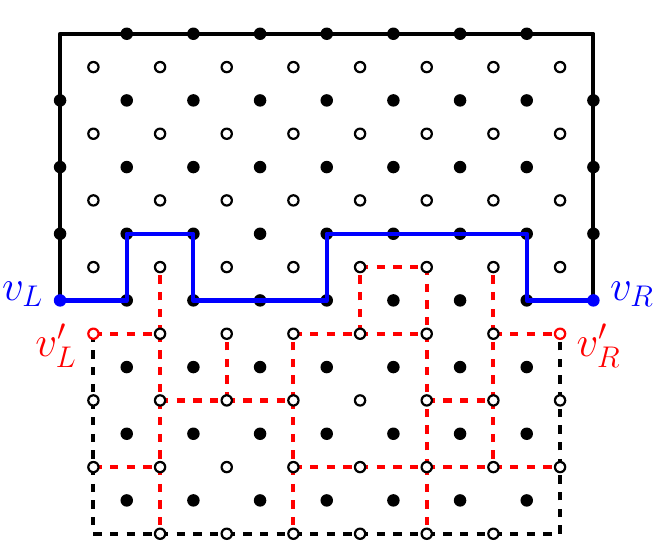}
\hspace{5pt}
\includegraphics[scale=0.42,page=2]{prox_int_fk_atrc}
\hspace{5pt}
\includegraphics[scale=0.42,page=3]{prox_int_fk_atrc}
\caption{For~\(n=m=3\). Left: vertices of~\(G\) (solid) and its dual (hollow) with~``\(1/0\)'' boundary conditions (black solid and dashed edges), the dual cluster~\(\calC'=C'\) (endpoints of dashed edges), and the lowermost path~\(p_{C'}^\shortuparrow\) (blue solid edges). Center: the set~\(\mathcal{D}_{C'}\) (solid and hollow vertices), the tiles in~\(A_{C'}\) that intersect them, with boundary tiles in~\(A_{C'}^\rmb\) in grey. Right: The associated edges in~\(\bar{E}^{C'}\), with those in~\(E_\rmb^{C'}\) contained in grey tiles.}
\label{fig:prox_int}
\end{figure}
Note that
\begin{equation}\label{eq:triang-ineq-c-prime-atrc}
	\rmd_{\mathrm{H}}(\Gamma_{\fk}^{+},\Gamma_{\atrc}) \leq \rmd_{\mathrm{H}}(\Gamma_{\fk}^{+},\calC') + \rmd_{\mathrm{H}}(\calC',\Gamma_{\atrc}) \leq \tfrac{1}{2} + \rmd_{\mathrm{H}}(\calC',\Gamma_{\atrc}).
\end{equation}
Let \(p_{-}\subset \bbV_{*E}\) be the uppermost path connecting \(v'_L\) to~\(v'_R:=(n+\tfrac{1}{2},-\tfrac{1}{2})\) on which \(\sigma_\circ\equiv-1\).
Recall the definition of~\(\calC=\calC_{v_L}\) in Section~\ref{sec:interface_matrc}.
Since \(*\partialedge\calC\) contains such a path, as well as a path connecting \((-n-\tfrac{3}{2},\tfrac{1}{2})\) to \((n+\tfrac{3}{2},\tfrac{1}{2})\) on which \(\sigma_\circ\equiv 1\) (see Fig.~\ref{fig:bc_FK_loops_spins}~and~\ref{fig:interface_matrc}), we have
\begin{equation}\label{eq:triang-ineq-c-prime-p-minus}
	\rmd_{\mathrm{H}}(\calC',\Gamma_{\atrc}) \leq \rmd_{\mathrm{H}}(\calC',p_-) + \rmd_{\mathrm{H}}(p_-,\Gamma_{\atrc}) \leq \rmd_{\mathrm{H}}(\calC',p_-) + \tfrac{1}{2}.
\end{equation}
Fix a realisation \(C'\) of \(\calC'\).
By the coupling, \(\sigma_\circ\equiv -1\) on \(\calC'=C'\) (see Fig.~\ref{fig:bc_FK_loops_spins}). Thus, the path \(p_-\) cannot contain vertices below \(C'\).
Hence, the random variable \(\rmd_{\mathrm{H}}(C',p_-)\) is measurable with respect to \(\sigma_\circ\) restricted to the vertices in \(\mathcal{D}\cap\bbL_\circ\) above \(C'\). 
Let us determine the conditional law of \((\sigma_\bullet,\sigma_\circ)\) restricted to the vertices in \(\mathcal{D}\) above \(C'\).

Let \(p_{C'}^\shortuparrow\) be the lowermost path in \(\omega\) connecting~\(v_L\) to~\(v_R\) above \(C'\) and \(\mathcal{D}_{C'}\) be the subset of vertices in \(\mathcal{D}\) above \(p_{C'}^\shortuparrow\).
Denote by \(A_{C'}\) the set of tiles with at least one corner in \(\mathcal{D}_{C'}\), and by \(A_{C'}^\rmb\) the set of tiles with precisely one corner in \(\mathcal{D}_{C'}\); see Fig.~\ref{fig:prox_int}.
Define the associated sets of edges
\begin{equation*}
\bar{E}^{C'}=\{e_t:t\in A_{C'}\},\quad E_\rmb^{C'}=\{e_t:t\in A_{C'}^\rmb\}\quad\text{and}\quad E^{C'}=\bar{E}^{C'}\setminus E_\rmb^{C'}.
\end{equation*}
Set \(G_{C'}=(\bbV_{E^{C'}},E^{C'})\), \(K_{C'}=(\bbV_{\bar{E}^{C'}},\bar{E}^{C'})\), and let \(K_{C'}^1\) be the graph obtained from \(K_{C'}\) by identifying vertices in \(\partialin_{\bbL_\bullet} \bbV_{\bar{E}^{C'}}\) and those connected by \(p_{C'}^\shortuparrow\).

By Lemma~\ref{lem:6V_spins_to_01FK} and the domain Markov property of FK-percolation,
\begin{equation*}
\Psi^{1/0}(\omega|_{E^{C'}}=\omega_0\given\calC'={C'})=\fk_{G}^{1/0}(\omega|_{E^{C'}}=\omega_0\given\calC'={C'})=\fk_{G_{C'}}^{1}(\omega|_{E^{C'}}=\omega_0)
\end{equation*}
for any \(\omega_0\in\{0,1\}^{E^{C'}}\).
The FK-percolation measure with wired boundary conditions can be coupled to the six-vertex spin measure with~\((+,+)\) boundary conditions.
This version of the BKW coupling was presented in~\cite{GlaPel23,Lis21} and is very closely related  to our Lemma~\ref{lem:6V_spins_to_01FK} that deals with the Dobrushin boundary conditions.
We provide only the statement.

We first define the six-vertex spin measure on~\(\mathcal{D}_{C'}\) under~\((+,+)\) boundary conditions (compare to~\eqref{eq:def_6v-spin_dobrushin}):
for \(\sigma=(\sigma_\bullet,\sigma_\circ)\in\{\pm1\}^{\bbL_\bullet}\times\{\pm1\}^{\bbL_\circ}\), it is defined by
\begin{equation*}
\spin_{\mathcal{D}_{C'}}^{+,+}(\sigma)\propto
\svc^{|T_{5,6}^\rmi(\sigma)|}\,\svcb^{|T_{5,6}^{\rmb}(\sigma)|}\,\mathds{1}_{\sigma=+1\text{ on }(\bbL_\bullet\cup\bbL_\circ)\setminus\mathcal{D}_{C'}}\,\mathds{1}_{\mathrm{ice}}(\sigma),
\end{equation*}
where \(T_{5,6}^\rmi(\sigma)\) and \(T_{5,6}^\rmb(\sigma)\) are respectively the sets of tiles in \(A_{C'}\setminus A_{C'}^\rmb\) and \(A_{C'}^\rmb\) that are of types 5,6 in \(\sigma\).
Then, as in Lemma~\ref{lem:6V_spins_to_01FK},
\begin{equation*}
\Psi^{1/0}\big((\sigma_\bullet|_{\mathcal{D}_{C'}\cap\bbL_\bullet},\sigma_\circ|_{\mathcal{D}_{C'}\cap\bbL_\circ})\in\cdot\bgiven\calC'={C'}\big)=\spin_{\mathcal{D}_{C'}}^{+,+}|_{\mathcal{D}_{C'}}.
\end{equation*}

In the same vein, we determine the conditional law of \((\omega_{\tau},\omega_{\tau\tau'})\) on~\(\bar{E}^{C'}\) by adapting Lemma~\ref{lem:6V_spins_to_AT} to~\(+,+\) boundary conditions.
We now define a modified ATRC measure on~\(K_{C'}\) by taking Definition~\ref{def:mATRC} and replacing~\((E,E_\rmb^+,E_\rmb^-)\) with~\((E^{C'},E_\rmb^{C'},\varnothing)\) (in particular, \(I(\calC)\equiv 0\) for all clusters): for any \((a,b)\in\{0,1\}^{\bar{E}^{C'}}\times \{0,1\}^{\bar{E}^{C'}}\), define
\begin{equation*}
\matrc_{K_{C'}}^\varnothing(a,b)\propto
\mathds{1}_{a\subseteq b}\mathds{1}_{b\setminus a\subseteq E^{C'}}
2^{\clusters_{K_{C'}}(a)+\clusters_{K_{C'}^1}(b)}2^{|a|}(\svc-2)^{|b\setminus a|}(\svcb-1)^{|E_\rmb^{C'}\setminus b|}.
\end{equation*}
Since \(\calC'={C'}\), we have~\(\sigma_\bullet\equiv 1\) at all endpoints of edges in~\(E_\rmb^{C'}\). 
Then, as in Lemma~\ref{lem:6V_spins_to_AT}, 
\begin{equation*}
\Psi^{1/0}\big((\omega_{\tau},\omega_{\tau\tau'})|_{\bar{E}^{C'}}\in\cdot\bgiven\calC'={C'}\big)=\matrc_{K_{C'}}^\varnothing.
\end{equation*}
We claim that this measure satisfies~\eqref{eq:strong-fkg}.
Indeed, all local terms in the right side of the display trivially satisfy the property and the non-local terms assign weight two to the clusters in~\(a\) and~\(b\).
It is standard that such terms satisfy~\eqref{eq:strong-fkg}: see the classical~\cite[Theorem~3.8]{Gri06} for the case of the random-cluster model; alternatively, the proof of Lemma~\ref{lem:fkg_matrc_+} applies by setting~\(\alpha\equiv\beta\equiv\tfrac12\).
Thus,
\begin{equation*}
\matrc_{K_{C'}}^\varnothing\leq_{\mathrm{st}}\matrc_{K_{C'}}^\varnothing(\cdot\given\omega_{\tau}(e)=1\text{ for all }e\in E_\rmb^{C'})=\atrc_{G_{C'}}^{1,1}.
\end{equation*}
As in the proof of Lemma~\ref{lem:InvPrinc:exp_dec}, uniform exponential decay of connection probabilities in \(\atrc_{G_{C'}}^{\atrcwired,\atrcwired}\) (Theorem~\ref{thm:Ale04_ATRC}) and the arguments of~\cite[Section 2]{Ale04} imply that connection probabilities in~\(\matrc_{K_{C'}}^\varnothing\) decay exponentially as well.

Finally, since \(\sigma_\circ\sim\omega_{\tau}^*\) and \(\sigma_\circ=1\) on \(\bbL_\circ\setminus\mathcal{D}_{C'}\), any connected component of \(\{\sigma_\circ=-1\}\) must be surrounded by a circuit in \(\omega_{\tau}\), whence
\begin{equation*}
\Psi^{1/0}(\rmd_{\mathrm{H}}(\calC',p_-)>k\given\calC'={C'})\leq \matrc_{K_{C'}}^\varnothing(\exists\,i\in p_{C'}^\shortuparrow:\,i\xleftrightarrow{\omega_{\tau}}i+\partialin\Lambda_k)\leq Cnmke^{-ck}.
\end{equation*}
Averaging over~\(C'\) and inserting this in~\eqref{eq:triang-ineq-c-prime-atrc} and~\eqref{eq:triang-ineq-c-prime-p-minus} completes the proof.
\end{proof}

\subsection{Proof of Theorem~\ref{thm:order-disorder:FK_loop}}
\label{sec:proof_thm_fk}
We will derive Theorem~\ref{thm:order-disorder:FK_loop} from the analogous statement for the cluster \(\calC=\calC_{v_L}\) of~\(v_L\) in~\(\omega_\tau\) with law the first marginal of~\(\matrc_{K_{n,m}}(\cdot\given v_L\xleftrightarrow{\omega_\tau} v_R)\) (Theorem~\ref{thm:main_Inv_principle_coupling_with_RW} and Lemma~\ref{lem:InvPrinc_mATRC}) and the proximity statement (Lemma~\ref{lem:closeness_interfaces_fk_matrc}). 

\begin{proof}[Proof of Theorem~\ref{thm:order-disorder:FK_loop}]
To lighten the notation, we omit \(p_c(q),q\) from the subscripts.
Recall the definition of the upper and lower envelopes~\(\Gamma_{\fk}^{\pm,n}\) in \(G_n\) defined in Section~\ref{sec:intro}, and of \(\Gamma_{\fk}^{\pm,n,m}\) in \(G=G_{n,m}\) defined analogously.
We will first show the statement for the FK measure \(\fk_{G_{n,Cn}}^{1/0}\) for \(C\) sufficiently large, and then from it derive the statement for~\(\fk_{G_n}^{1/0}\).

By Theorem~\ref{thm:main_Inv_principle_coupling_with_RW} and Lemmata~\ref{lem:InvPrinc_mATRC}~and~\ref{lem:closeness_interfaces_fk_matrc}, there exists \(C>1\) such that the statement holds for~\(\Gamma_{\fk}^{\pm,n,Cn}\) under the measure~\(\fk_{G_{n,Cn}}^{1/0}\).
The derivation for \(\Gamma_{\fk}^{\pm,n}\) under \(\fk_{G_n}^{1/0}\) follows from the FKG and DLR properties of FK percolation, and we only sketch the argument. 
Define the graphs \(G_{n,Cn}^\pm=(V_{n,Cn}^\pm,E_{n,Cn}^\pm)\) by
\begin{gather*}
\Lambda_{n,Cn}^- := \{-n, \dots, n\}\times\{-Cn, \dots, n\} ,\quad \Lambda_{n,Cn}^+ := \{-n, \dots, n\}\times\{-n, \dots, Cn\},\\
V_{n,Cn}^\pm=\Lambda_{n,Cn}^\pm\cup (\partialex\Lambda_{n,Cn}^\pm\cap\bbH^+),\quad E_{n,Cn}^\pm=\bbE_{V_{n,Cn}^\pm}\setminus\bbE_{(\Lambda_{n,Cn}^\pm)^c}.
\end{gather*}
Then, by the FKG and DLR properties of the FK measures, it holds that
\begin{equation*}
\fk_{G_{n,Cn}}^{1/0}\geq_\mathrm{st}\fk_{G_{n,Cn}^+}^{1/0}\leq_\mathrm{st}\fk_{G_{n}}^{1/0}\leq_\mathrm{st}\fk_{G_{n,Cn}^-}^{1/0}\geq_\mathrm{st}\fk_{G_{n,Cn}}^{1/0}.
\end{equation*} 
If \(\omega\) is distributed according to \(\fk_{G_{n,Cn}}^{1/0}\), by the statement for \(\Gamma_{\fk}^{\pm,n,Cn}\) under this measure, there exists a left-right crossing above \([-n,n]\times\{0\}\) in \(\omega\) and one in \(\omega^*\) below it at arbitrary small linear distance from \([-n,n]\times\{0\}\).
A chain of classical monotone coupling arguments finishes the proof.
\end{proof}

\subsection{Proof of Theorem~\ref{thm:potts}}
\label{sec:proof_thm_potts}

We will derive Theorem~\ref{thm:potts} from Theorem~\ref{thm:order-disorder:FK_loop}.
Recall the definition of the one-sided Hausdorff distance \(\rmd_{\mathrm{H}}\), and the notion of a subset of \(\R^2\) being above or below a connected set in \(\bbL_\bullet\) or \(\bbL_\circ\), introduced above Lemma~\ref{lem:closeness_interfaces_fk_matrc}.
\begin{proof}[Proof of Theorem~\ref{thm:potts}]
To simplify the notation, we omit \(n,q,T_c(q),p_c(q)\) from sub and superscripts. Consider the Edwards--Sokal coupling \(\edwardSokal_{\Lambda}^{1/0}\) of \(\sigma\sim\potts_{\Lambda}^{1/\rmf}\) and \(\omega\sim\fk_{G_n}^{1/0}\) described in the introduction; see~\cite[Section 1.4]{Gri06} for details.
Observe that the FK and Potts envelopes deterministically satisfy
\begin{equation*}
\Gamma_{\fk}^{-}(k)\leq\Gamma_{\fk}^{+}(k),\quad \Gamma_{\potts}^{-}(k)\leq\Gamma_{\potts}^{+}(k),\quad\Gamma_{\potts}^{\pm}(k)\leq\Gamma_{\fk}^{\pm}(k)\quad\forall k\in\{-n,\dots,n\}.
\end{equation*}
In particular, \(\Gamma_{\potts}^{+}\) is ``sandwiched'' between \(\Gamma_{\potts}^{-}\) and \(\Gamma_{\fk}^{+}\).
By Theorem~\ref{thm:order-disorder:FK_loop}, it suffices to show the existence of \(c,C>0\) for which, for any \(k\geq 1\),
\begin{equation*}
\edwardSokal_{\Lambda}^{1/0}(\rmd_{\mathrm{H}}(\Gamma_{\potts}^{-},\Gamma_{\fk}^{-})>k)<Cn^2e^{-ck}.
\end{equation*}

Let \(\calC\) be the cluster of the lower boundary in \(\{\sigma\neq 1\}\), that is, the set of~\(i\in\Lambda\) for which there exists a path \((i_0,\dots,i_\ell)\) in \(\Lambda\) with \(i_0=i,\,i_\ell\in\partialin\Lambda\cap\bbH^-\) and \(\sigma(i_k)\neq 1\) for \(0\leq k\leq\ell\).
By definition, each point in the lower envelope \(\Gamma_{\potts}^{-}\) is weakly above \(\calC\). 
Therefore, it suffices to show that \(\Gamma_{\fk}^{-}\) is not far above \(\calC\).
Fix a realisation \(C\) of \(\calC\), and define \(\Lambda_C=\Lambda\setminus (C\cup\partialex C)\). By the coupling and the spatial Markov property of the Potts model, it holds that
\begin{equation*}
\edwardSokal_{\Lambda}^{1/0}(\sigma|_{\Lambda_C}\in\cdot\given\calC=C)=\potts_{\Lambda}^{1/\rmf}(\sigma|_{\Lambda_C}\in\cdot\given\calC=C)=\potts_{\Lambda_C}^{1}.
\end{equation*}
Let \(G_C=(V_C,E_C)\) be defined by \(E_C=\bbE_{\Lambda_C}\cup\partialedge\Lambda_C\) and \(V_C=\bbV_{E_C}\). Then, by the above and by the coupling,
\begin{equation*}
\edwardSokal_{\Lambda}^{1/0}(\omega|_{E_C}\in\cdot\given\calC=C)=\fk_{G_C}^{1}(\omega|_{E_C}\in\cdot).
\end{equation*}
Now, conditional on \(\calC=C\), if \(\rmd_{\mathrm{H}}(C,\Gamma_{\fk}^{-})>k\), then there exists a dual path of length \(k\) in \(\omega^*|_{*E_C}\). Since \(\fk_{G_C}^{1}\) stochastically dominates the infinite-volume measure \(\fk^{1}\), which admits exponential decay of connection probabilites in its dual~\cite[Theorem 1.2]{DumGagHar21}, the proof is complete. 
\end{proof}

\subsection{Proof of Theorems~\ref{thm:oz-at}~and~\ref{thm:oz-atrc}}

Finally, the Ornstein-Zernike asymptotics for the two-point function of the AT and ATRC models (Theorems~\ref{thm:oz-at}~and~\ref{thm:oz-atrc}) are direct consequences of Theorem~\ref{thm:OZ_for_ATRC_infinite_vol}.

\begin{proof}[Proof of Theorems~\ref{thm:oz-at}~and~\ref{thm:oz-atrc}.]
	Theorem~\ref{thm:OZ_for_ATRC_infinite_vol} gives, in particular, that for \(s\in \bbS^1\), one can find \(t\in \partial\calW,\delta\in (0,1)\) such that (sums are over pieces confined in the cones/diamonds obtained using \(\fcone_{t,\delta}\))
	\begin{multline*}
		\Big| \rmC\sum_{\gamma_{L},\gamma_R}p_L(\gamma_L)p_R(\gamma_R)\sum_{k\geq 0}\sum_{\gamma_1,\dots,\gamma_k}  \mathds{1}_{\displace(\bar{\gamma})= ns} \prod_{i=1}^k p(\gamma_k) \\
		- e^{n\nu(s)}\atrc_{J,U}(0\xleftrightarrow{\omega_{\tau}} ns)\Big|\leq C_1 e^{-c_1n},
	\end{multline*}which is a comparison between \(e^{n\nu(s)}\atrc_{J,U}(0\leftrightarrow ns)\) and the Green functions of a directed random walk on \(\Z^2\) (the push-forward of \(p_L,p_R,p\) by \(\displace\)). One can then use the local limit theorem in dimension 2 as in~\cite[Section 8.1]{AouOttVel24} to obtain
	\begin{equation}
		\label{eq:OZ_asymp_ATRC}
		e^{n\nu(s)}\atrc_{J,U}(0\leftrightarrow ns) = \frac{c(s)}{\sqrt{n}}(1+o_n(1)),
	\end{equation}which is the wanted asymptotics for the ATRC model. The claim for the AT model follows from the coupling~\eqref{eq:at_atrc_coupling} between \(\at\) and \(\atrc\), .
\end{proof}

\appendix

\section{Total variation distance}
\label{app:tot_var_dist}

Let \(\Omega\) be countable or finite. Consider the measure space \(\Omega \equiv (\Omega, \calP(\Omega))\). Recall that for two finites measures \(\mu,\nu\) on \(\Omega\), one can define their total variation distance by
\begin{equation*}
	\tvd(\mu,\nu)
	\coloneq
	\sup_{A\subset \Omega} |\mu(A) - \nu(A)|
	=
	\tfrac{1}{2}\sum_{\omega\in \Omega} |\mu(\omega) - \nu(\omega)|
	=
	\sup_{\norm{f}_{\infty} \leq 1} \bigl|\mu[f] -\nu[f]\bigr|
\end{equation*}where the second supremum is over functions \(f:\Omega\to \R\), and we denoted \(\mu[f] = \int f d\mu\). In the case of \(\mu,\nu\) probability measures, \(\tvd(\mu,\nu) = \min_{\pi\in \Pi(\mu,\nu)} \pi(X\neq Y)\) where \(\Pi(\mu,\nu)\) is the set of couplings of \(\mu\) and \(\nu\), and \((X,Y)\sim \pi\).

\begin{lemma}
	\label{lem:tot_var_event_norm}
	Let \(\Omega\) be countable or finite. Let \(\mu,\nu\) be two finites measures on \((\Omega, \calP(\Omega))\). Suppose \(\mu(\Omega) = 1\). Let \(\epsilon\in [0,1)\), \(A\subset \Omega\). Then, one has the following.
	\begin{enumerate}
		\item If \(\tvd(\mu,\nu)\leq \epsilon\), \(\nu(\Omega)\in [1-\epsilon,1+\epsilon]\), and 
		\begin{equation*}
			\tvd(\nu,\bar{\nu}) \leq \tfrac{\epsilon}{2},
			\quad
			\tvd(\mu,\bar{\nu}) \leq \tfrac{3\epsilon}{2}.
		\end{equation*}where \(\bar{\nu}(B) = \frac{\nu(B)}{\nu(\Omega)}\) is the normalized version of \(\nu\).
		\item If \(\mu(A)\geq 1-\epsilon\), then,
		\begin{equation*}
			\tvd(\mu,\mu_A) \leq \tfrac{3\epsilon}{4},
		\end{equation*}where \(\mu_A(B) = \frac{\mu(A\cap B)}{\mu(A)}\) is the conditional measure.
	\end{enumerate}
\end{lemma}
\begin{proof}
	Start with the first point. By the assumption, \(|\nu(\Omega)-1| = |\nu(\Omega)-\mu(\Omega)|\leq \epsilon\). Then, \(\tvd(\nu,\bar{\nu}) = \frac{1}{2}\sum_{\omega}\nu(\omega)|1-\frac{1}{\nu(\Omega)}| = \frac{1}{2}|\nu(\Omega)-1| \leq \frac{\epsilon}{2} \). The second part of the display follows by triangle inequality.
	
	The second point follows form the first using the measure \(\nu(B) = \mu(A\cap B)\), which satisfies \(\tvd(\mu,\nu) = \frac{\nu(\Omega\setminus A)}{2}\leq \frac{\epsilon}{2}\).
\end{proof}

\section{Mixing to ratio mixing}

We prove here a technical Lemma whose use is recurrent in the proof that mixing implies ratio mixing under suitable conditions. It is a simplified version of the argument in~\cite[Section 5]{Ale98}.

\begin{lemma}
	\label{lem:app:mixing_to_ratioMixing}
	Let \(\Omega_i,\, i=1,2\) be finite sets. Let \(\Omega = \Omega_1\times \Omega_2\) and \(\calF_i =\{A\subset \Omega_i\}\), \(\calF = \{A\subset \Omega\}\). Let \(\mu,\nu\) be positive probability measures on \((\Omega,\calF)\). Let
	\begin{equation*}
		\pi_i :\Omega \to \Omega_i,\quad \pi_i((\omega_1,\omega_2)) = \omega_i,\quad
		\mu_i = \mu\circ \pi_i^{-1}.
	\end{equation*}Let \(\epsilon_1,\epsilon_2,\epsilon_3\in [0,1)\). Suppose that all of the following hold:
	\begin{enumerate}
		\item Mixing of \(\mu,\nu\): for every \(\xi,\xi'\in \Omega_1\), \(A\subset \Omega_2\), \(\rho\in \{\mu,\nu\}\)
		\begin{equation*}
			|\rho(\Omega_1\times A\given \{\xi\}\times \Omega_2) - \rho(\Omega_1\times A\given \{\xi'\}\times \Omega_2)|\leq \epsilon_1
		\end{equation*}
		\item Proximity between second marginals: \(\tvd(\mu_2,\nu_2)\leq \epsilon_2\);
		\item Conditional equality: there exists an event \(D\subset \Omega_2\) such that for \(\rho\in \{\mu_2,\nu_2\}\), \(\rho(D) \geq 1-\epsilon_3\), and for any \(y\in D\),
		\begin{equation*}
			\frac{\mu(x,y)}{\mu_2(y)} = \frac{\nu(x,y)}{\nu_2(y)},\ \forall x\in \Omega_1.
		\end{equation*}
	\end{enumerate}Then, if \(\epsilon = \max(\epsilon_1,\sqrt{\epsilon_2},\epsilon_3) \leq 0.1\), then, for any \(x\in \Omega_1\),
	\begin{equation*}
		1-9\epsilon \leq \frac{\mu_1(x)}{\nu_1(x)} \leq \frac{1}{1-9\epsilon}.
	\end{equation*}
\end{lemma}
\begin{proof}
	The proof goes by constructing a suitable coupling of \(\mu,\nu\). Let \(Q\) be a maximal coupling of \(\mu_2\) and \(\nu_2\). Then, sample a random vector \((X,Y) = \big((X_1,X_2),(Y_1,Y_2)\big)\) as follows:
	\begin{enumerate}
		\item sample \((X_2,Y_2)\) using \(Q\);
		\item sample \(X_1\) using \(\mu(\ \given \Omega_1\times \{X_2\})\circ \pi_1^{-1}\);
		\item if \(X_2=Y_2\in D\), set \(Y_1 = X_1\). Else, sample \(Y_1\sim \nu(\ \given \Omega_1\times \{Y_2\})\circ \pi_1^{-1}\) independently of \(X_1\).
	\end{enumerate}Denote \(\Psi\) the law of \((X,Y)\).
	\begin{claim}
		\label{prf:lem:mixing_to_ratioMixing:claim1}
		All the following points hold.
		\begin{enumerate}
			\item \(\Psi\) is a coupling of \(\mu\) and \(\nu\).
			\item \((X_2=Y_2\in D)\implies (X_1=Y_1)\).
			\item \(X_1\) and \(Y_2\) are independent conditionally on \(X_2\), and \(Y_1\) and \(X_2\) are independent conditionally on \(Y_2\).
		\end{enumerate}
	\end{claim}
	\begin{proof}
		The second point is by construction. We check the first point. For \(x=(x_1,x_2)\in \Omega\),
		\begin{align*}
			\Psi(X=x)
			&=
			\sum_{y_2\in \Omega_2}Q(x_2,y_2) \Psi(X_1=x_1\given X_2=x_2, Y_2 = y_2)\\
			&=
			\sum_{y_2\in \Omega_2}Q(x_2,y_2) \frac{\mu(x_1,x_2)}{\mu_2(x_2)}
			=
			\mu(x_1,x_2).
		\end{align*}Also, for \(y=(y_1,y_2)\in \Omega\),
		\begin{align*}
			\Psi(Y=y)
			&=
			\sum_{x_2\in \Omega_2}Q(x_2,y_2) \Psi(Y_1=y_1\given X_2=x_2, Y_2 = y_2)\\
			&=
			\sum_{x_2\in \Omega_2}Q(x_2,y_2) \Big(\mathds{1}_{y_2=x_2\in D}\frac{\mu(y_1,x_2)}{\mu_2(x_2)} + (1-\mathds{1}_{y_2=x_2\in D})\frac{\nu(y_1,y_2)}{\nu_2(y_2)}\Big)\\
			&=
			\sum_{x_2\in \Omega_2}Q(x_2,y_2) \frac{\nu(y_1,y_2)}{\nu_2(y_2)}
			=
			\nu(y_1,y_2),
		\end{align*}as, by hypotheses on \(D\), \(\mathds{1}_{y_2=x_2\in D}\frac{\mu(y_1,x_2)}{\mu_2(x_2)} = \mathds{1}_{y_2=x_2\in D}\frac{\mu(y_1,y_2)}{\mu_2(y_2)} = \mathds{1}_{y_2=x_2\in D}\frac{\nu(y_1,y_2)}{\nu_2(y_2)}\).
		
		Let us finally check the third point. First, for any \(x_1\in \Omega_1, x_2,y_2\in \Omega_2\),
		\begin{equation*}
			\Psi(X_1=x_1\given X_2 = x_2, Y_2=y_2) = \frac{\mu(x_1,x_2)}{\mu_2(x_2)}  = \Psi(X_1=x_1\given X_2 = x_2).
		\end{equation*}Then, for any \(y_1\in \Omega_1, x_2,y_2\in \Omega_2\),
		\begin{multline*}
			\Psi(Y_1=y_1\given X_2 = x_2, Y_2=y_2) = \mathds{1}_{y_2=x_2\in D}\frac{\mu(y_1,x_2)}{\mu_2(x_2)} + (1-\mathds{1}_{y_2=x_2\in D})\frac{\nu(y_1,y_2)}{\nu_2(y_2)} \\
			= \frac{\nu(y_1,y_2)}{\nu_2(y_2)} = \Psi(Y_1=y_1\given Y_2=y_2),
		\end{multline*}by hypotheses on \(D\), as before.
	\end{proof}
	Introduce then \(g,h:\Omega_2\to \R_+\) defined by
	\begin{align*}
		g(\xi) = \Psi(X_2\neq Y_2 \given X_2 = \xi),\\
		h(\xi) = \Psi(X_2\neq Y_2 \given Y_2 = \xi).
	\end{align*}Define the event \(H\) by
	\begin{equation*}
		H = \{X_2=Y_2\}\cap \{X_2\in D\}\cap \{Y_2\in D\}\cap \{g(X_2) \leq \sqrt{\epsilon_2}\} \cap \{h(Y_2)\leq \sqrt{\epsilon_2}\}.
	\end{equation*}
	\begin{claim}
		\label{prf:lem:mixing_to_ratioMixing:claim2}
		One has
		\begin{gather*}
			\max_{x_1\in \Omega_2}\Psi(H^c \given X_1=x_1) \leq 9\epsilon,\\
			\max_{y_1\in \Omega_2}\Psi(H^c \given Y_1=y_1) \leq 9\epsilon.
		\end{gather*}
	\end{claim}
	\begin{proof}
		Let \(a=\sqrt{\epsilon_2}\). First, let \(\{g>a\} = \{x_2\in \Omega_2:\ g(x_2) >a\}\), one has
		\begin{multline*}
			\mu(\Omega_1\times \{g>a\}) = \Psi(g(X_2)> a)
			\leq a^{-1}\Psi( g(X_2)) = a^{-1}\Psi( \Psi(X_2\neq Y_2 \given X_2))\\
			= a^{-1}\Psi(X_2\neq Y_2 ) \leq \frac{\epsilon_2}{a},
		\end{multline*}as \(\Psi(X_2\neq Y_2 ) = Q(X_2\neq Y_2 ) = \tvd(\mu_2,\nu_2) \leq \epsilon_2\), and the same for \(\nu(\Omega_1\times \{h> a\})\). Now, by our first hypotheses,
		\begin{multline*}
			\Psi(g(X_2)> a\given X_1 = x_1) = \mu(\Omega_1\times \{g>a\}\given \{x_1\}\times \Omega_2 ) \\
			\leq \mu(\Omega_1\times \{g>a\}) +\epsilon_1 = \frac{\epsilon_2}{a} + \epsilon_1,
		\end{multline*}and the same for \(\Psi(h(Y_2)> a\given Y_1 = y_1)\). Also,
		\begin{equation*}
			\Psi(X_2\in D^c\given X_1 = x_1) = \mu(\Omega_1\times D^c\given \{x_1\}\times \Omega_2 )
			\leq \mu(\Omega_1\times D^c) +\epsilon_1 = \epsilon_3 + \epsilon_1,
		\end{equation*}and the same for \(\Psi(Y_2\in D^c \given Y_1 = y_1)\).
		
		Then, as \(X_1\) and \(Y_2\) are independent conditionally on \(X_2\) (by Claim~\ref{prf:lem:mixing_to_ratioMixing:claim1}),
		\begin{align*}
			\Psi(&X_2\neq Y_2, g(X_2)\leq a\given X_1 = x_1) \\
			&\quad =
			\sum_{x_2: g(x_2)\leq a} \Psi(X_2=x_2\given X_1 = x_1)\Psi(Y_2\neq x_2\given X_2=x_2)\\
			&\quad =
			\sum_{x_2: g(x_2)\leq a} \Psi(X_2=x_2\given X_1 = x_1)g(x_2) \leq a.
		\end{align*}In the same fashion, \(\Psi(X_2\neq Y_2, h(Y_2)\leq a\given Y_1 = y_1) \leq a\). Finally,
		\begin{align*}
			\Psi(&X_2= Y_2\in D, h(Y_2)> a \given X_1 = x_1) \\
			&\qquad =
			\Psi(X_2= Y_2\in D\given X_1 = x_1) \Psi( h(Y_2)> a \given X_1 = x_1, X_2= Y_2\in D)\\
			&\qquad =
			\Psi(X_2= Y_2\in D\given X_1 = x_1) \Psi( h(Y_2)> a \given Y_1 = x_1, X_2= Y_2\in D)\\
			&\qquad =
			\frac{\Psi(X_2= Y_2\in D\given X_1 = x_1)}{\Psi(X_2= Y_2\in D\given Y_1 = x_1)} \Psi(X_2= Y_2\in D, h(Y_2)> a \given Y_1 = x_1)\\
			&\qquad \leq
			\frac{\epsilon_2/a + \epsilon_1}{1-2\epsilon_1 - \epsilon_3 - \epsilon_2/a - a},
		\end{align*}where the second equality is because \(X_2=Y_2\in D \implies X_1 = Y_1\), and the inequality is the previous bounds, and a union bound on \(\Psi(\{X_2= Y_2\in D\}^c\given Y_1 = x_1)\). Putting things together,
		\begin{align*}
			\Psi(H^c \given X_1 = x_1) &\leq \Psi(X_2\in D^c\given X_1 = x_1) + \Psi(g(X_2)> a\given X_1 = x_1)  \\
			&\qquad + \Psi(X_2\neq Y_2, g(X_2)\leq a\given X_1 = x_1) \\
			&\qquad + \Psi(X_2= Y_2\in D, h(Y_2)> a \given X_1 = x_1)\\
			&\leq \epsilon_3+\epsilon_1 + \frac{\epsilon_2}{a} + \epsilon_1 + a + \frac{\epsilon_2/a + \epsilon_1}{1-2\epsilon_1
				-\epsilon_3 -\epsilon_2/a - a}.
		\end{align*}One obtains the same bound on \(\Psi(H^c \given Y_1 = y_1)\) in the same way (with \(h\leftrightarrow g\) and \(Y_2\leftrightarrow X_2\)). Plugging in \(a=\sqrt{\epsilon_2}\) and using the definition of \(\epsilon\) gives that the last display is upper bounded by
		\begin{equation*}
			5\epsilon + \frac{2\epsilon}{1-5\epsilon} \leq 9\epsilon,
		\end{equation*}as \(\epsilon\leq 0.1\) by hypotheses.
	\end{proof}
	Let us now see how Claims~\ref{prf:lem:mixing_to_ratioMixing:claim1} and~\ref{prf:lem:mixing_to_ratioMixing:claim2} imply the wanted result. As \(\Psi\) is a coupling,
	\begin{equation*}
		\frac{\mu_1(\xi)}{\nu_1(\xi)}
		=
		\frac{\Psi(X_1 = \xi)}{\Psi(Y_1 = \xi)}
		=
		\frac{\Psi(X_1 = \xi)\Psi(Y_1 = \xi,H)}{\Psi(Y_1 = \xi)\Psi(X_1 = \xi,H)}
		=
		\frac{\Psi(H\given Y_1 = \xi)}{\Psi(H\given X_1 = \xi)},
	\end{equation*}where the second equality is because \(\{X_1=Y_1\}\) under \(H\) (as \(\{X_2=Y_2\in D\}\subset H\)). But now,
	\begin{equation*}
		1- 9\epsilon \leq 1-\Psi(H^c\given Y_1 = \xi) \leq \frac{\Psi(H\given Y_1 = \xi)}{\Psi(H\given X_1 = \xi)} \leq \frac{1}{1-\Psi(H^c\given X_1 = \xi)}\leq \frac{1}{1-9\epsilon}.
	\end{equation*}
\end{proof}

\section{A relative density bound for monotonic measures}

We prove here a small Lemma that morally says ``if one controls relative densities with maximal boundary conditions, then one control relative densities with any boundary conditions''. We stat it in a somewhat general form so that it can be re-used in the future.
\begin{lemma}
	\label{lem:ratio_FKG_measures}
	Let \(\Lambda\) be a finite set, and \(q\geq 1\) be an integer. Let \(\mu,\nu\) be measures on \(\Omega_{\Lambda} = (\{0,1\}^q)^{\Lambda}\) with mass functions (\(\omega = (\omega_{i,k})_{i\in \Lambda,k=1,\dots, q}\))
	\begin{equation*}
		\mu(\omega) = w_{\mu}(\omega) f_{\Lambda}(\omega),
		\quad
		\nu(\omega) = w_{\nu}(\omega) f_{\Lambda}(\omega),
	\end{equation*}where
	\begin{equation*}
		f_{\Lambda}(\omega)
		=
		\prod_{i\in \Lambda}f_i\big((\omega_{i,k})_{k=1}^q\big),
		\quad
		f_i\big((\omega_{i,k})_{k=1}^q\big)
		=
		\prod_{k=1}^{q-1}\mathds{1}_{\omega_{i,k}\leq \omega_{i,k+1}},
	\end{equation*}and \(w_{\mu},w_{\nu}\) are positive lattice FKG functions. Let \(\epsilon \in [0,1)\), \(\Delta\subset W\subset \Lambda\). Suppose that\footnote{We denote \(0,1\in \{0,1\}^q\) for the constant \(0,1\) sequence.}: for \(X\sim \mu\), \(Y\sim \nu\),
	\begin{enumerate}
		\item \(P(X_W\in \cdot \given X_{\Lambda\setminus W} = 1) = P(Y_W\in \cdot \given Y_{\Lambda\setminus W} = 1)\),
		\item \(P(X_W\in \cdot \given X_{\Lambda\setminus W} = 0) = P(Y_W\in \cdot \given Y_{\Lambda\setminus W} = 0)\),
		\item for any \(A\subset \Omega_{\Delta}\),
		\(
		1-\epsilon
		\leq
		\frac{P(X_{\Delta}\in A \given X_{\Lambda\setminus W} = 1)}{P(X_{\Delta}\in A \given  X_{\Lambda\setminus W} = 0)}
		\leq
		1+\epsilon
		\),
	\end{enumerate}where \(\frac{0}{0} = 1\).
	Then, for any \(A\subset (\{0,1\}^q)^{\Delta}\),
	\begin{equation*}
		\frac{1-\epsilon}{(1+\epsilon)^2}
		\leq
		\frac{P(X_{\Delta}\in A )}{P(Y_{\Delta}\in A)}
		\leq
		\frac{(1+\epsilon)^2}{1-\epsilon},
	\end{equation*}where again \(\frac{0}{0} = 1\).
\end{lemma}
\begin{proof}
	We refer to~\cite[Section~4]{GeoHagMae01} for more details and terminology about monotonic/lattice FKG measures. Let \(X\sim \mu\), \(Y\sim \nu\). Note that \(f\) is lattice FKG, so \(\mu,\nu\) are lattice FKG. Note moreover that \(\mu,\nu\) are supported on an irreducible sub-lattice of \((\{0,1\}^q)^V\), so for any \(\Lambda\subset V\), and almost every realisations \(\eta\leq \eta'\) of \(X_{V\setminus \Lambda}\),
	\begin{equation*}
		P(X_{\Lambda}\in \cdot \given X_{V\setminus \Lambda} = \eta)
		\preccurlyeq
		P(X_{\Lambda}\in \cdot \given X_{V\setminus \Lambda} = \eta'),
	\end{equation*}and the same for \(Y\). It is sufficient to show the claim for \(A\) being a singleton with positive \(\mu\)-probability (and therefore positive \(\nu\)-probability). Let \(\Delta\supset F_1\supset \dots \supset F_q\), \(E_i=\Delta\setminus F_i\). Then, letting
	\begin{gather*}
		A_{X}^{1} = \cap_{i=1}^q \{X_{E_i}(i)=1\},
		\quad
		A_{X}^{0} = \cap_{i=1}^q \{X_{F_i}(i)=0\},
		\\
		A_{Y}^{1} = \cap_{i=1}^q \{Y_{E_i}(i)=1\},
		\quad
		A_{Y}^{0} = \cap_{i=1}^q \{Y_{F_i}(i)=0\},
	\end{gather*}one has
	\begin{align*}
		&\frac{P(\cap_{i=1}^q \{X_{F_i}(i)=0, X_{E_i}(i)=1\})}{P(\cap_{i=1}^q \{Y_{F_i}(i)=0,Y_{E_i}(i)=1\})}
		=
		\frac{P(A_X^1)P(A_X^0 \given A_X^1)}{P(A_Y^1)P(A_Y^0 \given A_Y^1)}
		\\
		&\quad\leq
		\frac{P(A_X^1\given X_{V\setminus W} =1)P(A_X^0 \given A_X^1,\ X_{V\setminus W} = 0)}{P(A_Y^1\given Y_{V\setminus W} = 0)P(A_Y^0 \given A_Y^1,\ Y_{V\setminus W} = 1)}
		\\
		&\quad=
		\frac{P(A_X^1\given X_{V\setminus W} =1) P(A_X^0,\ A_X^1 \given X_{V\setminus W} = 0)P( A_Y^1 \given Y_{V\setminus W} = 1)}{P(A_X^1 \given X_{V\setminus W} = 0) P(A_Y^1 \given Y_{V\setminus W} = 0) P(A_Y^0,\ A_Y^1 \given Y_{V\setminus W} = 1)}
		\\
		&\quad=
		\frac{P(A_X^1\given X_{V\setminus W} =1)}{P(A_X^1 \given X_{V\setminus W} = 0)} \frac{P(A_X^0,\ A_X^1 \given X_{V\setminus W} = 0)}{P(A_X^0,\ A_X^1 \given X_{V\setminus W} = 1)} \frac{P( A_X^1 \given X_{V\setminus W} = 1)}{ P(A_X^1\given X_{V\setminus W} = 0)}
		\leq
		\frac{(1+\epsilon)^2}{1-\epsilon}
	\end{align*}where we used the one-monotonicity of \(\mu,\nu\) in the second line, definition of conditional probability in the third, the first two hypotheses of the Lemma in the fourth line, and the last hypotheses in the last inequality. The lower bound follows the same path:
	\begin{align*}
		&\frac{P(\cap_{i=1}^q \{X_{F_i}(i)=0, X_{E_i}(i)=1\})}{P(\cap_{i=1}^q \{Y_{F_i}(i)=0,Y_{E_i}(i)=1\})}
		\geq
		\frac{P(A_X^1\given X_{V\setminus W} =0)P(A_X^0 \given A_X^1,\ X_{V\setminus W} = 1)}{P(A_Y^1\given Y_{V\setminus W} = 1)P(A_Y^0 \given A_Y^1,\ Y_{V\setminus W} = 0)}
		\\
		&\quad=
		\frac{P(A_X^1\given X_{V\setminus W} =0)}{P(A_X^1 \given X_{V\setminus W} = 1)} \frac{P(A_X^0,\ A_X^1 \given X_{V\setminus W} = 1)}{P(A_X^0,\ A_X^1 \given X_{V\setminus W} = 0)} \frac{P( A_X^1 \given X_{V\setminus W} = 0)}{ P(A_X^1\given X_{V\setminus W} = 1)}
		\geq
		\frac{1-\epsilon}{(1+\epsilon)^2}.
	\end{align*}
\end{proof}

\bibliographystyle{amsalpha}
\bibliography{biblicomplete}

\end{document}